\documentclass[reqno, 12pt]{amsart}
\pdfoutput=1
\makeatletter
\@namedef{subjclassname@2020}{\textup{2020} Mathematics Subject Classification}
\makeatother

\usepackage[T1]{fontenc}
\usepackage{amsmath}
\usepackage{amssymb}
\usepackage{bm}
\usepackage{mathrsfs}
\usepackage{dsfont}
\usepackage{accents}
\usepackage{stmaryrd}
\allowdisplaybreaks

\usepackage[rgb,table]{xcolor} 
\definecolor{mygreen}{rgb}{0,0.7,0.3}
\definecolor{myblue}{rgb}{0,0.50,1.20}
\definecolor{myorange}{rgb}{1,0.5,0.1}

\definecolor{fillred}{rgb}{1,0.9,0.9}
\definecolor{fillgreen}{rgb}{0.9,1,0.9}

\definecolor{refkey}{rgb}{0,0.7,0.3}
\usepackage{comment}
\definecolor{labelkey}{rgb}{1,0,0}

\usepackage[colorlinks, linkcolor=blue, citecolor=red, urlcolor=cyan, pagebackref]{hyperref}

\usepackage{float}
\usepackage{graphicx,color}
\usepackage{ascmac}
\usepackage{harpoon}
\usepackage{mathtools}

\usepackage{tabularx}

\usepackage{tikz}
\usetikzlibrary{calc,decorations.markings,patterns,angles,quotes}
\usetikzlibrary{arrows.meta}
\usetikzlibrary{positioning}
\usetikzlibrary{cd, intersections, decorations.pathmorphing, arrows, decorations.pathreplacing}
\usepackage{pgfplots}
\pgfplotsset{compat=1.17}
\usetikzlibrary{pgfplots.fillbetween}

\usepackage{cleveref}
\crefname{thm}{Theorem}{Theorems}
\crefname{cor}{Corollary}{Corollaries}
\crefname{lem}{Lemma}{Lemmas}
\crefname{prop}{Proposition}{Propositions}
\crefname{dfn}{Definition}{Definitions}
\crefname{ex}{Example}{Examples}
\crefname{claim}{Claim}{Claims}
\crefname{conj}{Conjecture}{Conjectures}
\crefname{conv}{Notation}{Notations}
\crefname{rem}{Remark}{Remarks}
\crefname{prob}{Problem}{Problems}
\crefname{figure}{Figure}{Figures}
\crefname{table}{Table}{Tables}
\crefname{section}{Section}{Sections}
\crefname{subsection}{Section}{Sections}
\crefname{appendix}{Appendix}{Appendices}
\crefname{lemdef}{Lemma-Definition}{Lemma-Definitions}
\crefname{conv}{Convention}{Conventions}
\crefname{introthm}{Theorem}{Theorems}
\crefname{introcor}{Corollary}{Corollaries}
\crefname{introconj}{Conjecture}{Conjectures}

\usepackage{amsthm}

\newtheorem{thm}{Theorem}
\newtheorem{lem}[thm]{Lemma}
\newtheorem{prop}[thm]{Proposition}
\newtheorem{cor}[thm]{Corollary}
\newtheorem{conj}[thm]{Conjecture}
\newtheorem{introthm}{Theorem}

\newtheorem{introconj}[introthm]{Conjecture}

\theoremstyle{definition}
\newtheorem{dfn}[thm]{Definition}
\newtheorem{rem}[thm]{Remark}
\newtheorem{ex}[thm]{Example}

\numberwithin{figure}{section}
\numberwithin{equation}{section}
\numberwithin{thm}{section}

\newcommand{\bZ}{\mathbb{Z}}

\newcommand{\bC}{\mathbb{C}}

\newcommand{\bM}{\mathbb{M}}
\newcommand{\bN}{\mathbb{N}}

\newcommand{\bE}{\mathbb{E}}

\newcommand{\bs}{{\boldsymbol{s}}}
\newcommand{\sfs}{\mathsf{s}}

\newcommand{\sff}{\mathsf{f}}

\newcommand{\Exch}{\mathrm{Exch}}
\newcommand{\bExch}{\bE \mathrm{xch}}
\newcommand{\uf}{\mathrm{uf}}
\newcommand{\f}{\mathrm{f}}

\newcommand{\A}{\mathcal{A}}

\newcommand{\cC}{\mathcal{C}}
\newcommand{\cD}{\mathcal{D}}

\newcommand{\cF}{\mathcal{F}}

\newcommand{\cO}{\mathcal{O}}

\newcommand{\cR}{\mathcal{R}}

\newcommand{\Skein}[1]{\mathscr{S}_{\mathfrak{sp}_{4},#1}}
\newcommand{\Bweb}[1]{\mathsf{BWeb}_{\mathfrak{sp}_{4},{#1}}}
\newcommand{\Eweb}[1]{\mathsf{EWeb}_{\mathfrak{sp}_{4},{#1}}}
\newcommand{\Cweb}[1]{\mathsf{CWeb}_{\mathfrak{sp}_{4},{#1}}}
\newcommand{\Tang}[1]{\mathsf{Tang}_{\mathfrak{sp}_{4},{#1}}}
\newcommand{\Diag}[1]{\mathsf{Diag}_{\mathfrak{sp}_{4},{#1}}}
\newcommand{\Desc}[1]{\mathsf{Desc}_{\mathfrak{sp}_{4},{#1}}^{\varpi_1}}
\newcommand{\Tree}[1]{\mathsf{Tree}_{\mathfrak{sp}_{4},{#1}}}
\newcommand{\SimpWil}[1]{\mathsf{SimpWil}_{\mathfrak{sp}_{4},{#1}}^{\varpi_1}}
\newcommand{\DescWil}[1]{\mathsf{DescWil}_{\mathfrak{sp}_{4},{#1}}^{\varpi_1}}

\newcommand{\interior}{\mathrm{int}}
\newcommand{\CA}{\mathscr{A}}
\newcommand{\UCA}{\mathscr{U}}
\newcommand{\Conf}{\mathrm{Conf}}
\newcommand{\bD}{{\boldsymbol{\Delta}}}

\newcommand{\CV}{\mathsf{CV}}

\makeatletter
\newcommand{\oset}[3][0ex]{%
  \mathrel{\mathop{#3}\limits^{
    \vbox to#1{\kern-2\ex@
    \hbox{$\scriptstyle#2$}\vss}}}}
\makeatother
\newcommand{\overbar}[1]{\oset{#1}{-\!\!\!-\!\!\!-}}

\newcommand\dnode[2]{\filldraw[draw=#2,fill=#2!15,even odd rule](#1) circle(1.7pt) circle(3pt)}

\newcommand\qarrow[2]{\draw[-latex,shorten >=2pt,shorten <=2pt] (#1) -- (#2) [thick];} 
\newcommand\qsarrow[2]{\draw[-latex,shorten >=3pt,shorten <=3pt] (#1) -- (#2) [thick];} 
\newcommand\qsharrow[2]{\draw[-latex,shorten >=4pt,shorten <=2pt] (#1) -- (#2) [thick];} 
\newcommand\qstarrow[2]{\draw[-latex,shorten >=2pt,shorten <=4pt] (#1) -- (#2) [thick];} 
\newcommand{\uniarrow}[3]{\draw[-latex,thick,#3] (#1) to (#2);}

\newcommand\qshdarrow[2]{\draw[->,dashed,shorten >=4pt,shorten <=2pt] (#1) -- (#2) [thick];} 
\newcommand\qstdarrow[2]{\draw[->,dashed,shorten >=2pt,shorten <=4pt] (#1) -- (#2) [thick];} 

\def\centerarc(#1)(#2:#3:#4)
    {($(#1)+({#4*cos(#2)},{#4*sin(#2)})$) arc(#2:#3:#4) }

\tikzset{
  mid arrow/.style={postaction={decorate,decoration={
        markings,
        mark=at position .5 with {\arrow[#1]{stealth}}
      }}},
}

\newcommand{\markedpt}[1]{
\node[fill,circle,inner sep=1.2pt] at(#1) {};
}

\newcommand{\quiverplusC}[3]{
    \begin{scope}[>=latex]
    {\color{mygreen}
    \path(#1) coordinate(x1);
    \path(#2) coordinate(x2);
    \path(#3) coordinate(x3);
    \foreach \j in {1,2,3}
    {
        \foreach \k in {1,2,3}
        {
            \foreach \l in {1,2}
            {
            \path($(x\j)!0.333*\l!(x\k)$) coordinate(x\j\k\l);
            }
        }
    }
    \dnode{x121}{mygreen};
    \draw(x122) circle(2pt);
    \dnode{x312}{mygreen};
    \draw(x311) circle(2pt);
    {\color{myblue}
        \draw(x231) circle(2pt);
        \dnode{x232}{myblue};
    }
    \dnode{$(x121)!0.5!(x312)$}{mygreen} coordinate(G2);
    \draw($(x122)!0.5!(x311)$) circle(2pt) coordinate(G1);
    \qsarrow{G2}{x121}
    \qsarrow{x312}{G2}
    \qarrow{G1}{x122}
    \qarrow{x311}{G1}
    \qstarrow{x121}{G1}
    \qsharrow{G1}{G2}
    \qstarrow{G2}{x311}
    {\color{myblue}
    \qarrow{x122}{x231}
    \qarrow{x231}{G1}
    \qarrow{G2}{x232}
    \qarrow{x232}{x312}
    }
    }
    \draw($(x2)!0.5!(x3)$) coordinate(h);
    \draw($(x1)!0.1!(h)$) node{$\ast$};
    \end{scope}
}

\newcommand{\quiverminusC}[3]{
    \begin{scope}[>=latex]
    {\color{mygreen}
    \path(#1) coordinate(x1);
    \path(#2) coordinate(x2);
    \path(#3) coordinate(x3);
    \foreach \j in {1,2,3}
    {
        \foreach \k in {1,2,3}
        {
            \foreach \l in {1,2}
            {
            \path($(x\j)!0.333*\l!(x\k)$) coordinate(x\j\k\l);
            }
        }
    }
    \dnode{x121}{mygreen};
    \draw(x122) circle(2pt);
    \dnode{x312}{mygreen};
    \draw(x311) circle(2pt);
    {\color{myblue}
        \draw(x232) circle(2pt);
        \dnode{x231}{myblue};
    }
    \dnode{$(x121)!0.5!(x312)$}{mygreen} coordinate(G2);
    \draw($(x122)!0.5!(x311)$) circle(2pt) coordinate(G1);
    \qsarrow{G2}{x121}
    \qsarrow{x312}{G2}
    \qarrow{G1}{x122}
    \qarrow{x311}{G1}
    \qsharrow{x122}{G2}
    \qstarrow{G2}{G1}
    \qsharrow{G1}{x312}
    {\color{myblue}
    \qarrow{x121}{x231}
    \qarrow{x231}{G2}
    \qarrow{G1}{x232}
    \qarrow{x232}{x311}
    }
    }
    \draw($(x2)!0.5!(x3)$) coordinate(h);
    \draw($(x1)!0.1!(h)$) node{$\ast$};
    \end{scope}
}

\newcommand{\quiverplus}[3]{
    \begin{scope}[>=latex]
    {\color{mygreen}
    \path(#1) coordinate(x1);
    \path(#2) coordinate(x2);
    \path(#3) coordinate(x3);
    \foreach \j in {1,2,3}
    {
        \foreach \k in {1,2,3}
        {
            \foreach \l in {1,2}
            {
            \path($(x\j)!0.333*\l!(x\k)$) coordinate(x\j\k\l);
            }
        }
    }
    \dnode{x121}{mygreen};
    \draw(x122) circle(2pt);
    \dnode{x312}{mygreen};
    \draw(x311) circle(2pt);
    {\color{mygreen}
        \draw(x231) circle(2pt);
        \dnode{x232}{mygreen};
    }
    \dnode{$(x121)!0.5!(x312)$}{mygreen} coordinate(G2);
    \draw($(x122)!0.5!(x311)$) circle(2pt) coordinate(G1);
    \qsarrow{G2}{x121}
    \qsarrow{x312}{G2}
    \qarrow{G1}{x122}
    \qarrow{x311}{G1}
    \qstarrow{x121}{G1}
    \qsharrow{G1}{G2}
    \qstarrow{G2}{x311}
    {\color{mygreen}
    \qarrow{x122}{x231}
    \qarrow{x231}{G1}
    \qarrow{G2}{x232}
    \qarrow{x232}{x312}
    }
    }
    \draw($(x2)!0.5!(x3)$) coordinate(h);
    \draw($(x1)!0.1!(h)$) node{$\ast$};
    \end{scope}
}

\newcommand{\quiverminus}[3]{
    \begin{scope}[>=latex]
    {\color{mygreen}
    \path(#1) coordinate(x1);
    \path(#2) coordinate(x2);
    \path(#3) coordinate(x3);
    \foreach \j in {1,2,3}
    {
        \foreach \k in {1,2,3}
        {
            \foreach \l in {1,2}
            {
            \path($(x\j)!0.333*\l!(x\k)$) coordinate(x\j\k\l);
            }
        }
    }
    \dnode{x121}{mygreen};
    \draw(x122) circle(2pt);
    \dnode{x312}{mygreen};
    \draw(x311) circle(2pt);
    {\color{mygreen}
        \draw(x232) circle(2pt);
        \dnode{x231}{mygreen};
    }
    \dnode{$(x121)!0.5!(x312)$}{mygreen} coordinate(G2);
    \draw($(x122)!0.5!(x311)$) circle(2pt) coordinate(G1);
    \qsarrow{G2}{x121}
    \qsarrow{x312}{G2}
    \qarrow{G1}{x122}
    \qarrow{x311}{G1}
    \qsharrow{x122}{G2}
    \qstarrow{G2}{G1}
    \qsharrow{G1}{x312}
    {\color{mygreen}
    \qarrow{x121}{x231}
    \qarrow{x231}{G2}
    \qarrow{G1}{x232}
    \qarrow{x232}{x311}
    }
    }
    \draw($(x2)!0.5!(x3)$) coordinate(h);
    \draw($(x1)!0.1!(h)$) node{$\ast$};
    \end{scope}
}

\newcommand{\quiversquare}[4]{
\begin{scope}[>=latex]
    {\color{mygreen}
    \path(#1) coordinate(x1);
    \path(#2) coordinate(x2);
    \path(#3) coordinate(x3);
				\path(#4) coordinate(x4);
				\foreach \i in {1,2}
				{
    \path($(x1)!\i/3!(x4)$) coordinate(x14\i);
				\path($(x2)!\i/3!(x3)$) coordinate(x23\i);
				}
				\foreach \j in {0,1,2,3,4}
				{
				\draw($(x141)!\j/4!(x231)$) circle(2pt) coordinate(v1\j);
				\draw($(x142)!\j/4!(x232)$) coordinate(v2\j);
				\dnode{v2\j}{mygreen};
				}
				\draw[myblue]($(x1)!1/4!(x2)$) circle(2pt) coordinate(yl);
				\draw[myblue]($(x1)!3/4!(x2)$) circle(2pt) coordinate(yr);
				\draw($(x4)!1/4!(x3)$) coordinate(zl);
				\dnode{zl}{myblue};
				\draw($(x4)!3/4!(x3)$) coordinate(zr);
				\dnode{zr}{myblue};
				}
\end{scope}
}

\newcommand{\bline}[3]{
    \path (#1)++(0,-#3) coordinate(m1);
    \path (#2)++(0,-#3) coordinate(m2);
    \filldraw[gray!30] (m1) -- (#1) -- (#2) -- (m2) --cycle;
    \draw[thick] (#1) -- (#2);
}

\newcommand{\smallsq}[1]{\draw(0,0) coordinate(A) -- (0,1.2)  coordinate(B) -- (1.2,1.2)  coordinate(C) -- (1.2,0)  coordinate(D) --cycle;
#1;
\foreach \i in {0,1.2} \foreach \j in {0,1.2} \markedpt{\i,\j};}

\tikzset{
    symbol/.style={
        draw=none,
        every to/.append style={
            edge node={node [sloped, allow upside down, auto=false]{$#1$}}}
    }
}

\tikzset{
    ->-/.style 2 args={
        postaction={decorate},
        decoration={markings, mark=at position #1 with {\arrow[thick, #2]{>}}}
    },
    ->-/.default={0.5}{}
}

\tikzset{
    -<-/.style 2 args={
        postaction={decorate},
        decoration={markings, mark=at position #1 with {\arrow[thick, #2]{<}}} 
    },
    -<-/.default={0.5}{}
}

\tikzset{
    overarc/.style={
        white, double=red, double distance=1.6pt, line width=2.4pt
    }
}

\tikzset{
    overarcblack/.style={
        white, double=black, double distance=1.6pt, line width=2.4pt
    }
}

\tikzset{
    wline/.style={
        line width=1.6pt, red!15, preaction={draw, line width=2.8pt, red}
    }
}

\tikzset{
    overwline/.style={
        line width=1.6pt, red!15, preaction={preaction={draw, line width=6.8pt, white},draw, line width=2.8pt, red}
    }
}

\tikzset{
	webline/.style={
		red, ultra thick
	}
}

\tikzset{
    overarcblack/.style={
        white, double=black, double distance=1.6pt, line width=2.4pt
    }
}

\tikzset{
	weblineblack/.style={
		black, ultra thick
	}
}

\tikzset{
    wlineblack/.style={
        line width=1.6pt, black!15, preaction={draw, line width=2.8pt}
    }
}

\tikzset{
	weblinegreen/.style={
		mygreen, ultra thick
	}
}

\tikzset{
    wlinegreen/.style={
        line width=1.6pt, green!15, preaction={draw, line width=2.8pt, mygreen}
    }
}

\newcommand{\bdryline}[3]{
    \coordinate (temp1) at #1;
    \coordinate (temp2) at #2;
    \coordinate (temp11) at ($(temp1)!#3!-90:(temp2)$);
    \coordinate (temp22) at ($(temp2)!#3!90:(temp1)$);
    \filldraw[gray!30] (temp1) -- (temp11) -- (temp22) -- (temp2) --cycle;
    \draw[very thick] (temp1) -- (temp2);
}

\newcommand{\CoG}[3]{
    \path(#1) coordinate(x1);
    \path(#2) coordinate(x2);
    \path(#3) coordinate(x3);
    \path($(x1)!0.5!(x2)$) coordinate(H);
    \path($(x3)!0.667!(H)$) circle(2pt) coordinate(G);}
\newcommand{\triv}[3]{
    \CoG{#1}{#2}{#3}
    \draw[webline] (#1) -- (G);
    \draw[webline] (#2) -- (G);
    \draw[wline] (#3) -- (G);
}

\title[Skein and cluster algebras for $\mathfrak{sp}_4$]{Skein and cluster algebras of unpunctured surfaces for $\mathfrak{sp}_4$}

\author[Tsukasa Ishibashi]{Tsukasa Ishibashi}
\address{Tsukasa Ishibashi, Mathematical Institute, Tohoku University, 
6-3 Aoba, Aramaki, Aoba-ku, Sendai, Miyagi 980-8578, Japan.}
\email{tsukasa.ishibashi.a6@tohoku.ac.jp}
\urladdr{https://sites.google.com/view/tsukasa-ishibashi/home} 

\author[Wataru Yuasa]{Wataru Yuasa}
\address{Wataru Yuasa, Department of Mathematics, Kyoto University,
Kitashirakawa Oiwake-cho, Sakyo-ku, Kyoto 606-8502, Japan}
\email{yuasa.wataru.6m@kyoto-u.ac.jp}
\urladdr{https://wataruyuasa.github.io/math/}


\begin{document}

\begin{abstract}
Continuing to our previous work \cite{IYsl3} on the $\mathfrak{sl}_3$-case, we introduce a skein algebra $\mathscr{S}_{\mathfrak{sp}_4,\Sigma}^{q}$ consisting of $\mathfrak{sp}_4$-webs on a marked surface $\Sigma$ with certain ``clasped'' skein relations at special points, and investigate its cluster nature. 
We also introduce a natural $\mathbb{Z}_q$-form $\mathscr{S}_{\mathfrak{sp}_4,\Sigma}^{\mathbb{Z}_q} \subset \mathscr{S}_{\mathfrak{sp}_4,\Sigma}^q$, while the natural coefficient ring $\mathcal{R}$ of $\mathscr{S}_{\mathfrak{sp}_4,\Sigma}^q$ includes the inverse of the quantum integer $[2]_q$. 
We prove that its boundary-localization  $\mathscr{S}_{\mathfrak{sp}_4,\Sigma}^{\mathbb{Z}_q}[\partial^{-1}]$ is included into a quantum cluster algebra $\mathscr{A}^q_{\mathfrak{sp}_4,\Sigma}$ that quantizes the function ring of the moduli space $\mathcal{A}_{Sp_4,\Sigma}^\times$. 
Moreover, we obtain the positivity of Laurent expressions of elevation-preserving webs in a similar way to \cite{IYsl3}. 
We also propose a characterization of cluster variables in the spirit of Fomin--Pylyavksyy \cite{FP16} in terms of the $\mathfrak{sp}_4$-webs, and give infinitely many supporting examples on a quadrilateral.
\end{abstract}
\maketitle

\setcounter{tocdepth}{1}
\tableofcontents


\section{Introduction}\label{sec:intro}
We continue to investigate the relationship between the cluster and skein quantizations of the (decorated) $G$-character varieties of surfaces. The $\mathfrak{sl}_2$-case is established by Muller \cite{Muller16} and the $\mathfrak{sl}_3$-case is investigated (via the $\mathfrak{sl}_3$-skein algebra defined by \cite{FrohmanSikora20}) in our previous paper \cite{IYsl3}. In the present paper, we are going to deal with the $\mathfrak{sp}_4$-case. 
Kuperberg~\cite{Kuperberg96} has established the skein theory for rank two Lie algebras, including $\mathfrak{sp}_4$. He gave a diagrammatic interpretation of categories of finite-dimensional representations of their quantum groups. The diagrammatic category of $U_q(\mathfrak{sp}_4)$ was further studied by Bodish~\cite{Bod20,Bod22}, and that for $U_q(\mathfrak{sp}_{2n})$ by \cite{BERT}.   
On the other hand, the canonical cluster $K_2$-structure on the moduli space $\A_{G,\Sigma}$ of decorated twisted $G$-local system on a marked surface $\Sigma$ \cite{FG03} has been constructed by Le \cite{Le19} (for classical types, and particular cluster charts) and \cite{GS19} (for all semisimple types, and a more general class of cluster charts). Let $\sfs(\mathfrak{g},\Sigma)$ be the mutation class of seeds that encodes this structure, where $\mathfrak{g}=\mathrm{Lie}(G)$.  

\subsection{Comparison of skein and cluster algebras}
In this paper, we first define a skein algebra $\mathscr{S}_{\mathfrak{sp}_4,\Sigma}^q$ spanned by certain $\mathfrak{sp}_4$-webs on an unpunctured marked surface $\Sigma$, subject to certain ``clasped'' skein relations (\cref{def:bdry-skeinrel}) as well as the interior $\mathfrak{sp}_4$-skein relations studied by Kuperberg. Here the coefficient ring of the skein algebra is $\cR=\cR_q:=\bZ[q^{\pm 1/2},1/[2]_q]$, which includes the inverse of the quantum integer $[2]_q$. 
We also introduce the set $\Bweb{\Sigma} \subset \mathscr{S}_{\mathfrak{sp}_4,\Sigma}^q$ of \emph{basis webs}, which gives an $\cR$-basis (\cref{thm:basis-web}). Then we define the \emph{$\bZ_q$-form} $\mathscr{S}_{\mathfrak{sp}_4,\Sigma}^{\bZ_q} \subset \mathscr{S}_{\mathfrak{sp}_4,\Sigma}^q$ to be the $\bZ_q$-span of $\Bweb{\Sigma}$, which turns out to be closed under the multiplication (\cref{lem:crossroads-in-Zv}). 

On the other hand, we construct quantum clusters associated with decorated triangulations as \emph{web clusters} in the skein algebra $\mathscr{S}_{\mathfrak{sp}_4,\Sigma}^q$, and show that they are mutation-equivalent to each other as a comparison between the quantum exchange and skein relations (\cref{prop:compatibility,thm:q-mutation-equivalence}). Hence the associated quantum seeds give rise to a canonical mutation class $\sfs_q(\mathfrak{sp}_4,\Sigma)$ that quantizes $\sfs(\mathfrak{sp}_4,\Sigma)$, which defines a quantum (upper) cluster algebras $\CA_{\mathfrak{sp}_4,\Sigma}^q \subset \UCA_{\mathfrak{sp}_4,\Sigma}^q$ over $\bZ_q=\bZ[q^{\pm 1/2}]$ in the skew-field $\mathrm{Frac}\mathscr{S}_{\mathfrak{sp}_4,\Sigma}^q$ of fractions.  

The following is our first result, which gives a comparison of skein and cluster algebras:

\begin{introthm}[Comparison of skein and cluster algebras: \cref{subsec:S in A,subsec:S in U}]\label{introthm:comparison}
For any connected (triangulable) unpunctured marked surface $\Sigma$, we have an inclusion  $\mathscr{S}_{\mathfrak{sp}_4,\Sigma}^{\bZ_q}[\partial^{-1}] \subset \UCA_{\mathfrak{sp}_4,\Sigma}^q$. Here $\mathscr{S}_{\mathfrak{sp}_4,\Sigma}^{\bZ_q}[\partial^{-1}]$ denotes the boundary-localization of the $\bZ_q$-form of $\mathfrak{sp}_4$-skein algebra (\cref{def:localized-skein-alg}). 
Moreover if $\Sigma$ has at least two special points, then we have inclusions 
\begin{align*}
    \mathscr{S}_{\mathfrak{sp}_4,\Sigma}^{\bZ_q}[\partial^{-1}] \subset
    \CA_{\mathfrak{sp}_4,\Sigma}^q \subset  \UCA_{\mathfrak{sp}_4,\Sigma}^q.
\end{align*}
Moreover, 
\begin{itemize}
    \item the inclusions are mapping class group equivariant;
    \item the bar-involution on $\UCA_{\mathfrak{sp}_4,\Sigma}^q$ restricts to the mirror-reflections on $\mathscr{S}_{\mathfrak{sp}_4,\Sigma}^{\bZ_q}[\partial^{-1}]$;
    \item the ensemble grading of $\UCA_{\mathfrak{sp}_4,\Sigma}^q$ restricts to the endpoint grading of $\mathscr{S}_{\mathfrak{sp}_4,\Sigma}^{\bZ_q}[\partial^{-1}]$.
\end{itemize}
\end{introthm}
As in our previous study on the $\mathfrak{sl}_3$-case, the inclusion $\mathscr{S}_{\mathfrak{sp}_4,\Sigma}^{\bZ_q}[\partial^{-1}] \subset \UCA_{\mathfrak{sp}_4,\Sigma}^q$ is provided by the \emph{cutting trick} (\cref{lem:cutteing-trick}), which is the way to obtain the Laurent expression of a given web in a web cluster. The stronger inclusion $\mathscr{S}_{\mathfrak{sp}_4,\Sigma}^{\bZ_q}[\partial^{-1}] \subset \CA_{\mathfrak{sp}_4,\Sigma}^q$ is provided by the \emph{sticking trick} (\cref{lem:sticking-trick}), by which we can reduce the generators of the localized skein algebra into those corresponding to cluster variables. Actually, we show that the boundary-localized $\bZ_q$-form can be generated by the set $\SimpWil{\Sigma}$ (\cref{thm:generator-Zv-form}), which consists of the quantum counterpart of the matrix coefficients of \emph{simple Wilson lines} \cite{IOS} in the vector representation. Then it is easy to see that they are actually cluster variables up to boundary webs (cf.~\cite[Proposition 4.12]{IOS}). 

\begin{introconj}\label{introconj:coincidence}
For any  connected (triangulable) unpunctured marked surface $\Sigma$, we have the equalities
\begin{align*}
    \mathscr{S}_{\mathfrak{sp}_4,\Sigma}^{\bZ_q}[\partial^{-1}] =
    \CA_{\mathfrak{sp}_4,\Sigma}^q = \UCA_{\mathfrak{sp}_4,\Sigma}^q.
\end{align*}
\end{introconj}
When $\Sigma=T$ is a triangle, this conjecture is true (\cref{S=A_triangle}). 
The classical counterpart $q=1$ of this conjecture is proved by Ishibashi--Oya--Shen \cite{IOS}. See \cref{prop:Tree=A=U}.
In particular, $\Bweb{\Sigma}$ provides a basis of the classical cluster algebra $\CA_{\mathfrak{sp}_4,\Sigma}$ over $\bC$. 

We also obtain the following positivity result. 

\begin{introthm}[Quantum Laurent positivity of webs: \cref{thm:positivity_cluster}]\label{introthm:positivity}
Any elevation-preserving web with respect to an ideal triangulation $\Delta$ is expressed as a Laurent polynomial with coefficients in $\cR_+=\bZ_+[q^{\pm 1/2},1/[2]_q]$ in the quantum cluster associated with any decorated triangulation $\bD=(\Delta,m_\Delta,\bs_\Delta)$ over $\Delta$. In particular, the following webs are quantum GS-universally positive Laurent polynomials (\cref{def:GS-univ})
\begin{itemize}
     \item over $\bZ_q$: descending loops/arcs of type $1$ with or without legs (see \cref{def:diagram});
    \item over $\cR$: the geometric bracelets or the bangles
    (\cref{fig:bracelet}) of type $2$ along simple loops. 
\end{itemize}
\end{introthm}
Here the webs in the first item form a subset $\Desc{\Sigma}$ (see \cref{def:web-set}) that contains geometric bracelets/bangles of type~$1$ and give a generating set of $\mathscr{S}_{\mathfrak{sp}_4,\Sigma}^{\bZ_q}[\partial^{-1}]$ (\cref{thm:generator-Zv-form}). 

\begin{figure}
    \begin{tikzpicture}[scale=.1]
        \begin{scope}
            \draw[dashed] (0,0) circle [radius=5];
            \draw[dashed] (0,0) circle [radius=15];
            \draw[very thick] (0,0) circle [radius=10];
            \node at (0,-15) [below=5pt] {a simple loop $\gamma$};
        \end{scope}
        \begin{scope}[xshift=-45cm]
            \draw[dashed] (0,0) circle [radius=4];
            \draw[dashed] (0,0) circle [radius=16];
            \draw[very thick] (15:14) arc (15:345:14);
            \draw[very thick] (15:12) arc (15:345:12);
            \draw[very thick, dotted] (15:10) arc (15:30:10);
            \draw[very thick, dotted] (-15:10) arc (-15:-30:10);
            \draw[very thick] (15:8) arc (15:345:8);
            \draw[very thick] (15:6) arc (15:345:6);
            \draw[overarcblack, very thick] (-15:8) to[out=north, in=south] (15:6);
            \draw[overarcblack, very thick] (-15:10) to[out=north, in=south] (15:8);
            \draw[overarcblack, very thick] (-15:12) to[out=north, in=south] (15:10);
            \draw[overarcblack, very thick] (-15:14) to[out=north, in=south] (15:12);
            \draw[overarcblack, very thick] (-15:6) to[out=north, in=south] (15:14);
            \node at (180:10) [xscale=2.5, rotate=90] {$\}$};
            \node at (180:10) [above]{\scriptsize $n$};
            \node at (0,10) [scale=.3]{$\bullet$};
            \node at (0,10) [scale=.3, above=1pt]{$\bullet$};
            \node at (0,10) [scale=.3, below=1pt]{$\bullet$};
            \node at (0,-10) [scale=.3]{$\bullet$};
            \node at (0,-10) [scale=.3, above=1pt]{$\bullet$};
            \node at (0,-10) [scale=.3, below=1pt]{$\bullet$};
            \node at (0,-15) [below=5pt] {the geometric $n$-bracelet of $\gamma$};
        \end{scope}
        \begin{scope}[xshift=45cm]
            \draw[dashed] (0,0) circle [radius=4];
            \draw[dashed] (0,0) circle [radius=16];
            \draw[very thick] (0,0) circle [radius=14];
            \draw[very thick] (0,0) circle [radius=12];
            \draw[very thick] (0,0) circle [radius=8];
            \draw[very thick] (0,0) circle [radius=6];
            \node at (180:10) [xscale=2.5, rotate=90] {$\}$};
            \node at (180:10) [above]{\scriptsize $n$};
            \node at (0,10) [scale=.3]{$\bullet$};
            \node at (0,10) [scale=.3, above=1pt]{$\bullet$};
            \node at (0,10) [scale=.3, below=1pt]{$\bullet$};
            \node at (0,-10) [scale=.3]{$\bullet$};
            \node at (0,-10) [scale=.3, above=1pt]{$\bullet$};
            \node at (0,-10) [scale=.3, below=1pt]{$\bullet$};
            \node at (0,-15) [below=5pt] {the $n$-bangle of $\gamma$};
        \end{scope}
    \end{tikzpicture}
    \caption{The middle shows a tubular neighborhood of a simple loop $\gamma$. The \emph{geometric $n$-bracelet} (resp. \emph{$n$-bangle}) along $\gamma$ is obtained by replacing it with the graph shown in the left (right) colored by a fundamental representation of $\mathfrak{sp}_4$.}
    \label{fig:bracelet}
\end{figure}


\subsection{Characterization of cluster variables} 
Among its possible benefits, the interplay between the skein and cluster theory is expected to provide a topological model for cluster variables. In the simplest case $\mathfrak{g}=\mathfrak{sl}_2$, the cluster variables are actually in a one-to-one correspondence with the simple arcs in the $\mathfrak{sl}_2$-skein algebra. In the other cases, the mutation class $\sfs(\mathfrak{g},\Sigma)$ is typically of infinite mutation type. Namely, it contains infinitely many distinct exchange matrices (\emph{i.e.}, it includes ``infinitely complicated'' combinatorics), so it is a very hard problem to uniformly understand all the cluster variables. In the $\mathfrak{sl}_3$-case, Fomin--Pylyavskyy \cite{FP16} proposed a series of insightful conjectures on the characterization of cluster variables in terms of $\mathfrak{sl}_3$-webs. 

We are going to discuss a conjectural characterization of the cluster variables in our $\mathfrak{sp}_4$-case, in the same spirit as Fomin--Pylyavskyy. Let
\begin{itemize}
    \item $\Eweb{\Sigma} \subset \Bweb{\Sigma}$ be the set of \emph{elementary webs}; 
    \item $\Tree{\Sigma} \subset \Eweb{\Sigma}$ be the set of \emph{tree-type} elementary webs;
    \item $\CV_{\mathfrak{sp}_4,\Sigma} \subset \CA^q_{\mathfrak{sp}_4,\Sigma}$ be the set of cluster variables. 
\end{itemize}
See \cref{def:elementary-web} for a detail. 
Let $\mathrm{DT}$ denote the mapping class on the marked surface $(\Sigma,\bM)$ given by the composite
\begin{align}\label{eq:cluster_DT}
    \mathrm{DT}:=\prod_{h} \mathrm{rot}_{h},
\end{align}
where $h$ runs over the boundary components of $\Sigma$, and $\mathrm{rot}_h$ shifts by one the special points on $h$ in the positive direction with respect to the orientation of $\Sigma$. 
Then $\mathrm{DT}$ acts as algebra automorphisms on $\Skein{\Sigma}^q$ and $\CA^q_{\mathfrak{sp}_4,\Sigma}$. 
Following \cite{GS16}, we call $\mathrm{DT}$ the \emph{cluster Donaldson--Thomas transformation}. 
Then our proposal is the following:

\begin{introconj}\label{conj:tree-variable}
Regarding $\Skein{\Sigma}^{\bZ_q}$ as a subalgebra of $\CA^q_{\mathfrak{sp}_4,\Sigma}$ through the inclusion in \cref{introthm:comparison}, we expect:
\begin{align*}
    \mathsf{Tree}_{\mathfrak{sp}_4,\Sigma}=\CV_{\mathfrak{sp}_4,\Sigma}=\Eweb{\Sigma}\setminus(\Eweb{\Sigma})^{\mathrm{DT}}.
\end{align*}
Here $(\Eweb{\Sigma})^{\mathrm{DT}}$ is the set of elementary webs invariant under the cluster Donaldson--Thomas transformation.
\end{introconj}
Again, this conjecture is true when $\Sigma=T$ is a triangle. 
We remark that none of cluster variables in $\CA^q_{\mathfrak{sp}_4,\Sigma}$ are invariant under $\mathrm{DT}$, since it maps each $g$-vector to its negative. Moreover, it is clear from the construction that the elementary webs in the web cluster associated with any decorated triangulation $\bD$ are of tree type and not invariant under $\mathrm{DT}$. See \cref{ex:elementary_webs} for more examples. 
In \cref{sect:quad}, we will give examples of elementary webs on a quadrilateral that support this conjecture, including an infinite sequence. Here we summarize the relations among several sets of $\mathfrak{sp}_4$-webs that we have introduced:

\begin{equation*}
\begin{tikzcd}
    \Bweb{\Sigma}\arrow[symbol=\supset]{r}{}
    &[-1em]\Eweb{\Sigma}\arrow[symbol=\supset]{r}{}\arrow[symbol=\supset, shift right]{rd}{}
    &[-5em]\Tree{\Sigma}\arrow[symbol=\supset]{r}{}\arrow[equal, red]{d}[above right]{\text{Conj.~4}}
    &[-5em]\SimpWil{\Sigma}\\
    &&\Eweb{\Sigma}\setminus(\Eweb{\Sigma})^{\mathrm{DT}}\arrow[symbol=\supset, shift right]{ru}{}&
\end{tikzcd}
\end{equation*}
By \cref{thm:generator-Zv-form}, for each set $S$ in this diagram, the associated $\bZ_q$-subalgebra $\langle S \rangle_{\bZ_q}[\partial^{-1}]\subset \mathscr{S}^q_{\mathfrak{sp}_4,\Sigma}[\partial^{-1}]$ coincides with each other. 

\subsection*{Organization of the paper}
Below in this section, our notation on marked surfaces and their triangulations is summarized. In \cref{sect:skein}, we define the skein algebra $\mathscr{S}^v_{\mathfrak{sp}_4,\Sigma}$ and investigate its basic structures such as its $\cR$-basis, the triangle case $\Sigma=T$ and the Ore property. We continue to study its generators and Laurent positivity in \cref{sec:generator_positivity}, based on the cutting and sticking tricks. 

In \cref{sec:cluster}, we recall the general framework of the quantum cluster algebra and the construction of the classical mutation class $\sfs(\mathfrak{sp}_4,\Sigma)$. In \cref{sec:realization}, we realize the quantum mutation class $\sfs_q(\mathfrak{sp}_4,\Sigma)$ inside the skew-field of fractions of $\mathscr{S}^q_{\mathfrak{sp}_4,\Sigma}$. Utilizing the results in \cref{sect:skein,sec:generator_positivity}, we prove \cref{introthm:comparison,introthm:positivity}. 

In \cref{sect:quad}, we give a gallery of examples of web clusters on a quadrilateral, including an infinite sequence.

\subsection*{Acknowledgements}
We are grateful to the referees for carefully reading our paper.
T. I. is supported by JSPS KAKENHI Grant Number~JP20K22304, JP24K16914.
W. Y. is supported by JSPS KAKENHI Grant Numbers~JP19J00252, JP19K14528, and JP23K12972.

\subsection*{Notation on marked surfaces and their triangulations}\label{subsec:notation_marked_surface}

A \emph{marked surface} $(\Sigma,\bM)$ is a compact oriented surface $\Sigma$ with boundary equipped with a fixed non-empty finite set $\mathbb{M} \subset \partial\Sigma$ of \emph{special points}. In particular, we do not allow interior marked points (\lq\lq punctures'') in this paper. We also call $(\Sigma,\bM)$ an \emph{unpunctured marked surface} when we emphasize the absence of punctures. 
When the choice of $\bM$ is clear from the context, we simply denote a marked surface by $\Sigma$. Moreover, assume the following conditions:
\begin{enumerate}
\item Each boundary component has at least one special point.
\item $n(\Sigma):=-3\chi(\Sigma)+2|\mathbb{M}|>0$, and $\Sigma$ is not a disk with two special points (a biangle).
\end{enumerate}
These conditions ensure that the marked surface $\Sigma$ has an ideal triangulation, that is, the isotopy class of a collection $\Delta$ of simple arcs connecting special points whose interiors are mutually disjoint, which decomposes $\Sigma$ into triangles. 
The number $n(\Sigma)$ gives the number of edges of any ideal triangulation $\Delta$. 
We call a connected component of the punctured boundary $\partial^\ast \Sigma:=\partial\Sigma\setminus \mathbb{M}$ a \emph{boundary interval}.
Each boundary interval belongs to any ideal triangulation $\Delta$. We call an edge of $\Delta$ an \emph{interior edge} if it is not a boundary interval. Denote the set of edges (resp. interior edges, triangles) of $\Delta$ by $e(\Delta)$ (resp. $e_\interior(\Delta)$, $t(\Delta)$). 

More generally, we can consider an \emph{ideal cell decomposition} of $\Sigma$, which is a decomposition of $\Sigma$ into a union of polygons. When it is obtained from an ideal triangulation by removing $k$ interior edges, it is said to be \emph{of deficiency $k$}. 
In this paper, we only use an ideal cell decomposition of deficiency $0$ or $1$. The ideal cell decomposition of deficiency $1$ obtained from an ideal triangulation $\Delta$ by removing one interior edge $E$ is denoted by $(\Delta;E)$.


Finally, a \emph{decorated triangulation} is a triple $\bD=(\Delta,m_\Delta,\bs_\Delta)$, where
\begin{itemize}
    \item $\Delta$ is an ideal triangulation of $\Sigma$;
    \item $m_\Delta:t(\Delta) \to \bM$ is a choice of a vertex of each triangle;
    \item $\bs_\Delta:t(\Delta) \to \{+,-\}$ is a choice of a sign at each triangle.
\end{itemize}
The data $m_\Delta$ is indicated by the symbol $\ast$ in the figures. A typical seed in the mutation class $\sfs(\mathfrak{sp}_4,\Sigma)$ will be associated with such a decorated triangulation \cite{GS19}. 

\begin{figure}[ht]
\begin{tikzpicture}[scale=0.7]
\draw[blue] (2,2) -- (-2,2) -- (-2,-2) -- (2,-2) --cycle;
\draw[blue] (2,2) -- (-2,-2);
\filldraw[gray!30] (-2.5,-2) -- (2.5,-2) -- (2.5,-2.3) -- (-2.5,-2.3) --cycle;
\draw[thick] (-2.5,-2) -- (2.5,-2);
\foreach \x in {45,135,225,315}
\path(\x:2.8284) node [fill, circle, inner sep=1.2pt]{};
\node at (-1.8,1.8) {$\ast$};
\node at (1.8,1.6) {$\ast$};
\end{tikzpicture}
    \caption{A local picture of a datum $(\Delta,m_\Delta)$. By convention, a portion of $\partial \Sigma$ is drawn by a thick line together with a gray region indicating the ``outer side'' of $\Sigma$.}
    \label{fig:sl3_triangulation}
\end{figure}
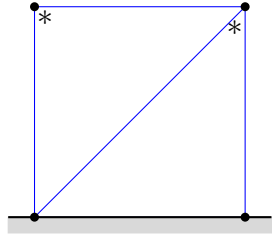

\section{The \texorpdfstring{$\mathfrak{sp}_4$}{sp4}-skein algebra \texorpdfstring{$\Skein{\Sigma}$}{S(sp4,S)} of an unpunctured marked surface}\label{sect:skein}
%

In this section, we define the $\mathfrak{sp}_4$-skein algebra $\Skein{\Sigma}=\Skein{\Sigma}^q$ for a surface $\Sigma$ with special points $\bM\subset\partial\Sigma$ by introducing the skein relation at a special point.
We also give some basic definitions and properties related to $\Skein{\Sigma}$.
In what follows, we use a ring $\cR=\cR_{q}\coloneqq\bZ[q^{\pm 1/2},1/[2]_q]$ of coefficients with a formal variable $q^{1/2}$ where a quantum integer is defined by $[n]=[n]_q\coloneqq (q^{n}-q^{-n})/(q-q^{-1})$.

\subsection{The skein relations for \texorpdfstring{$\mathfrak{sp}_4$}{sp4}}
Let us introduce the $\mathfrak{sp}_4$-webs and its skein relation introduced by Kuperberg~\cite{Kuperberg96}.
We firstly define tangled graphs for $\mathfrak{sp}_4$ on an unpunctured marked surface.

\begin{dfn}[tangled $\mathfrak{sp}_4$-graphs]\label{def:tangled-graph}
\begin{enumerate}
    \item 
    A \emph{tangled $\mathfrak{sp}_4$-graph diagram} on $\Sigma$ is a map (or its image) from a uni-trivalent graph to $\Sigma$ whose edges are colored by fundamental representations $\varpi_1$ (called \emph{type~$1$ edges}) or $\varpi_2$ (called \emph{type~$2$ edges}) of $\mathfrak{sp}_4$. Moreover:
    \begin{itemize}
        \item We assume that each triple of edges incident to a common trivalent vertex consists of two type~$1$ edges and a type~$2$ edge, and the only self-intersections in the interior of $\Sigma$ are transverse double points of edges (called \emph{internal crossings}). The image of the graph only intersects with $\partial\Sigma$ at special points. 
        \item The transverse double points of edges in the interior have over/under-passing information.
        For each $p\in\bM$, the set of half-edges incident to $p$ has a total order called the \emph{elevation} at $p$.
    \end{itemize}
    In pictures, an edge of type~$1$ (resp.~type~$2$) is described by a single (resp.~double) line. 
    Each point of the image of a tangled $\mathfrak{sp}_4$-graph diagram has a neighborhood described by
    $\mathord{
        \ 
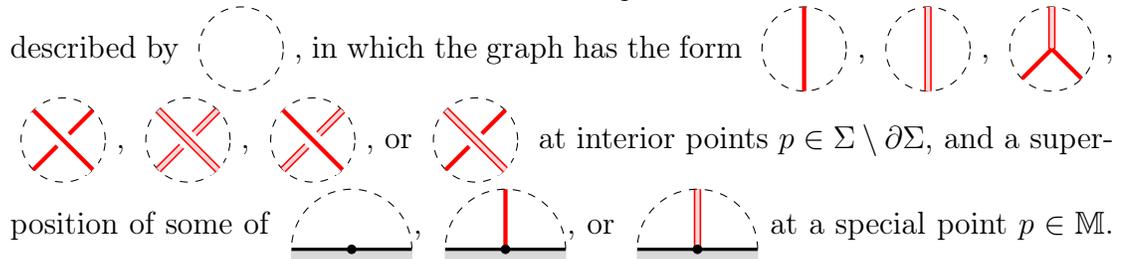
\tikz[baseline=-.6ex, scale=.08]{
            \draw[dashed, fill=white] (0,0) circle [radius=7];
        }
    \ }$, in which the graph has the form
    $\mathord{
        \ \tikz[baseline=-.6ex, scale=.08]{
            \draw[dashed, fill=white] (0,0) circle [radius=7];
            \draw[webline] (-90:7) -- (90:7);
        }
    \ }$,
    $\mathord{
        \ \tikz[baseline=-.6ex, scale=.08]{
            \draw[dashed, fill=white] (0,0) circle [radius=7];
            \draw[wline] (-90:7) -- (90:7);
        }
    \ }$,
    $\mathord{
        \ \tikz[baseline=-.6ex, scale=.08]{
            \draw[dashed, fill=white] (0,0) circle [radius=7];
            \draw[wline] (0:0) -- (90:7);
            \draw[webline] (0:0) -- (225:7);
            \draw[webline] (0:0) -- (-45:7);
        }
    \ }$,
    $\mathord{
        \ \tikz[baseline=-.6ex, scale=.08]{
            \draw[dashed, fill=white] (0,0) circle [radius=7];
            \draw[webline] (225:7) -- (45:7);
            \draw[overarc] (-45:7) -- (135:7);
        }
    \ }$,
    $\mathord{
        \ \tikz[baseline=-.6ex, scale=.08]{
            \draw[dashed, fill=white] (0,0) circle [radius=7];
            \draw[wline] (225:7) -- (45:7);
            \draw[overwline] (-45:7) -- (135:7);
        }
    \ }$,
    $\mathord{
        \ \tikz[baseline=-.6ex, scale=.08]{
            \draw[dashed, fill=white] (0,0) circle [radius=7];
            \draw[wline] (225:7) -- (45:7);
            \draw[overarc] (-45:7) -- (135:7);
        }
    \ }$, or
    $\mathord{
        \ \tikz[baseline=-.6ex, scale=.08]{
            \draw[dashed, fill=white] (0,0) circle [radius=7];
            \draw[webline] (225:7) -- (45:7);
            \draw[overwline] (-45:7) -- (135:7);
        }
    \ }$ at interior points $p\in\Sigma\setminus\partial\Sigma$, and a  superposition of some of
    $\mathord{
        \ \tikz[baseline=-.6ex, scale=.08, yshift={-4cm}]{
            \coordinate (P) at (0,0);
            \draw[dashed] (10,0) arc (0:180:10cm);
            \bdryline{(-10,0)}{(10,0)}{2cm}
            \draw[fill=black] (P) circle [radius=20pt];
        \ }
    }$,
    $\mathord{
        \ \tikz[baseline=-.6ex, scale=.08, yshift={-4cm}]{
            \coordinate (P) at (0,0);
            \draw[webline] (P) -- (90:10);
            \draw[dashed] (10,0) arc (0:180:10cm);
            \bdryline{(-10,0)}{(10,0)}{2cm}
            \draw[fill=black] (P) circle [radius=20pt];
        \ }
    }$, or
    $\mathord{
        \ \tikz[baseline=-.6ex, scale=.08, yshift={-4cm}]{
            \coordinate (P) at (0,0);
            \draw[wline] (P) -- (90:10);
            \draw[dashed] (10,0) arc (0:180:10cm);
            \bdryline{(-10,0)}{(10,0)}{2cm}
            \draw[fill=black] (P) circle [radius=20pt];
        \ }
    }$ at a special point $p\in\bM$. 
    Here an edge with a gap at the intersection point is under-passing. 
    The elevation is indicated by the distance from $p$ (see \cref{fig:elevation}).
    \item Two tangled $\mathfrak{sp}_4$-graphs are \emph{equivalent} if these graphs are related by a finite sequence of the Reidemeister moves (R1$'$), (R2), (R3), (R4), (bR) (see diagrams in \cref{lem:Reidemeister}) and isotopies of $\Sigma$ relative to $\partial\Sigma$.
    We call the equivalence class of a tangled $\mathfrak{sp}_4$-graph diagram a \emph{tangled $\mathfrak{sp}_4$-graph} on $\Sigma$.
    Let $\Tang{\Sigma}$ be the set of tangled $\mathfrak{sp}_4$-graphs on $\Sigma$. It has a multiplication defined by superposition of their diagrams: 
    the product $G_1G_2$ of $G_1,G_2\in\Tang{\Sigma}$ is defined by a superposing $G_1$ on $G_2$ so that $G_1$ passes over $G_2$ at all intersection points and the half-edges of $G_1$ has higher elevation than those of $G_2$ at any $p\in\bM$.
\end{enumerate}
\end{dfn}
\begin{figure}
    \begin{tikzpicture}[scale=.2]
        \coordinate (P) at (0,0);
        \draw[webline, shorten <= .3cm] (P) -- (60:10);
        \draw[wline, shorten <= .6cm] (P) -- (120:10);
        \draw[webline, shorten <= .9cm] (P) -- (90:10);
        \draw[webline, shorten <= 1.2cm] (P) -- (45:10);
        \draw[wline, shorten <= 1.5cm] (P) -- (135:10);
        \draw[wline] (P) -- (30:10);
        \draw[dashed] (10,0) arc (0:180:10cm);
        \draw[dashed] (7.5,0) arc (0:180:7.5cm);
        \draw[dashed] (6,0) arc (0:180:6cm);
        \draw[dashed] (4.5,0) arc (0:180:4.5cm);
        \draw[dashed] (3,0) arc (0:180:3cm);
        \draw[dashed] (1.5,0) arc (0:180:1.5cm);
        \bdryline{(-10,0)}{(10,0)}{2cm}
        \draw[fill=black] (P) circle [radius=20pt];
        \node at (30:11) {\scriptsize $6$};
        \node at (45:11) {\scriptsize $2$};
        \node at (60:11) {\scriptsize $5$};
        \node at (90:11) {\scriptsize $3$};
        \node at (120:11) {\scriptsize $4$};
        \node at (135:11) {\scriptsize $1$};
    \end{tikzpicture}
    \caption{Elevation at a special point. The labeling by natural numbers represents the total order.}
    \label{fig:elevation}
\end{figure}
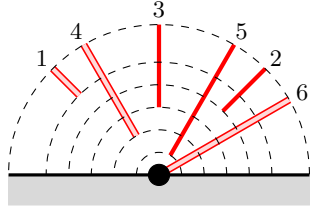

We define the following skein relations for formal linear combinations of tangled $\mathfrak{sp}_4$-graphs.
\begin{dfn}[the internal $\mathfrak{sp}_4$-skein relations~\cite{Kuperberg94,Kuperberg96}]\label{def:skeinrel}
    \begin{align}
    \mathord{
        \ \tikz[baseline=-.6ex, scale=.08]{
            \draw[dashed, fill=white] (0,0) circle [radius=7];
            \draw[webline] (0,0) circle [radius=3];
        }
    \ }
    &=-\frac{[2][6]}{[3]}
    \mathord{
        \ \tikz[baseline=-.6ex, scale=.08]{
            \draw[dashed, fill=white] (0,0) circle [radius=7];
        }
    \ }\label{rel:circle}\\
    \mathord{
        \ \tikz[baseline=-.6ex, scale=.08]{
            \draw[dashed, fill=white] (0,0) circle [radius=7];
            \draw[wline] (0,0) circle [radius=3];
        }
    \ }
    &=\frac{[5][6]}{[2][3]}
    \mathord{
        \ \tikz[baseline=-.6ex, scale=.08]{
            \draw[dashed, fill=white] (0,0) circle [radius=7];
        }
    \ }\label{rel:wcircle}\\
    \mathord{
        \ \tikz[baseline=-.6ex, scale=.08]{
            \draw[dashed, fill=white] (0,0) circle [radius=7];
            \draw[wline] (-90:3) -- (-90:7);
            \draw[webline] (0,0) circle [radius=3];
        }
    \ }
    &=0\label{rel:monogon}\\
    \mathord{
        \ \tikz[baseline=-.6ex, scale=.08]{
            \draw[dashed, fill=white] (0,0) circle [radius=7];
            \draw[wline] (0,-7) -- (0,-3);
            \draw[wline] (0,3) -- (0,7);
            \draw[webline] (0,-3) to[out=east, in=south] (3,0) to[out=north, in=east] (0,3);
            \draw[webline] (0,-3) to[out=west, in=south] (-3,0) to[out=north, in=west] (0,3);
        }
    \ }
    &=-[2]\mathord{
        \ \tikz[baseline=-.6ex, scale=.08]{
            \draw[dashed, fill=white] (0,0) circle [radius=7];
            \draw[wline] (0,-7) -- (0,7);
        }
    \ },\label{rel:bigon}\\
    \mathord{
        \ \tikz[baseline=-.6ex, scale=.08]{
            \draw[dashed, fill=white] (0,0) circle [radius=7];
            \draw[webline] (90:4) -- (210:4) -- (330:4) -- cycle;
            \draw[wline] (90:4) -- (90:7);
            \draw[wline] (210:4) -- (210:7);
            \draw[wline] (330:4) -- (330:7);
        }
    \ }
    &=0,\label{rel:trigon}\\
    \mathord{
        \ \tikz[baseline=-.6ex, scale=.08, rotate=90]{
            \draw[dashed, fill=white] (0,0) circle [radius=7];
            \draw[webline] (-45:7) to[out=north west, in=south] (3,0) to[out=north, in=south west] (45:7);
            \draw[webline] (-135:7) to[out=north east, in=south] (-3,0) to[out=north, in=south east] (135:7);
        }
    \ }
    -[2]
    \mathord{
        \ \tikz[baseline=-.6ex, scale=.08]{
            \draw[dashed, fill=white] (0,0) circle [radius=7];
            \draw[webline] (45:7) -- (90:3);
            \draw[webline] (135:7) -- (90:3);
            \draw[webline] (225:7) -- (-90:3);
            \draw[webline] (315:7) -- (-90:3);
            \draw[wline] (90:3) -- (-90:3);
            }
    \ }
    &=
    \mathord{
        \ \tikz[baseline=-.6ex, scale=.08]{
            \draw[dashed, fill=white] (0,0) circle [radius=7];
            \draw[webline] (-45:7) to[out=north west, in=south] (3,0) to[out=north, in=south west] (45:7);
            \draw[webline] (-135:7) to[out=north east, in=south] (-3,0) to[out=north, in=south east] (135:7);
        }
    \ }
    -[2]
    \mathord{
        \ \tikz[baseline=-.6ex, scale=.08, rotate=90]{
            \draw[dashed, fill=white] (0,0) circle [radius=7];
            \draw[webline] (45:7) -- (90:3);
            \draw[webline] (135:7) -- (90:3);
            \draw[webline] (225:7) -- (-90:3);
            \draw[webline] (315:7) -- (-90:3);
            \draw[wline] (90:3) -- (-90:3);
            }
    \ },\label{rel:IH}
\end{align}
\begin{align}
    \mathord{
        \ \tikz[baseline=-.6ex, scale=.08]{
            \draw[dashed, fill=white] (0,0) circle [radius=7];
            \draw[red, very thick] (-45:7) -- (135:7);
            \draw[overarc] (-135:7) -- (45:7);
            }
    \ }
    &=\frac{q^2}{[2]}
    \mathord{
        \ \tikz[baseline=-.6ex, scale=.08]{
            \draw[dashed, fill=white] (0,0) circle [radius=7];
            \draw[webline] (-45:7) to[out=north west, in=south] (3,0) to[out=north, in=south west] (45:7);
            \draw[webline] (-135:7) to[out=north east, in=south] (-3,0) to[out=north, in=south east] (135:7);
        }
    \ }
    +q^{-1}
    \mathord{
        \ \tikz[baseline=-.6ex, scale=.08, rotate=90]{
            \draw[dashed, fill=white] (0,0) circle [radius=7];
            \draw[webline] (-45:7) to[out=north west, in=south] (3,0) to[out=north, in=south west] (45:7);
            \draw[webline] (-135:7) to[out=north east, in=south] (-3,0) to[out=north, in=south east] (135:7);
        }
    \ }
    +
    \mathord{
        \ \tikz[baseline=-.6ex, scale=.08, rotate=90]{
            \draw[dashed, fill=white] (0,0) circle [radius=7];
            \draw[webline] (45:7) -- (90:3);
            \draw[webline] (135:7) -- (90:3);
            \draw[webline] (225:7) -- (-90:3);
            \draw[webline] (315:7) -- (-90:3);
            \draw[wline] (90:3) -- (-90:3);
            }
    \ },\notag\\
    &=q
    \mathord{
        \ \tikz[baseline=-.6ex, scale=.08]{
            \draw[dashed, fill=white] (0,0) circle [radius=7];
            \draw[webline] (-45:7) to[out=north west, in=south] (3,0) to[out=north, in=south west] (45:7);
            \draw[webline] (-135:7) to[out=north east, in=south] (-3,0) to[out=north, in=south east] (135:7);
        }
    \ }
    +\frac{q^{-2}}{[2]}
    \mathord{
        \ \tikz[baseline=-.6ex, scale=.08, rotate=90]{
            \draw[dashed, fill=white] (0,0) circle [radius=7];
            \draw[webline] (-45:7) to[out=north west, in=south] (3,0) to[out=north, in=south west] (45:7);
            \draw[webline] (-135:7) to[out=north east, in=south] (-3,0) to[out=north, in=south east] (135:7);
        }
    \ }
    +
    \mathord{
        \ \tikz[baseline=-.6ex, scale=.08]{
            \draw[dashed, fill=white] (0,0) circle [radius=7];
            \draw[webline] (45:7) -- (90:3);
            \draw[webline] (135:7) -- (90:3);
            \draw[webline] (225:7) -- (-90:3);
            \draw[webline] (315:7) -- (-90:3);
            \draw[wline] (90:3) -- (-90:3);
            }
    \ },\label{rel:ss-cross}\\
    \mathord{
        \ \tikz[baseline=-.6ex, scale=.08]{
            \draw[dashed, fill=white] (0,0) circle [radius=7];
            \draw[red, very thick] (-45:7) -- (135:7);
            \draw[overwline] (-135:7) -- (45:7);
            }
    \ }
    &=q\mathord{
        \ \tikz[baseline=-.6ex, scale=.08]{
            \draw[dashed, fill=white] (0,0) circle [radius=7];
            \draw[webline] (-45:7) -- (3,0);
            \draw[wline] (-135:7) -- (-3,0);
            \draw[wline] (45:7) -- (3,0);
            \draw[webline] (135:7) -- (-3,0);
            \draw[webline] (-3,0) -- (3,0);
        }
    \ }
    +q^{-1}\mathord{
        \ \tikz[baseline=-.6ex, scale=.08]{
            \draw[dashed, fill=white] (0,0) circle [radius=7];
            \draw[webline] (-45:7) -- (0,-3);
            \draw[wline] (-135:7) -- (0,-3);
            \draw[wline] (45:7) -- (0,3);
            \draw[webline] (135:7) -- (0,3);
            \draw[webline] (0,-3) -- (0,3);
        }
    \ },\label{rel:ws-cross}\\
    \mathord{
        \ \tikz[baseline=-.6ex, scale=.08]{
            \draw[dashed, fill=white] (0,0) circle [radius=7];
            \draw[wline] (-45:7) -- (135:7);
            \draw[overarc] (-135:7) -- (45:7);
            }
    \ }
    &=q\mathord{
        \ \tikz[baseline=-.6ex, scale=.08]{
            \draw[dashed, fill=white] (0,0) circle [radius=7];
            \draw[wline] (-45:7) -- (3,0);
            \draw[webline] (-135:7) -- (-3,0);
            \draw[webline] (45:7) -- (3,0);
            \draw[wline] (135:7) -- (-3,0);
            \draw[webline] (-3,0) -- (3,0);
        }
    \ }
    +q^{-1}\mathord{
        \ \tikz[baseline=-.6ex, scale=.08]{
            \draw[dashed, fill=white] (0,0) circle [radius=7];
            \draw[wline] (-45:7) -- (0,-3);
            \draw[webline] (-135:7) -- (0,-3);
            \draw[webline] (45:7) -- (0,3);
            \draw[wline] (135:7) -- (0,3);
            \draw[webline] (0,-3) -- (0,3);
        }
    \ },\label{rel:sw-cross}\\
    \mathord{
        \ \tikz[baseline=-.6ex, scale=.08]{
            \draw[dashed, fill=white] (0,0) circle [radius=7];
            \draw[wline] (-45:7) -- (135:7);
            \draw[overwline] (-135:7) -- (45:7);
            }
    \ }
    &=q^{2}\mathord{
        \ \tikz[baseline=-.6ex, scale=.08]{
            \draw[dashed, fill=white] (0,0) circle [radius=7];
            \draw[wline] (-45:7) to[out=north west, in=south] (3,0) to[out=north, in=south west] (45:7);
            \draw[wline] (-135:7) to[out=north east, in=south] (-3,0) to[out=north, in=south east] (135:7);
        }
    \ }
    +q^{-2}\mathord{
        \ \tikz[baseline=-.6ex, scale=.08, rotate=90]{
            \draw[dashed, fill=white] (0,0) circle [radius=7];
            \draw[wline] (-45:7) to[out=north west, in=south] (3,0) to[out=north, in=south west] (45:7);
            \draw[wline] (-135:7) to[out=north east, in=south] (-3,0) to[out=north, in=south east] (135:7);
        }
    \ }
    +\mathord{
        \ \tikz[baseline=-.6ex, scale=.08]{
            \draw[dashed, fill=white] (0,0) circle [radius=7];
            \draw[wline] (-45:7) -- (135:7);
            \draw[wline] (-135:7) -- (45:7);
            \draw[very thick, red, fill=white] ([shift=(-135:3)]0,0) rectangle (45:3);
        }
    \ }.\label{rel:ww-cross}
\end{align}
These relations are obtained from Kuperberg's skein relations in \cite{Kuperberg96} by replacing $-q^{\frac{1}{2}}$ with $q$ and rescaling trivanelt vertices as 
\begin{align*}
    \mathord{
        \ \tikz[baseline=-.6ex, scale=.08]{
            \draw[dashed, fill=white] (0,0) circle [radius=7];
            \draw[wline] (0:0) -- (90:7);
            \draw[webline] (0:0) -- (225:7);
            \draw[webline] (0:0) -- (-45:7);
        }
    \ }
    \leadsto
    \frac{1}{\sqrt{[2]}}
    \mathord{
        \ \tikz[baseline=-.6ex, scale=.08]{
            \draw[dashed, fill=white] (0,0) circle [radius=7];
            \draw[wline] (0:0) -- (90:7);
            \draw[webline] (0:0) -- (225:7);
            \draw[webline] (0:0) -- (-45:7);
        }
    \ }.
\end{align*}
\end{dfn}

Let us further introduce the ``clasped'' skein relations among half-edges incident to a special point.

\begin{dfn}[the clasped $\mathfrak{sp}_4$-skein relations]\label{def:bdry-skeinrel}
    \begin{align}
    &\begin{aligned}
        \mathord{
            \ \tikz[baseline=-.6ex, scale=.1, yshift=-4cm]{
                \coordinate (P) at (0,0);
                \draw[webline, shorten <=.2cm] ($(P)$) -- (135:10);
                \draw[webline] (P) -- (45:10);
                \draw[dashed] (10,0) arc (0:180:10cm);
                \bdryline{(-10,0)}{(10,0)}{2cm}
                \draw[fill=black] (P) circle [radius=20pt];
            \ }
        }
        &=q
        \mathord{
            \ \tikz[baseline=-.6ex, scale=.1, yshift=-4cm]{
                \coordinate (P) at (0,0);
                \draw[webline, shorten <=.2cm] (P) -- (45:10);
                \draw[webline] (P) -- (135:10);
                \draw[dashed] (10,0) arc (0:180:10cm);
                \bdryline{(-10,0)}{(10,0)}{2cm}
                \draw[fill=black] (P) circle [radius=20pt];
            \ }
        },&
        \mathord{
            \ \tikz[baseline=-.6ex, scale=.1, yshift=-4cm]{
                \coordinate (P) at (0,0);
                \draw[wline, shorten <=.2cm] (P) -- (135:10);
                \draw[wline] (P) -- (45:10);
                \draw[dashed] (10,0) arc (0:180:10cm);
                \bdryline{(-10,0)}{(10,0)}{2cm}
                \draw[fill=black] (P) circle [radius=20pt];
            \ }
        }
        &=q^2
        \mathord{
            \ \tikz[baseline=-.6ex, scale=.1, yshift=-4cm]{
                \coordinate (P) at (0,0);
                \draw[wline, shorten <=.2cm] (P) -- (45:10);
                \draw[wline] (P) -- (135:10);
                \draw[dashed] (10,0) arc (0:180:10cm);
                \bdryline{(-10,0)}{(10,0)}{2cm}
                \draw[fill=black] (P) circle [radius=20pt];
            \ }
        },\\
        \mathord{
            \ \tikz[baseline=-.6ex, scale=.1, yshift=-4cm]{
                \coordinate (P) at (0,0);
                \draw[wline, shorten <=.2cm] ($(P)$) -- (135:10);
                \draw[webline] (P) -- (45:10);
                \draw[dashed] (10,0) arc (0:180:10cm);
                \bdryline{(-10,0)}{(10,0)}{2cm}
                \draw[fill=black] (P) circle [radius=20pt];
            \ }
        }
        &=q
        \mathord{
            \ \tikz[baseline=-.6ex, scale=.1, yshift=-4cm]{
                \coordinate (P) at (0,0);
                \draw[webline, shorten <=.2cm] (P) -- (45:10);
                \draw[wline] (P) -- (135:10);
                \draw[dashed] (10,0) arc (0:180:10cm);
                \bdryline{(-10,0)}{(10,0)}{2cm}
                \draw[fill=black] (P) circle [radius=20pt];
            \ }
        },&
        \mathord{
            \ \tikz[baseline=-.6ex, scale=.1, yshift=-4cm]{
                \coordinate (P) at (0,0);
                \draw[webline, shorten <=.2cm] ($(P)$) -- (135:10);
                \draw[wline] (P) -- (45:10);
                \draw[dashed] (10,0) arc (0:180:10cm);
                \bdryline{(-10,0)}{(10,0)}{2cm}
                \draw[fill=black] (P) circle [radius=20pt];
            \ }
        }
        &=q
        \mathord{
            \ \tikz[baseline=-.6ex, scale=.1, yshift=-4cm]{
                \coordinate (P) at (0,0);
                \draw[wline, shorten <=.2cm] (P) -- (45:10);
                \draw[webline] (P) -- (135:10);
                \draw[dashed] (10,0) arc (0:180:10cm);
                \bdryline{(-10,0)}{(10,0)}{2cm}
                \draw[fill=black] (P) circle [radius=20pt];
            \ }
        },
    \end{aligned}\label{rel:bdry-cross}\\
    &\begin{aligned}
        \mathord{
            \ \tikz[baseline=-.6ex, scale=.1, yshift=-4cm]{
                \coordinate (P) at (0,0);
                \coordinate (R) at (45:7) {};
                \coordinate (L) at (135:7) {};
                \draw[webline] (L) -- (R);
                \draw[wline] (P) -- (L);
                \draw[wline] (R) -- (45:10);
                \draw[webline] (L) -- (135:10);
                \draw[webline] (P) -- (R);
                \draw[dashed] (10,0) arc (0:180:10cm);
                \bdryline{(-10,0)}{(10,0)}{2cm}
                \draw[fill=black] (P) circle [radius=20pt];
            \ }
        }
        &=
        \mathord{
            \ \tikz[baseline=-.6ex, scale=.1, yshift=-4cm]{
                \coordinate (P) at (0,0) {};
                \draw[webline] (P) -- (135:10);
                \draw[wline] (P) -- (45:10);
                \draw[dashed] (10,0) arc (0:180:10cm);
                \bdryline{(-10,0)}{(10,0)}{2cm}
                \draw[fill=black] (P) circle [radius=20pt];
            \ }
        },&
        \mathord{
            \ \tikz[baseline=-.6ex, scale=.1, yshift=-4cm]{
                \coordinate (P) at (0,0);
                \coordinate (R) at (45:7) {};
                \coordinate (L) at (135:7) {};
                \draw[webline] (L) -- (R);
                \draw[webline] (P) -- (L);
                \draw[webline] (R) -- (45:10);
                \draw[wline] (L) -- (135:10);
                \draw[wline] (P) -- (R);
                \draw[dashed] (10,0) arc (0:180:10cm);
                \bdryline{(-10,0)}{(10,0)}{2cm}
                \draw[fill=black] (P) circle [radius=20pt];
            \ }
        }
        &=
        \mathord{
            \ \tikz[baseline=-.6ex, scale=.1, yshift=-4cm]{
                \coordinate (P) at (0,0) {};
                \draw[wline] (P) -- (135:10);
                \draw[webline] (P) -- (45:10);
                \draw[dashed] (10,0) arc (0:180:10cm);
                \bdryline{(-10,0)}{(10,0)}{2cm}
                \draw[fill=black] (P) circle [radius=20pt];
            \ }
        },\\
        \mathord{
            \ \tikz[baseline=-.6ex, scale=.1, yshift=-4cm]{
                \coordinate (P) at (0,0);
                \coordinate (R) at (45:7) {};
                \coordinate (L) at (135:7) {};
                \draw[wline] (L) -- (R);
                \draw[webline] (P) -- (L);
                \draw[webline] (R) -- (45:10);
                \draw[webline] (L) -- (135:10);
                \draw[webline] (P) -- (R);
                \draw[dashed] (10,0) arc (0:180:10cm);
                \bdryline{(-10,0)}{(10,0)}{2cm}
                \draw[fill=black] (P) circle [radius=20pt];
            \ }
        }
        &=\frac{1}{[2]}
        \mathord{
            \ \tikz[baseline=-.6ex, scale=.1, yshift=-4cm]{
                \coordinate (P) at (0,0) {};
                \draw[webline] (P) -- (135:10);
                \draw[webline] (P) -- (45:10);
                \draw[dashed] (10,0) arc (0:180:10cm);
                \bdryline{(-10,0)}{(10,0)}{2cm}
                \draw[fill=black] (P) circle [radius=20pt];
            \ }
        },&
        \mathord{
            \ \tikz[baseline=-.6ex, scale=.1, yshift=-4cm]{
                \coordinate (P) at (0,0);
                \coordinate (R) at (45:7) {};
                \coordinate (L) at (135:7) {};
                \draw[webline] (L) -- (R);
                \draw[wline] (P) -- (L);
                \draw[webline] (R) -- (45:10);
                \draw[webline] (L) -- (135:10);
                \draw[wline] (P) -- (R);
                \draw[dashed] (10,0) arc (0:180:10cm);
                \bdryline{(-10,0)}{(10,0)}{2cm}
                \draw[fill=black] (P) circle [radius=20pt];
            \ }
        }
        &=0,      
    \end{aligned}\label{rel:bdry-tri}\\
    &\begin{aligned}
        \mathord{
            \ \tikz[baseline=-.6ex, scale=.1, yshift=-4cm]{
                \coordinate (P) at (0,0);
                \coordinate (C) at (90:7) {};
                \draw[webline] (P) to[out=north west, in=west] (C);
                \draw[webline] (P) to[out=north east, in=east] (C);
                \draw[wline] (C) -- (90:10);
                \draw[dashed] (10,0) arc (0:180:10cm);
                \bdryline{(-10,0)}{(10,0)}{2cm}
                \draw[fill=black] (P) circle [radius=20pt];
            \ }
        }
        &=0,&
        \mathord{
            \ \tikz[baseline=-.6ex, scale=.1, yshift=-4cm]{
                \coordinate (P) at (0,0);
                \coordinate (C) at (90:7) {};
                \draw[wline] (P) to[out=north west, in=west] (C);
                \draw[webline] (P) to[out=north east, in=east] (C);
                \draw[webline] (C) -- (90:10);
                \draw[dashed] (10,0) arc (0:180:10cm);
                \bdryline{(-10,0)}{(10,0)}{2cm}
                \draw[fill=black] (P) circle [radius=20pt];
            \ }
        }
        &=0,&
        \mathord{
            \ \tikz[baseline=-.6ex, scale=.1, yshift=-4cm]{
                \coordinate (P) at (0,0);
                \coordinate (C) at (90:7) {};
                \draw[webline] (P) to[out=north west, in=west] (C);
                \draw[wline] (P) to[out=north east, in=east] (C);
                \draw[webline] (C) -- (90:10);
                \draw[dashed] (10,0) arc (0:180:10cm);
                \bdryline{(-10,0)}{(10,0)}{2cm}
                \draw[fill=black] (P) circle [radius=20pt];
            \ }
        }
        &=0,            
    \end{aligned}\label{rel:bdry-bigon}\\
    &\begin{aligned}
        \mathord{
            \ \tikz[baseline=-.6ex, scale=.1, yshift=-4cm]{
                \coordinate (P) at (0,0);
                \coordinate (C) at (90:7);
                \draw[webline] (P) to[out=north west, in=west] (C);
                \draw[webline] (P) to[out=north east, in=east] (C);
                \draw[dashed] (10,0) arc (0:180:10cm);
                \bdryline{(-10,0)}{(10,0)}{2cm}
                \draw[fill=black] (P) circle [radius=20pt];
            \ }
        }
        &=0,&
        \mathord{
            \ \tikz[baseline=-.6ex, scale=.1, yshift=-4cm]{
                \coordinate (P) at (0,0) {};
                \coordinate (C) at (90:7);
                \draw[wline] (P) to[out=north west, in=west] (C);
                \draw[wline] (P) to[out=north east, in=east] (C);
                \draw[dashed] (10,0) arc (0:180:10cm);
                \bdryline{(-10,0)}{(10,0)}{2cm}
                \draw[fill=black] (P) circle [radius=20pt];
            \ }
        }
        &=0.        
    \end{aligned}\label{rel:bdry-monogon}
\end{align}
\end{dfn}

In the above definition, we have drawn half-edges with the same elevation (called \emph{simultaneous crossings}), which is defined as follows.

\begin{dfn}[simultaneous crossings]\label{def:simul-crossing}
    \begin{align}
        \mathord{
            \ \tikz[baseline=-.6ex, scale=.1, yshift=-4cm]{
                \coordinate (P) at (0,0);
                \draw[webline] (P) -- (45:10);
                \draw[webline] (P) -- (135:10);
                \draw[dashed] (10,0) arc (0:180:10cm);
                \bdryline{(-10,0)}{(10,0)}{2cm}
                \draw[fill=black] (P) circle [radius=20pt];
            \ }
        }
        &\coloneqq q^{-\frac{1}{2}}
        \mathord{
            \ \tikz[baseline=-.6ex, scale=.1, yshift=-4cm]{
                \coordinate (P) at (0,0);
                \draw[webline, shorten <=.2cm] ($(P)$) -- (135:10);
                \draw[webline] (P) -- (45:10);
                \draw[dashed] (10,0) arc (0:180:10cm);
                \bdryline{(-10,0)}{(10,0)}{2cm}
                \draw[fill=black] (P) circle [radius=20pt];
            \ }
        },\label{ss-simul-cross}\\
        \mathord{
            \ \tikz[baseline=-.6ex, scale=.1, yshift=-4cm]{
                \coordinate (P) at (0,0);
                \draw[webline] (P) -- (45:10);
                \draw[wline] (P) -- (135:10);
                \draw[dashed] (10,0) arc (0:180:10cm);
                \bdryline{(-10,0)}{(10,0)}{2cm}
                \draw[fill=black] (P) circle [radius=20pt];
            \ }
        }
        &\coloneqq q^{-\frac{1}{2}}
        \mathord{
            \ \tikz[baseline=-.6ex, scale=.1, yshift=-4cm]{
                \coordinate (P) at (0,0);
                \draw[wline, shorten <=.2cm] ($(P)$) -- (135:10);
                \draw[webline] (P) -- (45:10);
                \draw[dashed] (10,0) arc (0:180:10cm);
                \bdryline{(-10,0)}{(10,0)}{2cm}
                \draw[fill=black] (P) circle [radius=20pt];
            \ }
        },\quad
        \mathord{
            \ \tikz[baseline=-.6ex, scale=.1, yshift=-4cm]{
                \coordinate (P) at (0,0);
                \draw[wline] (P) -- (45:10);
                \draw[webline] (P) -- (135:10);
                \draw[dashed] (10,0) arc (0:180:10cm);
                \bdryline{(-10,0)}{(10,0)}{2cm}
                \draw[fill=black] (P) circle [radius=20pt];
            \ }
        }
        \coloneqq q^{-\frac{1}{2}}
        \mathord{
            \ \tikz[baseline=-.6ex, scale=.1, yshift=-4cm]{
                \coordinate (P) at (0,0);
                \draw[webline, shorten <=.2cm] ($(P)$) -- (135:10);
                \draw[wline] (P) -- (45:10);
                \draw[dashed] (10,0) arc (0:180:10cm);
                \bdryline{(-10,0)}{(10,0)}{2cm}
                \draw[fill=black] (P) circle [radius=20pt];
            \ }
        },\label{rel:sw-simul-cross}\\
        \mathord{
            \ \tikz[baseline=-.6ex, scale=.1, yshift=-4cm]{
                \coordinate (P) at (0,0);
                \draw[wline] (P) -- (45:10);
                \draw[wline] (P) -- (135:10);
                \draw[dashed] (10,0) arc (0:180:10cm);
                \bdryline{(-10,0)}{(10,0)}{2cm}
                \draw[fill=black] (P) circle [radius=20pt];
            \ }
        }
        &\coloneqq q^{-1}
        \mathord{
            \ \tikz[baseline=-.6ex, scale=.1, yshift=-4cm]{
                \coordinate (P) at (0,0);
                \draw[wline, shorten <=.2cm] ($(P)$) -- (135:10);
                \draw[wline] (P) -- (45:10);
                \draw[dashed] (10,0) arc (0:180:10cm);
                \bdryline{(-10,0)}{(10,0)}{2cm}
                \draw[fill=black] (P) circle [radius=20pt];
            \ }
        },\label{rel:ww-simul-cross}
    \end{align}
\end{dfn}

For a local diagram of a tangled $\mathfrak{sp}_4$-graph, one can show the following by using the (clasped) $\mathfrak{sp}_4$-skein relations.
\begin{lem}\label{lem:Reidemeister}
    The following Reidemeister moves are obtained by $\mathfrak{sp}_4$-skein relations.
    \begin{align*}
    &\mathord{
        \ \tikz[baseline=-.6ex, scale=.1]{
            \draw[dashed] (0,0) circle [radius=7];
            \draw[weblineblack] (3,-2) to[out=south, in=east] (2,-3) to[out=west, in=south] (-1,0) to[out=north, in=west] (2,3) to[out=east, in=north] (3,2);
            \draw[overarcblack] (0,-7) to[out=north, in=west] (2,-1) to[out=east, in=north] (3,-2);
            \draw[overarcblack] (3,2) to[out=south, in=east] (2,1) to[out=west, in=south] (0,7);
        }
    \ }
    \mathord{
        \tikz[baseline=-.6ex, scale=.1]{
            \draw[<->] (0,0)--(15,0) node[midway,above]{\text{(R1$'$)}};
        }
    } 
    \mathord{
        \ \tikz[baseline=-.6ex, scale=.1]{
            \draw[dashed] (0,0) circle [radius=7];
            \draw[weblineblack] (90:7) to (-90:7);
        }
    \ },&
    &\mathord{
        \ \tikz[baseline=-.6ex, scale=.1]{
            \draw[dashed] (0,0) circle [radius=7];
            \draw[weblineblack] (135:7) to[out=south east, in=west] (0,-2) to[out=east, in=south west](45:7);
            \draw[overarcblack] (-135:7) to[out=north east, in=west] (0,2) to[out=east, in=north west] (-45:7);
        }
    \ }
    \mathord{
        \tikz[baseline=-.6ex, scale=.1]{
            \draw[<->] (0,0)--(15,0) node[midway,above]{\text{(R2)}};
        }
    } 
    \mathord{
        \ \tikz[baseline=-.6ex, scale=.1]{
            \draw[dashed] (0,0) circle [radius=7];
            \draw[weblineblack] (135:7) to[out=south east, in=west](0,2) to[out=east, in=south west](45:7);
            \draw[weblineblack] (-135:7) to[out=north east, in=west](0,-2) to[out=east, in=north west] (-45:7);
        }
    \ },\\
    &\mathord{
        \ \tikz[baseline=-.6ex, scale=.1]{
            \draw[dashed] (0,0) circle [radius=7];
            \draw[weblineblack] (-135:7) -- (45:7);
            \draw[overarcblack] (135:7) -- (-45:7);
            \draw[overarcblack] (180:7) to[out=east, in=west](0,5) to[out=east, in=west] (0:7);
        }
    \ }
    \mathord{
        \tikz[baseline=-.6ex, scale=.1]{
            \draw[<->] (0,0)--(15,0) node[midway,above]{\text{(R3)}};
        }
    } 
    \mathord{
        \ \tikz[baseline=-.6ex, scale=.1]{
            \draw[dashed] (0,0) circle [radius=7];
            \draw[weblineblack] (-135:7) -- (45:7);
            \draw[overarcblack] (135:7) -- (-45:7);
            \draw[overarcblack] (180:7) to[out=east, in=west] (0,-5) to[out=east, in=west] (0:7);
        }
    \ },&
    &\mathord{
        \ \tikz[baseline=-.6ex, scale=.1]{
            \draw[dashed] (0,0) circle [radius=7];
            \draw[weblineblack] (0:0) -- (90:7); 
            \draw[weblineblack] (0:0) -- (210:7); 
            \draw[weblineblack] (0:0) -- (-30:7);
            \draw[overarcblack] (180:7) to[out=east, in=west] (0,5) to[out=east, in=west] (0:7);
        }
    \ }
    \mathord{
        \tikz[baseline=-.6ex, scale=.1]{
            \draw[<->] (0,0)--(15,0) node[midway,above]{\text{(R4)}};
        }
    } 
    \mathord{
        \ \tikz[baseline=-.6ex, scale=.1]{
            \draw[dashed] (0,0) circle [radius=7];
            \draw[weblineblack] (0:0) -- (90:7); 
            \draw[weblineblack] (0:0) -- (210:7); 
            \draw[weblineblack] (0:0) -- (-30:7);
            \draw[overarcblack] (180:7) to[out=east, in=west] (0,-5) to[out=east, in=west](0:7);
        }
    \ },\\
    &\mathord{
        \ \tikz[baseline=-.6ex, scale=.1, yshift=-2cm]{
            \coordinate (P) at (0,0);
            \draw[weblineblack, rounded corners, shorten <=.2cm] (P) -- (135:4) -- (60:8);
            \draw[rounded corners, overarcblack] (P) -- (45:4) -- (120:8);
            \draw[dashed] (8,0) arc (0:180:8cm);
            \bline{-8,0}{8,0}{2}
            \draw[fill=black] (P) circle [radius=20pt];
        \ }
    }\ 
    \mathord{
        \tikz[baseline=-.6ex, scale=.1]{
            \draw [<->] (0,0)--(10,0) node[midway,above]{\text{(bR)}};
        }
    }
    \mathord{
        \ \tikz[baseline=-.6ex, scale=.1, yshift=-2cm]{
            \coordinate (P) at (0,0);
            \draw[weblineblack] (P) -- (120:8);
            \draw[weblineblack, shorten <=.2cm] (P) -- (60:8);
            \draw[dashed] (8,0) arc (0:180:8cm);
            \bline{-8,0}{8,0}{2}
            \draw[fill=black] (P) circle [radius=20pt];
        \ }
    }.&&
\end{align*}

    Here the edge coloring is arbitrary under the condition in \cref{def:tangled-graph} at trivalent vertices. 
    One can also show those moves obtained by exchanging type~$1$ and type~$2$ edges for some strands, or reversing all the over/under-passing information and the elevation.
\end{lem}
\begin{proof}
    The moves (R1$'$), (R2), (R3), and (R4) have been proved by Kuperberg~\cite{Kuperberg94,Kuperberg96}.
    One can prove (bR) by straightforward calculation using the skein relations in \cref{def:skeinrel,def:bdry-skeinrel}.
\end{proof}

\begin{dfn}[the (clasped) $\mathfrak{sp}_4$-skein algebra of an unpunctured marked surface]\label{def:skein-alg}
    Let $\Sigma$ be a marked surface with special points $\bM\subset\partial\Sigma$.
    The \emph{(clasped) $\mathfrak{sp}_4$-skein algebra $\Skein{\Sigma}=\mathscr{S}_{\mathfrak{sp}_4,\Sigma}^q$} is the quotient of the free $\cR$-module $\cR\Tang{\Sigma}$ modulo the $\mathfrak{sp}_4$-skein relations in \cref{def:skeinrel} and the clasped $\mathfrak{sp}_4$-skein relations in \cref{def:bdry-skeinrel}.
    The multiplication of $\Skein{\Sigma}$ is induced from $\Tang{\Sigma}$.
    
    We call an element of $\Skein{\Sigma}$ coming from $\Tang{\Sigma}$ a \emph{tangled $\mathfrak{sp}_4$-graphs in $\Skein{\Sigma}$}, and a general element (a linear combination of tangled $\mathfrak{sp}_4$-graphs) of $\Skein{\Sigma}$ an \emph{$\mathfrak{sp}_4$-web}.
    A \emph{flat $\mathfrak{sp}_4$-graph} in $\Skein{\Sigma}$ is a tangled $\mathfrak{sp}_4$-graph in $\Skein{\Sigma}$ without internal crossings, and only with simultaneous crossings at special points.
    A \emph{boundary $\mathfrak{sp}_4$-web} is a tangled $\mathfrak{sp}_4$-graph in $\Skein{\Sigma}$ represented by an arc of any type parallel to a boundary interval.
    We denote the set of boundary webs by $\partial_{\Sigma} \subset \Skein{\Sigma}$.
\end{dfn}

Any boundary web $v$-commutes with any $\mathfrak{sp}_4$-tangled graphs in $\Skein{\Sigma}$ by the clasped skein relations in \cref{def:bdry-skeinrel}.
Therefore it is easy to see that the multiplicatively closed subset 
\begin{align*}
    \mathrm{mon}(\partial_{\Sigma})\coloneqq\{q^{\frac{k}{2}}[\prod_{\alpha\in\partial_{\Sigma}}\alpha^{f(\alpha)}]\mid f\colon\partial_{\Sigma}\to\bN, k\in\bZ\}
\end{align*}
of $\Skein{\Sigma}$ becomes a left and right Ore set of $\Skein{\Sigma}$.
It guarantees the existence of the following localization of $\Skein{\Sigma}$ at $\mathrm{mon}(\partial_{\Sigma})$.
Moreover it will be shown that $\Skein{\Sigma}$ is an Ore domain in \cref{subsec:Ore}.
\begin{dfn}[the boundary-localized $\mathfrak{sp}_4$-skein algebra]\label{def:localized-skein-alg}
    The boundary-localized $\mathfrak{sp}_4$-skein algebra $\Skein{\Sigma}[\partial^{-1}]$ is the Ore localization of $\Skein{\Sigma}$ at $\mathrm{mon}(\partial_{\Sigma})$.
\end{dfn}


The clasped skein relations preserve the number of edges of each type incident to $p\in\bM$.
Hence, the $\mathfrak{sp}_4$-skein algebra has the following gradings. 

\begin{dfn}[the endpoint grading]\label{def:egrading}
    For a tangled $\mathfrak{sp}_4$-graph $G$ and $p\in\bM$, associated is the vector $\mathrm{deg}_{p}(G)=(c^1_p,c^2_p) \in \bN \times \bN$, where $c^s_p$ is the number of type~$s$ edges of $G$ incident to $p$, for $s=1,2$.
    It defines the \emph{endpoint grading} $\mathrm{deg}(G)=(\mathrm{deg}_p(G)) \in (\bN \times \bN)^{\bM}$ of $G\in\Skein{\Sigma}$ for any tangled $\mathfrak{sp}_4$-graph $G$.
\end{dfn}

\begin{dfn}[the mirror-reflection]\label{def:mirror-ref}
    The \emph{mirror-reflection} $G^{\dagger}$ of a tangled $\mathfrak{sp}_4$-graph $G$ is defined by reversing the ordering of the univalent vertices on each special point and exchanging the over-/under-passing information at each internal crossing.
    The mirror-reflection is extended to an anti-involution  $\dagger:\Skein{\Sigma} \to \Skein{\Sigma}$ by $\bZ$-linearly and by setting $(q^{\pm 1/2})^{\dagger}:=q^{\mp 1/2}$. 
\end{dfn}


\subsection{Basis webs, elementary webs, and web clusters}
In \cite{Kuperberg96}, Kuperberg introduced a $4$-valent vertex to define a basis for a certain space of $\mathfrak{sp}_4$-webs.
The following ``crossroad'' is the Kuperberg's $4$-valent vertex multiplied by $(-1)$.
\begin{dfn}[crossroads and rungs]
    A \emph{crossroad} is a $4$-valent vertex defined by 
    \begin{align*}
        \mathord{
            \ \tikz[baseline=-.6ex, scale=.08, rotate=90]{
                \draw[dashed, fill=white] (0,0) circle [radius=7];
                \draw[webline] (45:7) -- (-135:7);
                \draw[webline] (135:7) -- (-45:7);
                \draw[fill=pink, thick] (0,0) circle [radius=30pt];
            }
        \ }
        :=
        \mathord{
            \ \tikz[baseline=-.6ex, scale=.08]{
                \draw[dashed, fill=white] (0,0) circle [radius=7];
                \draw[webline] (45:7) -- (90:3);
                \draw[webline] (135:7) -- (90:3);
                \draw[webline] (225:7) -- (-90:3);
                \draw[webline] (315:7) -- (-90:3);
                \draw[wline] (90:3) -- (-90:3);
                }
        \ }
        -\frac{1}{[2]}
        \mathord{
            \ \tikz[baseline=-.6ex, scale=.08, rotate=90]{
                \draw[dashed, fill=white] (0,0) circle [radius=7];
                \draw[webline] (-45:7) to[out=north west, in=south] (3,0) to[out=north, in=south west] (45:7);
                \draw[webline] (-135:7) to[out=north east, in=south] (-3,0) to[out=north, in=south east] (135:7);
            }
        \ }
        =
        \mathord{
            \ \tikz[baseline=-.6ex, scale=.08, rotate=90]{
                \draw[dashed, fill=white] (0,0) circle [radius=7];
                \draw[webline] (45:7) -- (90:3);
                \draw[webline] (135:7) -- (90:3);
                \draw[webline] (225:7) -- (-90:3);
                \draw[webline] (315:7) -- (-90:3);
                \draw[wline] (90:3) -- (-90:3);
                }
        \ }
        -\frac{1}{[2]}
        \mathord{
            \ \tikz[baseline=-.6ex, scale=.08]{
                \draw[dashed, fill=white] (0,0) circle [radius=7];
                \draw[webline] (-45:7) to[out=north west, in=south] (3,0) to[out=north, in=south west] (45:7);
                \draw[webline] (-135:7) to[out=north east, in=south] (-3,0) to[out=north, in=south east] (135:7);
            }
        \ }.
    \end{align*}
    A \emph{rung} is a type~$2$ edge in a tangled $\mathfrak{sp}_4$-graph which does not touch special points.
    A subset $S$ of rungs in a tangled $\mathfrak{sp}_4$-graph is \emph{essential} if the tangled $\mathfrak{sp}_4$-graph obtained by removing all rungs in $S$ becomes zero in the skein algebra. 
    An $\mathfrak{sp}_4$-web in $\Skein{\Sigma}$ is a \emph{crossroad web} 
    if it can be represented by a tangled $\mathfrak{sp}_4$-graph diagram on $\Sigma$ with crossroads and no rungs.
    The \emph{crossroad web associated with an $\mathfrak{sp}_4$-web $G$} is a crossroad web obtained by formally replacing all the rungs of $G$ with crossroads.
\end{dfn}
We remark here that a simple loop or arc of either type is also a crossroad web. 

In order to define a basis of $\Skein{\Sigma}$, let us apply the confluence theory for skein modules in Sikora--Westbury~\cite{SikoraWestbury07} to our $\mathfrak{sp}_4$-skein algebra $\Skein{\Sigma}$.
Roughly speaking, the confluence theory 
is a general method to obtain a basis of the quotient of a free module spanned by graphs (in a manifold) with diagrammatic relations.
A skein module of a surface is a typical example: 
let us consider such a module obtained as the quotient of the free $\cR$-module spanned by certain diagrams on $\Sigma$ modulo certain skein relations $\langle X_i-Y_i\rangle_{i\in I}$ corresponding to ``reduction rules $\{S_i\colon X_i\leadsto Y_i\}_{i\in I}$''. In the terminology in \cite{SikoraWestbury07}, the statement is as follows:

\begin{thm}[Sikora--Westbury~\cite{SikoraWestbury07}]\label{thm:confluence}
    If the reduction rules $\{S_i\}_{i\in I}$ is terminal and locally confluent, then the set of irreducible graphs on $\Sigma$ gives a basis for the skein module.
\end{thm}
We roughly explain the confluence theory in the case of the our skein algebra $\Skein{\Sigma}$.
A \emph{reduction rule} $S\colon X\leadsto Y$ is a specific pair of $X$ and $Y$ in $\cR\Diag{\Sigma}$ which gives a relation $X-Y=0$ where $\Diag{\Sigma}$ is the set of isotopy classes of tangled graph diagrams.
A set of reduction rules is \emph{terminal} if all the descending paths associated with the reduction rules are finite, \emph{i.e.} there is no infinite sequence of reduction rules. 
Examples of reduction rules for $\cR\Diag{\Sigma}$ are the left-hand sides to the right-hand sides in \eqref{rel:circle}--\eqref{rel:trigon}.
We will give the precise set of reduction rules $\{S_i\}_{i\in I}$ soon below. 
A fixed set $\{S_i\colon X_i\leadsto Y_i\}_{i\in I}$ of reduction rules is \emph{locally confluent} if for a tangled $\mathfrak{sp}_4$-graph diagram $\Gamma\in\Diag{\Sigma}$ having two reduction rules $S\colon\Gamma\leadsto Z$ and $T\colon\Gamma\leadsto W$, the diagrams $Z$ and $W$ have a common descendant via sequences of reduction rules in $\{S_i\}_{i\in I}$.
An \emph{irreducible $\mathfrak{sp}_4$-graph diagram} is a $\mathfrak{sp}_4$-graph diagram which admits no reductions.
We define reduction rules for $\cR\Diag{\Sigma}$ by adding reduction rules in \cite[Section~6]{SikoraWestbury07}.
It is useful for us to replace each marked point with a small interval which we call an \emph{external clasp} (shown by a thick black line), and describe an $\mathfrak{sp}_4$-graph as a diagram with distinct ends on external clasps. 
In fact, the skein relations in \cref{def:bdry-skeinrel} correspond to relations at external clasps in the sense of  Kuperberg~\cite{Kuperberg96}.
We will use such a description using external clasps in \cref{def:reduction,prop:Bweb-T}.

\begin{dfn}[reduction rules for $\cR\Diag{\Sigma}$]\label{def:reduction}
    Let $(\Sigma,\bM)$ be an unpunctured marked surface, and fix an enumeration $\{1,2,\dots,|\bM|\} \xrightarrow{\sim} \bM$ of special points.

\begin{align*}
&\begin{alignedat}{3}
    &B_1\colon
    \mathord{
        \ \tikz[baseline=-.6ex, scale=.08, yshift=-4cm]{
            \fill[pattern=north west lines, pattern color=lightgray] (-11,0) rectangle (11,11);
            \draw[dashed, fill=white] (10,0) arc (0:180:10cm);
            \coordinate (P) at (0,0);
            \coordinate (C) at (90:7);
            \draw[webline, fill=white] (5,0) arc (0:180:5cm);
            \draw[line width=.1cm] (-10,0) -- (10,0);
        \ }
    }
    \leadsto 0,\quad&
    &B_2\colon
    \mathord{
        \ \tikz[baseline=-.6ex, scale=.08, yshift=-4cm]{
            \fill[pattern=north west lines, pattern color=lightgray] (-11,0) rectangle (11,11);
            \draw[dashed, fill=white] (10,0) arc (0:180:10cm);
            \coordinate (P) at (0,0) {};
            \coordinate (C) at (90:7);
            \draw[wline] (5,0) arc (0:180:5cm);
            \draw[line width=.1cm] (-10,0) -- (10,0);
        \ }
    }
    \leadsto 0,\quad&
    &B_3\colon
    \mathord{
        \ \tikz[baseline=-.6ex, scale=.08, yshift=-4cm]{
            \fill[pattern=north west lines, pattern color=lightgray] (-11,0) rectangle (11,11);
            \draw[dashed, fill=white] (10,0) arc (0:180:10cm);
            \coordinate (P) at (0,0);
            \coordinate (C) at (90:5);
            \draw[webline] (-5,0) -- (C) -- (5,0);
            \draw[wline] (C) -- (90:10);
            \draw[line width=.1cm] (-10,0) -- (10,0);
        \ }
    }
    \leadsto 0,\\
    &B_4\colon
    \mathord{
        \ \tikz[baseline=-.6ex, scale=.08, yshift=-4cm]{
            \fill[pattern=north west lines, pattern color=lightgray] (-11,0) rectangle (11,11);
            \draw[dashed, fill=white] (10,0) arc (0:180:10cm);
            \coordinate (P) at (0,0);
            \coordinate (C) at (90:5);
            \draw[wline] (-5,0) -- (C);
            \draw[webline] (5,0) -- (C);
            \draw[webline] (C) -- (90:10);
            \draw[line width=.1cm] (-10,0) -- (10,0);
        \ }
    }
    \leadsto 0,\quad&
    &B_5\colon
    \mathord{
        \ \tikz[baseline=-.6ex, scale=.08, yshift=-4cm]{
            \fill[pattern=north west lines, pattern color=lightgray] (-11,0) rectangle (11,11);
            \draw[dashed, fill=white] (10,0) arc (0:180:10cm);
            \coordinate (P) at (0,0);
            \coordinate (C) at (90:5);
            \draw[wline] (5,0) -- (C);
            \draw[webline] (-5,0) -- (C);
            \draw[webline] (C) -- (90:10);
            \draw[line width=.1cm] (-10,0) -- (10,0);
        \ }
    }
    \leadsto 0,\quad&
    &B_6\colon
    \mathord{
        \ \tikz[baseline=-.6ex, scale=.08, yshift=-4cm]{
            \fill[pattern=north west lines, pattern color=lightgray] (-11,0) rectangle (11,11);
            \draw[dashed, fill=white] (10,0) arc (0:180:10cm);
            \coordinate (P) at (0,0);
            \coordinate (C) at (90:5);
            \begin{scope}
                \clip (10,0) arc (0:180:10cm);
                \draw[webline, yshift=5cm] (225:10) -- (45:10);
                \draw[webline, yshift=5cm] (-45:10) -- (135:10);
                \draw[fill=pink, yshift=5cm] (0,0) circle [radius=30pt];
            \end{scope}
            \draw[line width=.1cm] (-10,0) -- (10,0);
        \ }
    }
    \leadsto 0,\\
    &B_7\colon
    \mathord{
        \ \tikz[baseline=-.6ex, scale=.08, yshift=-4cm]{
            \fill[pattern=north west lines, pattern color=lightgray] (-11,0) rectangle (11,11);
            \draw[dashed, fill=white] (10,0) arc (0:180:10cm);
            \coordinate (P) at (0,0);
            \coordinate (C) at (90:5);
            \begin{scope}
                \clip (10,0) arc (0:180:10cm);
                \draw[webline] (-5,10) -- (-5,5) -- (5,5) -- (5,10);
                \draw[wline] (5,5) -- (5,0);
                \draw[wline] (-5,5) -- (-5,0);
            \end{scope}
            \draw[line width=.1cm] (-10,0) -- (10,0);
        \ }
    }
    \leadsto 0,\quad&
    &B_8\colon
    \mathord{
        \ \tikz[baseline=-.6ex, scale=.08, yshift=-4cm]{
            \fill[pattern=north west lines, pattern color=lightgray] (-11,0) rectangle (11,11);
            \draw[dashed, fill=white] (10,0) arc (0:180:10cm);
            \coordinate (P) at (0,0);
            \coordinate (C) at (90:5);
            \begin{scope}
                \clip (10,0) arc (0:180:10cm);
                \draw[webline] (-5,10) -- (-5,0);
                \draw[webline] (5,10) -- (5,0);
                \draw[wline] (-5,5) -- (5,5);
            \end{scope}
            \draw[line width=.1cm] (-10,0) -- (10,0);
        \ }
    }
    \leadsto\frac{1}{[2]}
    \mathord{
        \ \tikz[baseline=-.6ex, scale=.08, yshift=-4cm]{
            \fill[pattern=north west lines, pattern color=lightgray] (-11,0) rectangle (11,11);
            \draw[dashed, fill=white] (10,0) arc (0:180:10cm);
            \coordinate (P) at (0,0);
            \coordinate (C) at (90:5);
            \begin{scope}
                \clip (10,0) arc (0:180:10cm);
                \draw[webline] (-5,10) -- (-5,0);
                \draw[webline] (5,10) -- (5,0);
            \end{scope}
            \draw[line width=.1cm] (-10,0) -- (10,0);
        \ }
    },
\end{alignedat}\\
&\begin{alignedat}{2}
    &B_9\colon
    \mathord{
        \ \tikz[baseline=-.6ex, scale=.08, yshift=-4cm]{
            \fill[pattern=north west lines, pattern color=lightgray] (-11,0) rectangle (11,11);
            \draw[dashed, fill=white] (10,0) arc (0:180:10cm);
            \coordinate (P) at (0,0);
            \coordinate (C) at (90:5);
            \begin{scope}
                \clip (10,0) arc (0:180:10cm);
                \draw[webline] (-5,10) -- (-5,5);
                \draw[wline] (-5,5) -- (-5,0);
                \draw[wline] (5,10) -- (5,5);
                \draw[webline] (5,5) -- (5,0);
                \draw[webline] (-5,5) -- (5,5);
            \end{scope}
            \draw[line width=.1cm] (-10,0) -- (10,0);
        \ }
    }
    \leadsto
    \mathord{
        \ \tikz[baseline=-.6ex, scale=.08, yshift=-4cm]{
            \fill[pattern=north west lines, pattern color=lightgray] (-11,0) rectangle (11,11);
            \draw[dashed, fill=white] (10,0) arc (0:180:10cm);
            \coordinate (P) at (0,0);
            \coordinate (C) at (90:5);
            \begin{scope}
                \clip (10,0) arc (0:180:10cm);
                \draw[webline] (-5,10) -- (-5,0);
                \draw[wline] (5,10) -- (5,0);
            \end{scope}
            \draw[line width=.1cm] (-10,0) -- (10,0);
        \ }
    },\quad&
    &B_{10}\colon
    \mathord{
        \ \tikz[baseline=-.6ex, scale=.08, yshift=-4cm]{
            \fill[pattern=north west lines, pattern color=lightgray] (-11,0) rectangle (11,11);
            \draw[dashed, fill=white] (10,0) arc (0:180:10cm);
            \coordinate (P) at (0,0);
            \coordinate (C) at (90:5);
            \begin{scope}
                \clip (10,0) arc (0:180:10cm);
                \draw[wline] (-5,10) -- (-5,5);
                \draw[webline] (-5,5) -- (-5,0);
                \draw[webline] (5,10) -- (5,5);
                \draw[wline] (5,5) -- (5,0);
                \draw[webline] (-5,5) -- (5,5);
            \end{scope}
            \draw[line width=.1cm] (-10,0) -- (10,0);
        \ }
    }
    \leadsto
    \mathord{
        \ \tikz[baseline=-.6ex, scale=.08, yshift=-4cm]{
            \fill[pattern=north west lines, pattern color=lightgray] (-11,0) rectangle (11,11);
            \draw[dashed, fill=white] (10,0) arc (0:180:10cm);
            \begin{scope}
                \clip (10,0) arc (0:180:10cm);
                \draw[wline] (-5,10) -- (-5,0);
                \draw[webline] (5,10) -- (5,0);
            \end{scope}
            \draw[line width=.1cm] (-10,0) -- (10,0);
        \ }
    },\\
    &B_{11}\colon
    \mathord{
        \ \tikz[baseline=-.6ex, scale=.08, yshift=-4cm]{
            \fill[pattern=north west lines, pattern color=lightgray] (-11,0) rectangle (11,11);
            \draw[dashed, fill=white] (10,0) arc (0:180:10cm);
            \begin{scope}
                \clip (10,0) arc (0:180:10cm);
                \draw[wline] (-3,0) -- (-3,5);
                \draw[webline] (-3,5) -- (10,5);
                \draw[webline] (-3,5) -- (-3,10);
                \draw[webline] (3,10) -- (3,0);
            \end{scope}
            \draw[line width=.1cm] (-10,0) -- (10,0);
            \draw[fill=pink] (3,5) circle [radius=30pt];
        \ }
    }
    \leadsto
    \mathord{
        \ \tikz[baseline=-.6ex, scale=.08, yshift=-4cm]{
            \fill[pattern=north west lines, pattern color=lightgray] (-11,0) rectangle (11,11);
            \draw[dashed, fill=white] (10,0) arc (0:180:10cm);
            \begin{scope}
                \clip (10,0) arc (0:180:10cm);
                \draw[wline] (3,0) -- (3,5);
                \draw[webline] (3,5) -- (10,5);
                \draw[webline] (3,5) -- (3,10);
                \draw[webline] (-3,10) -- (-3,0);
            \end{scope}
            \draw[line width=.1cm] (-10,0) -- (10,0);
        \ }
    },\quad&
    &B_{12}\colon
    \mathord{
        \ \tikz[baseline=-.6ex, scale=.08, yshift=-4cm]{
            \fill[pattern=north west lines, pattern color=lightgray] (-11,0) rectangle (11,11);
            \draw[dashed, fill=white] (10,0) arc (0:180:10cm);
            \begin{scope}[xscale=-1]
                \clip (10,0) arc (0:180:10cm);
                \draw[wline] (-3,0) -- (-3,5);
                \draw[webline] (-3,5) -- (10,5);
                \draw[webline] (-3,5) -- (-3,10);
                \draw[webline] (3,10) -- (3,0);
            \end{scope}
            \draw[line width=.1cm] (-10,0) -- (10,0);
            \draw[fill=pink] (-3,5) circle [radius=30pt];
        \ }
    }
    \leadsto
    \mathord{
        \ \tikz[baseline=-.6ex, scale=.08, yshift=-4cm]{
            \fill[pattern=north west lines, pattern color=lightgray] (-11,0) rectangle (11,11);
            \draw[dashed, fill=white] (10,0) arc (0:180:10cm);
            \begin{scope}[xscale=-1]
                \clip (10,0) arc (0:180:10cm);
                \draw[wline] (3,0) -- (3,5);
                \draw[webline] (3,5) -- (10,5);
                \draw[webline] (3,5) -- (3,10);
                \draw[webline] (-3,10) -- (-3,0);
            \end{scope}
            \draw[line width=.1cm] (-10,0) -- (10,0);
        \ }
    },\\
    &B_{13}\colon
    \mathord{
        \ \tikz[baseline=-.6ex, scale=.08, yshift=-4cm]{
            \fill[pattern=north west lines, pattern color=lightgray] (-12,0) rectangle (12,12);
            \draw[dashed, fill=white] (11,0) arc (0:180:11cm);
            \draw[wline] (9,0) arc (0:180:9cm);
            \draw[webline] (7,0) arc (0:180:7cm);
            \draw[dashed, pattern=north west lines, pattern color=lightgray] (5,0) arc (0:180:5cm);
            \draw[line width=.1cm] (-12,0) -- (12,0);
        \ }
    }
    \leadsto 
    \mathord{
        \ \tikz[baseline=-.6ex, scale=.08, yshift=-4cm]{
            \fill[pattern=north west lines, pattern color=lightgray] (-12,0) rectangle (12,12);
            \draw[dashed, fill=white] (11,0) arc (0:180:11cm);
            \draw[webline] (9,0) arc (0:180:9cm);
            \draw[wline] (7,0) arc (0:180:7cm);
            \draw[dashed, pattern=north west lines, pattern color=lightgray] (5,0) arc (0:180:5cm);
            \draw[line width=.1cm] (-12,0) -- (12,0);
        \ }
    },\quad&
    &B_{14}\colon
    \mathord{
        \ \tikz[baseline=-.6ex, scale=.08]{
            \fill[pattern=north west lines, pattern color=lightgray] (-7,-7) rectangle (7,7);
            \draw[dashed, fill=white] (-5,-7) rectangle (5,7);
            \coordinate (N1) at (2,7);
            \coordinate (S1) at (2,-7);
            \coordinate (N2) at (-2,7);
            \coordinate (S2) at (-2,-7);
            \draw[webline] (S1) -- (N1);
            \draw[wline] (S2) -- (N2);
            \draw[line width=.1cm] (-5,7) -- (5,7);
            \draw[line width=.1cm] (-5,-7) -- (5,-7);
            \node at (0,-7) [below]{$i$};
            \node at (0,7) [above]{$j$};
        \ }
    }
    \leadsto
    \mathord{
        \ \tikz[baseline=-.6ex, scale=.08]{
            \fill[pattern=north west lines, pattern color=lightgray] (-7,-7) rectangle (7,7);
            \draw[dashed, fill=white] (-5,-7) rectangle (5,7);
            \coordinate (N1) at (2,7);
            \coordinate (S1) at (2,-7);
            \coordinate (N2) at (-2,7);
            \coordinate (S2) at (-2,-7);
            \draw[wline] (S1) -- (N1);
            \draw[webline] (S2) -- (N2);
            \draw[line width=.1cm] (-5,7) -- (5,7);
            \draw[line width=.1cm] (-5,-7) -- (5,-7);
            \node at (0,-7) [below]{$i$};
            \node at (0,7) [above]{$j$};
        \ }
    } \text{ for $i<j$ },
\end{alignedat}
\end{align*}
\begin{align*}
    &C_{1}\colon
    \mathord{
        \ \tikz[baseline=-.6ex, scale=.08]{
            \fill[pattern=north west lines, pattern color=lightgray] (-8,-8) rectangle (8,8);
            \draw[dashed, fill=white] (0,0) circle [radius=7];
            \draw[red, very thick] (-45:7) -- (135:7);
            \draw[overarc] (-135:7) -- (45:7);
            }
    \ }
    \leadsto
    q
    \mathord{
        \ \tikz[baseline=-.6ex, scale=.08]{
            \fill[pattern=north west lines, pattern color=lightgray] (-8,-8) rectangle (8,8);
            \draw[dashed, fill=white] (0,0) circle [radius=7];
            \draw[webline] (-45:7) to[out=north west, in=south] (3,0) to[out=north, in=south west] (45:7);
            \draw[webline] (-135:7) to[out=north east, in=south] (-3,0) to[out=north, in=south east] (135:7);
        }
    \ }
    +q^{-1}
    \mathord{
        \ \tikz[baseline=-.6ex, scale=.08, rotate=90]{
            \fill[pattern=north west lines, pattern color=lightgray] (-8,-8) rectangle (8,8);
            \draw[dashed, fill=white] (0,0) circle [radius=7];
            \draw[webline] (-45:7) to[out=north west, in=south] (3,0) to[out=north, in=south west] (45:7);
            \draw[webline] (-135:7) to[out=north east, in=south] (-3,0) to[out=north, in=south east] (135:7);
        }
    \ }
    +
    \mathord{
        \ \tikz[baseline=-.6ex, scale=.08]{
            \fill[pattern=north west lines, pattern color=lightgray] (-8,-8) rectangle (8,8);
            \draw[dashed, fill=white] (0,0) circle [radius=7];
            \draw[webline] (45:7) -- (225:7);
            \draw[webline] (-45:7) -- (135:7);
            \draw[fill=pink] (0,0) circle [radius=30pt];
            }
    \ },\\
    &C_{2}\colon
    \mathord{
        \ \tikz[baseline=-.6ex, scale=.08]{
            \fill[pattern=north west lines, pattern color=lightgray] (-8,-8) rectangle (8,8);
            \draw[dashed, fill=white] (0,0) circle [radius=7];
            \draw[red, very thick] (-45:7) -- (135:7);
            \draw[overwline] (-135:7) -- (45:7);
            }
    \ }
    \leadsto q\mathord{
        \ \tikz[baseline=-.6ex, scale=.08]{
            \fill[pattern=north west lines, pattern color=lightgray] (-8,-8) rectangle (8,8);
            \draw[dashed, fill=white] (0,0) circle [radius=7];
            \draw[webline] (-45:7) -- (3,0);
            \draw[wline] (-135:7) -- (-3,0);
            \draw[wline] (45:7) -- (3,0);
            \draw[webline] (135:7) -- (-3,0);
            \draw[webline] (-3,0) -- (3,0);
        }
    \ }
    +q^{-1}\mathord{
        \ \tikz[baseline=-.6ex, scale=.08]{
            \fill[pattern=north west lines, pattern color=lightgray] (-8,-8) rectangle (8,8);
            \draw[dashed, fill=white] (0,0) circle [radius=7];
            \draw[webline] (-45:7) -- (0,-3);
            \draw[wline] (-135:7) -- (0,-3);
            \draw[wline] (45:7) -- (0,3);
            \draw[webline] (135:7) -- (0,3);
            \draw[webline] (0,-3) -- (0,3);
        }
    \ },\\
    \quad
    &C_{3}\colon
    \mathord{
        \ \tikz[baseline=-.6ex, scale=.08]{
            \fill[pattern=north west lines, pattern color=lightgray] (-8,-8) rectangle (8,8);
            \draw[dashed, fill=white] (0,0) circle [radius=7];
            \draw[wline] (-45:7) -- (135:7);
            \draw[overarc] (-135:7) -- (45:7);
            }
    \ }
    \leadsto q
    \mathord{
        \ \tikz[baseline=-.6ex, scale=.08]{
            \fill[pattern=north west lines, pattern color=lightgray] (-8,-8) rectangle (8,8);
            \draw[dashed, fill=white] (0,0) circle [radius=7];
            \draw[wline] (-45:7) -- (3,0);
            \draw[webline] (-135:7) -- (-3,0);
            \draw[webline] (45:7) -- (3,0);
            \draw[wline] (135:7) -- (-3,0);
            \draw[webline] (-3,0) -- (3,0);
        }
    \ }
    +q^{-1}
    \mathord{
        \ \tikz[baseline=-.6ex, scale=.08]{
            \fill[pattern=north west lines, pattern color=lightgray] (-8,-8) rectangle (8,8);
            \draw[dashed, fill=white] (0,0) circle [radius=7];
            \draw[wline] (-45:7) -- (0,-3);
            \draw[webline] (-135:7) -- (0,-3);
            \draw[webline] (45:7) -- (0,3);
            \draw[wline] (135:7) -- (0,3);
            \draw[webline] (0,-3) -- (0,3);
        }
    \ },\\
    \quad
    &C_{4}\colon
    \mathord{
        \ \tikz[baseline=-.6ex, scale=.08]{
            \fill[pattern=north west lines, pattern color=lightgray] (-8,-8) rectangle (8,8);
            \draw[dashed, fill=white] (0,0) circle [radius=7];
            \draw[wline] (-45:7) -- (135:7);
            \draw[overwline] (-135:7) -- (45:7);
            }
    \ }
    \leadsto q^{2}
    \mathord{
        \ \tikz[baseline=-.6ex, scale=.08]{
            \fill[pattern=north west lines, pattern color=lightgray] (-8,-8) rectangle (8,8);
            \draw[dashed, fill=white] (0,0) circle [radius=7];
            \draw[wline] (-45:7) to[out=north west, in=south] (3,0) to[out=north, in=south west] (45:7);
            \draw[wline] (-135:7) to[out=north east, in=south] (-3,0) to[out=north, in=south east] (135:7);
        }
    \ }
    +q^{-2}\mathord{
        \ \tikz[baseline=-.6ex, scale=.08, rotate=90]{
            \fill[pattern=north west lines, pattern color=lightgray] (-8,-8) rectangle (8,8);
            \draw[dashed, fill=white] (0,0) circle [radius=7];
            \draw[wline] (-45:7) to[out=north west, in=south] (3,0) to[out=north, in=south west] (45:7);
            \draw[wline] (-135:7) to[out=north east, in=south] (-3,0) to[out=north, in=south east] (135:7);
        }
    \ }
    +\mathord{
        \ \tikz[baseline=-.6ex, scale=.08]{
            \fill[pattern=north west lines, pattern color=lightgray] (-8,-8) rectangle (8,8);
            \draw[dashed, fill=white] (0,0) circle [radius=7];
            \draw[wline] (-45:7) -- (135:7);
            \draw[wline] (-135:7) -- (45:7);
            \draw[very thick, red, fill=white] ([shift=(-135:3)]0,0) rectangle (45:3);
        }
    \ }.
\end{align*}
\end{dfn}

Here for instance, the reduction rule $B_1$ is related to the first relation in \eqref{rel:bdry-monogon}. 
$B_{13}$ is a reduction rule related to arcs incident to a special point. 
It moves arcs of type $2$ towards inside of other arcs of type $1$.
The rule $B_{14}$ moves arcs of type $2$ towards the right of other arcs of type $1$, referring to the labels of the relevant special points (or external clasps).
We need reduction rules $B_{13}$ and $B_{14}$ because an overlap between two applications of $B_9$ is not locally confluent as shown below:
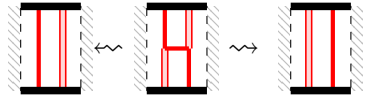
\[
    \mathord{
        \ \tikz[baseline=-.6ex, scale=.08]{
            \fill[pattern=north west lines, pattern color=lightgray] (-7,-7) rectangle (7,7);
            \draw[dashed, fill=white] (-5,-7) rectangle (5,7);
            \coordinate (N1) at (2,7);
            \coordinate (S1) at (2,-7);
            \coordinate (N2) at (-2,7);
            \coordinate (S2) at (-2,-7);
            \draw[wline] (S1) -- (N1);
            \draw[webline] (S2) -- (N2);
            \draw[line width=.1cm] (-5,7) -- (5,7);
            \draw[line width=.1cm] (-5,-7) -- (5,-7);
        \ }
    }
    \reflectbox{$\leadsto$}
    \mathord{
        \ \tikz[baseline=-.6ex, scale=.08]{
            \fill[pattern=north west lines, pattern color=lightgray] (-7,-7) rectangle (7,7);
            \draw[dashed, fill=white] (-5,-7) rectangle (5,7);
            \coordinate (N1) at (2,7);
            \coordinate (S1) at (2,-7);
            \coordinate (N2) at (-2,7);
            \coordinate (S2) at (-2,-7);
            \coordinate (L) at (-2,0);
            \coordinate (R) at (2,0);
            \draw[webline] (S1) -- (R);
            \draw[wline] (S2) -- (L);
            \draw[wline] (N1) -- (R);
            \draw[webline] (N2) -- (L);
            \draw[webline] (L) -- (R);
            \draw[line width=.1cm] (-5,7) -- (5,7);
            \draw[line width=.1cm] (-5,-7) -- (5,-7);
        \ }
    }
    \leadsto
    \mathord{
        \ \tikz[baseline=-.6ex, scale=.08]{
            \fill[pattern=north west lines, pattern color=lightgray] (-7,-7) rectangle (7,7);
            \draw[dashed, fill=white] (-5,-7) rectangle (5,7);
            \coordinate (N1) at (2,7);
            \coordinate (S1) at (2,-7);
            \coordinate (N2) at (-2,7);
            \coordinate (S2) at (-2,-7);
            \draw[webline] (S1) -- (N1);
            \draw[wline] (S2) -- (N2);
            \draw[line width=.1cm] (-5,7) -- (5,7);
            \draw[line width=.1cm] (-5,-7) -- (5,-7);
        \ }
    }.
\]

\begin{lem}\label{lem:sp4-confluence}
    The set $\{S_1,\ldots,S_{17},B_1,\ldots,B_{14},C_{1},\ldots,C_{4}\}$ of reduction rules introduced above is locally confluent and terminal.
\end{lem}
\begin{proof}
    For a $\mathfrak{sp}_4$-graph diagram, let us consider the following quintuple $(x,v_3,v_4,c,l)\in (\bZ_{\geq 0})^5$ where $x$ is the number of crossings, $v_3$ the number of trivalent vertices, $v_4$ the number of crossroads, $c$ the number of connected components, $l$ the number of type $2$ arcs lying in the left-side or outer-side of another type $1$ arcs in a parallel manner.
    Each reduction rule decreases the quintuple with respect to the lexicographic order of $(\bZ_{\geq 0})^5$. 
    Any descending path of the reduction rules is terminal because the quintuple belongs to $(\bZ_{\geq 0})^5$ with the lexicographic order.
    One can also confirm that all ``overlaps'' are locally confluent in a similar way to \cite[Section~6]{SikoraWestbury07}.
    For example, the previous overlap between $B_9$ and itself is locally confluent thanks to $B_{13}$ or $B_{14}$.
\end{proof}

We call faces bounded by edges (resp.~ and an external clasp) appearing in the left-hand sides of $S_{1}$ to $S_{17}$ (resp.~$B_{1}$ to $B_{12}$) except for $S_8$ \emph{elliptic faces} (cf.~\cite{Kuperberg96,SikoraWestbury07})].
We remark that $S_8$ is used to eliminate rungs.

\begin{dfn}[basis webs]\label{def:basis-web}
    A \emph{basis web} is a flat crossroad web without elliptic faces.
    We denote the set of all basis webs in $\Skein{\Sigma}$ by $\Bweb{\Sigma}$.
\end{dfn}

\begin{thm}\label{thm:basis-web}
    $\Bweb{\Sigma}$ is an $\cR$-basis of $\Skein{\Sigma}$.
\end{thm}
\begin{proof}
    One can apply \cref{thm:confluence} to $\cR\Diag{\Sigma}/\langle S_{i}, B_{j}, C_{k}\rangle$, thanks to \cref{lem:sp4-confluence}.
    We can confirm that the $\mathfrak{sp}_4$-webs in both sides of these reduction rules are related by skein relations in \cref{def:skeinrel,def:bdry-skeinrel}.
    In addition, the skein relations are deduced from the reduction rules.
    Thus, the $\cR$-submodule $\langle S_{i}, B_{j}, C_{k}\rangle$ is equal to the $\cR$-submodule generated by the $\mathfrak{sp}_4$-skein relations.
    It concludes that the irreducible $\mathfrak{sp}_4$-graph diagrams with respect to the reduction rules in \cref{def:reduction} give an $\cR$-basis of $\Skein{\Sigma}$.
    Note that if $\mathfrak{sp}_4$-graphs in the left hand side of $B_{13}$ or $B_{14}$ has no elliptic faces, then the right has also no elliptic faces. However, these $\mathfrak{sp}_4$-graphs determine the same element in $\Skein{\Sigma}$.
    It concludes that $\Bweb{\Sigma}$ is an $\cR$-basis of $\Skein{\Sigma}$.
\end{proof}

\begin{dfn}[elementary webs]\label{def:elementary-web}
    An \emph{elementary web} $G$ is a flat $\mathfrak{sp}_4$-graph satisfying the following:
    \begin{itemize}
        \item all subsets of rungs of $G$ are essential,
        \item $G$ is indecomposable to a product of more than two basis webs, and
        \item the associated crossroad web is a basis web.
    \end{itemize}
    It is said to be of weight $d=2$ if the number $\epsilon(\mathrm{deg}_p(G)) \in \bN$ is even for all $p \in \bM$, and $d=1$ otherwise. Here $\epsilon(c^1,c^2):=c^1+2c^2$.
    An elementary web is \emph{tree-type} if it is represented by a tangled tree $G$ (\emph{i.e.}, an $\mathfrak{sp}_4$-graph whose underlying graph is a tree) up to a power of $q^{1/2}$ such that any rung of $G$ is adjacent to a triangle 
    $\mathord{
        \ \tikz[baseline=-.6ex, scale=.1, yshift=-3cm]{
            \coordinate (P) at (0,0);
            \coordinate (Q) at (90:7);
            \coordinate (A1) at (135:7);
            \coordinate (A2) at (45:7);
            \coordinate (B1) at (120:5);
            \coordinate (B2) at (60:5);
            \draw[overarc] (P) -- (B1) -- (B2) -- (P);
            \draw[wline] (B1) -- (A1);
            \draw[wline] (B2) -- (A2);
            \draw[dashed] (7,0) arc (0:180:7cm);
            \bdryline{(-7,0)}{(7,0)}{1cm}
            \draw[fill=black] (P) circle [radius=20pt];
        \ }
    }$
    bounded by type $1$ edges. We call such a tangled tree $G$ a \emph{lift} of the tree-type elementary web.
    We denote the set of all elementary webs in $\Skein{\Sigma}$ by $\Eweb{\Sigma}$, and the subset of tree-type ones by $\Tree{\Sigma} \subset \Eweb{\Sigma}$. 
\end{dfn}

\begin{rem}
    Any elementary web is, in fact, equal to its associated crossroad web in $\Skein{\Sigma}$ by definition of the essential rung and the crossroad.
    Hence, one obtain inclusions $\Tree{\Sigma}\subset\Eweb{\Sigma}\subset\Bweb{\Sigma}$.
\end{rem}

\begin{ex}\label{ex:elementary_webs}
    Let $\Sigma$ be a quadrilateral.
    \begin{enumerate}
        \item The $\mathfrak{sp}_4$-graphs
        \ \tikz[baseline=-.6ex, scale=.1, yshift=-5cm]{
            \coordinate (A) at (0,0);
            \coordinate (B) at (0,10);
            \coordinate (C) at (10,10);
            \coordinate (D) at (10,0);
            \coordinate (AB) at ($(A)!0.5!(B)+(3,0)$);
            \coordinate (BC) at ($(B)!0.5!(C)+(0,-3)$);
            \coordinate (CD) at ($(C)!0.5!(D)+(-3,0)$);
            \coordinate (DA) at ($(D)!0.5!(A)+(0,3)$);
            \draw[webline] (A) -- (AB) -- (B);
            \draw[webline] (C) -- (CD) -- (D);
            \draw[wline] (AB) -- (CD);
            \draw[very thick] (A) -- (B) -- (C) -- (D) -- cycle;
            \filldraw (A) circle [radius=20pt];
            \filldraw (B) circle [radius=20pt];
            \filldraw (C) circle [radius=20pt];
            \filldraw (D) circle [radius=20pt];
        }\ 
        and
        \ \tikz[baseline=-.6ex, scale=.1, yshift=-5cm]{
            \coordinate (A) at (0,0);
            \coordinate (B) at (0,10);
            \coordinate (C) at (10,10);
            \coordinate (D) at (10,0);
            \coordinate (AB) at ($(A)!0.5!(B)+(3,0)$);
            \coordinate (BC) at ($(B)!0.5!(C)+(0,-3)$);
            \coordinate (CD) at ($(C)!0.5!(D)+(-3,0)$);
            \coordinate (DA) at ($(D)!0.5!(A)+(0,3)$);
            \draw[webline] (B) -- (BC) -- (C);
            \draw[webline] (D) -- (DA) -- (A);
            \draw[wline] (DA) -- (BC);
            \draw[very thick] (A) -- (B) -- (C) -- (D) -- cycle;
            \filldraw (A) circle [radius=20pt];
            \filldraw (B) circle [radius=20pt];
            \filldraw (C) circle [radius=20pt];
            \filldraw (D) circle [radius=20pt];
        }\ 
        are indecomposable.
        Their common associated crossroad web
        \ \tikz[baseline=-.6ex, scale=.1, yshift=-5cm]{
            \coordinate (A) at (0,0);
            \coordinate (B) at (0,10);
            \coordinate (C) at (10,10);
            \coordinate (D) at (10,0);
            \draw[webline] (A) -- (C);
            \draw[webline] (B) -- (D);
            \draw[fill=pink, thick] (5,5) circle [radius=20pt];
            \draw[very thick] (A) -- (B) -- (C) -- (D) -- cycle;
            \filldraw (A) circle [radius=20pt];
            \filldraw (B) circle [radius=20pt];
            \filldraw (C) circle [radius=20pt];
            \filldraw (D) circle [radius=20pt];
        }\ 
        belongs to $\Bweb{\Sigma}$.
        However, they are not elementary because their rungs are not essential.
        \item The flat $\mathfrak{sp}_4$-graph
        \ \tikz[baseline=-.6ex, scale=.1, yshift=-5cm, rotate=90]{
            \coordinate (A) at (0,0);
            \coordinate (B) at (0,10);
            \coordinate (C) at (10,10);
            \coordinate (D) at (10,0);
            \coordinate (AB) at ($(A)!0.5!(B)+(2,0)$);
            \coordinate (BC) at ($(B)!0.5!(C)+(0,-2)$);
            \coordinate (CD) at ($(C)!0.5!(D)+(-2,0)$);
            \coordinate (DA) at ($(D)!0.5!(A)+(0,2)$);
            \coordinate (CD1) at ($(C)!.3!(D)+(-2,0)$);
            \coordinate (CD2) at ($(C)!.7!(D)+(-2,0)$);
            \draw[webline] (C) -- (BC) -- (CD1) -- (C);
            \draw[webline] (D) -- (DA) -- (CD2) -- (D);
            \draw[wline] (A) -- (DA);
            \draw[wline] (B) -- (BC);
            \draw[wline] (CD1) -- (CD2);
            \draw[very thick] (A) -- (B) -- (C) -- (D) -- cycle;
            \filldraw (A) circle [radius=20pt];
            \filldraw (B) circle [radius=20pt];
            \filldraw (C) circle [radius=20pt];
            \filldraw (D) circle [radius=20pt];
        }\ 
        belongs to $\Tree{\Sigma}$ with its associated crossroad web
        \ \tikz[baseline=-.6ex, scale=.1, yshift=-5cm, rotate=90]{
            \coordinate (A) at (0,0);
            \coordinate (B) at (0,10);
            \coordinate (C) at (10,10);
            \coordinate (D) at (10,0);
            \coordinate (AB) at ($(A)!0.5!(B)+(2,0)$);
            \coordinate (BC) at ($(B)!0.5!(C)+(0,-2)$);
            \coordinate (CD) at ($(C)!0.5!(D)+(-2,0)$);
            \coordinate (DA) at ($(D)!0.5!(A)+(0,2)$);
            \draw[webline] (C) -- (BC);
            \draw[webline] (D) -- (DA);
            \draw[wline] (A) -- (DA);
            \draw[wline] (B) -- (BC);
            \draw[webline] (DA) -- (C);
            \draw[webline] (BC) -- (D);
            \draw[fill=pink, thick] ($(CD)+(-1,0)$) circle [radius=20pt];
            \draw[very thick] (A) -- (B) -- (C) -- (D) -- cycle;
            \filldraw (A) circle [radius=20pt];
            \filldraw (B) circle [radius=20pt];
            \filldraw (C) circle [radius=20pt];
            \filldraw (D) circle [radius=20pt];
        }\ 
        in $\Bweb{\Sigma}$.
        It will correspond to a cluster variable via the inclusion  $\Skein{\Sigma}^{\bZ_q}\subset\CA^q_{\mathfrak{sp}_4,\Sigma}$ given in \cref{thm:S in A}.
        \item The indecomposable $\mathfrak{sp}_4$-graph 
        \ \tikz[baseline=-.6ex, scale=.1, yshift=-5cm, rotate=90]{
            \coordinate (A) at (0,0);
            \coordinate (B) at (0,10);
            \coordinate (C) at (10,10);
            \coordinate (D) at (10,0);
            \coordinate (AB) at ($(A)!0.5!(B)+(2,0)$);
            \coordinate (BC) at ($(B)!0.5!(C)+(0,-2)$);
            \coordinate (CD) at ($(C)!0.5!(D)+(-2,0)$);
            \coordinate (DA) at ($(D)!0.5!(A)+(0,2)$);
            \draw[webline] (C) -- (CD) -- (D);
            \draw[webline] (BC) -- (DA);
            \draw[webline] (C) -- (BC);
            \draw[webline] (D) -- (DA);
            \draw[wline] (A) -- (DA);
            \draw[wline] (B) -- (BC);
            \draw[wline] ($(CD)-(3,0)$) -- (CD);
            \draw[very thick] (A) -- (B) -- (C) -- (D) -- cycle;
            \filldraw (A) circle [radius=20pt];
            \filldraw (B) circle [radius=20pt];
            \filldraw (C) circle [radius=20pt];
            \filldraw (D) circle [radius=20pt];
        }\ 
        has the same associated crossroad web as the previous example.
        However, it is not an elementary web because its rung is not essential.
        \item The elementary web
        \ \tikz[baseline=-.6ex, scale=.1]{
            \coordinate (P1) at (45:7);
            \coordinate (P2) at (135:7);
            \coordinate (P3) at (-135:7);
            \coordinate (P4) at (-45:7);
            \coordinate (P11) at (30:4);
            \coordinate (P12) at (60:4);
            \coordinate (P21) at (120:4);
            \coordinate (P22) at (150:4);
            \coordinate (P31) at (210:4);
            \coordinate (P32) at (240:4);
            \coordinate (P41) at (300:4);
            \coordinate (P42) at (330:4);
            \draw[wline] (P12) -- (P21);
            \draw[wline] (P22) -- (P31);
            \draw[wline] (P32) -- (P41);
            \draw[wline] (P42) -- (P11);
            \draw[webline] (P1) -- (P11) -- (P12) -- (P1);
            \draw[webline] (P2) -- (P21) -- (P22) -- (P2);
            \draw[webline] (P3) -- (P31) -- (P32) -- (P3);
            \draw[webline] (P4) -- (P41) -- (P42) -- (P4);
            \draw[very thick] (P1) -- (P2) -- (P3) -- (P4) -- cycle;
            \filldraw (P1) circle [radius=20pt];
            \filldraw (P2) circle [radius=20pt];
            \filldraw (P3) circle [radius=20pt];
            \filldraw (P4) circle [radius=20pt];
        }\ 
        is not tree-type, and it is invariant under the  $\pi/2$-rotations of the quadrilateral, which is the cluster Donaldson--Thomas transformation in this case: see \eqref{eq:cluster_DT}.
        \item The underlying graph of the flat $\mathfrak{sp}_4$-graph
        \ \tikz[baseline=-.6ex, scale=.1]{
            \coordinate (P1) at (45:7);
            \coordinate (P2) at (135:7);
            \coordinate (P3) at (-135:7);
            \coordinate (P4) at (-45:7);
            \coordinate (P11) at (30:4);
            \coordinate (P12) at (60:4);
            \coordinate (P21) at (120:4);
            \coordinate (P22) at (150:4);
            \coordinate (P31) at (210:4);
            \coordinate (P32) at (240:4);
            \coordinate (P41) at (300:4);
            \coordinate (P42) at (330:4);
            \draw[wline] (P12) -- (P21);
            \draw[wline] (P22) -- (P31);
            \draw[webline] (P2) -- (P21) -- (P22) -- (P2);
            \draw[webline] (P31) -- (P3);
            \draw[webline] (P31) -- (P4);
            \draw[webline] (P12) -- (P1);
            \draw[webline] (P12) -- (P4);
            \draw[very thick] (P1) -- (P2) -- (P3) -- (P4) -- cycle;
            \filldraw (P1) circle [radius=20pt];
            \filldraw (P2) circle [radius=20pt];
            \filldraw (P3) circle [radius=20pt];
            \filldraw (P4) circle [radius=20pt];
        }\ 
        is a tree.
        However, it decomposes into the product of
        \ \tikz[baseline=-.6ex, scale=.1, yshift=-5cm]{
            \coordinate (A) at (0,0);
            \coordinate (B) at (0,10);
            \coordinate (C) at (10,10);
            \coordinate (D) at (10,0);
            \draw[webline] (B) -- (D);
            \draw[very thick] (A) -- (B) -- (C) -- (D) -- cycle;
            \filldraw (A) circle [radius=20pt];
            \filldraw (B) circle [radius=20pt];
            \filldraw (C) circle [radius=20pt];
            \filldraw (D) circle [radius=20pt];
        }\ 
        and
        \ \tikz[baseline=-.6ex, scale=.1, yshift=-5cm]{
            \coordinate (A) at (0,0);
            \coordinate (B) at (0,10);
            \coordinate (C) at (10,10);
            \coordinate (D) at (10,0);
            \draw[webline] (A) -- (C);
            \draw[webline] (B) -- (D);
            \draw[fill=pink, thick] (5,5) circle [radius=20pt];
            \draw[very thick] (A) -- (B) -- (C) -- (D) -- cycle;
            \filldraw (A) circle [radius=20pt];
            \filldraw (B) circle [radius=20pt];
            \filldraw (C) circle [radius=20pt];
            \filldraw (D) circle [radius=20pt];
        }\ .
    \end{enumerate}
\end{ex}
Certain decompositions of $\mathfrak{sp}_4$-webs into products of elementary webs are obtained by the following \emph{arborization relations}, which are similar to those introduced by Fomin--Pylyavskyy~\cite{FP16} for $\mathfrak{sl}_3$-webs.
\begin{lem}[arborization relations for $\mathfrak{sp}_4$-webs]\label{lem:arborization}
    For $\mathfrak{sp}_4$-webs, we have the following relations
    \begin{align*}
        \mathord{
            \ 
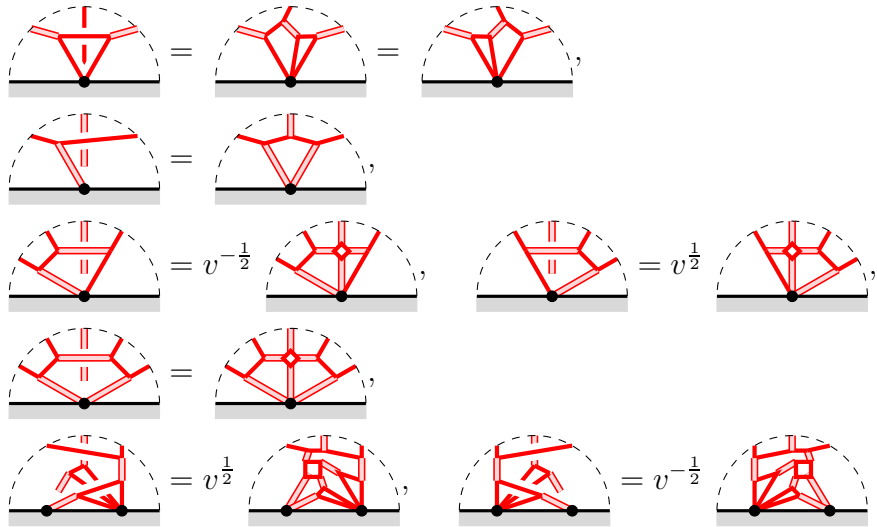
\tikz[baseline=-.6ex, scale=.1, yshift=-4cm]{
                \coordinate (P) at (0,0);
                \coordinate (Q) at (90:10);
                \coordinate (A1) at (135:10);
                \coordinate (A2) at (45:10);
                \coordinate (B1) at (120:7);
                \coordinate (B2) at (60:7);
                \draw[webline, shorten <=.2cm] (P) -- (Q);
                \draw[overarc] (P) -- (B1) -- (B2) -- (P);
                \draw[wline] (B1) -- (A1);
                \draw[wline] (B2) -- (A2);
                \draw[dashed] (10,0) arc (0:180:10cm);
                \bdryline{(-10,0)}{(10,0)}{2cm}
                \draw[fill=black] (P) circle [radius=20pt];
            \ }
        }
        &=
        \mathord{
            \ \tikz[baseline=-.6ex, scale=.1, yshift=-4cm]{
                \coordinate (P) at (0,0);
                \coordinate (Q) at (90:10);
                \coordinate (R) at (90:7);
                \coordinate (A1) at (135:10);
                \coordinate (A2) at (45:10);
                \coordinate (B1) at (120:7);
                \coordinate (B2) at (60:7);
                \draw[webline] (B1) -- ($(R)+(-1,1)$);
                \draw[wline] ($(R)+(1,-1)$) -- ($(R)+(-1,1)$);
                \draw[webline] (P) -- (B2) -- ($(R)+(1,-1)$) -- (P);
                \draw[webline] (P) -- (B1);
                \draw[webline] ($(R)+(-1,1)$)--(Q);
                \draw[wline] (B1) -- (A1);
                \draw[wline] (B2) -- (A2);
                \draw[dashed] (10,0) arc (0:180:10cm);
                \bdryline{(-10,0)}{(10,0)}{2cm}
                \draw[fill=black] (P) circle [radius=20pt];
            \ }
        }
        =
        \mathord{
            \ \tikz[baseline=-.6ex, scale=.1, yshift=-4cm, xscale=-1]{
                \coordinate (P) at (0,0);
                \coordinate (Q) at (90:10);
                \coordinate (R) at (90:7);
                \coordinate (A1) at (135:10);
                \coordinate (A2) at (45:10);
                \coordinate (B1) at (120:7);
                \coordinate (B2) at (60:7);
                \draw[webline] (B1) -- ($(R)+(-1,1)$);
                \draw[wline] ($(R)+(1,-1)$) -- ($(R)+(-1,1)$);
                \draw[webline] (P) -- (B2) -- ($(R)+(1,-1)$) -- (P);
                \draw[webline] (P) -- (B1);
                \draw[webline] ($(R)+(-1,1)$)--(Q);
                \draw[wline] (B1) -- (A1);
                \draw[wline] (B2) -- (A2);
                \draw[dashed] (10,0) arc (0:180:10cm);
                \bdryline{(-10,0)}{(10,0)}{2cm}
                \draw[fill=black] (P) circle [radius=20pt];
            \ }
        },\\
        \mathord{
            \ \tikz[baseline=-.6ex, scale=.1, yshift=-4cm]{
                \coordinate (P) at (0,0);
                \coordinate (Q) at (90:10);
                \coordinate (A1) at (135:10);
                \coordinate (A2) at (45:10);
                \coordinate (B1) at (120:7);
                \coordinate (B2) at (60:7);
                \draw[wline, shorten <=.3cm] (P) -- (Q);
                \draw[overarc] (A1) -- (B1) -- (A2);
                \draw[wline] (P) -- (B1);
                \draw[dashed] (10,0) arc (0:180:10cm);
                \bdryline{(-10,0)}{(10,0)}{2cm}
                \draw[fill=black] (P) circle [radius=20pt];
            \ }
        }
        &=
        \mathord{
            \ \tikz[baseline=-.6ex, scale=.1, yshift=-4cm]{
                \coordinate (P) at (0,0);
                \coordinate (Q) at (90:10);
                \coordinate (R) at (90:7);
                \coordinate (A1) at (135:10);
                \coordinate (A2) at (45:10);
                \coordinate (B1) at (120:7);
                \coordinate (B2) at (60:7);
                \draw[wline] (P) -- (B1);
                \draw[wline] (P) -- (B2);
                \draw[webline] (A1) -- (B1) -- (R) -- (B2) -- (A2);
                \draw[wline] (R) -- (Q);
                \draw[dashed] (10,0) arc (0:180:10cm);
                \bdryline{(-10,0)}{(10,0)}{2cm}
                \draw[fill=black] (P) circle [radius=20pt];
            \ }
        },\\
        \mathord{
            \ \tikz[baseline=-.6ex, scale=.1, yshift=-4cm]{
                \coordinate (P) at (0,0);
                \coordinate (Q) at (90:10);
                \foreach \i in {0,1,...,6}
                {
                    \coordinate (A\i) at (\i*30:7);
                    \coordinate (P\i) at (\i*30:10);
                }
                \draw[wline, shorten <=.3cm] (P) -- (Q);
                \draw[overwline] (A2) -- (A4);
                \draw[webline] (P) -- (A2);
                \draw[wline] (P) -- (A5);
                \draw[webline] (A4) -- (A5);
                \draw[webline] (P5) -- (A5);
                \draw[webline] (P4) -- (A4);
                \draw[webline] (P2) -- (A2);
                \draw[dashed] (10,0) arc (0:180:10cm);
                \bdryline{(-10,0)}{(10,0)}{2cm}
                \draw[fill=black] (P) circle [radius=20pt];
            \ }
        }
        &=q^{-\frac{1}{2}}
        \mathord{
            \ \tikz[baseline=-.6ex, scale=.1, yshift=-4cm]{
                \coordinate (P) at (0,0);
                \coordinate (Q) at (90:10);
                \foreach \i in {0,1,...,6}
                {
                    \coordinate (A\i) at (\i*30:7);
                    \coordinate (P\i) at (\i*30:10);
                }
                \draw[wline] (P) -- (Q);
                \draw[wline] (A2) -- (A4);
                \draw[webline] (P) -- (A2);
                \draw[wline] (P) -- (A5);
                \draw[webline] (A4) -- (A5);
                \draw[webline] (P5) -- (A5);
                \draw[webline] (P4) -- (A4);
                \draw[webline] (P2) -- (A2);
                \draw[webline, fill=white] ($(A3)+(-1,0)-(0,1)$) -- ($(A3)+(0,-1)-(0,1)$) -- ($(A3)+(1,0)-(0,1)$) -- ($(A3)+(0,1)-(0,1)$) -- cycle;
                \draw[dashed] (10,0) arc (0:180:10cm);
                \bdryline{(-10,0)}{(10,0)}{2cm}
                \draw[fill=black] (P) circle [radius=20pt];
            \ }
        },
        \quad
        \mathord{
            \ \tikz[baseline=-.6ex, scale=.1, yshift=-4cm]{
                \coordinate (P) at (0,0);
                \coordinate (Q) at (90:10);
                \foreach \i in {0,1,...,6}
                {
                    \coordinate (A\i) at (\i*30:7);
                    \coordinate (P\i) at (\i*30:10);
                }
                \draw[wline, shorten <=.3cm] (P) -- (Q);
                \draw[overwline] (A2) -- (A4);
                \draw[webline] (P) -- (A4);
                \draw[wline] (P) -- (A1);
                \draw[webline] (A1) -- (A2);
                \draw[webline] (P1) -- (A1);
                \draw[webline] (P2) -- (A2);
                \draw[webline] (P4) -- (A4);
                \draw[dashed] (10,0) arc (0:180:10cm);
                \bdryline{(-10,0)}{(10,0)}{2cm}
                \draw[fill=black] (P) circle [radius=20pt];
            \ }
        }
        =q^{\frac{1}{2}}
        \mathord{
            \ \tikz[baseline=-.6ex, scale=.1, yshift=-4cm]{
                \coordinate (P) at (0,0);
                \coordinate (Q) at (90:10);
                \foreach \i in {0,1,...,6}
                {
                    \coordinate (A\i) at (\i*30:7);
                    \coordinate (P\i) at (\i*30:10);
                }
                \draw[wline] (P) -- (Q);
                \draw[wline] (A2) -- (A4);
                \draw[webline] (P) -- (A4);
                \draw[wline] (P) -- (A1);
                \draw[webline] (A1) -- (A2);
                \draw[webline] (P1) -- (A1);
                \draw[webline] (P2) -- (A2);
                \draw[webline] (P4) -- (A4);
                \draw[webline, fill=white] ($(A3)+(-1,0)-(0,1)$) -- ($(A3)+(0,-1)-(0,1)$) -- ($(A3)+(1,0)-(0,1)$) -- ($(A3)+(0,1)-(0,1)$) -- cycle;
                \draw[dashed] (10,0) arc (0:180:10cm);
                \bdryline{(-10,0)}{(10,0)}{2cm}
                \draw[fill=black] (P) circle [radius=20pt];
            \ }
        },\\
        \mathord{
            \ \tikz[baseline=-.6ex, scale=.1, yshift=-4cm]{
                \coordinate (P) at (0,0);
                \coordinate (Q) at (90:10);
                \foreach \i in {0,1,...,6}
                {
                    \coordinate (A\i) at (\i*30:7);
                    \coordinate (P\i) at (\i*30:10);
                }
                \draw[wline, shorten <=.3cm] (P) -- (Q);
                \draw[overwline] (A2) -- (A4);
                \draw[wline] (P) -- (A1);
                \draw[wline] (P) -- (A5);
                \draw[webline] (A1) -- (A2);
                \draw[webline] (A4) -- (A5);
                \draw[webline] (P5) -- (A5);
                \draw[webline] (P4) -- (A4);
                \draw[webline] (P2) -- (A2);
                \draw[webline] (P1) -- (A1);
                \draw[dashed] (10,0) arc (0:180:10cm);
                \bdryline{(-10,0)}{(10,0)}{2cm}
                \draw[fill=black] (P) circle [radius=20pt];
            \ }
        }
        &=
        \mathord{
            \ \tikz[baseline=-.6ex, scale=.1, yshift=-4cm]{
                \coordinate (P) at (0,0);
                \coordinate (Q) at (90:10);
                \foreach \i in {0,1,...,6}
                {
                    \coordinate (A\i) at (\i*30:7);
                    \coordinate (P\i) at (\i*30:10);
                }
                \draw[wline] (P) -- (Q);
                \draw[wline] (A2) -- (A4);
                \draw[wline] (P) -- (A1);
                \draw[wline] (P) -- (A5);
                \draw[webline] (A1) -- (A2);
                \draw[webline] (A4) -- (A5);
                \draw[webline] (P5) -- (A5);
                \draw[webline] (P4) -- (A4);
                \draw[webline] (P2) -- (A2);
                \draw[webline] (P1) -- (A1);
                \draw[webline, fill=white] ($(A3)+(-1,0)-(0,1)$) -- ($(A3)+(0,-1)-(0,1)$) -- ($(A3)+(1,0)-(0,1)$) -- ($(A3)+(0,1)-(0,1)$) -- cycle;
                \draw[dashed] (10,0) arc (0:180:10cm);
                \bdryline{(-10,0)}{(10,0)}{2cm}
                \draw[fill=black] (P) circle [radius=20pt];
            \ }
        },\\
        \mathord{
            \ \tikz[baseline=-.6ex, scale=.1, yshift=-4cm]{
                \coordinate (P) at (0,0);
                \coordinate (P1) at (-5,0);
                \coordinate (P2) at (5,0);
                \coordinate (Q) at (90:10);
                \coordinate (A1) at (135:10);
                \coordinate (A2) at (45:10);
                \coordinate (B1) at (120:7);
                \coordinate (B2) at (60:7);
                \draw[wline, shorten <=.3cm] (P1) -- ($(P1)+(3,5)$); 
                \draw[wline] ($(P)+(0,6)$) -- ($(P)+(0,10)$);
                \draw[webline] ($(P1)+(3,5)$) -- ($(P)+(0,6)$);
                \draw[webline] (P2) -- ($(P)+(0,6)$);
                \draw[webline] (P2) -- ($(P1)+(3,5)$);
                \draw[wline] (P1) -- ($(P)+(-1,2)$); 
                \draw[overarc] ($(P)+(-1,2)$) -- ($(P2)+(0,4)$);
                \draw[webline] (P2) -- ($(P2)+(0,4)$);
                \draw[webline] (P2) -- ($(P)+(-1,2)$);
                \draw[overarc] ($(P2)+(0,7)$) -- (120:10);
                \draw[webline] ($(P2)+(0,7)$) -- (60:10);
                \draw[wline] ($(P2)+(0,4)$) -- ($(P2)+(0,7)$);
                \draw[dashed] (10,0) arc (0:180:10cm);
                \bdryline{(-10,0)}{(10,0)}{2cm}
                \draw[fill=black] (P1) circle [radius=20pt];
                \draw[fill=black] (P2) circle [radius=20pt];
            \ }
        }
        &=q^{\frac{1}{2}}
        \mathord{
            \ \tikz[baseline=-.6ex, scale=.1, yshift=-4cm]{
                \coordinate (P) at (0,0);
                \coordinate (P1) at (-5,0);
                \coordinate (P2) at (5,0);
                \coordinate (S1) at (-2.5,6.5);
                \coordinate (S2) at (-2.5,4.5);
                \coordinate (S3) at (-.5,4.5);
                \coordinate (S4) at (-.5,6.5);
                \coordinate (T1) at (2,5);
                \coordinate (T2) at (0,3);
                \draw[wline] (S4) -- (T1);
                \draw[wline] (S3) -- (T2);
                \draw[webline] (P2) -- (-1,2) -- (T2) -- (P2);
                \draw[webline] (P2) -- (T1) -- ($(P2)+(0,4)$) -- (P2);
                \draw[wline] (P1) -- ($(P1)+(3,5)$); 
                \draw[wline] (S1) -- (-2,8);
                \draw[wline] (P1) -- (-1,2); 
                \draw[webline] ($(P2)+(0,7)$) -- (0,8);
                \draw[webline] ($(P2)+(0,7)$) -- (60:10);
                \draw[wline] ($(P2)+(0,4)$) -- ($(P2)+(0,7)$);
                \draw[webline] (-2,8) -- (0,8);
                \draw[wline] (0,8) -- (0,10);
                \draw[webline] (-2,8) -- (120:10);
                \draw[webline] (S1) -- (S2) -- (S3) -- (S4) -- cycle;
                \draw[dashed] (10,0) arc (0:180:10cm);
                \bdryline{(-10,0)}{(10,0)}{2cm}
                \draw[fill=black] (P1) circle [radius=20pt];
                \draw[fill=black] (P2) circle [radius=20pt];
            \ }
        },
        \quad
        \mathord{
            \ \tikz[baseline=-.6ex, scale=.1, yshift=-4cm, xscale=-1]{
                \coordinate (P) at (0,0);
                \coordinate (P1) at (-5,0);
                \coordinate (P2) at (5,0);
                \coordinate (Q) at (90:10);
                \coordinate (A1) at (135:10);
                \coordinate (A2) at (45:10);
                \coordinate (B1) at (120:7);
                \coordinate (B2) at (60:7);
                \draw[wline, shorten <=.3cm] (P1) -- ($(P1)+(3,5)$); 
                \draw[wline] ($(P)+(0,6)$) -- ($(P)+(0,10)$);
                \draw[webline] ($(P1)+(3,5)$) -- ($(P)+(0,6)$);
                \draw[webline] (P2) -- ($(P)+(0,6)$);
                \draw[webline] (P2) -- ($(P1)+(3,5)$);
                \draw[wline] (P1) -- ($(P)+(-1,2)$); 
                \draw[overarc] ($(P)+(-1,2)$) -- ($(P2)+(0,4)$);
                \draw[webline] (P2) -- ($(P2)+(0,4)$);
                \draw[webline] (P2) -- ($(P)+(-1,2)$);
                \draw[overarc] ($(P2)+(0,7)$) -- (120:10);
                \draw[webline] ($(P2)+(0,7)$) -- (60:10);
                \draw[wline] ($(P2)+(0,4)$) -- ($(P2)+(0,7)$);
                \draw[dashed] (10,0) arc (0:180:10cm);
                \bdryline{(-10,0)}{(10,0)}{2cm}
                \draw[fill=black] (P1) circle [radius=20pt];
                \draw[fill=black] (P2) circle [radius=20pt];
            \ }
        }
        =q^{-\frac{1}{2}}
        \mathord{
            \ \tikz[baseline=-.6ex, scale=.1, yshift=-4cm, xscale=-1]{
                \coordinate (P) at (0,0);
                \coordinate (P1) at (-5,0);
                \coordinate (P2) at (5,0);
                \coordinate (S1) at (-2.5,6.5);
                \coordinate (S2) at (-2.5,4.5);
                \coordinate (S3) at (-.5,4.5);
                \coordinate (S4) at (-.5,6.5);
                \coordinate (T1) at (2,5);
                \coordinate (T2) at (0,3);
                \draw[wline] (S4) -- (T1);
                \draw[wline] (S3) -- (T2);
                \draw[webline] (P2) -- (-1,2) -- (T2) -- (P2);
                \draw[webline] (P2) -- (T1) -- ($(P2)+(0,4)$) -- (P2);
                \draw[wline] (P1) -- ($(P1)+(3,5)$); 
                \draw[wline] (S1) -- (-2,8);
                \draw[wline] (P1) -- (-1,2); 
                \draw[webline] ($(P2)+(0,7)$) -- (0,8);
                \draw[webline] ($(P2)+(0,7)$) -- (60:10);
                \draw[wline] ($(P2)+(0,4)$) -- ($(P2)+(0,7)$);
                \draw[webline] (-2,8) -- (0,8);
                \draw[wline] (0,8) -- (0,10);
                \draw[webline] (-2,8) -- (120:10);
                \draw[webline] (S1) -- (S2) -- (S3) -- (S4) -- cycle;
                \draw[dashed] (10,0) arc (0:180:10cm);
                \bdryline{(-10,0)}{(10,0)}{2cm}
                \draw[fill=black] (P1) circle [radius=20pt];
                \draw[fill=black] (P2) circle [radius=20pt];
            \ }
        }
    \end{align*}
which we call the arborization relations. 
\end{lem}
\begin{proof}
    One can easily prove them by the skein relations except for the final relation.
    The final relation follow from \cref{lem:face-vanishing} in below.
\end{proof}

\begin{lem}\label{lem:face-vanishing}
    \begin{align*}
        \mathord{
            \ \tikz[baseline=-.6ex, scale=.1]{
                \foreach \i in {0,1,...,7}
                {
                    \coordinate (B\i) at (45*\i:3);
                    \coordinate (A\i) at (45*\i:7);
                }
                \coordinate (P) at (0,0);
                \draw[webline] (A3) -- (B3) -- (B5) -- (B7) -- (B1) -- (A1);
                \draw[wline] (A5) -- (B5);
                \draw[wline] (A7) -- (B7);
                \draw[wline] (B1) -- (B3);
            }\ 
        }
        &=
        \mathord{
            \ \tikz[baseline=-.6ex, scale=.1]{
                \foreach \i in {0,1,...,7}
                {
                    \coordinate (B\i) at (45*\i:3);
                    \coordinate (A\i) at (45*\i:7);
                }
                \coordinate (P) at (0,0);
                \draw[webline] (A3) to[bend right] (A1);
                \draw[wline] (A5) to[bend left] (A7);
            }\ 
        }
        +\frac{1}{[2]}
        \mathord{
            \ \tikz[baseline=-.6ex, scale=.1]{
                \foreach \i in {0,1,...,7}
                {
                    \coordinate (B\i) at (45*\i:3);
                    \coordinate (A\i) at (45*\i:7);
                }
                \coordinate (P) at (0,0);
                \draw[webline] (A3) -- (B4) -- (B0) -- (A1);
                \draw[wline] (A5) -- (B4);
                \draw[wline] (A7) -- (B0);
            }\ 
        },&
        \mathord{
            \ \tikz[baseline=-.6ex, scale=.1, yshift=-5cm]{
                \foreach \i in {0,1,...,6}
                {
                    \coordinate (A\i) at (30*\i:8);
                    \coordinate (B\i) at (30*\i:14);
                }
                \coordinate (P) at (0,0);
                \draw[webline] (P) -- (A2) -- (A4) -- (P);
                \draw[overarc] (A1) -- (A5);
                \draw[webline] (P) -- (A1);
                \draw[webline] (P) -- (A5);
                \draw[wline] (A1) -- (B1);
                \draw[wline] (A2) -- (B2);
                \draw[wline] (A4) -- ($(A4)!.5!(B4)$);
                \draw[wline] (A5) -- ($(A5)!.5!(B5)$);
                \draw[webline] (B4) -- ($(A4)!.5!(B4)$) -- ($(A5)!.5!(B5)$) -- (B5);
                \bdryline{(B6)}{(B0)}{2cm}
                \draw[fill] (P) circle (20pt);
            }\ 
        }
        &=
        \mathord{
            \ \tikz[baseline=-.6ex, scale=.1, yshift=-5cm]{
                \foreach \i in {0,1,...,6}
                {
                    \coordinate (A\i) at (30*\i:8);
                    \coordinate (B\i) at (30*\i:14);
                }
                \coordinate (P) at (0,0);
                \draw[webline] (P) -- (A1) -- (A2) -- (P);
                \draw[webline] (P) -- (B4);
                \draw[webline] (P) -- (B5);
                \draw[wline] (A1) -- (B1);
                \draw[wline] (A2) -- (B2);
                \bdryline{(B6)}{(B0)}{2cm}
                \draw[fill] (P) circle (20pt);
            }\ 
        }.
    \end{align*}
\end{lem}
\begin{proof}
    The first is obtained by replacing the horizontal rung with a vertical one by using \eqref{rel:IH}.
    In the left-hand side of the second equation, we replace a short edge of type~$2$ by using \eqref{rel:IH}, and resolve the left crossing.
    Then, one can apply the first equality to obtain the second equality. 
\end{proof}

If a lift $\tilde{G}$ of a tree-type elementary web $G$ only has internal crossings appearing in the arborization relations, then any tangled tree $\tilde{G}'$ related to $\tilde{G}$ by crossing changes is also a lift of $G$.
Hence we expect the following.
\begin{conj}\label{conj:unique-lift}
    Let $\tilde{G}$ and $\tilde{G}'$ be any tangled tree diagrams such that $\tilde{G}$ is related to $\tilde{G}'$ by crossing changes.
    If $\tilde{G}$ is a lift of a tree-type elementary web $G\in\Tree{\Sigma}$, then $\tilde{G}'$ is also a lift of $G$.
\end{conj}
We remark that the above conjecture is trivial in the classical case $v=1$.
We believe that this conjecture is closely related to \cref{conj:tree-variable}. 

The following notion will be related to the clusters.
\begin{dfn}[web clusters]\label{def:web-cluster}
    A web cluster of $\Skein{\Sigma}$ is a $v$-commutative subset of $\Eweb{\Sigma}$ with cardinality $2\# e(\Delta)+2\# t(\Delta)$, see \cref{subsec:cluster_sp4}.
    We denote the set of web clusters by $\Cweb{\Sigma}$.
\end{dfn}

\begin{dfn}\label{def:web-exchange}
	For two elementary webs $G_{1},G_{2}\in\Eweb{\Sigma}$ contained in a common web cluster, define an integer $\Pi(G_1,G_2)\in\bZ$ by
	\[
		G_{1}G_{2}=q^{\Pi(G_{1},G_{2})}G_{2}G_{1}.
	\]
\end{dfn}

Below, we begin to investigate the cluster nature of our skein algebra $\Skein{\Sigma}$ by discussing the web clusters of triangles.

\subsection{\texorpdfstring{$\mathfrak{sp}_4$}{sp4}-webs in a triangle}
Let $T$ be a triangle, \emph{i.e.} a disk with three special points. In some of the figures below, each special point is ``stretched'' and shown by a thick black segment (like an \emph{external clasp} of Kuperberg). 

\begin{prop}\label{prop:Bweb-T}
    $\Bweb{T}$ consists of the $\mathfrak{sp}_4$-webs represented by the following diagram (and its rotations and reflections along the vertical axis):
    
    
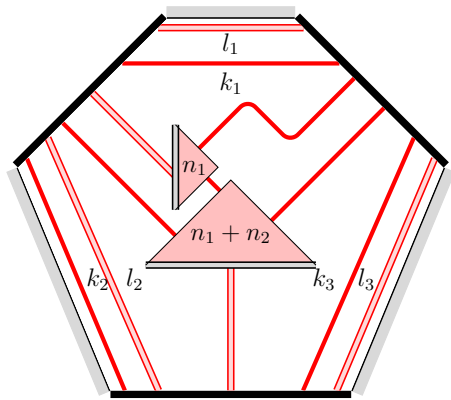
\begin{center}
            \begin{tikzpicture}[scale=.08]
                \coordinate (S11) at ($(135:30)+(45:15)$);
                \coordinate (S12) at ($(135:30)+(45:-20)$);
                \coordinate (S21) at ($(-90:30)+(180:20)$);
                \coordinate (S22) at ($(-90:30)+(180:-20)$);
                \coordinate (S31) at ($(45:30)+(-45:20)$);
                \coordinate (S32) at ($(45:30)+(-45:-15)$);
                \foreach \i [evaluate=\i as \j using int(mod{(\i,3)}+1)] in {1,2,3}
                    {
                        \bdryline{(S\i2)}{(S\j1)}{2cm}
                        \draw[line width=.2cm] (S\i1) -- (S\i2);
                        \fill[white] (S11) -- (S12) -- (S21) -- (S22) -- (S31) -- (S32);
                    }
                \begin{scope}
                    \clip (S11) -- (S12) -- (S21) -- (S22) -- (S31) -- (S32);
                    \draw[wline] ($(S11)!.5!(S12)$) -- +(-45:20);
                    \draw[webline] ($(S11)!.4!(S12)+(-45:20)$) -- ($(S11)!.4!(S12)+(-45:30)$);
                    \draw[webline] ($(S11)!.7!(S12)$) -- +(-45:30);
                    \draw[wline] ($(S21)!.5!(S22)$) -- +(90:30);
                    \draw[webline] ($(S31)!.4!(S32)$) -- +(-135:30);
                    \draw[webline, rounded corners] ($(S31)!.6!(S32)$) -- ++(-135:15) -- ++(135:10) --  ++(-135:15);
                    \draw[wline] ($(S11)!.06!(S12)$) -- ($(S32)!.06!(S31)$);
                    \draw[webline] ($(S11)!.3!(S12)$) -- ($(S32)!.3!(S31)$);
                    \draw[wline] ($(S21)!.2!(S22)$) -- ($(S12)!.2!(S11)$);
                    \draw[webline] ($(S21)!.06!(S22)$) -- ($(S12)!.06!(S11)$);
                    \draw[webline] ($(S31)!.2!(S32)$) -- ($(S22)!.2!(S21)$);
                    \draw[wline] ($(S31)!.06!(S32)$) -- ($(S22)!.06!(S21)$);
                \end{scope}
                \begin{scope}[yshift=5cm]
                    \begin{scope}[rotate=-135]
                        \fill[pink] (0,20) -- (20,0) -- (0,0) -- cycle;
                        \draw[wlineblack] (0,20) -- (20,0);
                        \draw (0,20) -- (0,0) -- (20,0);
                    \end{scope}
                    \begin{scope}[rotate=135, xshift=3cm]
                        \fill[pink] (0,10) -- (10,0) -- (0,0) -- cycle;
                        \draw[wlineblack] (0,10) -- (10,0);
                        \draw (0,10) -- (0,0) -- (10,0);
                    \end{scope}    
                \end{scope}
                \node[scale=0.8] at ($(S11)!.5!(S32)$) [below=17pt]{$k_1$};
                \node[scale=0.8] at ($(S11)!.5!(S32)$) [below=2pt]{$l_1$};
                \node[scale=0.8] at ($(S12)!.5!(S21)$) [right=5pt]{$k_2$};
                \node[scale=0.8] at ($(S12)!.5!(S21)$) [right=20pt]{$l_2$};
                \node[scale=0.8] at ($(S22)!.5!(S31)$) [left=5pt]{$l_3$};
                \node[scale=0.8] at ($(S22)!.5!(S31)$) [left=20pt]{$k_3$};
                \node[scale=0.8] at (0,-4) {$n_1+n_2$};
                \node[scale=0.8] at (-6,7) {$n_1$};
        \end{tikzpicture}
    \end{center}
    where $k_1,k_2,k_3,l_1,l_2,l_3,n_1,n_2\in\bZ_{\geq 0}$.
    Here an edge with a positive integer $k$ represents the $k$-parallelization of the edge, and a triangle with $n$ represents the crossroad web defined by
    \[
        \mathord{
            \ 
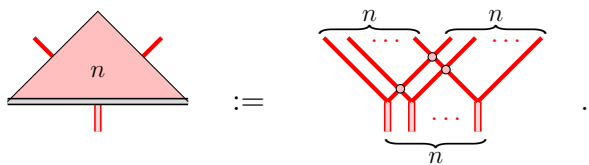
\tikz[baseline=-.6ex, scale=.08]{
                \draw[wline] (0,0) -- (-90:5);
                \draw[webline] (0,0) -- (45:15);
                \draw[webline] (0,0) -- (135:15);
                \fill[pink] (90:15) -- (15,0) -- (-15,0) -- cycle;
                \draw[wlineblack] (-15,0) -- (15,0);
                \draw (-15,0) -- (90:15) -- (15,0);
                \node[scale=.8] at (0,5) {$n$};
        \ }
        }
        \quad :=
        \mathord{
            \ \tikz[baseline=-.6ex, scale=.08]{
                \begin{scope}[xshift=0cm]
                    \draw[wline] (0,0) -- (-90:5);
                    \draw[webline] (0,0) -- (45:15);
                    \draw[webline] (0,0) -- (135:15);    
                \end{scope}
                \begin{scope}[xshift=4cm]
                    \draw[wline] (0,0) -- (-90:5);
                    \draw[webline] (0,0) -- (45:15);
                    \draw[webline] (0,0) -- (135:15);    
                \end{scope}
                \begin{scope}[xshift=15cm]
                    \draw[wline] (0,0) -- (-90:5);
                    \draw[webline] (0,0) -- (45:15);
                    \draw[webline] (0,0) -- (135:15);    
                \end{scope}
                \draw[fill=pink] ($(0,0)!.2!(45:15)$) circle [radius=20pt];
                \draw[fill=pink] ($(0,0)!.7!(45:15)$) circle [radius=20pt];
                \draw[fill=pink] ($(0,0)!.5!(135:15)+(15,0)$) circle [radius=20pt];
                \node[scale=.8] at (0,10) [red]{$\cdots$};
                \node[scale=.8] at (20,10) [red]{$\cdots$};
                \node[scale=.8] at (10,-3) [red]{$\cdots$};
                \node[scale=.8] at (-3,12) [rotate=-90, yscale=4]{$\{$};
                \node[scale=.8] at (18,12) [rotate=-90, yscale=4]{$\{$};
                \node[scale=.8] at (8,-7) [rotate=90, yscale=4]{$\{$};
                \node[scale=.8] at (-3,12) [above]{$n$};
                \node[scale=.8] at (18,12) [above]{$n$};
                \node[scale=.8] at (8,-7) [below]{$n$};
        \ }
        }.
    \]
\end{prop}
\begin{proof}
    We use a formula on the angular defect of a flat crossroad web developed in \cite{Kuperberg96}.
    For a trivalent vertex and a crossroad, we assign angles between their edges by 
    \[
        \mathord{
            \ 
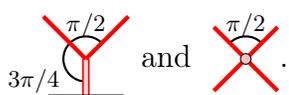
\tikz[baseline=-.6ex, scale=.1]{
                \coordinate (C) at (0,0);
                \coordinate (L) at (135:8);
                \coordinate (R) at (45:8);
                \coordinate (S) at (-90:5);
                \pic [draw,thick,"\scriptsize $\pi/2$", angle eccentricity=1.5, angle radius=.3cm] {angle = R--C--L};
                \pic [draw,thick,"\scriptsize $3\pi/4$", angle eccentricity=2.5, angle radius=.3cm] {angle = L--C--S};
                \draw[wline] (0,0) -- (-90:5);
                \draw[webline] (0,0) -- (45:8);
                \draw[webline] (0,0) -- (135:8);
                \draw[weblineblack] ($(-90:5)+(-5,0)$) -- +(10,0);
        \ }
        } \text{ and }
        \mathord{
            \ \tikz[baseline=-.6ex, scale=.1]{
                \coordinate (C) at (0,0);
                \coordinate (NW) at (135:6);
                \coordinate (NE) at (45:6);
                \coordinate (SE) at (-45:6);
                \coordinate (SW) at (-135:6);
                \pic [draw,thick,"\scriptsize $\pi/2$", angle eccentricity=1.5, angle radius=.3cm] {angle = NE--C--NW};
                \draw[webline] (C) -- (NW);
                \draw[webline] (C) -- (NE);
                \draw[webline] (C) -- (SW);
                \draw[webline] (C) -- (SE);
                \draw[fill=pink] (C) circle [radius=20pt];
        \ }
        }.
    \]
    Let $G$ be a connected component of a basis web in a triangle whose ends are located on the three external clasps.
    We consider the dual graph $G^{\ast}$ and assign angles, as shown in green:
    \[
        \mathord{
            \ 
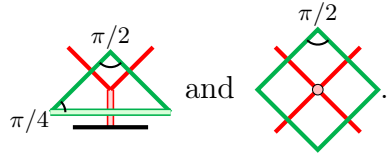
\tikz[baseline=-.6ex, scale=.1]{
                \coordinate (C) at (0,0);
                \coordinate (NW) at (135:8);
                \coordinate (NE) at (45:8);
                \coordinate (S) at (-90:5);
                \coordinate (N) at ($(90:8)+(0,-3)$);
                \coordinate (SW) at ($(-180:8)+(0,-3)$);
                \coordinate (SE) at ($(0:8)+(0,-3)$);
                \draw[wline] (C) -- (S);
                \draw[webline] (C) -- (NE);
                \draw[webline] (C) -- (NW);
                \draw[weblinegreen] (SW) -- (N) -- (SE);
                \draw[wlinegreen] (SW) -- (SE);
                \pic [draw,thick,"\scriptsize $\pi/2$", angle eccentricity=-1, angle radius=.2cm] {angle = SW--N--SE};
                \pic [draw,thick,"\scriptsize $\pi/4$", angle eccentricity=-1.5, angle radius=.2cm] {angle = SE--SW--N};
                \draw[weblineblack] ($(-90:5)+(-5,0)$) -- +(10,0);
        \ }
        } \text{ and }
        \mathord{
            \ \tikz[baseline=-.6ex, scale=.1]{
                \coordinate (C) at (0,0);
                \coordinate (NW) at (135:8);
                \coordinate (NE) at (45:8);
                \coordinate (N) at (90:8);
                \coordinate (SW) at (-135:8);
                \coordinate (SE) at (-45:8);
                \coordinate (S) at (-90:8);
                \coordinate (E) at (180:8);
                \coordinate (W) at (0:8);
                \draw[webline] (C) -- (NW);
                \draw[webline] (C) -- (NE);
                \draw[webline] (C) -- (SW);
                \draw[webline] (C) -- (SE);
                \draw[weblinegreen] (N) -- (W) -- (S) -- (E) -- cycle;
                \pic [draw,thick,"\scriptsize $\pi/2$", angle eccentricity=-1, angle radius=.2cm] {angle = E--N--W};
                \draw[fill=pink] (C) circle [radius=20pt];
        \ }
        }.
    \]
    The angular defect of $G$ is defined by
    \[
        \sum_{v^{\mathrm{int}}}(2\pi-\theta_{v^{\mathrm{int}}})+\sum_{v^{\mathrm{ex}}}(\pi-\theta_{v^{\mathrm{ex}}}).
    \]
    The summation takes over the set of all internal vertices $v^{\mathrm{int}}$ and external vertices $v^{\mathrm{ex}}$ of $G^{\ast}$, and $\theta_{v^{\mathrm{ex}}}$ (resp.~$\theta_{v^{\mathrm{ex}}}$) is the summation of angles at $v^{\mathrm{int}}$ (resp.~$v^{\mathrm{ex}}$). 
    If $v^{\mathrm{int}}$ or $v^{\mathrm{ex}}$ lies in an elliptic face, then the corresponding summand becomes negative. 
    A non-elliptic face of $G$ corresponds to an interior vertex of $G^{\ast}$ which has $n>4$ edges of type $1$, or an exterior vertex of $G^{\ast}$ which has $n_1$ edges of type $1$ and $n_2$ edges of type $2$.
    Then, $n_1$ and $n_2$ satisfy either
    \begin{enumerate}
        \item $n_1\geq 3$ and $n_2=0$, 
        \item $n_1\geq 3$ and $n_2=1$, or
        \item $n_1\geq 2$ and $n_2=2$
    \end{enumerate}
    for the vertex bounded by an external clasp, and either
    \begin{enumerate}\setcounter{enumi}{3}
        \item $n_1\geq 2$ and $n_2=0$,
        \item $n_1\geq 1$ and $n_2=1$, or
        \item $n_1\geq 1$ and $n_2=2$
    \end{enumerate}
    for the vertex bounded by a boundary interval.
    By the Gauss-Bonnet theorem, the total angular defect should be equal to $2\pi$.
    Hence, one obtains the following formula:
    \begin{align*}
        &2\pi=\sum_{n\geq 4}v_{n}^{\mathrm{int}}\left(2\pi-\frac{\pi}{2}n\right)+\sum_{n_1,n_2}v_{n_1,n_2}^{\mathrm{ex}}\left(\frac{3\pi}{2}-\frac{\pi}{4}(n+n_1)\right)+\sum_{n_1,n_2}v_{n_1,n_2}^{\mathrm{cl}}\left(\frac{3\pi}{2}-\frac{\pi}{4}(n+n_1)\right),\\
        &\sum_{n}v_n^{\mathrm{ex}}=3,
    \end{align*}
    where $v_{n}^{\mathrm{int}}$, $v_{n_1,n_2}^{\mathrm{ex}}$, and $v_{n_1,n_2}^{\mathrm{cl}}$ are the number of interior vertices with valency $n$, exterior vertices with valency $n=n_1+n_2$ bounded by boundary intervals, and exterior vertices with valency $n=n_1+n_2$ bounded by external clasps, respectively.
    In the right-hand side of the first formula, the first term related to the interior vertices is non-negative, and the summand $3\pi/2-\pi(n+n_1)/4$ in the second and the third terms are shown in \cref{tbl:deficiency}.
    The angular defect at an internal vertex belongs to $\{0,-\pi/2,-\pi,\ldots\}$.
    \begin{table}
        \begin{center}
            \begin{tabular}{|| c || c | c | c | c | c | c ||} 
                \hline
                $n_1$ & $1$ & $2$ & $3$ & $4$ & $5$ & $\cdots$\\ 
                \hline\hline
                (1) $n_2=0$ & -- & -- & $0$ & $-\pi/2$ & $-\pi$ & $\cdots$ \\ 
                \hline
                (2) $n_2=1$ & -- & -- & $-\pi/4$ & $-3\pi/4$ & $-5\pi/4$ & $\cdots$ \\
                \hline
                (3) $n_2=2$ & -- & $0$ & $-\pi/2$ & $-\pi$ & $-3\pi/2$ & $\cdots$ \\
                \hline\hline
                (4) $n_2=0$ & -- & \cellcolor{yellow} $\pi/2$ & $0$ & $-\pi/2$ & $-\pi$ & $\cdots$\\
                \hline
                (5) $n_2=1$ & \cellcolor{yellow} $3\pi/4$ & \cellcolor{yellow} $\pi/4$ & $-\pi/4$ & $-3\pi/4$ & $-5\pi/4$ & $\cdots$ \\ 
                \hline
                (6) $n_2=2$ & \cellcolor{yellow} $\pi/2$ & $0$ & $-\pi/2$ & $-\pi$ & $-3\pi/2$ & $\cdots$ \\ 
                \hline
            \end{tabular}
        \end{center}
        \caption{Entries are angular defect $3\pi/2-\pi(n+n_1)/4$ at external vertices where $n=n_1+n_2$. Yellow-colored cells have positive defects.}
        \label{tbl:deficiency}
    \end{table}
    Therefore the possibility of the total angular defect at each external clasps, boundary edges, and the interior is one of the following patterns (up to rotations and reflections):
    \begin{gather*}
        \mathord{
            \ 
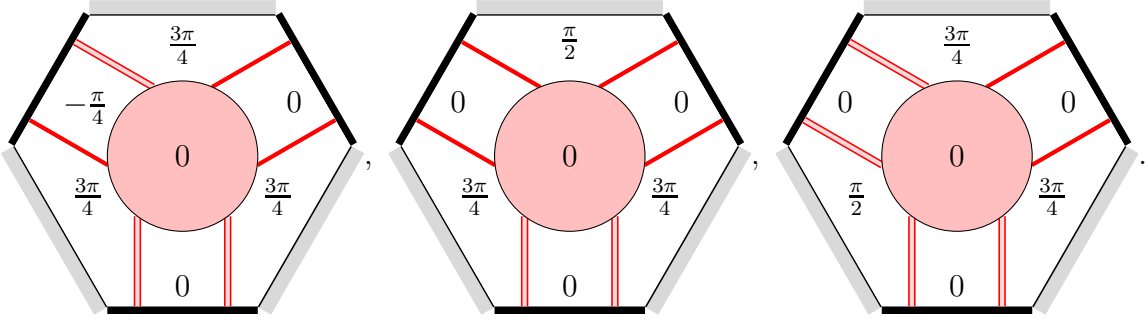
\tikz[baseline=-.6ex, scale=.1]{
                \coordinate (S11) at ($(30:20)+(120:-10)$);
                \coordinate (S12) at ($(30:20)+(120:10)$);
                \coordinate (S21) at ($(150:20)+(60:10)$);
                \coordinate (S22) at ($(150:20)+(60:-10)$);
                \coordinate (S31) at ($(270:20)+(0:-10)$);
                \coordinate (S32) at ($(270:20)+(0:10)$);
                \foreach \i [evaluate=\i as \j using int(mod{(\i,3)}+1)] in {1,2,3}
                    {
                        \bdryline{(S\i2)}{(S\j1)}{2cm}
                        \draw[line width=.2cm] (S\i1) -- (S\i2);
                        \fill[white] (S11) -- (S12) -- (S21) -- (S22) -- (S31) -- (S32);
                    }
                \begin{scope}
                    \clip (S11) -- (S12) -- (S21) -- (S22) -- (S31) -- (S32);
                    \draw[webline] ($(S11)!.2!(S12)$) -- +(210:12);
                    \draw[webline] ($(S11)!.8!(S12)$) -- +(210:12);
                    \draw[wline] ($(S21)!.2!(S22)$) -- +(-30:12);
                    \draw[webline] ($(S21)!.8!(S22)$) -- +(-30:12);
                    \draw[wline] ($(S31)!.2!(S32)$) -- +(90:12);
                    \draw[wline] ($(S31)!.8!(S32)$) -- +(90:12);
                \end{scope}
                \draw[fill=pink] (0,0) circle [radius=10cm]; 
                \node at ($(S11)!.5!(S12)$) [below left]{$0$};
                \node at ($(S21)!.5!(S22)$) [below right]{$-\frac{\pi}{4}$};
                \node at ($(S31)!.5!(S32)$) [above]{$0$};
                \node at ($(S11)!.5!(S32)$) [above left]{$\frac{3\pi}{4}$};
                \node at ($(S12)!.5!(S21)$) [below]{$\frac{3\pi}{4}$};
                \node at ($(S22)!.5!(S31)$) [above right]{$\frac{3\pi}{4}$};
                \node at (0,0) {$0$};
        \ }
        },
        \mathord{
            \ \tikz[baseline=-.6ex, scale=.1]{
                \coordinate (S11) at ($(30:20)+(120:-10)$);
                \coordinate (S12) at ($(30:20)+(120:10)$);
                \coordinate (S21) at ($(150:20)+(60:10)$);
                \coordinate (S22) at ($(150:20)+(60:-10)$);
                \coordinate (S31) at ($(270:20)+(0:-10)$);
                \coordinate (S32) at ($(270:20)+(0:10)$);
                \foreach \i [evaluate=\i as \j using int(mod{(\i,3)}+1)] in {1,2,3}
                    {
                        \bdryline{(S\i2)}{(S\j1)}{2cm}
                        \draw[line width=.2cm] (S\i1) -- (S\i2);
                        \fill[white] (S11) -- (S12) -- (S21) -- (S22) -- (S31) -- (S32);
                    }
                \begin{scope}
                    \clip (S11) -- (S12) -- (S21) -- (S22) -- (S31) -- (S32);
                    \draw[webline] ($(S11)!.2!(S12)$) -- +(210:12);
                    \draw[webline] ($(S11)!.8!(S12)$) -- +(210:12);
                    \draw[webline] ($(S21)!.2!(S22)$) -- +(-30:12);
                    \draw[webline] ($(S21)!.8!(S22)$) -- +(-30:12);
                    \draw[wline] ($(S31)!.2!(S32)$) -- +(90:12);
                    \draw[wline] ($(S31)!.8!(S32)$) -- +(90:12);
                \end{scope}
                \draw[fill=pink] (0,0) circle [radius=10cm]; 
                \node at ($(S11)!.5!(S12)$) [below left]{$0$};
                \node at ($(S21)!.5!(S22)$) [below right]{$0$};
                \node at ($(S31)!.5!(S32)$) [above]{$0$};
                \node at ($(S11)!.5!(S32)$) [above left]{$\frac{3\pi}{4}$};
                \node at ($(S12)!.5!(S21)$) [below]{$\frac{\pi}{2}$};
                \node at ($(S22)!.5!(S31)$) [above right]{$\frac{3\pi}{4}$};
                \node at (0,0) {$0$};
        \ }
        },
        \mathord{
            \ \tikz[baseline=-.6ex, scale=.1]{
                \coordinate (S11) at ($(30:20)+(120:-10)$);
                \coordinate (S12) at ($(30:20)+(120:10)$);
                \coordinate (S21) at ($(150:20)+(60:10)$);
                \coordinate (S22) at ($(150:20)+(60:-10)$);
                \coordinate (S31) at ($(270:20)+(0:-10)$);
                \coordinate (S32) at ($(270:20)+(0:10)$);
                \foreach \i [evaluate=\i as \j using int(mod{(\i,3)}+1)] in {1,2,3}
                    {
                        \bdryline{(S\i2)}{(S\j1)}{2cm}
                        \draw[line width=.2cm] (S\i1) -- (S\i2);
                        \fill[white] (S11) -- (S12) -- (S21) -- (S22) -- (S31) -- (S32);
                    }
                \begin{scope}
                    \clip (S11) -- (S12) -- (S21) -- (S22) -- (S31) -- (S32);
                    \draw[webline] ($(S11)!.2!(S12)$) -- +(210:12);
                    \draw[webline] ($(S11)!.8!(S12)$) -- +(210:12);
                    \draw[wline] ($(S21)!.2!(S22)$) -- +(-30:12);
                    \draw[wline] ($(S21)!.8!(S22)$) -- +(-30:12);
                    \draw[wline] ($(S31)!.2!(S32)$) -- +(90:12);
                    \draw[wline] ($(S31)!.8!(S32)$) -- +(90:12);
                \end{scope}
                \draw[fill=pink] (0,0) circle [radius=10cm]; 
                \node at ($(S11)!.5!(S12)$) [below left]{$0$};
                \node at ($(S21)!.5!(S22)$) [below right]{$0$};
                \node at ($(S31)!.5!(S32)$) [above]{$0$};
                \node at ($(S11)!.5!(S32)$) [above left]{$\frac{3\pi}{4}$};
                \node at ($(S12)!.5!(S21)$) [below]{$\frac{3\pi}{4}$};
                \node at ($(S22)!.5!(S31)$) [above right]{$\frac{\pi}{2}$};
                \node at (0,0) {$0$};
        \ }
        }.
    \end{gather*}
    The faces of $G$ corresponding to the vertices of $G^{\ast}$ with the above angular defects are
    \[
        \mathord{
            \ 
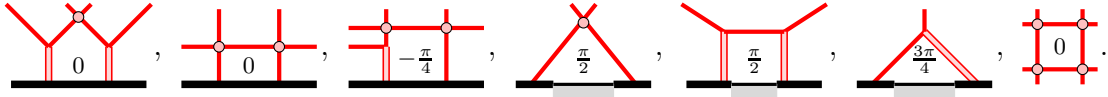
\tikz[baseline=-.6ex, scale=.1]{
                \coordinate (C) at (0,0);
                \coordinate (NW) at (135:8);
                \coordinate (NE) at (45:8);
                \coordinate (S) at (-90:5);
                \coordinate (N) at ($(90:8)+(0,-3)$);
                \coordinate (SW) at ($(-180:8)+(0,-3)$);
                \coordinate (SE) at ($(0:8)+(0,-3)$);
                \coordinate (C2) at ([xshift=8cm]C);
                \coordinate (NW2) at ([xshift=8cm]NW);
                \coordinate (NE2) at ([xshift=8cm]NE);
                \coordinate (S2) at ([xshift=8cm]S);
                \coordinate (N2) at ([xshift=8cm]N);
                \coordinate (SW2) at ([xshift=8cm]SW);
                \coordinate (SE2) at ([xshift=8cm]SE);
                \begin{scope}
                    \draw[wline] (C) -- (S);
                    \draw[webline] (C) -- (NE);
                    \draw[webline] (C) -- (NW);
                \end{scope}
                \begin{scope}
                    \draw[wline] (C2) -- (S2);
                    \draw[webline] (C2) -- (NE2);
                    \draw[webline] (C2) -- (NW2);
                \end{scope}
                \draw[fill=pink] ($(C)!.7!(NE)$) circle [radius=20pt];
                \draw[line width=.1cm] ($(-90:5)+(-5,0)$) -- +(18,0);
                \node at ($(S)+(4,0)$)[above]{\scriptsize $0$};
        \ }
        },
        \mathord{
            \ \tikz[baseline=-.6ex, scale=.1]{
                \coordinate (C) at (0,0);
                \coordinate (NW) at (135:8);
                \coordinate (NE) at (45:8);
                \coordinate (S) at (-90:5);
                \coordinate (N) at ($(90:8)+(0,-3)$);
                \coordinate (SW) at ($(-180:8)+(0,-3)$);
                \coordinate (SE) at ($(0:8)+(0,-3)$);
                \coordinate (C2) at ([xshift=8cm]C);
                \coordinate (NW2) at ([xshift=8cm]NW);
                \coordinate (NE2) at ([xshift=8cm]NE);
                \coordinate (S2) at ([xshift=8cm]S);
                \coordinate (N2) at ([xshift=8cm]N);
                \coordinate (SW2) at ([xshift=8cm]SW);
                \coordinate (SE2) at ([xshift=8cm]SE);
                \begin{scope}
                    \draw[webline] (S) -- (N);
                    \draw[webline] (S2) -- (N2);
                    \draw[webline] ($(C)+(-5,0)$) -- ($(C)+(13,0)$);
                \end{scope}
                \draw[fill=pink] (C) circle [radius=20pt];
                \draw[fill=pink] (C2) circle [radius=20pt];
                \draw[line width=.1cm] ($(-90:5)+(-5,0)$) -- +(18,0);
                \node at ($(S)+(4,0)$)[above]{\scriptsize $0$};
        \ }
        },
        \mathord{
            \ \tikz[baseline=-.6ex, scale=.1]{
                \coordinate (C) at (0,0);
                \coordinate (NW) at (135:8);
                \coordinate (NE) at (45:8);
                \coordinate (S) at (-90:5);
                \coordinate (N) at ($(90:8)+(0,-3)$);
                \coordinate (SW) at ($(-180:8)+(0,-3)$);
                \coordinate (SE) at ($(0:8)+(0,-3)$);
                \coordinate (C2) at ([xshift=8cm]C);
                \coordinate (NW2) at ([xshift=8cm]NW);
                \coordinate (NE2) at ([xshift=8cm]NE);
                \coordinate (S2) at ([xshift=8cm]S);
                \coordinate (N2) at ([xshift=8cm]N);
                \coordinate (SW2) at ([xshift=8cm]SW);
                \coordinate (SE2) at ([xshift=8cm]SE);
                \begin{scope}
                    \draw[webline] ($(C)+(-5,0)$) -- (C) -- (N);
                    \draw[wline] (S) -- (C);
                    \draw[webline] ($(C)!.5!(N)+(-5,0)$) -- ($(C)!.5!(N)+(13,0)$);
                    \draw[webline] (S2) -- (N2);
                \end{scope}
                \draw[fill=pink] ($(C)!.5!(N)$) circle [radius=20pt];
                \draw[fill=pink] ($(C2)!.5!(N2)$) circle [radius=20pt];
                \draw[line width=.1cm] ($(-90:5)+(-5,0)$) -- +(18,0);
                \node at ($(S)+(4,0)$)[above]{\scriptsize $-\frac{\pi}{4}$};
        \ }
        },
        \mathord{
            \ \tikz[baseline=-.6ex, scale=.1]{
                \coordinate (C) at (0,0);
                \coordinate (NW) at (135:8);
                \coordinate (NE) at (45:8);
                \coordinate (S) at (-90:5);
                \coordinate (N) at ($(90:8)+(0,-3)$);
                \coordinate (SW) at ($(-180:8)+(0,-3)$);
                \coordinate (SE) at ($(0:8)+(0,-3)$);
                \coordinate (C2) at ([xshift=8cm]C);
                \coordinate (NW2) at ([xshift=8cm]NW);
                \coordinate (NE2) at ([xshift=8cm]NE);
                \coordinate (S2) at ([xshift=8cm]S);
                \coordinate (N2) at ([xshift=8cm]N);
                \coordinate (SW2) at ([xshift=8cm]SW);
                \coordinate (SE2) at ([xshift=8cm]SE);
                \begin{scope}
                    \draw[webline] ($(S)+(-3,0)$) -- (NE);
                    \draw[webline] ($(S2)+(3,0)$) -- (NW2);
                \end{scope}
                \draw[fill=pink] ($(S)!.5!(S2)+(0,8.2)$) circle [radius=20pt];
                \draw[line width=.1cm] ($(S)+(-5,0)$) -- (S);
                \draw[line width=.1cm] ($(S2)+(5,0)$) -- (S2);
                \bdryline{(S)}{(S2)}{2cm}
                \node at ($(S)+(4,0)$)[above]{\scriptsize $\frac{\pi}{2}$};
        \ }
        },
        \mathord{
            \ \tikz[baseline=-.6ex, scale=.1]{
                \coordinate (C) at (0,0);
                \coordinate (NW) at (135:8);
                \coordinate (NE) at (45:8);
                \coordinate (S) at (-90:5);
                \coordinate (N) at ($(90:8)+(0,-3)$);
                \coordinate (SW) at ($(-180:8)+(0,-3)$);
                \coordinate (SE) at ($(0:8)+(0,-3)$);
                \coordinate (C2) at ([xshift=8cm]C);
                \coordinate (NW2) at ([xshift=8cm]NW);
                \coordinate (NE2) at ([xshift=8cm]NE);
                \coordinate (S2) at ([xshift=8cm]S);
                \coordinate (N2) at ([xshift=8cm]N);
                \coordinate (SW2) at ([xshift=8cm]SW);
                \coordinate (SE2) at ([xshift=8cm]SE);
                \begin{scope}
                    \draw[webline] ($(C)+(0,2)$) -- (NW);
                    \draw[webline] ($(C2)+(0,2)$) -- (NE2);
                    \draw[wline] ($(C)+(0,2)$) -- (S);
                    \draw[wline] ($(C2)+(0,2)$) -- (S2);
                    \draw[webline] ($(C)+(0,2)$) -- ($(C2)+(0,2)$);
                \end{scope}
                \draw[line width=.1cm] ($(S)+(-5,0)$) -- ($(S)+(1,0)$);
                \draw[line width=.1cm] ($(S2)+(5,0)$) -- ($(S2)+(-1,0)$);
                \bdryline{($(S)+(1,0)$)}{($(S2)+(-1,0)$)}{2cm}
                \node at ($(S)!.5!(S2)$)[above]{\scriptsize $\frac{\pi}{2}$};
        \ }
        },
        \mathord{
            \ \tikz[baseline=-.6ex, scale=.1]{
                \coordinate (C) at (0,0);
                \coordinate (NW) at (135:8);
                \coordinate (NE) at (45:8);
                \coordinate (S) at (-90:5);
                \coordinate (N) at ($(90:8)+(0,-3)$);
                \coordinate (SW) at ($(-180:8)+(0,-3)$);
                \coordinate (SE) at ($(0:8)+(0,-3)$);
                \coordinate (C2) at ([xshift=8cm]C);
                \coordinate (NW2) at ([xshift=8cm]NW);
                \coordinate (NE2) at ([xshift=8cm]NE);
                \coordinate (S2) at ([xshift=8cm]S);
                \coordinate (N2) at ([xshift=8cm]N);
                \coordinate (SW2) at ([xshift=8cm]SW);
                \coordinate (SE2) at ([xshift=8cm]SE);
                \coordinate (CC) at ($(S)!.5!(S2)+(0,7)$);
                \coordinate (NCC) at ($(S)!.5!(S2)+(0,10)$);
                \begin{scope}
                    \draw[webline] ($(S)+(-3,0)$) -- (CC);
                    \draw[wline] ($(S2)+(3,0)$) -- (CC);
                    \draw[webline] (CC) -- (NCC);
                \end{scope}
                \draw[line width=.1cm] ($(S)+(-5,0)$) -- (S);
                \draw[line width=.1cm] ($(S2)+(5,0)$) -- (S2);
                \bdryline{(S)}{(S2)}{2cm}
                \node at ($(S)+(4,0)$)[above]{\scriptsize $\frac{3\pi}{4}$};
        \ }
        },
        \mathord{
            \ \tikz[baseline=-.6ex, scale=.1]{
                \draw[webline] (-3,-5) -- +(0,10);
                \draw[webline] (3,-5) -- +(0,10);
                \draw[webline] (-5,3) -- +(10,0);
                \draw[webline] (-5,-3) -- +(10,0);
                \draw[fill=pink] (3,3) circle [radius=20pt];
                \draw[fill=pink] (-3,3) circle [radius=20pt];
                \draw[fill=pink] (3,-3) circle [radius=20pt];
                \draw[fill=pink] (-3,-3) circle [radius=20pt];
                \node at (0,0){\scriptsize $0$};
        \ }
        }.
    \]
    As a consequence, we obtain the diagrams in the statement.
    The second and the third diagrams are the cases of $n_1=0$ and $n_2=0$, respectively.
\end{proof}

\begin{thm}\label{thm:Eweb-T}
    $\Eweb{T}$ consists of the boundary webs in $\partial_T$ and the following $6$ webs:
    \begin{align*}
        \mathord{
            \ 
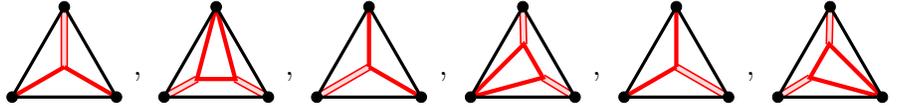
\tikz[baseline=-.6ex, scale=.1]{
                \coordinate (A) at (90:8);
                \coordinate (B) at (210:8);
                \coordinate (C) at (-30:8);
                \coordinate (P) at (0,0);
                \draw[wline] (P) -- (A);
                \draw[webline] (P) -- (B);
                \draw[webline] (P) -- (C);
                \draw[very thick] (A) -- (B) -- (C) -- cycle;
                \filldraw (A) circle [radius=20pt];
                \filldraw (B) circle [radius=20pt];
                \filldraw (C) circle [radius=20pt];
            }\ 
        },
        \mathord{
            \ \tikz[baseline=-.6ex, scale=.1]{
                \coordinate (A) at (90:8);
                \coordinate (B) at (210:8);
                \coordinate (C) at (-30:8);
                \coordinate (P1) at ($(A)!.8!(B)+(3,0)$);
                \coordinate (P2) at ($(A)!.8!(C)+(-3,0)$);
                \coordinate (P) at (0,0);
                \draw[wline] (P1) -- (B);
                \draw[wline] (P2) -- (C);
                \draw[webline] (A) -- (P1) -- (P2) -- (A);
                \draw[very thick] (A) -- (B) -- (C) -- cycle;
                \filldraw (A) circle [radius=20pt];
                \filldraw (B) circle [radius=20pt];
                \filldraw (C) circle [radius=20pt];
            }\ 
        },
        \mathord{
            \ \tikz[baseline=-.6ex, scale=.1, rotate=120]{
                \coordinate (A) at (90:8);
                \coordinate (B) at (210:8);
                \coordinate (C) at (-30:8);
                \coordinate (P) at (0,0);
                \draw[wline] (P) -- (A);
                \draw[webline] (P) -- (B);
                \draw[webline] (P) -- (C);
                \draw[very thick] (A) -- (B) -- (C) -- cycle;
                \filldraw (A) circle [radius=20pt];
                \filldraw (B) circle [radius=20pt];
                \filldraw (C) circle [radius=20pt];
            }\ 
        },
        \mathord{
            \ \tikz[baseline=-.6ex, scale=.1, rotate=120]{
                \coordinate (A) at (90:8);
                \coordinate (B) at (210:8);
                \coordinate (C) at (-30:8);
                \coordinate (P1) at ($(A)!.8!(B)+(3,0)$);
                \coordinate (P2) at ($(A)!.8!(C)+(-3,0)$);
                \coordinate (P) at (0,0);
                \draw[wline] (P1) -- (B);
                \draw[wline] (P2) -- (C);
                \draw[webline] (A) -- (P1) -- (P2) -- (A);
                \draw[very thick] (A) -- (B) -- (C) -- cycle;
                \filldraw (A) circle [radius=20pt];
                \filldraw (B) circle [radius=20pt];
                \filldraw (C) circle [radius=20pt];
            }\ 
        },
        \mathord{
            \ \tikz[baseline=-.6ex, scale=.1, rotate=-120]{
                \coordinate (A) at (90:8);
                \coordinate (B) at (210:8);
                \coordinate (C) at (-30:8);
                \coordinate (P) at (0,0);
                \draw[wline] (P) -- (A);
                \draw[webline] (P) -- (B);
                \draw[webline] (P) -- (C);
                \draw[very thick] (A) -- (B) -- (C) -- cycle;
                \filldraw (A) circle [radius=20pt];
                \filldraw (B) circle [radius=20pt];
                \filldraw (C) circle [radius=20pt];
            }\ 
        },
        \mathord{
            \ \tikz[baseline=-.6ex, scale=.1, rotate=-120]{
                \coordinate (A) at (90:8);
                \coordinate (B) at (210:8);
                \coordinate (C) at (-30:8);
                \coordinate (P1) at ($(A)!.8!(B)+(3,0)$);
                \coordinate (P2) at ($(A)!.8!(C)+(-3,0)$);
                \coordinate (P) at (0,0);
                \draw[wline] (P1) -- (B);
                \draw[wline] (P2) -- (C);
                \draw[webline] (A) -- (P1) -- (P2) -- (A);
                \draw[very thick] (A) -- (B) -- (C) -- cycle;
                \filldraw (A) circle [radius=20pt];
                \filldraw (B) circle [radius=20pt];
                \filldraw (C) circle [radius=20pt];
            }\ 
        },
    \end{align*}
    and a $\bZ_q$-subalgebra generated by these webs coincides with $\bZ_q\Bweb{T}$.
\end{thm}
\begin{proof}
    These webs are indecomposable because of the degrees at special points.
    One can see that a basis web decomposes into the above pieces by using the arborization relations in \cref{lem:arborization}.
    In fact, a flat crossroad web represented by the diagram in \cref{prop:Bweb-T} is decomposed into
    \[
        \prod_{i=1}^{3}(e_{i}^{(1)})^{k_i}(e_{i}^{(2)})^{l_i}x^{n_1}y^{n_2}
    \]
    up to multiplication by $v$, where $e_{i}^{(1)}$ (resp.~$e_{i}^{(2)}$) are arcs corresponding to labels $k_{i}$ (resp.~$l_{i}$) and
    \begin{align*}
        x
        =
        \mathord{
            \ \tikz[baseline=-.6ex, scale=.08, rotate=-60]{
                \coordinate (A) at (90:8);
                \coordinate (B) at (210:8);
                \coordinate (C) at (-30:8);
                \coordinate (P1) at ($(A)!.8!(B)+(3,0)$);
                \coordinate (P2) at ($(A)!.8!(C)+(-3,0)$);
                \coordinate (P) at (0,0);
                \draw[wline] (P1) -- (B);
                \draw[wline] (P2) -- (C);
                \draw[webline] (A) -- (P1) -- (P2) -- (A);
                \draw[very thick] (A) -- (B) -- (C) -- cycle;
                \filldraw (A) circle [radius=20pt];
                \filldraw (B) circle [radius=20pt];
                \filldraw (C) circle [radius=20pt];
            }\ 
        },
        y
        =
        \mathord{
            \ \tikz[baseline=-.6ex, scale=.08, rotate=180]{
                \coordinate (A) at (90:8);
                \coordinate (B) at (210:8);
                \coordinate (C) at (-30:8);
                \coordinate (P) at (0,0);
                \draw[wline] (P) -- (A);
                \draw[webline] (P) -- (B);
                \draw[webline] (P) -- (C);
                \draw[very thick] (A) -- (B) -- (C) -- cycle;
                \filldraw (A) circle [radius=20pt];
                \filldraw (B) circle [radius=20pt];
                \filldraw (C) circle [radius=20pt];
            }\ 
        }.
    \end{align*}
\end{proof}

\begin{cor}\label{cor:Cweb-T}
    $\Cweb{T}$ consists of the following $6$ subsets of $\Eweb{T}$:
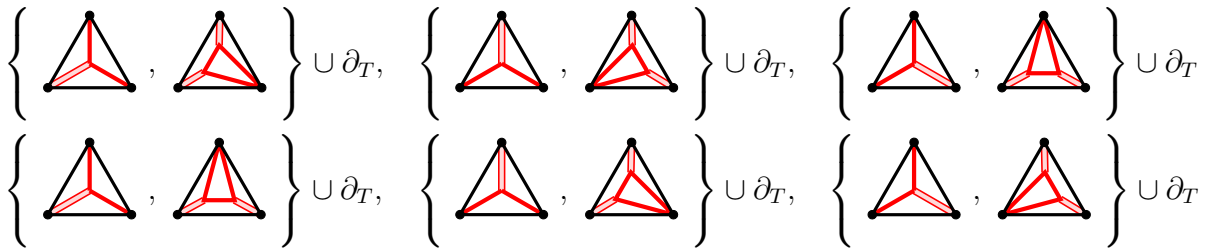
    \begin{align*}
        &\left\{
            \mathord{
                \ \tikz[baseline=-.6ex, scale=.08, rotate=120]{
                    \coordinate (A) at (90:8);
                    \coordinate (B) at (210:8);
                    \coordinate (C) at (-30:8);
                    \coordinate (P) at (0,0);
                    \draw[wline] (P) -- (A);
                    \draw[webline] (P) -- (B);
                    \draw[webline] (P) -- (C);
                    \draw[very thick] (A) -- (B) -- (C) -- cycle;
                    \filldraw (A) circle [radius=20pt];
                    \filldraw (B) circle [radius=20pt];
                    \filldraw (C) circle [radius=20pt];
                }\ 
            },
            \mathord{
                \ \tikz[baseline=-.6ex, scale=.08, rotate=-120]{
                    \coordinate (A) at (90:8);
                    \coordinate (B) at (210:8);
                    \coordinate (C) at (-30:8);
                    \coordinate (P1) at ($(A)!.8!(B)+(3,0)$);
                    \coordinate (P2) at ($(A)!.8!(C)+(-3,0)$);
                    \coordinate (P) at (0,0);
                    \draw[wline] (P1) -- (B);
                    \draw[wline] (P2) -- (C);
                    \draw[webline] (A) -- (P1) -- (P2) -- (A);
                    \draw[very thick] (A) -- (B) -- (C) -- cycle;
                    \filldraw (A) circle [radius=20pt];
                    \filldraw (B) circle [radius=20pt];
                    \filldraw (C) circle [radius=20pt];
                }\ 
            }
        \right\}\cup \partial_{T},&
        &\left\{
            \mathord{
                \ \tikz[baseline=-.6ex, scale=.08]{
                    \coordinate (A) at (90:8);
                    \coordinate (B) at (210:8);
                    \coordinate (C) at (-30:8);
                    \coordinate (P) at (0,0);
                    \draw[wline] (P) -- (A);
                    \draw[webline] (P) -- (B);
                    \draw[webline] (P) -- (C);
                    \draw[very thick] (A) -- (B) -- (C) -- cycle;
                    \filldraw (A) circle [radius=20pt];
                    \filldraw (B) circle [radius=20pt];
                    \filldraw (C) circle [radius=20pt];
                }\ 
            },
            \mathord{
                \ \tikz[baseline=-.6ex, scale=.08, rotate=120]{
                    \coordinate (A) at (90:8);
                    \coordinate (B) at (210:8);
                    \coordinate (C) at (-30:8);
                    \coordinate (P1) at ($(A)!.8!(B)+(3,0)$);
                    \coordinate (P2) at ($(A)!.8!(C)+(-3,0)$);
                    \coordinate (P) at (0,0);
                    \draw[wline] (P1) -- (B);
                    \draw[wline] (P2) -- (C);
                    \draw[webline] (A) -- (P1) -- (P2) -- (A);
                    \draw[very thick] (A) -- (B) -- (C) -- cycle;
                    \filldraw (A) circle [radius=20pt];
                    \filldraw (B) circle [radius=20pt];
                    \filldraw (C) circle [radius=20pt];
                }\ 
            }
        \right\}\cup \partial_{T},&
        &\left\{
            \mathord{
                \ \tikz[baseline=-.6ex, scale=.08, rotate=-120]{
                    \coordinate (A) at (90:8);
                    \coordinate (B) at (210:8);
                    \coordinate (C) at (-30:8);
                    \coordinate (P) at (0,0);
                    \draw[wline] (P) -- (A);
                    \draw[webline] (P) -- (B);
                    \draw[webline] (P) -- (C);
                    \draw[very thick] (A) -- (B) -- (C) -- cycle;
                    \filldraw (A) circle [radius=20pt];
                    \filldraw (B) circle [radius=20pt];
                    \filldraw (C) circle [radius=20pt];
                }\ 
            },
            \mathord{
                \ \tikz[baseline=-.6ex, scale=.08]{
                    \coordinate (A) at (90:8);
                    \coordinate (B) at (210:8);
                    \coordinate (C) at (-30:8);
                    \coordinate (P1) at ($(A)!.8!(B)+(3,0)$);
                    \coordinate (P2) at ($(A)!.8!(C)+(-3,0)$);
                    \coordinate (P) at (0,0);
                    \draw[wline] (P1) -- (B);
                    \draw[wline] (P2) -- (C);
                    \draw[webline] (A) -- (P1) -- (P2) -- (A);
                    \draw[very thick] (A) -- (B) -- (C) -- cycle;
                    \filldraw (A) circle [radius=20pt];
                    \filldraw (B) circle [radius=20pt];
                    \filldraw (C) circle [radius=20pt];
                }\ 
            }
        \right\}\cup \partial_{T}\\
        &\left\{
            \mathord{
                \ \tikz[baseline=-.6ex, scale=.08, rotate=120]{
                    \coordinate (A) at (90:8);
                    \coordinate (B) at (210:8);
                    \coordinate (C) at (-30:8);
                    \coordinate (P) at (0,0);
                    \draw[wline] (P) -- (A);
                    \draw[webline] (P) -- (B);
                    \draw[webline] (P) -- (C);
                    \draw[very thick] (A) -- (B) -- (C) -- cycle;
                    \filldraw (A) circle [radius=20pt];
                    \filldraw (B) circle [radius=20pt];
                    \filldraw (C) circle [radius=20pt];
                }\ 
            },
            \mathord{
                \ \tikz[baseline=-.6ex, scale=.08]{
                    \coordinate (A) at (90:8);
                    \coordinate (B) at (210:8);
                    \coordinate (C) at (-30:8);
                    \coordinate (P1) at ($(A)!.8!(B)+(3,0)$);
                    \coordinate (P2) at ($(A)!.8!(C)+(-3,0)$);
                    \coordinate (P) at (0,0);
                    \draw[wline] (P1) -- (B);
                    \draw[wline] (P2) -- (C);
                    \draw[webline] (A) -- (P1) -- (P2) -- (A);
                    \draw[very thick] (A) -- (B) -- (C) -- cycle;
                    \filldraw (A) circle [radius=20pt];
                    \filldraw (B) circle [radius=20pt];
                    \filldraw (C) circle [radius=20pt];
                }\ 
            }
        \right\}\cup \partial_{T},&
        &\left\{
            \mathord{
                \ \tikz[baseline=-.6ex, scale=.08]{
                    \coordinate (A) at (90:8);
                    \coordinate (B) at (210:8);
                    \coordinate (C) at (-30:8);
                    \coordinate (P) at (0,0);
                    \draw[wline] (P) -- (A);
                    \draw[webline] (P) -- (B);
                    \draw[webline] (P) -- (C);
                    \draw[very thick] (A) -- (B) -- (C) -- cycle;
                    \filldraw (A) circle [radius=20pt];
                    \filldraw (B) circle [radius=20pt];
                    \filldraw (C) circle [radius=20pt];
                }\ 
            },
            \mathord{
                \ \tikz[baseline=-.6ex, scale=.08, rotate=-120]{
                    \coordinate (A) at (90:8);
                    \coordinate (B) at (210:8);
                    \coordinate (C) at (-30:8);
                    \coordinate (P1) at ($(A)!.8!(B)+(3,0)$);
                    \coordinate (P2) at ($(A)!.8!(C)+(-3,0)$);
                    \coordinate (P) at (0,0);
                    \draw[wline] (P1) -- (B);
                    \draw[wline] (P2) -- (C);
                    \draw[webline] (A) -- (P1) -- (P2) -- (A);
                    \draw[very thick] (A) -- (B) -- (C) -- cycle;
                    \filldraw (A) circle [radius=20pt];
                    \filldraw (B) circle [radius=20pt];
                    \filldraw (C) circle [radius=20pt];
                }\ 
            }
        \right\}\cup \partial_{T},&
        &\left\{
            \mathord{
                \ \tikz[baseline=-.6ex, scale=.08, rotate=-120]{
                    \coordinate (A) at (90:8);
                    \coordinate (B) at (210:8);
                    \coordinate (C) at (-30:8);
                    \coordinate (P) at (0,0);
                    \draw[wline] (P) -- (A);
                    \draw[webline] (P) -- (B);
                    \draw[webline] (P) -- (C);
                    \draw[very thick] (A) -- (B) -- (C) -- cycle;
                    \filldraw (A) circle [radius=20pt];
                    \filldraw (B) circle [radius=20pt];
                    \filldraw (C) circle [radius=20pt];
                }\ 
            },
            \mathord{
                \ \tikz[baseline=-.6ex, scale=.08, rotate=120]{
                    \coordinate (A) at (90:8);
                    \coordinate (B) at (210:8);
                    \coordinate (C) at (-30:8);
                    \coordinate (P1) at ($(A)!.8!(B)+(3,0)$);
                    \coordinate (P2) at ($(A)!.8!(C)+(-3,0)$);
                    \coordinate (P) at (0,0);
                    \draw[wline] (P1) -- (B);
                    \draw[wline] (P2) -- (C);
                    \draw[webline] (A) -- (P1) -- (P2) -- (A);
                    \draw[very thick] (A) -- (B) -- (C) -- cycle;
                    \filldraw (A) circle [radius=20pt];
                    \filldraw (B) circle [radius=20pt];
                    \filldraw (C) circle [radius=20pt];
                }\ 
            }
        \right\}\cup \partial_{T}
    \end{align*}
\end{cor}
\begin{proof}
    For each set in above, one can confirm the $v$-commutativity by the arborization relations in \cref{lem:arborization}.
\end{proof}

\subsection{\texorpdfstring{$\Skein{\Sigma}$}{S(sp4,S)} is an Ore domain}\label{subsec:Ore}
In this section, we are going to see that $\Skein{\Sigma}$ is embedded into its skew-field of fractions $\mathrm{Frac}\Skein{\Sigma}$. This is needed in studying the relationship between the localized $\mathfrak{sp}_4$-skein algebra $\Skein{\Sigma}[\partial^{-1}]$ and the quantum cluster algebra $\CA^q_{\mathfrak{sp}_4,\Sigma}$, where the latter will be naturally constructed in $\mathrm{Frac}\Skein{\Sigma}$.

\begin{prop}\label{prop:Ore_triangle}
    For a triangle $T$, the boundary-localized skein algebra $\Skein{T}[\partial^{-1}]$ is an Ore domain.
\end{prop}

This proposition follows from the isomorphism between $\Skein{T}[\partial^{-1}]$ and the quantum cluster algebra $\CA^q_{\mathfrak{sp}_4,T}$ (\cref{S=A_triangle}).
In what follows, we prove that $\Skein{\Sigma}[\partial^{-1}]$ is an Ore domain for any unpunctured marked surface, assuming that the triangle case is established. 
The proof is based on a relation to the \emph{reduced stated $\mathfrak{sp}_4$-skein algebra} and its splitting property.
The stated $\mathfrak{sp}_4$-skein algebra $\mathscr{S}_{\mathfrak{sp}_4}^{\mathsf{st}}(\Sigma)$ and its reduced version $\mathscr{S}^{\mathsf{st}}_{\mathfrak{sp}_4}(\Sigma)_{\mathrm{rd}}$ is introduced in \cite{IYstmk}.
In that paper, we construct an isomorphism
\[
    \Phi_{\Sigma}\colon\mathscr{S}^{\mathsf{st}}_{\mathfrak{sp}_4}(\Sigma)_{\mathrm{rd}}\to\Skein{\Sigma}[\partial^{-1}]
\]
and an injective homomorphism
\[
    \theta_{\alpha}^{\mathrm{rd}}\colon\mathscr{S}^{\mathsf{st}}_{\mathfrak{sp}_4}(\Sigma)_{\mathrm{rd}}\xhookrightarrow[]{}\mathscr{S}^{\mathsf{st}}_{\mathfrak{sp}_4}(\Sigma')_{\mathrm{rd}}
\]
called the \emph{splitting homomorphism} (cf.~\cite{TTQLe18} for $\mathfrak{sl}_2$, \cite{Higgins20} for $\mathfrak{sl}_3$, \cite{LS22} for $\mathfrak{sl}_n$),
where a marked surface $\Sigma$ is obtained from another marked surface $\Sigma'$ by gluing two distinct boundary intervals, which project to a common ideal arc $\alpha$ in $\Sigma$. 
We remark that if $\Sigma'$ consists of two connected components $\Sigma_1$ and $\Sigma_2$, then $\mathscr{S}^{\mathsf{st},\mathrm{rd}}_{\mathfrak{sp}_4}(\Sigma')\cong\mathscr{S}^{\mathsf{st}}_{\mathfrak{sp}_4}(\Sigma_1)_{\mathrm{rd}}\otimes\mathscr{S}^{\mathsf{st}}_{\mathfrak{sp}_4}(\Sigma_2)_{\mathrm{rd}}$.

\begin{thm}\label{thm:Ore_general}
    For any unpunctured marked surface $\Sigma$, the skein algebra $\Skein{\Sigma}$ is an Ore domain. Moreover, $\mathrm{Frac}\Skein{\Sigma}$ is isomorphic to the skew-field of fractions of a quantum torus.
\end{thm}

\begin{proof}
    Given an ideal triangulation $\Delta=\{\alpha_1,\alpha_2,\dots,\alpha_n\}$ of $\Sigma$, one can decompose $\mathscr{S}^{\mathsf{st}}_{\mathfrak{sp}_4}(\Sigma)_{\mathrm{rd}}$ into $\mathscr{S}^{\mathsf{st}}_{\mathfrak{sp}_4}(\sqcup_{T\in t(\Delta)}T)_{\mathrm{rd}}\cong\bigotimes_{T\in t(\Delta)}\mathscr{S}^{\mathsf{st}}_{\mathfrak{sp}_4}(T)_{\mathrm{rd}}$ via the composite $\theta_{\Delta}^{\mathrm{rd}}:=\theta_{\alpha_n}^{\mathrm{rd}}\circ\cdots\circ\theta_{\alpha_2}^{\mathrm{rd}}\circ\theta_{\alpha_1}^{\mathrm{rd}}$ of splitting homomorphisms. Then we get the following sequence of homomorphisms:
    \[
        \Skein{\Sigma}\xhookrightarrow[]{}\Skein{\Sigma}[\partial^{-1}]\xrightarrow[\Phi^{-1}_{\Sigma}]{\cong}\mathscr{S}^{\mathsf{st}}_{\mathfrak{sp}_4}(\Sigma)_{\mathrm{rd}}\xhookrightarrow[\theta_{\Delta}^{\mathrm{rd}}]{}\bigotimes_{T\in t(\Delta)}\mathscr{S}^{\mathsf{st}}_{\mathfrak{sp}_4}(T)_{\mathrm{rd}}\xrightarrow[\bigotimes_{T}\Phi_{T}]{\cong}\bigotimes_{T\in t(\Delta)}\Skein{T}[\partial^{-1}].
    \]
    The first localization map becomes inclusion because every boundary web is a non-zero divisor.
    The right-most algebra is an Ore domain by \cref{prop:Ore_triangle}. 
    Then we have an embedding of $\Skein{\Sigma}$ into an Ore domain, which implies that $\Skein{\Sigma}$ is a domain. Then we can apply \cite[Proposition~2.2]{LY20} to the domain $\Skein{\Sigma}$ and the quantum plane $\langle\cC_{\bD}\rangle_{\text{alg}}$ generated by a web cluster $\cC_{\bD}$ as in \cref{thm:Cweb-exp}, where the hypothesis there is satisfied thanks to \cref{thm:Cweb-exp}. It concludes that $\Skein{\Sigma}$ is an Ore domain, and $\mathrm{Frac}\Skein{\Sigma}$ is isomorphic to the skew-field of fraction of the quantum torus generated by the web cluster $\cC_{\bD}$.
\end{proof}

\begin{cor}\label{cor:embed_into_frac}
    We have inclusions $\Skein{\Sigma}\subset\Skein{\Sigma}[\partial^{-1}]\subset\mathrm{Frac}\Skein{\Sigma}$.
\end{cor}

\section{Generators and Laurent positivity for \texorpdfstring{$\mathfrak{sp}_4$}{sp4}-skein algebras}\label{sec:generator_positivity}
\subsection{Cutting and Sticking tricks}
Let us introduce two fundamental lemmas on ``the sticking trick'' and ``the cutting trick'', which will turn out to be useful techniques to investigate the relation between our skein algebra and the quantum cluster algebra.
In~\cite{IYsl3}, the authors showed that the sticking trick could be used to expand a $\mathfrak{sl}_3$-web into a Laurent polynomial in the boundary-localized $\mathfrak{sl}_3$-skein algebra.
The cutting trick was used to prove the positivity of the coefficients for cluster expansions of elevation-preserving $\mathfrak{sl}_3$-webs in the same paper.
These tricks will play a similar role in the current case of $\mathfrak{sp}_4$.

\begin{lem}[The cutting trick]\label{lem:cutteing-trick}
    \begin{align}
        \mathord{
            \ \tikz[baseline=-.6ex, scale=.08]{
            \coordinate (A) at (-10,0);
            \coordinate (B) at (10,0);
            \coordinate (NW) at ($(A)+(0,7)$);
            \coordinate (SW) at ($(A)+(0,-7)$);
            \coordinate (NE) at ($(B)+(0,7)$);
            \coordinate (SE) at ($(B)+(0,-7)$);
            \coordinate (T1) at ($(A)!0.5!(B)+(0,7)$);
            \coordinate (T2) at ($(A)!0.5!(B)+(0,-7)$);
            \draw[very thick] (NW) -- (SW);
            \draw[very thick] (NE) -- (SE);
            \draw[wline] (B) to[bend right] (A);
            \draw[webline] (B) to[bend left] (A);
            \draw[overarc] (0,-10) -- (0,10);
            \bdryline{(NE)}{(SE)}{-2cm}
            \bdryline{(NW)}{(SW)}{2cm}
            \draw[fill] (A) circle (20pt);
            \draw[fill] (B) circle (20pt);
            }\ 
        }
        &=q^2
        \mathord{
            \ \tikz[baseline=-.6ex, scale=.08]{
                \coordinate (A) at (-10,0);
                \coordinate (B) at (10,0);
                \coordinate (NW) at ($(A)+(0,7)$);
                \coordinate (SW) at ($(A)+(0,-7)$);
                \coordinate (NE) at ($(B)+(0,7)$);
                \coordinate (SE) at ($(B)+(0,-7)$);
                \coordinate (T1) at ($(A)!0.5!(B)+(0,7)$);
                \coordinate (T2) at ($(A)!0.5!(B)+(0,-7)$);
                \draw[very thick] (NW) -- (SW);
                \draw[very thick] (NE) -- (SE);
                \draw[wline] (A) -- (B);
                \draw[webline, rounded corners] (B) -- (T1) -- ($(T1)+(0,3)$);
                \draw[webline, rounded corners] (A) -- (T2) -- ($(T2)+(0,-3)$);
                \bdryline{(NE)}{(SE)}{-2cm}
                \bdryline{(NW)}{(SW)}{2cm}
                \draw[fill] (A) circle (20pt);
                \draw[fill] (B) circle (20pt);
            }\ 
        }
        +q
        \mathord{
            \ \tikz[baseline=-.6ex, scale=.08]{
                \coordinate (A) at (-10,0);
                \coordinate (B) at (10,0);
                \coordinate (NW) at ($(A)+(0,7)$);
                \coordinate (SW) at ($(A)+(0,-7)$);
                \coordinate (NE) at ($(B)+(0,7)$);
                \coordinate (SE) at ($(B)+(0,-7)$);
                \coordinate (T1) at ($(A)!0.5!(B)+(0,3)$);
                \coordinate (T2) at ($(A)!0.5!(B)+(0,-3)$);
                \draw[very thick] (NW) -- (SW);
                \draw[very thick] (NE) -- (SE);
                \draw[webline] (A) -- (T1);
                \draw[wline] (B) -- (T1);
                \draw[wline] (A) -- (T2);
                \draw[webline] (B) -- (T2);
                \draw[webline] (T1) -- ($(T1)+(0,7)$);
                \draw[webline] (T2) -- ($(T2)+(0,-7)$);
                \bdryline{(NE)}{(SE)}{-2cm}
                \bdryline{(NW)}{(SW)}{2cm}
                \draw[fill] (A) circle (20pt);
                \draw[fill] (B) circle (20pt);
            }\ 
        }
        +q^{-1}
        \mathord{
            \ \tikz[baseline=-.6ex, scale=.08]{
                \coordinate (A) at (-10,0);
                \coordinate (B) at (10,0);
                \coordinate (NW) at ($(A)+(0,7)$);
                \coordinate (SW) at ($(A)+(0,-7)$);
                \coordinate (NE) at ($(B)+(0,7)$);
                \coordinate (SE) at ($(B)+(0,-7)$);
                \coordinate (T1) at ($(A)!0.5!(B)+(0,3)$);
                \coordinate (T2) at ($(A)!0.5!(B)+(0,-3)$);
                \draw[very thick] (NW) -- (SW);
                \draw[very thick] (NE) -- (SE);
                \draw[wline] (A) -- (T1);
                \draw[webline] (B) -- (T1);
                \draw[webline] (A) -- (T2);
                \draw[wline] (B) -- (T2);
                \draw[webline] (T1) -- ($(T1)+(0,7)$);
                \draw[webline] (T2) -- ($(T2)+(0,-7)$);
                \bdryline{(NE)}{(SE)}{-2cm}
                \bdryline{(NW)}{(SW)}{2cm}
                \draw[fill] (A) circle (20pt);
                \draw[fill] (B) circle (20pt);
            }\ 
        }
        +q^{-2}
        \mathord{
            \ \tikz[baseline=-.6ex, scale=.08]{
                \coordinate (A) at (-10,0);
                \coordinate (B) at (10,0);
                \coordinate (NW) at ($(A)+(0,7)$);
                \coordinate (SW) at ($(A)+(0,-7)$);
                \coordinate (NE) at ($(B)+(0,7)$);
                \coordinate (SE) at ($(B)+(0,-7)$);
                \coordinate (T1) at ($(A)!0.5!(B)+(0,7)$);
                \coordinate (T2) at ($(A)!0.5!(B)+(0,-7)$);
                \draw[very thick] (NW) -- (SW);
                \draw[very thick] (NE) -- (SE);
                \draw[wline] (A) -- (B);
                \draw[webline, rounded corners] (A) -- (T1) -- ($(T1)+(0,3)$);
                \draw[webline, rounded corners] (B) -- (T2) -- ($(T2)+(0,-3)$);
                \bdryline{(NE)}{(SE)}{-2cm}
                \bdryline{(NW)}{(SW)}{2cm}
                \draw[fill] (A) circle (20pt);
                \draw[fill] (B) circle (20pt);
            }\ 
        }\label{eq:s-cutting}\\
        \mathord{
            \ \tikz[baseline=-.6ex, scale=.08]{
            \coordinate (A) at (-10,0);
            \coordinate (B) at (10,0);
            \coordinate (NW) at ($(A)+(0,7)$);
            \coordinate (SW) at ($(A)+(0,-7)$);
            \coordinate (NE) at ($(B)+(0,7)$);
            \coordinate (SE) at ($(B)+(0,-7)$);
            \coordinate (T1) at ($(A)!0.5!(B)+(0,7)$);
            \coordinate (T2) at ($(A)!0.5!(B)+(0,-7)$);
            \draw[very thick] (NW) -- (SW);
            \draw[very thick] (NE) -- (SE);
            \draw[wline] (B) to[bend right] (A);
            \draw[webline] (B) to[bend left=15] (A);
            \draw[webline] (B) to[bend left=45] (A);
            \draw[overwline] (0,-10) -- (0,10);
            \bdryline{(NE)}{(SE)}{-2cm}
            \bdryline{(NW)}{(SW)}{2cm}
        \draw[fill] (A) circle (20pt);
            \draw[fill] (B) circle (20pt);
            }\ 
        }
        &=q^{4}
        \mathord{
            \ \tikz[baseline=-.6ex, scale=.08]{
                \coordinate (A) at (-10,0);
                \coordinate (B) at (10,0);
                \coordinate (NW) at ($(A)+(0,7)$);
                \coordinate (SW) at ($(A)+(0,-7)$);
                \coordinate (NE) at ($(B)+(0,7)$);
                \coordinate (SE) at ($(B)+(0,-7)$);
                \coordinate (T1) at ($(A)!0.5!(B)+(0,7)$);
                \coordinate (T2) at ($(A)!0.5!(B)+(0,-7)$);
                \draw[very thick] (NW) -- (SW);
                \draw[very thick] (NE) -- (SE);
                \draw[webline] (A) to[bend right=15] (B);
                \draw[webline] (A) to[bend left=15] (B);
                \draw[wline, rounded corners] (B) -- (T1) -- ($(T1)+(0,3)$);
                \draw[wline, rounded corners] (A) -- (T2) -- ($(T2)+(0,-3)$);
                \bdryline{(NE)}{(SE)}{-2cm}
                \bdryline{(NW)}{(SW)}{2cm}
                \draw[fill] (A) circle (20pt);
                \draw[fill] (B) circle (20pt);
            }
        }
        +q^{2}
        \mathord{
            \ \tikz[baseline=-.6ex, scale=.08]{
                \coordinate (A) at (-10,0);
                \coordinate (B) at (10,0);
                \coordinate (NW) at ($(A)+(0,7)$);
                \coordinate (SW) at ($(A)+(0,-7)$);
                \coordinate (NE) at ($(B)+(0,7)$);
                \coordinate (SE) at ($(B)+(0,-7)$);
                \coordinate (T1) at ($(A)!0.5!(B)+(0,4)$);
                \coordinate (T1east) at ($(A)!0.7!(B)+(0,2)$);
                \coordinate (T1west) at ($(A)!0.3!(B)+(0,2)$);
                \coordinate (T2) at ($(A)!0.5!(B)+(0,-4)$);
                \coordinate (T2east) at ($(A)!0.7!(B)+(0,-2)$);
                \coordinate (T2west) at ($(A)!0.3!(B)+(0,-2)$);
                \draw[very thick] (NW) -- (SW);
                \draw[very thick] (NE) -- (SE);
                \draw[wline] (A) -- (T1west);
                \draw[webline] (B) -- (T1west);
                \draw[webline] (B) -- (T1);
                \draw[webline] (T1) -- (T1west);
                \draw[wline] (B) -- (T2east);
                \draw[webline] (A) -- (T2east);
                \draw[webline] (A) -- (T2);
                \draw[webline] (T2) -- (T2east);
                \draw[wline] (T1) -- ($(T1)+(0,7)$);
                \draw[wline] (T2) -- ($(T2)+(0,-7)$);
                \bdryline{(NE)}{(SE)}{-2cm}
                \bdryline{(NW)}{(SW)}{2cm}
                \draw[fill] (A) circle (20pt);
                \draw[fill] (B) circle (20pt);
            }
        }\notag\\
        &\quad+[2]
        \mathord{
            \ \tikz[baseline=-.6ex, scale=.08]{
                \coordinate (A) at (-10,0);
                \coordinate (B) at (10,0);
                \coordinate (NW) at ($(A)+(0,7)$);
                \coordinate (SW) at ($(A)+(0,-7)$);
                \coordinate (NE) at ($(B)+(0,7)$);
                \coordinate (SE) at ($(B)+(0,-7)$);
                \coordinate (T1) at ($(A)!0.5!(B)+(0,3)$);
                \coordinate (T2) at ($(A)!0.5!(B)+(0,-3)$);
                \draw[very thick] (NW) -- (SW);
                \draw[very thick] (NE) -- (SE);
                \draw[wline] (A) -- (B);
                \draw[webline] (A) -- (T1);
                \draw[webline] (B) -- (T1);
                \draw[webline] (A) -- (T2);
                \draw[webline] (B) -- (T2);
                \draw[wline] (T1) -- ($(T1)+(0,7)$);
                \draw[wline] (T2) -- ($(T2)+(0,-7)$);
                \bdryline{(NE)}{(SE)}{-2cm}
                \bdryline{(NW)}{(SW)}{2cm}
                \draw[fill] (A) circle (20pt);
                \draw[fill] (B) circle (20pt);
            }
        }
        +q^{-2}
        \mathord{
            \ \tikz[baseline=-.6ex, scale=.08]{
                \coordinate (A) at (-10,0);
                \coordinate (B) at (10,0);
                \coordinate (NW) at ($(A)+(0,7)$);
                \coordinate (SW) at ($(A)+(0,-7)$);
                \coordinate (NE) at ($(B)+(0,7)$);
                \coordinate (SE) at ($(B)+(0,-7)$);
                \coordinate (T1) at ($(A)!0.5!(B)+(0,4)$);
                \coordinate (T1east) at ($(A)!0.7!(B)+(0,2)$);
                \coordinate (T1west) at ($(A)!0.3!(B)+(0,2)$);
                \coordinate (T2) at ($(A)!0.5!(B)+(0,-4)$);
                \coordinate (T2east) at ($(A)!0.7!(B)+(0,-2)$);
                \coordinate (T2west) at ($(A)!0.3!(B)+(0,-2)$);
                \draw[very thick] (NW) -- (SW);
                \draw[very thick] (NE) -- (SE);
                \draw[wline] (B) -- (T1east);
                \draw[webline] (A) -- (T1east);
                \draw[webline] (A) -- (T1);
                \draw[webline] (T1) -- (T1east);
                \draw[wline] (A) -- (T2west);
                \draw[webline] (B) -- (T2west);
                \draw[webline] (B) -- (T2);
                \draw[webline] (T2) -- (T2west);
                \draw[wline] (T1) -- ($(T1)+(0,7)$);
                \draw[wline] (T2) -- ($(T2)+(0,-7)$);
                \bdryline{(NE)}{(SE)}{-2cm}
                \bdryline{(NW)}{(SW)}{2cm}
                \draw[fill] (A) circle (20pt);
                \draw[fill] (B) circle (20pt);
            }
        }
        +q^{-4}
        \mathord{
            \ \tikz[baseline=-.6ex, scale=.08]{
                \coordinate (A) at (-10,0);
                \coordinate (B) at (10,0);
                \coordinate (NW) at ($(A)+(0,7)$);
                \coordinate (SW) at ($(A)+(0,-7)$);
                \coordinate (NE) at ($(B)+(0,7)$);
                \coordinate (SE) at ($(B)+(0,-7)$);
                \coordinate (T1) at ($(A)!0.5!(B)+(0,7)$);
                \coordinate (T2) at ($(A)!0.5!(B)+(0,-7)$);
                \draw[very thick] (NW) -- (SW);
                \draw[very thick] (NE) -- (SE);
                \draw[webline] (A) to[bend right=15] (B);
                \draw[webline] (A) to[bend left=15] (B);
                \draw[wline, rounded corners] (A) -- (T1) -- ($(T1)+(0,3)$);
                \draw[wline, rounded corners] (B) -- (T2) -- ($(T2)+(0,-3)$);
                \bdryline{(NE)}{(SE)}{-2cm}
                \bdryline{(NW)}{(SW)}{2cm}
                \draw[fill] (A) circle (20pt);
                \draw[fill] (B) circle (20pt);
            }
        }\label{eq:w-cutting}
    \end{align}
\end{lem}
\begin{proof}
    Apply skein relations to crossings.
\end{proof}
\begin{lem}[The sticking trick]\label{lem:sticking-trick}
    \begin{align}
        &\begin{aligned}
            \mathord{
                \ \tikz[baseline=-.6ex, scale=.08, yshift=-5cm]{
                    \coordinate (A) at (-10,0);
                    \coordinate (B) at (10,0);
                    \coordinate (T11) at ($(A)!0.2!(B)+(0,15)$);
                    \coordinate (T21) at ($(A)!0.2!(B)+(0,10)$);
                    \coordinate (T12) at ($(A)!0.8!(B)+(0,15)$);
                    \coordinate (T22) at ($(A)!0.8!(B)+(0,10)$);
                    \draw[wline] (A) to[bend left=60] (B);
                    \draw[webline] (A) to[bend left=20] (B);
                    \draw[overarc, rounded corners] (T11) -- (T21) -- (T22) -- (T12);
                    \bdryline{(A)}{(B)}{2cm}
                    \draw[fill] (A) circle (20pt);
                    \draw[fill] (B) circle (20pt);
                }
            \ }
            &=q
            \mathord{
                \ \tikz[baseline=-.6ex, scale=.08, yshift=-5cm]{
                    \coordinate (A) at (-10,0);
                    \coordinate (B) at (10,0);
                    \coordinate (T11) at ($(A)!0.2!(B)+(0,15)$);
                    \coordinate (T21) at ($(A)!0.2!(B)+(0,10)$);
                    \coordinate (T12) at ($(A)!0.8!(B)+(0,15)$);
                    \coordinate (T22) at ($(A)!0.8!(B)+(0,10)$);
                    \draw[overarc] (T11) to[out=south, in=north] (B);
                    \draw[overarc] (T12) to[out=south, in=north] (A);
                    \draw[wline, shorten <=.2cm, shorten >=.2cm] (A) to[bend left] (B);
                    \bdryline{(A)}{(B)}{2cm}
                    \draw[fill] (A) circle (20pt);
                    \draw[fill] (B) circle (20pt);
                }\ 
            }
            -q^2
            \mathord{
                \ \tikz[baseline=-.6ex, scale=.08, yshift=-5cm]{
                    \coordinate (A) at (-10,0);
                    \coordinate (B) at (10,0);
                    \coordinate (T11) at ($(A)!0.2!(B)+(0,15)$);
                    \coordinate (T21) at ($(A)!0.2!(B)+(0,7)$);
                    \coordinate (T12) at ($(A)!0.8!(B)+(0,15)$);
                    \coordinate (T22) at ($(A)!0.8!(B)+(0,7)$);
                    \draw[webline] (T11) to[out=south, in=north] (T22);
                    \draw[overarc] (T12) to[out=south, in=north] (T21);
                    \draw[webline, shorten <=.3cm] (A) -- (T22);
                    \draw[overarc] (B) -- (T21);
                    \draw[wline] (A) -- (T21);
                    \draw[wline, shorten <=.3cm] (B) -- (T22);
                    \bdryline{(A)}{(B)}{2cm}
                    \draw[fill] (A) circle (20pt);
                    \draw[fill] (B) circle (20pt);
                }\ 
            }
            +q^3
            \mathord{
                \ \tikz[baseline=-.6ex, scale=.08, yshift=-5cm]{
                    \coordinate (A) at (-10,0);
                    \coordinate (B) at (10,0);
                    \coordinate (T11) at ($(A)!0.2!(B)+(0,15)$);
                    \coordinate (T21) at ($(A)!0.2!(B)+(0,7)$);
                    \coordinate (T12) at ($(A)!0.8!(B)+(0,15)$);
                    \coordinate (T22) at ($(A)!0.8!(B)+(0,7)$);
                    \draw[webline] (T11) -- (T21);
                    \draw[webline] (T12) -- (T22);
                    \draw[webline, shorten <=.3cm] (B) -- (T21);
                    \draw[overarc] (A) -- (T22);
                    \draw[wline, shorten <=.3cm] (A) -- (T21);
                    \draw[wline] (B) -- (T22);
                    \bdryline{(A)}{(B)}{2cm}
                    \draw[fill] (A) circle (20pt);
                    \draw[fill] (B) circle (20pt);
                }\ 
            }
            -q^4
            \mathord{
                \ \tikz[baseline=-.6ex, scale=.08, yshift=-5cm]{
                    \coordinate (A) at (-10,0);
                    \coordinate (B) at (10,0);
                    \coordinate (T11) at ($(A)!0.2!(B)+(0,15)$);
                    \coordinate (T21) at ($(A)!0.2!(B)+(0,10)$);
                    \coordinate (T12) at ($(A)!0.8!(B)+(0,15)$);
                    \coordinate (T22) at ($(A)!0.8!(B)+(0,10)$);
                    \draw[webline] (T11) to[out=south, in=north] (A);
                    \draw[webline] (T12) to[out=south, in=north] (B);
                    \draw[wline, shorten <=.2cm, shorten >=.2cm] (A) to[bend left] (B);
                    \bdryline{(A)}{(B)}{2cm}
                    \draw[fill] (A) circle (20pt);
                    \draw[fill] (B) circle (20pt);
                }\ 
            },
        \end{aligned}\label{eq:stick-1}\\
        &\begin{aligned}
            \mathord{
                \ \tikz[baseline=-.6ex, scale=.08, yshift=-5cm]{
                    \coordinate (A) at (-10,0);
                    \coordinate (B) at (10,0);
                    \coordinate (T11) at ($(A)!0.2!(B)+(0,15)$);
                    \coordinate (T21) at ($(A)!0.2!(B)+(0,10)$);
                    \coordinate (T12) at ($(A)!0.8!(B)+(0,15)$);
                    \coordinate (T22) at ($(A)!0.8!(B)+(0,10)$);
                    \draw[wline] (A) to[bend left=80] (B);
                    \draw[webline] (A) to[bend left=20] (B);
                    \draw[webline] (A) to[bend left=40] (B);
                    \draw[wline, rounded corners] (T11) -- (T21) -- (T22) -- (T12);
                    \bdryline{(A)}{(B)}{2cm}
                    \draw[fill] (A) circle (20pt);
                    \draw[fill] (B) circle (20pt);
                }
            \ }
            &=q^{2}
            \mathord{
                \ \tikz[baseline=-.6ex, scale=.08, yshift=-5cm]{
                    \coordinate (A) at (-10,0);
                    \coordinate (B) at (10,0);
                    \coordinate (T11) at ($(A)!0.2!(B)+(0,15)$);
                    \coordinate (T21) at ($(A)!0.2!(B)+(0,10)$);
                    \coordinate (T12) at ($(A)!0.8!(B)+(0,15)$);
                    \coordinate (T22) at ($(A)!0.8!(B)+(0,10)$);
                    \draw[overwline] (T11) to[out=south, in=north] (B);
                    \draw[overwline] (T12) to[out=south, in=north] (A);
                    \draw[webline, shorten <=.2cm, shorten >=.2cm] (A) to[bend left=20] (B);
                    \draw[webline, shorten <=.2cm, shorten >=.2cm] (A) to[bend left=40] (B);
                    \bdryline{(A)}{(B)}{2cm}
                    \draw[fill] (A) circle (20pt);
                    \draw[fill] (B) circle (20pt);
                }\ 
            }
            -q^4
            \mathord{
                \ \tikz[baseline=-.6ex, scale=.08, yshift=-5cm]{
                    \coordinate (A) at (-10,0);
                    \coordinate (B) at (10,0);
                    \coordinate (T11) at ($(A)!0.2!(B)+(0,15)$);
                    \coordinate (T21) at ($(A)!0.1!(B)+(0,10)$);
                    \coordinate (T31) at ($(A)!0.3!(B)+(0,7)$);
                    \coordinate (T12) at ($(A)!0.8!(B)+(0,15)$);
                    \coordinate (T22) at ($(A)!0.9!(B)+(0,10)$);
                    \coordinate (T32) at ($(A)!0.7!(B)+(0,7)$);
                    \draw[wline] (T11) to[out=south, in=north west] (T22);
                    \draw[overwline] (T12) to[out=south, in=north east] (T21);
                    \draw[webline] (T21) -- (T31);
                    \draw[wline, shorten <=.3cm] (A) -- (T32);
                    \draw[overwline] (B) -- (T31);
                    \draw[webline] (T22) -- (T32);
                    \draw[webline, shorten <=.3cm] (B) -- (T22);
                    \draw[webline, shorten <=.3cm] (B) -- (T32);
                    \draw[webline] (A) -- (T21);
                    \draw[webline] (A) -- (T31);
                    \bdryline{(A)}{(B)}{2cm}
                    \draw[fill] (A) circle (20pt);
                    \draw[fill] (B) circle (20pt);
                }\ 
            }\\
            &\quad +q^4[2]
            \mathord{
                \ \tikz[baseline=-.6ex, scale=.08, yshift=-5cm]{
                    \coordinate (A) at (-10,0);
                    \coordinate (B) at (10,0);
                    \coordinate (T11) at ($(A)!0.2!(B)+(0,15)$);
                    \coordinate (T21) at ($(A)!0.2!(B)+(0,10)$);
                    \coordinate (T12) at ($(A)!0.8!(B)+(0,15)$);
                    \coordinate (T22) at ($(A)!0.8!(B)+(0,10)$);
                    \draw[wline] (T11) -- (T21);
                    \draw[wline] (T12) -- (T22);
                    \draw[overarc, shorten <=.2cm] (B) -- (T21);
                    \draw[overarc] (A) -- (T22);
                    \draw[webline, shorten <=.2cm] (A) -- (T21);
                    \draw[webline] (B) -- (T22);
                    \draw[wline, shorten <=.5cm, shorten >=.5cm] (A) to[bend left=15] (B);
                    \bdryline{(A)}{(B)}{2cm}
                    \draw[fill] (A) circle (20pt);
                    \draw[fill] (B) circle (20pt);
                }\ 
            }
            -q^{4}
            \mathord{
                \ \tikz[baseline=-.6ex, scale=.08, yshift=-5cm]{
                    \coordinate (A) at (-10,0);
                    \coordinate (B) at (10,0);
                    \coordinate (T11) at ($(A)!0.2!(B)+(0,15)$);
                    \coordinate (T21) at ($(A)!0.1!(B)+(0,10)$);
                    \coordinate (T31) at ($(A)!0.3!(B)+(0,8)$);
                    \coordinate (T12) at ($(A)!0.8!(B)+(0,15)$);
                    \coordinate (T22) at ($(A)!0.9!(B)+(0,10)$);
                    \coordinate (T32) at ($(A)!0.7!(B)+(0,8)$);
                    \draw[wline] (T11) to[out=south, in=north] (T21);
                    \draw[wline] (T12) to[out=south, in=north] (T22);
                    \draw[webline] (T21) -- (T31);
                    \draw[overwline, shorten <=.3cm] (B) -- (T31);
                    \draw[webline] (T22) -- (T32);
                    \draw[webline] (B) -- (T22);
                    \draw[webline] (B) -- (T32);
                    \draw[overwline] (A) -- (T32);
                    \draw[webline, shorten <=.3cm] (A) -- (T21);
                    \draw[webline, shorten <=.3cm] (A) -- (T31);
                    \bdryline{(A)}{(B)}{2cm}
                    \draw[fill] (A) circle (20pt);
                    \draw[fill] (B) circle (20pt);
                }\ 
            }
            +q^{7}
            \mathord{
                \ \tikz[baseline=-.6ex, scale=.08, yshift=-5cm]{
                    \coordinate (A) at (-10,0);
                    \coordinate (B) at (10,0);
                    \coordinate (T11) at ($(A)!0.2!(B)+(0,15)$);
                    \coordinate (T21) at ($(A)!0.2!(B)+(0,10)$);
                    \coordinate (T12) at ($(A)!0.8!(B)+(0,15)$);
                    \coordinate (T22) at ($(A)!0.8!(B)+(0,10)$);
                    \draw[wline] (T11) to[out=south, in=north] (A);
                    \draw[wline] (T12) to[out=south, in=north] (B);
                    \draw[webline, shorten <=.2cm, shorten >=.2cm] (A) to[bend left=20] (B);
                    \draw[webline, shorten <=.2cm, shorten >=.2cm] (A) to[bend left=40] (B);
                    \bdryline{(A)}{(B)}{2cm}
                    \draw[fill] (A) circle (20pt);
                    \draw[fill] (B) circle (20pt);
                }\ 
            }.
        \end{aligned}\label{eq:stick-2}
    \end{align}
\end{lem}
\begin{proof}
    These formulas are derived from the cutting trick.
    In \cref{lem:cutteing-trick}, glue the bottom ends of the left and right boundary intervals and consider the returning arc obtained by bending the middle line.
    Then the statement follows by a straightforward calculation.
    The following formulas are useful in the calculation:
    \begin{align*}
        \mathord{
            \ \tikz[baseline=-.6ex, scale=.08, yshift=-5cm]{
                \coordinate (A) at (-8,0);
                \coordinate (B) at (8,0);
                \coordinate (NW) at ($(A)!0.2!(B)+(0,15)$);
                \coordinate (NW2) at ($(A)!0.3!(B)+(0,7)$);
                \coordinate (NE) at ($(B)!0.2!(A)+(0,15)$);
                \coordinate (NE2) at ($(B)!0.3!(A)+(0,7)$);
                \draw[wline] (NW2) to[out=north west, in=south west] (NE);
                \draw[overwline] (NE2) to[out=north east, in=south east] (NW);
                \draw[webline] (NW2) -- (NE2);
                \draw[webline] (0,0) -- (NE2);
                \draw[webline] (0,0) -- (NW2);
                \bdryline{(A)}{(B)}{2cm}
                \draw[fill] (0,0) circle (20pt);
            }\ 
        }
        &=-q^{2}
        \mathord{
            \ \tikz[baseline=-.6ex, scale=.08, yshift=-5cm]{
                \coordinate (A) at (-8,0);
                \coordinate (B) at (8,0);
                \coordinate (NW) at ($(A)!0.2!(B)+(0,15)$);
                \coordinate (NW2) at ($(A)!0.3!(B)+(0,7)$);
                \coordinate (NE) at ($(B)!0.2!(A)+(0,15)$);
                \coordinate (NE2) at ($(B)!0.3!(A)+(0,7)$);
                \draw[wline] (NW2) -- (NW);
                \draw[wline] (NE2) -- (NE);
                \draw[webline] (NW2) -- (NE2);
                \draw[webline] (0,0) -- (NE2);
                \draw[webline] (0,0) -- (NW2);
                \bdryline{(A)}{(B)}{2cm}
                \draw[fill] (0,0) circle (20pt);
            }\ 
        },
        &
        \mathord{
            \ \tikz[baseline=-.6ex, scale=.08, yshift=-5cm]{
                \coordinate (A) at (-8,0);
                \coordinate (B) at (8,0);
                \coordinate (NW) at ($(A)!0.2!(B)+(0,15)$);
                \coordinate (NW2) at ($(A)!0.3!(B)+(0,7)$);
                \coordinate (NE) at ($(B)!0.2!(A)+(0,15)$);
                \coordinate (NE2) at ($(B)!0.3!(A)+(0,7)$);
                \draw[webline] (0,5) to[out=east, in=east] (NW);
                \draw[overarc] (0,5) to[out=west, in=west] (NE);
                \draw[wline] (0,0) -- (0,5);
            }\ 
        }
        &=-q^{-1}
        \mathord{
            \ \tikz[baseline=-.6ex, scale=.08, yshift=-5cm]{
                \coordinate (A) at (-8,0);
                \coordinate (B) at (8,0);
                \coordinate (NW) at ($(A)!0.2!(B)+(0,15)$);
                \coordinate (NW2) at ($(A)!0.3!(B)+(0,7)$);
                \coordinate (NE) at ($(B)!0.2!(A)+(0,15)$);
                \coordinate (NE2) at ($(B)!0.3!(A)+(0,7)$);
                \draw[webline] (0,5) -- (NW);
                \draw[webline] (0,5) -- (NE);
                \draw[wline] (0,0) -- (0,5);
            }\ 
        },
        &
        \mathord{
            \ \tikz[baseline=-.6ex, scale=.08, yshift=-5cm]{
                \coordinate (A) at (-8,0);
                \coordinate (B) at (8,0);
                \coordinate (NW) at ($(A)!0.2!(B)+(0,15)$);
                \coordinate (NW2) at ($(A)!0.3!(B)+(0,7)$);
                \coordinate (NE) at ($(B)!0.2!(A)+(0,15)$);
                \coordinate (NE2) at ($(B)!0.3!(A)+(0,7)$);
                \draw[webline] (0,5) to[out=east, in=east] (NW);
                \draw[overwline] (0,5) to[out=west, in=west] (NE);
                \draw[webline] (0,0) -- (0,5);
            }\ 
        }
        &=q^{-4}
        \mathord{
            \ \tikz[baseline=-.6ex, scale=.08, yshift=-5cm]{
                \coordinate (A) at (-8,0);
                \coordinate (B) at (8,0);
                \coordinate (NW) at ($(A)!0.2!(B)+(0,15)$);
                \coordinate (NW2) at ($(A)!0.3!(B)+(0,7)$);
                \coordinate (NE) at ($(B)!0.2!(A)+(0,15)$);
                \coordinate (NE2) at ($(B)!0.3!(A)+(0,7)$);
                \draw[webline] (0,5) -- (NW);
                \draw[wline] (0,5) -- (NE);
                \draw[webline] (0,0) -- (0,5);
            }\ 
        },
        &
        \mathord{
            \ \tikz[baseline=-.6ex, scale=.08, yshift=-5cm]{
                \coordinate (A) at (-8,0);
                \coordinate (B) at (8,0);
                \coordinate (NW) at ($(A)!0.2!(B)+(0,15)$);
                \coordinate (NW2) at ($(A)!0.3!(B)+(0,7)$);
                \coordinate (NE) at ($(B)!0.2!(A)+(0,15)$);
                \coordinate (NE2) at ($(B)!0.3!(A)+(0,7)$);
                \draw[wline] (0,5) to[out=east, in=east] (NW);
                \draw[overarc] (0,5) to[out=west, in=west] (NE);
                \draw[webline] (0,0) -- (0,5);
            }\ 
        }
        &=q^{-4}
        \mathord{
            \ \tikz[baseline=-.6ex, scale=.08, yshift=-5cm]{
                \coordinate (A) at (-8,0);
                \coordinate (B) at (8,0);
                \coordinate (NW) at ($(A)!0.2!(B)+(0,15)$);
                \coordinate (NW2) at ($(A)!0.3!(B)+(0,7)$);
                \coordinate (NE) at ($(B)!0.2!(A)+(0,15)$);
                \coordinate (NE2) at ($(B)!0.3!(A)+(0,7)$);
                \draw[wline] (0,5) -- (NW);
                \draw[webline] (0,5) -- (NE);
                \draw[webline] (0,0) -- (0,5);
            }\ 
        }.
    \end{align*}
\end{proof}

\begin{rem}
    The sticking trick in this form corresponds to the sticking trick in the stated $\mathfrak{sp}_4$-skein algebra via an isomorphism constructed in \cite{IYstmk}.
\end{rem}

\subsection{A generating set of the \texorpdfstring{$\mathfrak{sp}_4$}{sp4}-skein algebra}
Roughly speaking, in order to construct an inclusion of the boundary-localized $\mathfrak{sp}_4$-skein algebra 
into the quantum cluster algebra 
we need to write an $\mathfrak{sp}_4$-web as a (not necessarily positive) polynomial of $\mathfrak{sp}_4$-webs that correspond to cluster variables.
For this, let us first prepare some classes of $\mathfrak{sp}_4$-graph diagrams that give rise to nice generating sets of $\Skein{\Sigma}$ and $\Skein{\Sigma}[\partial^{-1}]$.

\begin{dfn}\label{def:diagram}
    A \emph{tangled loop (resp.~arc) diagram} on $\Sigma$ is a tangled $\mathfrak{sp}_4$-graph diagram given by a map from a circle (resp. interval) of type~$1$ or type~$2$.
    A \emph{tangled loop (resp.~arc) diagram with legs} on $\Sigma$ is a connected tangled $\mathfrak{sp}_4$-graph diagram on $\Sigma$ obtained by connecting a tangled loop (resp.~arc) diagram of type~$1$ to special points by attaching several edges of type~$2$ (called \emph{legs}).
    A \emph{tangled triad diagram} on $\Sigma$ is a tangled arc diagram of type~$1$ with only one leg.
    We say that a tangled $\mathfrak{sp}_4$-graph diagram is \emph{simple} if it has no internal crossings.
    
    A tangled arc diagram on $\Sigma$ is said to be \emph{descending} if one passes every self-crossing points through an over-pass first, following some fixed orientation.
    A tangled loop diagram on $\Sigma$ with a basepoint is said to be \emph{descending} if an oriented tangled arc starting from the basepoint satisfies the descending property. 
    A tangled loop or arc diagram with legs on $\Sigma$ is said to be \emph{descending} if its legs have no internal crossings, and the diagram obtained by removing the legs is descending. 
\end{dfn}

\begin{dfn}\label{def:web-set}
    \begin{enumerate}
        \item We denote by $\Desc{\Sigma}$ the set of all 
        the descending loop/arc diagrams of type~$1$ with or without legs. (Here notice that it is possible that their ends share a common special point.)
        \item A \emph{stated end} of type~$1$ means an end of either types
        $
        \mathord{
            \ \tikz[baseline=-.6ex, scale=.05, yshift=-5cm]{
                \coordinate (A1) at (-10,0);
                \coordinate (A2) at (10,0);
                \coordinate (C0) at ($(A1)!0.5!(A2)+(0,12)$);
                \coordinate (B0) at ($(A1)!0.5!(A2)+(0,5)$);
                \draw[webline] (C0) -- (B0);
                \draw[wline] (A1) -- (B0);
                \draw[webline] (A2) -- (B0);
                \bdryline{(A1)}{(A2)}{2cm}
                \draw[fill] (A1) circle (30pt);
                \draw[fill] (A2) circle (30pt);
            }\ 
        }
        $
        ,
        $
        \mathord{
            \ \tikz[baseline=-.6ex, scale=.05, yshift=-5cm]{
                \coordinate (A1) at (-10,0);
                \coordinate (A2) at (10,0);
                \coordinate (C0) at ($(A1)!0.5!(A2)+(0,12)$);
                \coordinate (B0) at ($(A1)!0.5!(A2)+(0,5)$);
                \draw[webline] (C0) -- (B0);
                \draw[webline] (A1) -- (B0);
                \draw[wline] (A2) -- (B0);
                \bdryline{(A1)}{(A2)}{2cm}
                \draw[fill] (A1) circle (30pt);
                \draw[fill] (A2) circle (30pt);
            }\ 
        }
        $
        , or
        $
        \mathord{
            \ \tikz[baseline=-.6ex, scale=.05, yshift=-5cm]{
                \coordinate (A0) at (0,0);
                \coordinate (A1) at (-10,0);
                \coordinate (A2) at (10,0);
                \coordinate (C0) at ($(A1)!0.5!(A2)+(0,12)$);
                \coordinate (B0) at ($(A1)!0.5!(A2)+(0,5)$);
                \draw[webline] (C0) -- (A0);
                \bdryline{(A1)}{(A2)}{2cm}
                \draw[fill] (A0) circle (30pt);
            }\ 
        }
        $.
        \item The set $\DescWil{\Sigma}$ of \emph{descending Wilson lines of type $1$} consists of all descending $\mathfrak{sp}_4$-graphs obtained by connecting two stated ends on distinct boundary intervals by a descending tangled arc of type $1$.
        \item The set $\SimpWil{\Sigma}$ of \emph{simple Wilson lines of type $1$} consists of all simple $\mathfrak{sp}_4$-graphs obtained by connecting two stated ends on distinct boundary intervals by a simple arc of type $1$.
\end{enumerate}
\end{dfn}

\begin{rem}\label{rem:Wilson_line}
A simple Wilson line of type $1$ corresponds to a matrix entry of a simple Wilson line (\cite{IOS}) in the representation $V(\varpi_1)$, up to boundary webs. See also \cite{IYstmk}. 
\end{rem}
We are going to discuss a generating set of the $\mathfrak{sp}_4$-skein algebra. 
For example, $\Skein{T}$ is generated $\SimpWil{T}$ by \cref{thm:Eweb-T} for a triangle $T$.

\begin{prop}\label{prop:desc-gen}
    For any unpunctured marked surface $\Sigma$, any basis web is expressed as a polynomial in elements of $\Desc{\Sigma}$, simple loops and arcs of type $2$ in $\Skein{\Sigma}$ with coefficients in $\bZ_q$.
    In particular, $\Desc{\Sigma}$ generates $\Skein{\Sigma}$ as an $\cR$-algebra.
\end{prop}
\begin{proof}
    Let us define the \emph{complexity} $|G|$ of a crossroad diagram $G$ to be the number of crossroads and internal crossings of $G$, where a \emph{crossroad diagram} is a tangled $\mathfrak{sp}_4$-graph with $4$-valent vertices corresponding to crossroads and internal crossings and no rungs.
    Note that if a connected crossroad diagram $G$ has the complexity $|G|=0$, then $G$ is a simple loop or arc diagram with or without legs.
    We claim that a connected flat crossroad diagram $G$ is written as a $\bZ_q$-linear combination of $XG'$ and flat crossroad diagrams with complexity less than $|G|$, where $X\in\Desc{\Sigma}$ and $G'$ is a flat crossroad diagram with $|G'|<|G|$.
    Then by induction on the complexity, one can show that any basis web is written as a polynomial in $\Desc{\Sigma}$ over $\bZ_q$ because a basis web is represented as a product of connected flat crossroad diagrams.
    It concludes that $\Desc{\Sigma}$ generates $\Skein{\Sigma}$ as an $\cR$-algebra.
    
    Let $G$ be any connected flat crossroad diagram on $\Sigma$, and choose its basepoint on a type~$1$ edge.
    Here if $G$ has a type~$1$ edge incident to a special point, then we take the basepoint on such an edge.    
    We start from the basepoint, and arriving at a crossroad, replace it with an internal crossing by:
    \begin{align}
        \mathord{
            \ \tikz[baseline=-.6ex, scale=.08, rotate=90]{
                \draw[dashed, fill=white] (0,0) circle [radius=7];
                \draw[webline] (45:7) -- (-135:7);
                \draw[webline] (135:7) -- (-45:7);
                \draw[fill=pink, thick] (0,0) circle [radius=30pt];
            }
        \ }
        =
        \mathord{
            \ \tikz[baseline=-.6ex, scale=.08]{
                \draw[dashed, fill=white] (0,0) circle [radius=7];
                \draw[red, very thick] (-45:7) -- (135:7);
                \draw[overarc] (-135:7) -- (45:7);
                }
        \ }
        -q
        \mathord{
            \ \tikz[baseline=-.6ex, scale=.08]{
                \draw[dashed, fill=white] (0,0) circle [radius=7];
                \draw[webline] (-45:7) to[out=north west, in=south] (3,0) to[out=north, in=south west] (45:7);
                \draw[webline] (-135:7) to[out=north east, in=south] (-3,0) to[out=north, in=south east] (135:7);
            }
        \ }
        -q^{-1}
        \mathord{
            \ \tikz[baseline=-.6ex, scale=.08, rotate=90]{
                \draw[dashed, fill=white] (0,0) circle [radius=7];
                \draw[webline] (-45:7) to[out=north west, in=south] (3,0) to[out=north, in=south west] (45:7);
                \draw[webline] (-135:7) to[out=north east, in=south] (-3,0) to[out=north, in=south east] (135:7);
            }
        \ }.\label{eq:crossroad-to-crossing}
    \end{align}

    Here we choose the crossing in such a way that our chosen path becomes an over-passing subarc.
    This operation produces two extra diagrams with complexity less than $|G|-1$: see the lower two diagrams in the middle column in \cref{fig:desc-gen}.
    We repeat this procedure until we return to the basepoint or arrive at a special point.
    See diagrams in the top route in \cref{fig:desc-gen}.
    If one arrives at a type~$2$ edge (incident to a special point) in this procedure, turn at the vertex there and go to the other type~$1$ edge.
    
    By construction, our trail $X$ becomes a descending loop or an arc diagram with/without legs, and another component is a flat crossroad diagram $G'$ with $|G'|<|G|$ since $X$ removes at least one crossroad from $G$.
    Although crossroad diagrams other than $XG'$ arise in this procedure, they have complexity less than $|G|$.
    These redundant crossroad diagrams become a $\bZ_q$-linear combination of $\Bweb{\Sigma}$ by the reduction rules in \cref{def:reduction} and \eqref{eq:crossroad-to-crossing}, since the reduction rules related to crossroads and \eqref{eq:crossroad-to-crossing} all have coefficients in $\bZ_q$ and do not increase the complexity.
    Consequently, we can decompose $G$ into $XG'+\sum_{i}\lambda_iG_i$ such that $X\in\Desc{\Sigma}$, $G'$ and $G_i$ are flat crossroad diagrams with $|G'|<|G|$ and $|G_i|<|G|$. Thus the first assertion is proved.
    
    We can exclude simple loops and arcs of type~$2$ from generators of $\Skein{\Sigma}$ as $\cR$-algebra. Indeed, one can produce a bigon by deforming a simple loop of type~$2$ using \eqref{rel:bigon} and apply it to the algorithm above (notice that this produces a factor $[2]^{-1}$). In a similar way, we consider a simple arc of type $2$ as a trivial loop with two legs. Thus the second assertion is proved.
\end{proof}

\begin{figure}
    \begin{tikzcd}
        \begin{tikzpicture}[scale=.1, yshift=-2cm]
            \foreach \i in {0,1,...,4}
            \foreach \j in {0,1,...,4}
            {
                \coordinate (P\i\j) at (4*\i,3*\j);
            }
            \draw[webline] (P00) -- (P22);
            \draw[webline] (P20) -- (P02);
            \draw[webline] (P22) -- (P44);
            \draw[webline] (P42) -- (P24);
            \draw[fill=pink, thick] (P11) circle [radius=20pt];
            \draw[fill=pink, thick] (P33) circle [radius=20pt];
            \node at (P20) [below]{\scriptsize $G$};
        \end{tikzpicture}
        \arrow{r}{}\arrow{rd}{}\arrow{rdd}{}
        &
        \begin{tikzpicture}[scale=.1, yshift=-2cm]
            \foreach \i in {0,1,...,4}
            \foreach \j in {0,1,...,4}
            {
                \coordinate (P\i\j) at (4*\i,3*\j);
            }
            \draw[webline] (P20) -- (P02);
            \draw[overarc] (P00) -- (P22);
            \draw[webline] (P22) -- (P44);
            \draw[webline] (P42) -- (P24);
            \draw[fill=pink, thick] (P33) circle [radius=20pt];
            \node at (P20) [below]{\scriptsize \text{complexity} $|G|$};
        \end{tikzpicture}
        \arrow{r}{}
        &
        \begin{tikzpicture}[scale=.1, yshift=-2cm]
            \foreach \i in {0,1,...,4}
            \foreach \j in {0,1,...,4}
            {
                \coordinate (P\i\j) at (4*\i,3*\j);
            }
            \draw[webline] (P20) -- (P02);
            \draw[webline] (P42) -- (P24);
            \draw[overarc] (P00) -- (P22);
            \draw[overarc] (P22) -- (P44);
            \node at (P20) [below]{\scriptsize \text{complexity} $|G|$};
        \end{tikzpicture}
        \\[-2em]
        &
        \begin{tikzpicture}[scale=.1]
            \foreach \i in {0,1,...,4}
            \foreach \j in {0,1,...,4}
            {
                \coordinate (P\i\j) at (4*\i,3*\j);
            }
            \draw[webline] (P00) to[out=north east, in=north west] (P20);
            \draw[webline] (P02) to[out=south east, in=south west] (P22);
            \draw[webline] (P22) -- (P44);
            \draw[webline] (P42) -- (P24);
            \draw[fill=pink, thick] (P33) circle [radius=20pt];
            \node at (P20) [below]{\scriptsize \text{complexity} $|G|-1$};
        \end{tikzpicture}
        &
        \begin{tikzpicture}[scale=.1]
            \foreach \i in {0,1,...,4}
            \foreach \j in {0,1,...,4}
            {
                \coordinate (P\i\j) at (4*\i,3*\j);
            }
            \draw[webline] (P20) -- (P02);
            \draw[overarc] (P00) -- (P22);
            \draw[webline] (P22) to[out=north east, in=south east] (P24);
            \draw[webline] (P42) to[out=north west, in=south west] (P44);
            \node at (P20) [below]{\scriptsize \text{complexity} $|G|-1$};
        \end{tikzpicture}
        \\[-2em]
        &
        \begin{tikzpicture}[scale=.1]
            \foreach \i in {0,1,...,4}
            \foreach \j in {0,1,...,4}
            {
                \coordinate (P\i\j) at (4*\i,3*\j);
            }
            \draw[webline] (P00) to[out=north east, in=north west] (P20);
            \draw[webline] (P02) to[out=south east, in=south west] (P22);
            \draw[webline] (P22) -- (P44);
            \draw[webline] (P42) -- (P24);
            \draw[fill=pink, thick] (P33) circle [radius=20pt];
            \node at (P20) [below]{\scriptsize \text{complexity} $|G|-1$};
        \end{tikzpicture}
        &
        \begin{tikzpicture}[scale=.1]
            \foreach \i in {0,1,...,4}
            \foreach \j in {0,1,...,4}
            {
                \coordinate (P\i\j) at (4*\i,3*\j);
            }
            \draw[webline] (P20) -- (P02);
            \draw[overarc] (P00) -- (P22);
            \draw[webline] (P22) to[out=north east, in=north west] (P42);
            \draw[webline] (P24) to[out=south east, in=south west] (P44);
            \node at (P20) [below]{\scriptsize \text{complexity} $|G|-1$};
        \end{tikzpicture}
    \end{tikzcd}
    \caption{The top route removes a descending loop or arc $X$ from a flat crossroad diagram $G$. The other resulting graphs are described as a $\bZ_q$-linear combination of crossroad diagrams with complexity less than $|G|$.}
    \label{fig:desc-gen}
\end{figure}
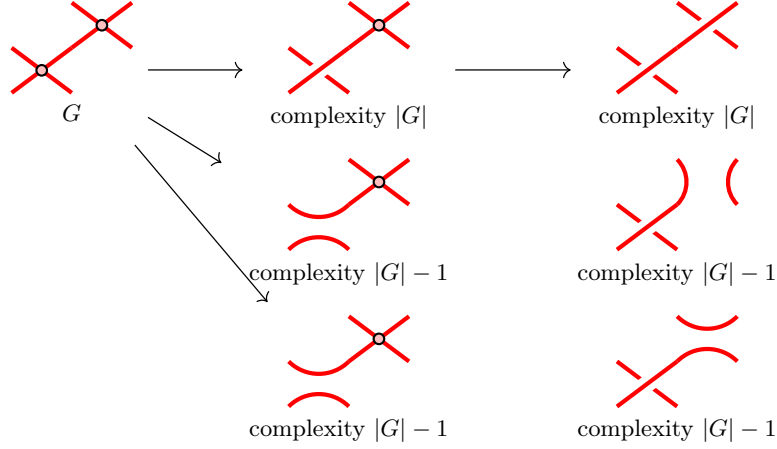

\begin{rem}\label{rem:basept}
        In the proof of \cref{prop:desc-gen}, we can choose a basepoint of a descending loop diagram (with legs) on an arbitrary subarc.
        Indeed, if $G$ and $G'$ are two descending loop diagrams with the same projection to $\Sigma$ such that their basepoints are incident to each other across a single subarc, then these two descending loop diagrams are related by
        \begin{align}
            \mathord{
                \ \tikz[baseline=-.6ex, scale=.08]{
                    \draw[dashed, fill=white] (0,0) circle [radius=7];
                    \draw[red, very thick] (-45:7) -- (135:7);
                    \draw[overarc] (-135:7) -- (45:7);
                    }
            \ }
            -
            \mathord{
                \ \tikz[baseline=-.6ex, scale=.08, rotate=90]{
                    \draw[dashed, fill=white] (0,0) circle [radius=7];
                    \draw[red, very thick] (-45:7) -- (135:7);
                    \draw[overarc] (-135:7) -- (45:7);
                    }
            \ }
            =(q-q^{-1})\left(
            \mathord{
                \ \tikz[baseline=-.6ex, scale=.08]{
                    \draw[dashed, fill=white] (0,0) circle [radius=7];
                    \draw[webline] (-45:7) to[out=north west, in=south] (3,0) to[out=north, in=south west] (45:7);
                    \draw[webline] (-135:7) to[out=north east, in=south] (-3,0) to[out=north, in=south east] (135:7);
                }
            \ }
            -
            \mathord{
                \ \tikz[baseline=-.6ex, scale=.08, rotate=90]{
                    \draw[dashed, fill=white] (0,0) circle [radius=7];
                    \draw[webline] (-45:7) to[out=north west, in=south] (3,0) to[out=north, in=south west] (45:7);
                    \draw[webline] (-135:7) to[out=north east, in=south] (-3,0) to[out=north, in=south east] (135:7);
                }
            \ }
            \right).\label{rel:crossing-change-11}
        \end{align}
        Hence, the complexities of tangled graph diagrams appearing in the difference between $G$ and $G'$ are less than $|G|$.
        In a similar way to the proof in \cref{prop:desc-gen}, the generator $G$ can be replaced by $G'$.
\end{rem}

One can obtain simple generators of the boundary-localized skein algebra $\Skein{\Sigma}[\partial^{-1}]$ by refining descending generators in $\Desc{\Sigma}$.
\begin{thm}\label{thm:localized-gen}
    Let $\Sigma$ be a connected unpunctured marked surface with at least two special points. 
    Then, any element in $\Desc{\Sigma}$ and simple loops/arcs of type~$2$ are expressed as polynomials in $\SimpWil{\Sigma}\cup\partial_{\Sigma}^{-1}$ with coefficients in $\bZ_q$, where $\partial_{\Sigma}^{-1}$ is the set of inverses of boundary webs in $\Skein{\Sigma}[\partial^{-1}]$.
    In particular, $\SimpWil{\Sigma}\cup\partial_{\Sigma}^{-1}$ is a generating set of $\Skein{\Sigma}[\partial^{-1}]$ as $\cR$-algebra. 
\end{thm}
\begin{proof}
    We are going to use the sticking tricks to further expand an element in $\Desc{\Sigma}$ or a simple loop/arc of type~$2$. 
    
    Firstly, by using the following relation, we can expand a discending arc of type~$2$ to a linear combination of tangles type~$1$ with or without legs and inverses of boundary intervals:
    \begin{align*}
        \mathord{
            \ \tikz[baseline=-.6ex, scale=.1]{
                \foreach \i in {0,1,...,8}
                \foreach \j in {0,1,...,8}
                {
                    \coordinate (P\i\j) at (\i*2-8,\j-4);
                }
                \draw[wline] (P01) -- (P87);
                \draw[webline] (P01) to[bend right=60] (P07);
                \draw[webline] (P81) to[bend left=60] (P87);
                \bdryline{(P00)}{(P08)}{-2cm}
                \bdryline{(P80)}{(P88)}{2cm}
                \draw[fill] (P01) circle (20pt);
                \draw[fill] (P07) circle (20pt);
                \draw[fill] (P81) circle (20pt);
                \draw[fill] (P87) circle (20pt);
            }\ 
        }
        =
        \mathord{
            \ \tikz[baseline=-.6ex, scale=.1]{
                \foreach \i in {0,1,...,8}
                \foreach \j in {0,1,...,8}
                {
                    \coordinate (P\i\j) at (\i*2-8,\j-4);
                }
                \draw[webline] (P14) -- (P74);
                \draw[overarc] (P01) to[out=east, in=south west] (P44) to[out=north east, in=west] (P87);
                \draw[webline] (P07) -- (P14);
                \draw[wline] (P01) -- (P14);
                \draw[webline] (P81) -- (P74);
                \draw[wline] (P87) -- (P74);
                \bdryline{(P00)}{(P08)}{-2cm}
                \bdryline{(P80)}{(P88)}{2cm}
                \draw[fill] (P01) circle (20pt);
                \draw[fill] (P07) circle (20pt);
                \draw[fill] (P81) circle (20pt);
                \draw[fill] (P87) circle (20pt);
            }\ 
        }
        -q^{-1}
        \mathord{
            \ \tikz[baseline=-.6ex, scale=.1]{
                \foreach \i in {0,1,...,8}
                \foreach \j in {0,1,...,8}
                {
                    \coordinate (P\i\j) at (\i*2-8,\j-4);
                }
                \draw[webline] (P07) -- (P14);
                \draw[wline] (P01) -- (P14);
                \draw[webline] (P81) -- (P74);
                \draw[wline] (P87) -- (P74);
                \draw[webline] (P01) -- (P74);
                \draw[webline] (P14) -- (P87);
                \bdryline{(P00)}{(P08)}{-2cm}
                \bdryline{(P80)}{(P88)}{2cm}
                \draw[fill] (P01) circle (20pt);
                \draw[fill] (P07) circle (20pt);
                \draw[fill] (P81) circle (20pt);
                \draw[fill] (P87) circle (20pt);
            }\ 
        }.
    \end{align*}
    
    Secondly, for a simple loop of type~$2$, stick it to a boundary interval $E$ by \eqref{eq:stick-2}. Then the resulting $\mathfrak{sp}_4$-graphs in the first and final terms in \eqref{eq:stick-2} are descending arcs of type~$2$, which we have discussed. 
    The second and forth terms in \eqref{eq:stick-2} can be replaced with tangles of type $1$ with or without legs by
    \begin{align*}
        \mathord{
            \ \tikz[baseline=-.6ex, scale=.1]{
                \foreach \i in {0,1,...,8}
                \foreach \j in {0,1,...,8}
                {
                    \coordinate (P\i\j) at (\i*2-8,\j-4);
                }
                \draw[wline] (P34) -- (P54);
                \draw[wline] (P07) -- (P14);
                \draw[webline] (P01) -- (P14);
                \draw[webline] (P01) -- (P34);
                \draw[webline] (P14) -- (P34);
                \draw[wline] (P87) -- (P74);
                \draw[webline] (P81) -- (P74);
                \draw[webline] (P81) -- (P54);
                \draw[webline] (P74) -- (P54);
                \bdryline{(P00)}{(P08)}{-2cm}
                \bdryline{(P80)}{(P88)}{2cm}
                \draw[fill] (P01) circle (20pt);
                \draw[fill] (P07) circle (20pt);
                \draw[fill] (P81) circle (20pt);
                \draw[fill] (P87) circle (20pt);
            }\ 
        }
        &=
        \mathord{
            \ \tikz[baseline=-.6ex, scale=.1]{
                \foreach \i in {0,1,...,8}
                \foreach \j in {0,1,...,8}
                {
                    \coordinate (P\i\j) at (\i*2-8,\j-4);
                }
                \draw[wline] (P87) -- (P74);
                \draw[webline] (P81) -- (P74);
                \draw[webline] (P81) -- (P14);
                \draw[wline] (P07) -- (P14);
                \draw[overarc] (P01) -- (P74);
                \draw[webline] (P01) -- (P14);
                \bdryline{(P00)}{(P08)}{-2cm}
                \bdryline{(P80)}{(P88)}{2cm}
                \draw[fill] (P01) circle (20pt);
                \draw[fill] (P07) circle (20pt);
                \draw[fill] (P81) circle (20pt);
                \draw[fill] (P87) circle (20pt);
            }\ 
        }
        -q^{-1}
        \mathord{
            \ \tikz[baseline=-.6ex, scale=.1]{
                \foreach \i in {0,1,...,8}
                \foreach \j in {0,1,...,8}
                {
                    \coordinate (P\i\j) at (\i*2-8,\j-4);
                }
                \draw[webline] (P14) -- (P74);
                \draw[wline] (P07) -- (P14);
                \draw[webline] (P01) -- (P14);
                \draw[wline] (P87) -- (P74);
                \draw[webline] (P81) -- (P74);
                \draw[webline] (P01) -- (P81);
                \bdryline{(P00)}{(P08)}{-2cm}
                \bdryline{(P80)}{(P88)}{2cm}
                \draw[fill] (P01) circle (20pt);
                \draw[fill] (P07) circle (20pt);
                \draw[fill] (P81) circle (20pt);
                \draw[fill] (P87) circle (20pt);
            }\ 
        }, \text{or}\\
        \mathord{
            \ \tikz[baseline=-.6ex, scale=.1]{
                \foreach \i in {0,1,...,8}
                \foreach \j in {0,1,...,8}
                {
                    \coordinate (P\i\j) at (\i*2-8,\j-4);
                }
                \draw[wline] (P34) -- (P54);
                \draw[wline] (P07) -- (P14);
                \draw[webline] (P01) -- (P14);
                \draw[webline] (P01) -- (P34);
                \draw[webline] (P14) -- (P34);
                \draw[wline] (P81) -- (P74);
                \draw[webline] (P87) -- (P74);
                \draw[webline] (P87) -- (P54);
                \draw[webline] (P74) -- (P54);
                \bdryline{(P00)}{(P08)}{-2cm}
                \bdryline{(P80)}{(P88)}{2cm}
                \draw[fill] (P01) circle (20pt);
                \draw[fill] (P07) circle (20pt);
                \draw[fill] (P81) circle (20pt);
                \draw[fill] (P87) circle (20pt);
            }\ 
        }
        &=
        \mathord{
            \ \tikz[baseline=-.6ex, scale=.1]{
                \foreach \i in {0,1,...,8}
                \foreach \j in {0,1,...,8}
                {
                    \coordinate (P\i\j) at (\i*2-8,\j-4);
                }
                \draw[webline] (P14) -- (P74);
                \draw[overarc] (P01) to[out=east, in=south west] (P44) to[out=north east, in=west] (P87);
                \draw[wline] (P07) -- (P14);
                \draw[webline] (P01) -- (P14);
                \draw[wline] (P81) -- (P74);
                \draw[webline] (P87) -- (P74);
                \bdryline{(P00)}{(P08)}{-2cm}
                \bdryline{(P80)}{(P88)}{2cm}
                \draw[fill] (P01) circle (20pt);
                \draw[fill] (P07) circle (20pt);
                \draw[fill] (P81) circle (20pt);
                \draw[fill] (P87) circle (20pt);
            }\ 
        }
        -q^{-1}
        \mathord{
            \ \tikz[baseline=-.6ex, scale=.1]{
                \foreach \i in {0,1,...,8}
                \foreach \j in {0,1,...,8}
                {
                    \coordinate (P\i\j) at (\i*2-8,\j-4);
                }
                \draw[wline] (P07) -- (P14);
                \draw[webline] (P01) -- (P14);
                \draw[wline] (P81) -- (P74);
                \draw[webline] (P87) -- (P74);
                \draw[webline] (P01) -- (P74);
                \draw[webline] (P14) -- (P87);
                \bdryline{(P00)}{(P08)}{-2cm}
                \bdryline{(P80)}{(P88)}{2cm}
                \draw[fill] (P01) circle (20pt);
                \draw[fill] (P07) circle (20pt);
                \draw[fill] (P81) circle (20pt);
                \draw[fill] (P87) circle (20pt);
            }\ 
        }.
    \end{align*}
    The third term has a coefficient $[2]$ thus it expressed as 
    \begin{align*}
        [2]
        \mathord{
            \ \tikz[baseline=-.6ex, scale=.1]{
                \foreach \i in {0,1,...,8}
                \foreach \j in {0,1,...,8}
                {
                    \coordinate (P\i\j) at (\i*2-8,\j-4);
                }
                \draw[webline] (P07) -- (P24);
                \draw[webline] (P01) -- (P24);
                \draw[wline] (P24) -- (P64);
                \draw[webline] (P81) -- (P64);
                \draw[webline] (P87) -- (P64);
                \bdryline{(P00)}{(P08)}{-2cm}
                \bdryline{(P80)}{(P88)}{2cm}
                \draw[fill] (P01) circle (20pt);
                \draw[fill] (P07) circle (20pt);
                \draw[fill] (P81) circle (20pt);
                \draw[fill] (P87) circle (20pt);
            }\ 
        }
        &=[2]
        \mathord{
            \ \tikz[baseline=-.6ex, scale=.1]{
                \foreach \i in {0,1,...,8}
                \foreach \j in {0,1,...,8}
                {
                    \coordinate (P\i\j) at (\i*2-8,\j-4);
                }
                \draw[webline] (P81) -- (P07);
                \draw[overarc] (P01) -- (P87);
                \bdryline{(P00)}{(P08)}{-2cm}
                \bdryline{(P80)}{(P88)}{2cm}
                \draw[fill] (P01) circle (20pt);
                \draw[fill] (P07) circle (20pt);
                \draw[fill] (P81) circle (20pt);
                \draw[fill] (P87) circle (20pt);
            }\ 
        }
        -q^{-2}
        \mathord{
            \ \tikz[baseline=-.6ex, scale=.1]{
                \foreach \i in {0,1,...,8}
                \foreach \j in {0,1,...,8}
                {
                    \coordinate (P\i\j) at (\i*2-8,\j-4);
                }
                \draw[webline] (P81) to[out=west, in=west] (P87);
                \draw[webline] (P01) to[out=east, in=east] (P07);
                \bdryline{(P00)}{(P08)}{-2cm}
                \bdryline{(P80)}{(P88)}{2cm}
                \draw[fill] (P01) circle (20pt);
                \draw[fill] (P07) circle (20pt);
                \draw[fill] (P81) circle (20pt);
                \draw[fill] (P87) circle (20pt);
            }\ 
        }
        -q
        \mathord{
            \ \tikz[baseline=-.6ex, scale=.1]{
                \foreach \i in {0,1,...,8}
                \foreach \j in {0,1,...,8}
                {
                    \coordinate (P\i\j) at (\i*2-8,\j-4);
                }
                \draw[webline] (P87) -- (P07);
                \draw[webline] (P01) -- (P81);
                \bdryline{(P00)}{(P08)}{-2cm}
                \bdryline{(P80)}{(P88)}{2cm}
                \draw[fill] (P01) circle (20pt);
                \draw[fill] (P07) circle (20pt);
                \draw[fill] (P81) circle (20pt);
                \draw[fill] (P87) circle (20pt);
            }\ 
        }.
    \end{align*}
    We can make these resulting tangles of type $1$ with or without legs into polynomials of descending arc of type $1$ with or without legs by crossing chenges \cref{rel:crossing-change-11}.
    
    Let us fix a boundary interval $E$ and denote the corresponding boundary web of type~$s$ by $e_{s}$ for $s=1,2$.
    Given a descending loop diagram $G$ with or without legs, we can take its basepoint on a subarc $\gamma$ between two crossing points of $G$ such that $\gamma$ and $E$ share the same face in $\Sigma\setminus G$ by \cref{rem:basept}.
    Then, one can apply the sticking trick  \eqref{eq:stick-1} to $\gamma$ by multiplying $e_1e_2$ to $G$ to expand it into a $\bZ_q$-linear combination of descending arc diagrams with/without legs. 
    Then it remains to consider a descending arc diagram $G$ with/without legs.
    For each leg of $G$, one can apply the sticking trick \eqref{eq:stick-1} as follows:
    \begin{align*}
        \mathord{
            \ \tikz[baseline=-.6ex, scale=.08, yshift=-5cm]{
                \coordinate (A1) at (-10,0);
                \coordinate (A2) at (10,0);
                \coordinate (B0) at (-10,10);
                \coordinate (C0) at (-13,15);
                \coordinate (C1) at ($(A1)!0.2!(A2)+(0,15)$);
                \coordinate (C2) at ($(A1)!0.8!(A2)+(0,15)$);
                \coordinate (B1) at ($(A1)!0.2!(A2)+(0,10)$);
                \coordinate (B2) at ($(A1)!0.8!(A2)+(0,10)$);
                \draw[wline] (A1) to[bend left=60] (A2);
                \draw[webline] (A1) to[bend left=20] (A2);
                \draw[wline, shorten >=.3cm] (B0) -- (A1);
                \draw[webline, rounded corners] (C0) -- (B0) -- (B2) -- (C2);
                \bdryline{(A1)}{(A2)}{2cm}
                \draw[fill] (A1) circle (20pt);
                \draw[fill] (A2) circle (20pt);
                \node at (-2,10) [above]{\scriptsize $\longleftarrow$};
            }
        \ }
        &=q
        \mathord{
            \ \tikz[baseline=-.6ex, scale=.08, yshift=-5cm]{
                \coordinate (A1) at (-10,0);
                \coordinate (A2) at (10,0);
                \coordinate (C1) at ($(A1)!0.2!(A2)+(0,15)$);
                \coordinate (B1) at ($(A1)!0.2!(A2)+(0,10)$);
                \coordinate (C2) at ($(A1)!0.8!(A2)+(0,15)$);
                \coordinate (B2) at ($(A1)!0.8!(A2)+(0,10)$);
                \draw[overarc] (C0) -- (B0) to[out=east, in=north] (A2);
                \draw[overarc] (C2) to[out=south, in=north] (A1);
                \draw[wline, shorten <=.2cm, shorten >=.2cm] (A1) to[bend left] (A2);
                \draw[wline, shorten >=.4cm] (B0) -- (A1);
                \bdryline{(A1)}{(A2)}{2cm}
                \draw[fill] (A1) circle (20pt);
                \draw[fill] (A2) circle (20pt);
            }\ 
        }
        -q^2
        \mathord{
            \ \tikz[baseline=-.6ex, scale=.08, yshift=-5cm]{
                \coordinate (A1) at (-10,0);
                \coordinate (A2) at (10,0);
                \coordinate (C1) at ($(A1)!0.2!(A2)+(0,15)$);
                \coordinate (B1) at ($(A1)!0.2!(A2)+(0,7)$);
                \coordinate (C2) at ($(A1)!0.8!(A2)+(0,15)$);
                \coordinate (B2) at ($(A1)!0.8!(A2)+(0,7)$);
                \draw[webline] (C0) -- (B0) to[out=east, in=north] (B2);
                \draw[overarc] (C2) to[out=south, in=north] (B1);
                \draw[webline, shorten <=.3cm] (A1) -- (B2);
                \draw[overarc] (A2) -- (B1);
                \draw[wline] (A1) -- (B1);
                \draw[wline, shorten <=.3cm] (A2) -- (B2);
                \draw[wline, shorten >=.4cm] (B0) -- (A1);
                \bdryline{(A1)}{(A2)}{2cm}
                \draw[fill] (A1) circle (20pt);
                \draw[fill] (A2) circle (20pt);
            }\ 
        }
        +q^3
        \mathord{
            \ \tikz[baseline=-.6ex, scale=.08, yshift=-5cm]{
                \coordinate (A1) at (-10,0);
                \coordinate (A2) at (10,0);
                \coordinate (C1) at ($(A1)!0.2!(A2)+(0,15)$);
                \coordinate (B1) at ($(A1)!0.2!(A2)+(0,7)$);
                \coordinate (C2) at ($(A1)!0.8!(A2)+(0,15)$);
                \coordinate (B2) at ($(A1)!0.8!(A2)+(0,7)$);
                \draw[webline] (C0) -- (B0) -- (B1);
                \draw[webline] (C2) -- (B2);
                \draw[webline, shorten <=.3cm] (A2) -- (B1);
                \draw[overarc] (A1) -- (B2);
                \draw[wline, shorten <=.2cm] (A1) -- (B1);
                \draw[wline] (A2) -- (B2);
                \draw[wline, shorten >=.4cm] (B0) -- (A1);
                \bdryline{(A1)}{(A2)}{2cm}
                \draw[fill] (A1) circle (20pt);
                \draw[fill] (B) circle (20pt);
            }\ 
        }
        -q^4
        \mathord{
            \ \tikz[baseline=-.6ex, scale=.08, yshift=-5cm]{
                \coordinate (A1) at (-10,0);
                \coordinate (B) at (10,0);
                \coordinate (C1) at ($(A1)!0.2!(B)+(0,15)$);
                \coordinate (B1) at ($(A1)!0.2!(B)+(0,10)$);
                \coordinate (C2) at ($(A1)!0.8!(B)+(0,15)$);
                \coordinate (B2) at ($(A1)!0.8!(B)+(0,10)$);
                \draw[webline] (C0) -- (B0) to[out=east, in=north east] (A1);
                \draw[webline] (C2) to[out=south, in=north] (B);
                \draw[wline, shorten <=.2cm, shorten >=.2cm] (A1) to[bend left] (B);
                \draw[wline, shorten >=.4cm] (B0) -- (A1);
                \bdryline{(A1)}{(B)}{2cm}
                \draw[fill] (A1) circle (20pt);
                \draw[fill] (B) circle (20pt);
            }\ 
        }\\
        &=q^3
        \mathord{
            \ \tikz[baseline=-.6ex, scale=.08, yshift=-5cm]{
                \coordinate (A1) at (-10,0);
                \coordinate (A2) at (10,0);
                \coordinate (C1) at ($(A1)!0.2!(A2)+(0,15)$);
                \coordinate (B1) at ($(A1)!0.2!(A2)+(0,10)$);
                \coordinate (C2) at ($(A1)!0.8!(A2)+(0,15)$);
                \coordinate (B2) at ($(A1)!0.8!(A2)+(0,10)$);
                \draw[overarc] (C0) -- (B0) to[out=east, in=north] (A2);
                \draw[overarc] (C2) to[out=south, in=north] (A1);
                \draw[wline, shorten <=.4cm, shorten >=.2cm] (A1) to[bend left] (A2);
                \draw[wline, shorten >=.2cm] (B0) to[bend right] (A1);
                \bdryline{(A1)}{(A2)}{2cm}
                \draw[fill] (A1) circle (20pt);
                \draw[fill] (A2) circle (20pt);
            }\ 
        }
        -q^3
        \mathord{
            \ \tikz[baseline=-.6ex, scale=.08, yshift=-5cm]{
                \coordinate (A1) at (-10,0);
                \coordinate (A2) at (10,0);
                \coordinate (C1) at ($(A1)!0.2!(A2)+(0,15)$);
                \coordinate (B1) at ($(A1)!0.2!(A2)+(0,7)$);
                \coordinate (C2) at ($(A1)!0.8!(A2)+(0,15)$);
                \coordinate (B2) at ($(A1)!0.8!(A2)+(0,7)$);
                \draw[webline, shorten >=.2cm] (C0) -- (A1);
                \draw[overarc] (C2) to[out=south, in=north] (B1);
                \draw[webline] (A2) to[out=north, in=east] (B1);
                \draw[wline] (A1) -- (B1);
                \draw[wline, shorten <=.3cm, shorten >=.4cm] (A2) to[bend right] (A1);
                \bdryline{(A1)}{(A2)}{2cm}
                \draw[fill] (A1) circle (20pt);
                \draw[fill] (A2) circle (20pt);
            }\ 
        },
    \end{align*}
    where $G$ is descending along the arrow in the first diagram.
    Observe that the resulting diagrams have stated ends on boundary intervals.
    In other words, $G$ is expanded into a polynomial in elements of $\DescWil{\Sigma}$, (\cref{def:web-set} (3)) up to boundary webs.
    Let us denote by $S_{n} \subset \DescWil{\Sigma}$ the subset of descending Wilson lines whose minimal diagram has $n$ self-crossing points.
    We remark that $S_{0}=\SimpWil{\Sigma}$.
    
    We then claim that any tangled graph in $S_{n}$ can be written as a $\bZ_q$-linear combination of the product of elements in $S_{<n}:=\cup_{i<n} S_{i}$ and simple Wilson lines of type~$1$.
    For a tangled graph $G\in S_{n}$, we can describe a neighborhood of the subarc from an end to the first crossing of $G$ by forgetting arcs that sit in the lower layer as follows:
    \[
        \mathord{
            \ \tikz[baseline=-.6ex, scale=.08, yshift=-5cm]{
                \coordinate (A) at (0,0);
                \coordinate (A1) at (-10,0);
                \coordinate (A2) at (10,0);
                \coordinate (C) at ($(A1)!0.5!(A2)+(0,15)$);
                \coordinate (C1) at ($(A1)+(0,15)$);
                \coordinate (C2) at ($(A2)+(0,15)$);
                \coordinate (B) at ($(A1)!0.5!(A2)+(0,5)$);
                \coordinate (B1) at ($(A1)!0.2!(A2)+(0,5)$);
                \coordinate (B2) at ($(A1)!0.8!(A2)+(0,5)$);
                \draw[wline] (A1) -- (B);
                \draw[webline] (A2) -- (B);
                \draw[webline] (C) to[out=north, in=north] (C2);
                \draw[webline] (C2) to[out=south, in=south east] (C1);
                \draw[overarc] (C) -- (B);
                \draw[fill=gray] ($(C)!.5!(C2)+(0,-1)$) circle [radius=2];
                \bdryline{(A1)}{(A2)}{2cm}
                \draw[fill] (A1) circle (30pt);
                \draw[fill] (A2) circle (30pt);
            }\ 
        },
        \mathord{
            \ \tikz[baseline=-.6ex, scale=.08, yshift=-5cm]{
                \coordinate (A) at (0,0);
                \coordinate (A1) at (-10,0);
                \coordinate (A2) at (10,0);
                \coordinate (C) at ($(A1)!0.5!(A2)+(0,15)$);
                \coordinate (C1) at ($(A1)+(0,15)$);
                \coordinate (C2) at ($(A2)+(0,15)$);
                \coordinate (B) at ($(A1)!0.5!(A2)+(0,5)$);
                \coordinate (B1) at ($(A1)!0.2!(A2)+(0,5)$);
                \coordinate (B2) at ($(A1)!0.8!(A2)+(0,5)$);
                \draw[webline] (A1) -- (B);
                \draw[wline] (A2) -- (B);
                \draw[webline] (C) to[out=north, in=north] (C2);
                \draw[webline] (C2) to[out=south, in=south east] (C1);
                \draw[overarc] (C) -- (B);
                \draw[fill=gray] ($(C)!.5!(C2)+(0,-1)$) circle [radius=2];
                \bdryline{(A1)}{(A2)}{2cm}
                \draw[fill] (A1) circle (30pt);
                \draw[fill] (A2) circle (30pt);
            }\ 
        },
        \mathord{
            \ \tikz[baseline=-.6ex, scale=.08, yshift=-5cm]{
                \coordinate (A) at (0,0);
                \coordinate (A1) at (-10,0);
                \coordinate (A2) at (10,0);
                \coordinate (C) at ($(A1)!0.5!(A2)+(0,15)$);
                \coordinate (C1) at ($(A1)+(0,15)$);
                \coordinate (C2) at ($(A2)+(0,15)$);
                \coordinate (B) at ($(A1)!0.5!(A2)+(0,5)$);
                \coordinate (B1) at ($(A1)!0.2!(A2)+(0,5)$);
                \coordinate (B2) at ($(A1)!0.8!(A2)+(0,5)$);
                \draw[webline] (A1) to[out=north, in=south] (B);
                \draw[webline] (C) to[out=north, in=north] (C2);
                \draw[webline] (C2) to[out=south, in=south east] (C1);
                \draw[overarc] (C) -- (B);
                \draw[fill=gray] ($(C)!.5!(C2)+(0,-1)$) circle [radius=2];
                \bdryline{(A1)}{(A2)}{2cm}
                \draw[fill] (A1) circle (30pt);
                \draw[fill] (A2) circle (30pt);
            }\ 
        }\text{ and the left-turn version.}  
    \]
    Apply the sticking trick to the under-passing subarc of the first crossing point of $G$.
    One can see that it becomes $G=\sum_{i} a_i\gamma_i G_i$ where $G_i\in S_{<n}$ and
    \[
        \gamma_i=
        \mathord{
            \ \tikz[baseline=-.6ex, scale=.08, yshift=-5cm]{
                \coordinate (A) at (0,0);
                \coordinate (A1) at (-10,0);
                \coordinate (A2) at (10,0);
                \coordinate (C) at ($(A1)!0.5!(A2)+(0,15)$);
                \coordinate (C1) at ($(A1)+(0,15)$);
                \coordinate (C2) at ($(A2)+(0,15)$);
                \coordinate (B) at ($(A1)!0.5!(A2)+(0,5)$);
                \coordinate (B1) at ($(A1)!0.2!(A2)+(0,5)$);
                \coordinate (B2) at ($(A1)!0.8!(A2)+(0,5)$);
                \draw[wline] (A1) -- (B);
                \draw[webline] (A2) -- (B);
                \draw[webline] (C) to[out=north, in=north] (C2);
                \draw[webline] (C2) to[out=south, in=north] (A);
                \draw[fill=gray] ($(C)!.5!(C2)+(0,-1)$) circle [radius=2];
                \draw[fill=red!30] ($(A1)+(-2,0)$) rectangle ($(A2)+(2,10)$);
                \draw[fill=blue!40] (A1) rectangle ($(A2)+(0,5)$);
                \draw[overarc] (C) -- ($(A)+(0,5)$);
                \bdryline{(A1)}{(A2)}{2cm}
                \draw[fill] (A1) circle (30pt);
                \draw[fill] (A2) circle (30pt);
            }\ 
        } 
    \]
    where the filled area is one of the above ends.
    Finally, one can apply the sticking trick to the top subarc of $\gamma_i$ and the other boundary interval (which exists since $\Sigma$ has at least two special points), and the resulting graphs become simple Wilson line of type $1$.
    By induction on $n$, we conclude that $G$ is expressed as a $\bZ_q$-linear combination of $\SimpWil{\Sigma}$. Observe that all the coefficients appearing in the above argument are in $\bZ_q$.
\end{proof}

\begin{rem}
See \cite[Corollary 3.20 and Proposition 2.1]{IOS} for a closely related statement. 
\end{rem}

\subsection{The \texorpdfstring{$\bZ_{q}$}{Zq}-form of the \texorpdfstring{$\mathfrak{sp}_4$}{sp4}-skein algebra}
To compare the $\mathfrak{sp}_4$-skein algebra $\Skein{\Sigma}[\partial^{-1}]$ with the quantum cluster algebra with its natural coefficient ring $\bZ_q=\bZ[q^{\pm 1/2}]$, we introduce the following notion.
\begin{dfn}\label{def:Zv-form}
     The $\bZ_q$-span $\Skein{\Sigma}^{\bZ_q}:=\bZ_q\Bweb{\Sigma} \subset \Skein{\Sigma}$ of the basis webs on $\Sigma$ is called the \emph{$\bZ_q$-form of the $\mathfrak{sp}_4$-skein algebra $\Skein{\Sigma}$}.
\end{dfn}

The following lemma implies that the $\bZ_q$-form of $\Skein{\Sigma}$ is a $\bZ_q$-subalgebra of $\Skein{\Sigma}$, namely, the product $G_1G_2$ lies in $\Skein{\Sigma}^{\bZ_q}$ for any basis webs $G_1,G_2 \in \Bweb{\Sigma}$.
\begin{lem}\label{lem:crossroads-in-Zv}
    Any crossroad web belongs to $\Skein{\Sigma}^{\bZ_q}$.
\end{lem}
\begin{proof}
    Let $G$ be an $\mathfrak{sp}_4$-graph diagram with crossroads whose legs have no internal crossings.
    We know that $G$ reduces to a $\cR$-linear combination of $\Bweb{\Sigma}$ by the reduction rules in \cref{def:reduction}.
    Let us show that the coefficients in this expansion lie in $\bZ_q$ by induction on the \emph{weighted crossing number} $c(G)$ defined as 
    $
    c(G):=
    \#\mathord{
        \ \tikz[baseline=-.6ex, scale=.08]{
            \draw[dashed, fill=white] (0,0) circle [radius=5];
            \draw[webline] (-45:5) -- (135:5);
            \draw[overarc] (-135:5) -- (45:5);
            }
    \ }
    +2
    \#\mathord{
        \ \tikz[baseline=-.6ex, scale=.08]{
            \draw[dashed, fill=white] (0,0) circle [radius=5];
            \draw[webline] (-45:5) -- (135:5);
            \draw[overwline] (-135:5) -- (45:5);
            }
    \ }
    +2
    \#\mathord{
        \ \tikz[baseline=-.6ex, scale=.08]{
            \draw[dashed, fill=white] (0,0) circle [radius=5];
            \draw[wline] (-45:5) -- (135:5);
            \draw[overarc] (-135:5) -- (45:5);
            }
    \ }
    +4\#
    \mathord{
        \ \tikz[baseline=-.6ex, scale=.08]{
            \draw[dashed, fill=white] (0,0) circle [radius=5];
            \draw[wline] (-45:5) -- (135:5);
            \draw[overwline] (-135:5) -- (45:5);
            }
    \ }
    $.
    We remark that $c(G)$ is invariant under the Reidemeister moves:
    \begin{align}\label{eq:R4-preserve}
    \mathord{
        \ \tikz[baseline=-.6ex, scale=.1]{
            \draw[dashed, fill=white] (0,0) circle [radius=5];
            \draw[wline] (90:5) -- (0,-3);
            \draw[webline] (-60:5) -- (0,-3);
            \draw[webline] (-120:5) -- (0,-3);
            \draw[overarc] (0:5) to[bend right] (180:5);
            }
    \ }
    \leftrightarrow
    \mathord{
        \ \tikz[baseline=-.6ex, scale=.1]{
            \draw[dashed, fill=white] (0,0) circle [radius=5];
            \draw[wline] (90:5) -- (0,3);
            \draw[webline] (-60:5) -- (0,3);
            \draw[webline] (-120:5) -- (0,3);
            \draw[overarc] (0:5) to[bend left] (180:5);
            }
    \ },\quad
    \mathord{
        \ \tikz[baseline=-.6ex, scale=.1]{
            \draw[dashed, fill=white] (0,0) circle [radius=5];
            \draw[webline] (0:5) to[bend right] (180:5);
            \draw[overwline] (90:5) -- (0,-3);
            \draw[webline] (-60:5) -- (0,-3);
            \draw[webline] (-120:5) -- (0,-3);
            }
    \ }
    \leftrightarrow
    \mathord{
        \ \tikz[baseline=-.6ex, scale=.1]{
            \draw[dashed, fill=white] (0,0) circle [radius=5];
            \draw[webline] (0:5) to[bend left] (180:5);
            \draw[overarc] (-60:5) -- (0,3);
            \draw[overarc] (-120:5) -- (0,3);
            \draw[webline] (-60:5) -- (0,3);
            \draw[wline] (90:5) -- (0,3);
            }
    \ }.
    \end{align}
    If $G$ has no internal crossings, then it belongs to $\Skein{\Sigma}^{\bZ_q}$ since those among the reduction rules $S_i$ and $B_i$ in \cref{def:reduction} used in this procedure have coefficients in $\bZ_q$, thanks to the absence of rungs in $G$. 
    Here note that $[4]/[2]=q^2+q^{-2},[6]/([2][3])=q^{2}-1+q^{-2} \in \bZ_q$.

    Let us consider the cases where $G$ has an internal crossing (i) 
    $
    \mathord{
        \ \tikz[baseline=-.6ex, scale=.08]{
            \draw[dashed, fill=white] (0,0) circle [radius=5];
            \draw[webline] (-45:5) -- (135:5);
            \draw[overarc] (-135:5) -- (45:5);
            }
    \ }
    $,
    (ii)
    $
    \mathord{
        \ \tikz[baseline=-.6ex, scale=.08]{
            \draw[dashed, fill=white] (0,0) circle [radius=5];
            \draw[webline] (-45:5) -- (135:5);
            \draw[overwline] (-135:5) -- (45:5);
            }
    \ }
    $,
    $
    \mathord{
        \ \tikz[baseline=-.6ex, scale=.08]{
            \draw[dashed, fill=white] (0,0) circle [radius=5];
            \draw[wline] (-45:5) -- (135:5);
            \draw[overarc] (-135:5) -- (45:5);
            }
    \ }
    $,
    or
    (iii)
    $
    \mathord{
        \ \tikz[baseline=-.6ex, scale=.08]{
            \draw[dashed, fill=white] (0,0) circle [radius=5];
            \draw[wline] (-45:5) -- (135:5);
            \draw[overwline] (-135:5) -- (45:5);
            }
    \ }
    $.
    We apply the reduction rule $C_1$ to the crossing (i), which expands $G$ into a $\bZ_q$-linear combination of crossroad webs with weighted crossing numbers less than $c(G)$.
    Since $G$ has no rungs, the edge of type~$2$ in (ii) must be a tangled loop or an arc component.
    In the latter case, 
    $G$ is expanded as $G=qX+q^{-1}Y$ by the reduction rule $C_2$ or $C_3$, where $X$ and $Y$ are crossroad webs with $c(X)<c(G)$ and $c(Y)<c(G)$. Observe that the resulting edges of type~$2$ are legs. 
    In the former case, we expand similarly as $G=qX+q^{-1}Y$, where 
    $X$ and $Y$ have rungs.
    These rungs are replaced by crossroads as follows:
    \begin{align*}
        \mathord{
            \ \tikz[baseline=-.6ex, scale=.1]{
            \coordinate (P1) at (-6,6);
            \coordinate (P2) at (-6,-6);
            \coordinate (P3) at (6,-6);
            \coordinate (P4) at (6,6);
            \coordinate (Q1) at (0,6);
            \coordinate (Q2) at (-6,0);
            \coordinate (Q3) at (0,-6);
            \coordinate (Q4) at (6,0);
            \draw[webline] (Q2) -- (Q4);
            \draw[overwline] (Q1) -- (Q3);
            \draw[dashed] (P1) -- (P2) -- (P3) -- (P4) -- cycle;
            \draw[webline, dotted] (Q1) -- ($(Q1)+(0,2)$) -- ($(P4)+(2,2)$) -- ($(P3)+(2,-2)$) -- ($(Q3)+(0,-2)$) -- (Q3);
            }\ 
        }
        &=q
        \mathord{
            \ \tikz[baseline=-.6ex, scale=.1]{
            \coordinate (P1) at (-6,6);
            \coordinate (P2) at (-6,-6);
            \coordinate (P3) at (6,-6);
            \coordinate (P4) at (6,6);
            \coordinate (Q1) at (0,6);
            \coordinate (Q2) at (-6,0);
            \coordinate (Q3) at (0,-6);
            \coordinate (Q4) at (6,0);
            \draw[webline] (Q2) -- (Q4);
            \draw[wline] (Q1) -- (2,0);
            \draw[wline] (-2,0) -- (Q3);
            \draw[dashed] (P1) -- (P2) -- (P3) -- (P4) -- cycle;
            }\ 
        }
        +q^{-1}
        \mathord{
            \ \tikz[baseline=-.6ex, scale=.1]{
            \coordinate (P1) at (-6,6);
            \coordinate (P2) at (-6,-6);
            \coordinate (P3) at (6,-6);
            \coordinate (P4) at (6,6);
            \coordinate (Q1) at (0,6);
            \coordinate (Q2) at (-6,0);
            \coordinate (Q3) at (0,-6);
            \coordinate (Q4) at (6,0);
            \draw[webline] (Q2) -- (Q4);
            \draw[wline] (Q1) -- (-2,0);
            \draw[wline] (2,0) -- (Q3);
            \draw[dashed] (P1) -- (P2) -- (P3) -- (P4) -- cycle;
            }\ 
        }
        =q
        \mathord{
            \ \tikz[baseline=-.6ex, scale=.1]{
            \coordinate (P1) at (-6,6);
            \coordinate (P2) at (-6,-6);
            \coordinate (P3) at (6,-6);
            \coordinate (P4) at (6,6);
            \coordinate (Q1) at (0,6);
            \coordinate (Q2) at (-6,0);
            \coordinate (Q3) at (0,-6);
            \coordinate (Q4) at (6,0);
            \draw[webline, rounded corners] (Q2) -- (0,0) -- ($(Q1)-(1,0)$);
            \draw[webline, rounded corners] ($(Q1)+(1,0)$) -- (2,0) -- (Q4);
            \draw[webline, rounded corners] ($(Q3)-(1,0)$) -- (-3,-2) -- (-1,-2) -- ($(Q3)+(1,0)$);
            \draw[wline] (-2,0) -- (-2,-2);
            \draw[dashed] (P1) -- (P2) -- (P3) -- (P4) -- cycle;
            }\ 
        }
        +q^{-1}
        \mathord{
            \ \tikz[baseline=-.6ex, scale=.1, xscale=-1]{
            \coordinate (P1) at (-6,6);
            \coordinate (P2) at (-6,-6);
            \coordinate (P3) at (6,-6);
            \coordinate (P4) at (6,6);
            \coordinate (Q1) at (0,6);
            \coordinate (Q2) at (-6,0);
            \coordinate (Q3) at (0,-6);
            \coordinate (Q4) at (6,0);
            \draw[webline, rounded corners] (Q2) -- (0,0) -- ($(Q1)-(1,0)$);
            \draw[webline, rounded corners] ($(Q1)+(1,0)$) -- (2,0) -- (Q4);
            \draw[webline, rounded corners] ($(Q3)-(1,0)$) -- (-3,-2) -- (-1,-2) -- ($(Q3)+(1,0)$);
            \draw[wline] (-2,0) -- (-2,-2);
            \draw[dashed] (P1) -- (P2) -- (P3) -- (P4) -- cycle;
            }\ 
        }\\
        &=q
        \mathord{
            \ \tikz[baseline=-.6ex, scale=.1]{
            \coordinate (P1) at (-6,6);
            \coordinate (P2) at (-6,-6);
            \coordinate (P3) at (6,-6);
            \coordinate (P4) at (6,6);
            \coordinate (Q1) at (0,6);
            \coordinate (Q2) at (-6,0);
            \coordinate (Q3) at (0,-6);
            \coordinate (Q4) at (6,0);
            \draw[webline, rounded corners] (Q2) -- (0,0) -- ($(Q3)+(1,0)$);
            \draw[webline, rounded corners] ($(Q1)+(1,0)$) -- (2,0) -- (Q4);
            \draw[webline, rounded corners] ($(Q3)+(-1,0)$) -- (-1,0) -- ($(Q1)+(-1,0)$);
            \draw[fill=pink, thick] (-1,0) circle [radius=20pt];
            \draw[dashed] (P1) -- (P2) -- (P3) -- (P4) -- cycle;
            }\ 
        }
        +q^{-1}
        \mathord{
            \ \tikz[baseline=-.6ex, scale=.1, xscale=-1]{
            \coordinate (P1) at (-6,6);
            \coordinate (P2) at (-6,-6);
            \coordinate (P3) at (6,-6);
            \coordinate (P4) at (6,6);
            \coordinate (Q1) at (0,6);
            \coordinate (Q2) at (-6,0);
            \coordinate (Q3) at (0,-6);
            \coordinate (Q4) at (6,0);
            \draw[webline, rounded corners] (Q2) -- (0,0) -- ($(Q3)+(1,0)$);
            \draw[webline, rounded corners] ($(Q1)+(1,0)$) -- (2,0) -- (Q4);
            \draw[webline, rounded corners] ($(Q3)+(-1,0)$) -- (-1,0) -- ($(Q1)+(-1,0)$);
            \draw[fill=pink, thick] (-1,0) circle [radius=20pt];
            \draw[dashed] (P1) -- (P2) -- (P3) -- (P4) -- cycle;
            }\ 
        }
        +\frac{q}{[2]}
        \mathord{
            \ \tikz[baseline=-.6ex, scale=.1]{
            \coordinate (P1) at (-6,6);
            \coordinate (P2) at (-6,-6);
            \coordinate (P3) at (6,-6);
            \coordinate (P4) at (6,6);
            \coordinate (Q1) at (0,6);
            \coordinate (Q2) at (-6,0);
            \coordinate (Q3) at (0,-6);
            \coordinate (Q4) at (6,0);
            \draw[webline, rounded corners] (Q2) -- (0,0) -- ($(Q1)-(1,0)$);
            \draw[webline, rounded corners] ($(Q1)+(1,0)$) -- (2,0) -- (Q4);
            \draw[webline, rounded corners] ($(Q3)-(1,0)$) -- (-3,-2) -- (-1,-2) -- ($(Q3)+(1,0)$);
            \draw[dashed] (P1) -- (P2) -- (P3) -- (P4) -- cycle;
            }\ 
        }
        +\frac{q^{-1}}{[2]}
        \mathord{
            \ \tikz[baseline=-.6ex, scale=.1, xscale=-1]{
            \coordinate (P1) at (-6,6);
            \coordinate (P2) at (-6,-6);
            \coordinate (P3) at (6,-6);
            \coordinate (P4) at (6,6);
            \coordinate (Q1) at (0,6);
            \coordinate (Q2) at (-6,0);
            \coordinate (Q3) at (0,-6);
            \coordinate (Q4) at (6,0);
            \draw[webline, rounded corners] (Q2) -- (0,0) -- ($(Q1)-(1,0)$);
            \draw[webline, rounded corners] ($(Q1)+(1,0)$) -- (2,0) -- (Q4);
            \draw[webline, rounded corners] ($(Q3)-(1,0)$) -- (-3,-2) -- (-1,-2) -- ($(Q3)+(1,0)$);
            \draw[dashed] (P1) -- (P2) -- (P3) -- (P4) -- cycle;
            }\ 
        }\\
        &=q
        \mathord{
            \ \tikz[baseline=-.6ex, scale=.1]{
            \coordinate (P1) at (-6,6);
            \coordinate (P2) at (-6,-6);
            \coordinate (P3) at (6,-6);
            \coordinate (P4) at (6,6);
            \coordinate (Q1) at (0,6);
            \coordinate (Q2) at (-6,0);
            \coordinate (Q3) at (0,-6);
            \coordinate (Q4) at (6,0);
            \draw[webline, rounded corners] (Q2) -- (0,0) -- ($(Q3)+(1,0)$);
            \draw[webline, rounded corners] ($(Q1)+(1,0)$) -- (2,0) -- (Q4);
            \draw[webline, rounded corners] ($(Q3)+(-1,0)$) -- (-1,0) -- ($(Q1)+(-1,0)$);
            \draw[fill=pink, thick] (-1,0) circle [radius=20pt];
            \draw[dashed] (P1) -- (P2) -- (P3) -- (P4) -- cycle;
            }\ 
        }
        +q^{-1}
        \mathord{
            \ \tikz[baseline=-.6ex, scale=.1, xscale=-1]{
            \coordinate (P1) at (-6,6);
            \coordinate (P2) at (-6,-6);
            \coordinate (P3) at (6,-6);
            \coordinate (P4) at (6,6);
            \coordinate (Q1) at (0,6);
            \coordinate (Q2) at (-6,0);
            \coordinate (Q3) at (0,-6);
            \coordinate (Q4) at (6,0);
            \draw[webline, rounded corners] (Q2) -- (0,0) -- ($(Q3)+(1,0)$);
            \draw[webline, rounded corners] ($(Q1)+(1,0)$) -- (2,0) -- (Q4);
            \draw[webline, rounded corners] ($(Q3)+(-1,0)$) -- (-1,0) -- ($(Q1)+(-1,0)$);
            \draw[fill=pink, thick] (-1,0) circle [radius=20pt];
            \draw[dashed] (P1) -- (P2) -- (P3) -- (P4) -- cycle;
            }\ 
        }
        +
        \mathord{
            \ \tikz[baseline=-.6ex, scale=.1]{
            \coordinate (P1) at (-6,6);
            \coordinate (P2) at (-6,-6);
            \coordinate (P3) at (6,-6);
            \coordinate (P4) at (6,6);
            \coordinate (Q1) at (0,6);
            \coordinate (Q2) at (-6,0);
            \coordinate (Q3) at (0,-6);
            \coordinate (Q4) at (6,0);
            \draw[webline] (Q2) -- (Q4);
            \draw[dashed] (P1) -- (P2) -- (P3) -- (P4) -- cycle;
            }\ 
        }.
    \end{align*}
    Here we only use the Reidemeister move \eqref{eq:R4-preserve} in the second equation. The resulting crossroad webs have weighted intersection numbers less than $c(G)$.
    Let us consider an internal crossing (iii) of two tangled loops of type~$2$.
    We apply the skein relation \eqref{rel:ww-cross} to the internal crossing of $G$:
    \begin{align*}
        \mathord{
            \ \tikz[baseline=-.6ex, scale=.1]{
            \coordinate (P1) at (-6,6);
            \coordinate (P2) at (-6,-6);
            \coordinate (P3) at (6,-6);
            \coordinate (P4) at (6,6);
            \coordinate (Q1) at (0,6);
            \coordinate (Q2) at (-6,0);
            \coordinate (Q3) at (0,-6);
            \coordinate (Q4) at (6,0);
            \draw[wline] (Q2) -- (Q4);
            \draw[overwline] (Q1) -- (Q3);
            \draw[dashed] (P1) -- (P2) -- (P3) -- (P4) -- cycle;
            }\ 
        }
        &=q^2
        \mathord{
            \ \tikz[baseline=-.6ex, scale=.1]{
            \coordinate (P1) at (-6,6);
            \coordinate (P2) at (-6,-6);
            \coordinate (P3) at (6,-6);
            \coordinate (P4) at (6,6);
            \coordinate (Q1) at (0,6);
            \coordinate (Q2) at (-6,0);
            \coordinate (Q3) at (0,-6);
            \coordinate (Q4) at (6,0);
            \draw[wline] (Q2) to[out=east, in=north] (Q3);
            \draw[wline] (Q4) to[out=west, in=south] (Q1);
            \draw[dashed] (P1) -- (P2) -- (P3) -- (P4) -- cycle;
            }\ 
        }
        +q^{-2}
        \mathord{
            \ \tikz[baseline=-.6ex, scale=.1, xscale=-1]{
            \coordinate (P1) at (-6,6);
            \coordinate (P2) at (-6,-6);
            \coordinate (P3) at (6,-6);
            \coordinate (P4) at (6,6);
            \coordinate (Q1) at (0,6);
            \coordinate (Q2) at (-6,0);
            \coordinate (Q3) at (0,-6);
            \coordinate (Q4) at (6,0);
            \draw[wline] (Q2) to[out=east, in=north] (Q3);
            \draw[wline] (Q4) to[out=west, in=south] (Q1);
            \draw[dashed] (P1) -- (P2) -- (P3) -- (P4) -- cycle;
            }\ 
        }
        +
        \mathord{
            \ \tikz[baseline=-.6ex, scale=.1, xscale=-1]{
            \coordinate (P1) at (-6,6);
            \coordinate (P2) at (-6,-6);
            \coordinate (P3) at (6,-6);
            \coordinate (P4) at (6,6);
            \coordinate (Q1) at (0,6);
            \coordinate (Q2) at (-6,0);
            \coordinate (Q3) at (0,-6);
            \coordinate (Q4) at (6,0);
            \draw[wline] (Q1) -- (0,3);
            \draw[wline] (Q2) -- (-3,0);
            \draw[wline] (Q3) -- (0,-3);
            \draw[wline] (Q4) -- (3,0);
            \draw[webline] (0,3) -- (-3,0) -- (0,-3) -- (3,0) -- cycle;
            \draw[dashed] (P1) -- (P2) -- (P3) -- (P4) -- cycle;
            }\ 
        }.
    \end{align*}
    The first and second $\mathfrak{sp}_4$-graphs are crossroad webs with the weighted crossing number less than $c(G)$.
    We are going to deform the third $\mathfrak{sp}_4$-graph into a $\bZ_q$-linear combination of crossroad webs.
    If the internal crossing is formed by different loop components, then:
    \begin{align*}
        \mathord{
            \ \tikz[baseline=-.6ex, scale=.1]{
            \coordinate (P1) at (-6,6);
            \coordinate (P2) at (-6,-6);
            \coordinate (P3) at (6,-6);
            \coordinate (P4) at (6,6);
            \coordinate (Q1) at (0,6);
            \coordinate (Q2) at (-6,0);
            \coordinate (Q3) at (0,-6);
            \coordinate (Q4) at (6,0);
            \draw[wline] (Q1) -- (0,3);
            \draw[wline] (Q2) -- (-3,0);
            \draw[wline] (Q3) -- (0,-3);
            \draw[wline] (Q4) -- (3,0);
            \draw[webline] (0,3) -- (-3,0) -- (0,-3) -- (3,0) -- cycle;
            \draw[dashed] (P1) -- (P2) -- (P3) -- (P4) -- cycle;
            \draw[webline, dotted] (Q1) -- ($(Q1)+(0,2)$) -- ($(P4)+(2,2)$) -- ($(P3)+(2,-2)$) -- ($(Q3)+(0,-2)$) -- (Q3);
            \draw[webline, dotted] (Q2) -- ($(Q2)+(-4,0)$) -- ($(P2)+(-4,-4)$) -- ($(P3)+(4,-4)$) -- ($(Q4)+(4,0)$) -- (Q4);
            }\ 
        }
        &=
        \mathord{
            \ \tikz[baseline=-.6ex, scale=.1]{
            \coordinate (P1) at (-6,6);
            \coordinate (P2) at (-6,-6);
            \coordinate (P3) at (6,-6);
            \coordinate (P4) at (6,6);
            \coordinate (Q1) at (0,6);
            \coordinate (Q2) at (-6,0);
            \coordinate (Q3) at (0,-6);
            \coordinate (Q4) at (6,0);
            \draw[wline] (Q2) -- (-2,0);
            \draw[wline] (0,-4) -- (0,-2);
            \draw[wline] (Q4) -- (2,0);
            \draw[webline] ($(Q1)+(-2,0)$) -- (-2,0) -- (0,-2) -- (2,0) -- ($(Q1)+(2,0)$);
            \draw[webline] ($(Q3)+(-2,0)$) -- (0,-4) -- ($(Q3)+(2,0)$);
            \draw[dashed] (P1) -- (P2) -- (P3) -- (P4) -- cycle;
            }\ 
        }
        =
        \mathord{
            \ \tikz[baseline=-.6ex, scale=.1]{
            \coordinate (P1) at (-6,6);
            \coordinate (P2) at (-6,-6);
            \coordinate (P3) at (6,-6);
            \coordinate (P4) at (6,6);
            \coordinate (Q1) at (0,6);
            \coordinate (Q2) at (-6,0);
            \coordinate (Q3) at (0,-6);
            \coordinate (Q4) at (6,0);
            \draw[wline] (Q2) -- (-2,0);
            \draw[wline] (Q4) -- (2,0);
            \draw[webline] (-2,0) -- ($(Q1)+(-2,0)$);
            \draw[webline] (2,0) -- ($(Q1)+(2,0)$);
            \draw[webline] ($(Q3)+(-2,0)$) -- (0,-3) -- (2,0);
            \draw[webline] ($(Q3)+(2,0)$) -- (0,-3) -- (-2,0);
            \draw[dashed] (P1) -- (P2) -- (P3) -- (P4) -- cycle;
            \draw[fill=pink, thick] (0,-3) circle [radius=20pt];
            }\ 
        }
        +\frac{1}{[2]}
        \mathord{
            \ \tikz[baseline=-.6ex, scale=.1]{
            \coordinate (P1) at (-6,6);
            \coordinate (P2) at (-6,-6);
            \coordinate (P3) at (6,-6);
            \coordinate (P4) at (6,6);
            \coordinate (Q1) at (0,6);
            \coordinate (Q2) at (-6,0);
            \coordinate (Q3) at (0,-6);
            \coordinate (Q4) at (6,0);
            \draw[wline] (Q2) -- (-2,0);
            \draw[wline] (Q4) -- (2,0);
            \draw[webline] ($(Q1)+(-2,0)$) -- (-2,0) to[out=south, in=south] (2,0) -- ($(Q1)+(2,0)$);
            \draw[webline, rounded corners] ($(Q3)+(-2,0)$) -- (0,-4) -- ($(Q3)+(2,0)$);
            \draw[dashed] (P1) -- (P2) -- (P3) -- (P4) -- cycle;
            }\ 
        }\\
        &=
        \mathord{
            \ \tikz[baseline=-.6ex, scale=.1]{
            \coordinate (P1) at (-6,6);
            \coordinate (P2) at (-6,-6);
            \coordinate (P3) at (6,-6);
            \coordinate (P4) at (6,6);
            \coordinate (Q1) at (0,6);
            \coordinate (Q2) at (-6,0);
            \coordinate (Q3) at (0,-6);
            \coordinate (Q4) at (6,0);
            \draw[wline] (Q2) -- (-2,0);
            \draw[wline] (Q4) -- (2,0);
            \draw[webline] (-2,0) -- ($(Q1)+(-2,0)$);
            \draw[webline] (2,0) -- ($(Q1)+(2,0)$);
            \draw[webline] ($(Q3)+(-2,0)$) -- (0,-3) -- (2,0);
            \draw[webline] ($(Q3)+(2,0)$) -- (0,-3) -- (-2,0);
            \draw[dashed] (P1) -- (P2) -- (P3) -- (P4) -- cycle;
            \draw[fill=pink, thick] (0,-3) circle [radius=20pt];
            }\ 
        }
        -
        \mathord{
            \ \tikz[baseline=-.6ex, scale=.1]{
            \coordinate (P1) at (-6,6);
            \coordinate (P2) at (-6,-6);
            \coordinate (P3) at (6,-6);
            \coordinate (P4) at (6,6);
            \coordinate (Q1) at (0,6);
            \coordinate (Q2) at (-6,0);
            \coordinate (Q3) at (0,-6);
            \coordinate (Q4) at (6,0);
            \draw[wline] (Q2) -- (Q4);
            \draw[dashed] (P1) -- (P2) -- (P3) -- (P4) -- cycle;
            }\ 
        }
        =
        \mathord{
            \ \tikz[baseline=-.6ex, scale=.1]{
            \coordinate (P1) at (-6,6);
            \coordinate (P2) at (-6,-6);
            \coordinate (P3) at (6,-6);
            \coordinate (P4) at (6,6);
            \coordinate (Q1) at (0,6);
            \coordinate (Q2) at (-6,0);
            \coordinate (Q3) at (0,-6);
            \coordinate (Q4) at (6,0);
            \draw[wline] (-4,0) -- (-2,0);
            \draw[webline] (-2,0) -- ($(Q1)+(-2,0)$);
            \draw[webline] ($(Q4)+(0,2)$) to[out=west, in=south] ($(Q1)+(2,0)$);
            \draw[webline] ($(Q3)+(-2,0)$) -- (0,-3) to[out=north east, in=west] ($(Q4)+(0,-2)$);
            \draw[webline] ($(Q3)+(2,0)$) -- (0,-3) -- (-2,0);
            \draw[webline] ($(Q2)+(0,-2)$) -- (-4,0) -- ($(Q2)+(0,2)$);
            \draw[dashed] (P1) -- (P2) -- (P3) -- (P4) -- cycle;
            \draw[fill=pink, thick] (0,-3) circle [radius=20pt];
            }\ 
        }
        -
        \mathord{
            \ \tikz[baseline=-.6ex, scale=.1, xscale=-1]{
            \coordinate (P1) at (-6,6);
            \coordinate (P2) at (-6,-6);
            \coordinate (P3) at (6,-6);
            \coordinate (P4) at (6,6);
            \coordinate (Q1) at (0,6);
            \coordinate (Q2) at (-6,0);
            \coordinate (Q3) at (0,-6);
            \coordinate (Q4) at (6,0);
            \draw[wline] (Q2) -- (Q4);
            \draw[dashed] (P1) -- (P2) -- (P3) -- (P4) -- cycle;
            }\ 
        }\\
        &=
        \mathord{
            \ \tikz[baseline=-.6ex, scale=.1]{
            \coordinate (P1) at (-6,6);
            \coordinate (P2) at (-6,-6);
            \coordinate (P3) at (6,-6);
            \coordinate (P4) at (6,6);
            \coordinate (Q1) at (0,6);
            \coordinate (Q2) at (-6,0);
            \coordinate (Q3) at (0,-6);
            \coordinate (Q4) at (6,0);
            \draw[webline] ($(Q4)+(0,2)$) to[out=west, in=south] ($(Q1)+(2,0)$);
            \draw[webline] ($(Q3)+(-2,0)$) -- (0,-3) to[out=north east, in=west] ($(Q4)+(0,-2)$);
            \draw[webline] ($(Q3)+(2,0)$) -- (0,-3) -- (-3,0) -- ($(Q2)+(0,2)$);
            \draw[webline] ($(Q2)+(0,-2)$) -- (-3,0) to[out=north east, in=south] ($(Q1)+(-2,0)$);
            \draw[dashed] (P1) -- (P2) -- (P3) -- (P4) -- cycle;
            \draw[fill=pink, thick] (0,-3) circle [radius=20pt];
            \draw[fill=pink, thick] (-3,0) circle [radius=20pt];
            }\ 
        }
        +\frac{1}{[2]}
        \mathord{
            \ \tikz[baseline=-.6ex, scale=.1]{
            \coordinate (P1) at (-6,6);
            \coordinate (P2) at (-6,-6);
            \coordinate (P3) at (6,-6);
            \coordinate (P4) at (6,6);
            \coordinate (Q1) at (0,6);
            \coordinate (Q2) at (-6,0);
            \coordinate (Q3) at (0,-6);
            \coordinate (Q4) at (6,0);
            \draw[webline] (-2,0) -- ($(Q1)+(-2,0)$);
            \draw[webline] ($(Q4)+(0,2)$) to[out=west, in=south] ($(Q1)+(2,0)$);
            \draw[webline] ($(Q3)+(-2,0)$) -- (0,-3) to[out=north east, in=west] ($(Q4)+(0,-2)$);
            \draw[webline] ($(Q3)+(2,0)$) -- (0,-3) -- (-2,0);
            \draw[webline, rounded corners] ($(Q2)+(0,-2)$) -- (-4,0) -- ($(Q2)+(0,2)$);
            \draw[dashed] (P1) -- (P2) -- (P3) -- (P4) -- cycle;
            \draw[fill=pink, thick] (0,-3) circle [radius=20pt];
            }\ 
        }
        -
        \mathord{
            \ \tikz[baseline=-.6ex, scale=.1, xscale=-1]{
            \coordinate (P1) at (-6,6);
            \coordinate (P2) at (-6,-6);
            \coordinate (P3) at (6,-6);
            \coordinate (P4) at (6,6);
            \coordinate (Q1) at (0,6);
            \coordinate (Q2) at (-6,0);
            \coordinate (Q3) at (0,-6);
            \coordinate (Q4) at (6,0);
            \draw[wline] (Q2) -- (Q4);
            \draw[dashed] (P1) -- (P2) -- (P3) -- (P4) -- cycle;
            }\ 
        }\\
        &=
        \mathord{
            \ \tikz[baseline=-.6ex, scale=.1]{
            \coordinate (P1) at (-6,6);
            \coordinate (P2) at (-6,-6);
            \coordinate (P3) at (6,-6);
            \coordinate (P4) at (6,6);
            \coordinate (Q1) at (0,6);
            \coordinate (Q2) at (-6,0);
            \coordinate (Q3) at (0,-6);
            \coordinate (Q4) at (6,0);
            \draw[webline] ($(Q4)+(0,2)$) to[out=west, in=south] ($(Q1)+(2,0)$);
            \draw[webline] ($(Q3)+(-2,0)$) -- (0,-3) to[out=north east, in=west] ($(Q4)+(0,-2)$);
            \draw[webline] ($(Q3)+(2,0)$) -- (0,-3) -- (-3,0) -- ($(Q2)+(0,2)$);
            \draw[webline] ($(Q2)+(0,-2)$) -- (-3,0) to[out=north east, in=south] ($(Q1)+(-2,0)$);
            \draw[dashed] (P1) -- (P2) -- (P3) -- (P4) -- cycle;
            \draw[fill=pink, thick] (0,-3) circle [radius=20pt];
            \draw[fill=pink, thick] (-3,0) circle [radius=20pt];
            }\ 
        }
        -
        \mathord{
            \ \tikz[baseline=-.6ex, scale=.1]{
            \coordinate (P1) at (-6,6);
            \coordinate (P2) at (-6,-6);
            \coordinate (P3) at (6,-6);
            \coordinate (P4) at (6,6);
            \coordinate (Q1) at (0,6);
            \coordinate (Q2) at (-6,0);
            \coordinate (Q3) at (0,-6);
            \coordinate (Q4) at (6,0);
            \draw[wline] (Q1) -- (Q3);
            \draw[dashed] (P1) -- (P2) -- (P3) -- (P4) -- cycle;
            }\ 
        }
        -
        \mathord{
            \ \tikz[baseline=-.6ex, scale=.1, xscale=-1]{
            \coordinate (P1) at (-6,6);
            \coordinate (P2) at (-6,-6);
            \coordinate (P3) at (6,-6);
            \coordinate (P4) at (6,6);
            \coordinate (Q1) at (0,6);
            \coordinate (Q2) at (-6,0);
            \coordinate (Q3) at (0,-6);
            \coordinate (Q4) at (6,0);
            \draw[wline] (Q2) -- (Q4);
            \draw[dashed] (P1) -- (P2) -- (P3) -- (P4) -- cycle;
            }\ 
        }.
    \end{align*}
    If it is a self-crossing, then:
    \begin{align*}
        \mathord{
            \ \tikz[baseline=-.6ex, scale=.1]{
            \coordinate (P1) at (-6,6);
            \coordinate (P2) at (-6,-6);
            \coordinate (P3) at (6,-6);
            \coordinate (P4) at (6,6);
            \coordinate (Q1) at (0,6);
            \coordinate (Q2) at (-6,0);
            \coordinate (Q3) at (0,-6);
            \coordinate (Q4) at (6,0);
            \draw[wline] (Q1) -- (0,3);
            \draw[wline] (Q2) -- (-3,0);
            \draw[wline] (Q3) -- (0,-3);
            \draw[wline] (Q4) -- (3,0);
            \draw[webline] (0,3) -- (-3,0) -- (0,-3) -- (3,0) -- cycle;
            \draw[dashed] (P1) -- (P2) -- (P3) -- (P4) -- cycle;
            \draw[webline, dotted] (Q1) -- ($(Q1)+(0,2)$) -- ($(P4)+(2,2)$) -- ($(Q4)+(2,0)$) -- (Q4);
            \draw[webline, dotted] (Q2) -- ($(Q2)+(-2,0)$) -- ($(P2)+(-2,-2)$) -- ($(Q3)+(0,-2)$) -- (Q3);
            }\ 
        }
        &=
        \mathord{
            \ \tikz[baseline=-.6ex, scale=.1]{
            \coordinate (P1) at (-6,6);
            \coordinate (P2) at (-6,-6);
            \coordinate (P3) at (6,-6);
            \coordinate (P4) at (6,6);
            \coordinate (Q1) at (0,6);
            \coordinate (Q2) at (-6,0);
            \coordinate (Q3) at (0,-6);
            \coordinate (Q4) at (6,0);
            \draw[wline] (Q2) -- (-2,0);
            \draw[wline] (Q3) -- (0,-2);
            \draw[wline] (4,0) -- (2,0);
            \draw[webline] ($(Q1)+(-2,0)$) -- (-2,0) -- (0,-2) -- (2,0) -- ($(Q1)+(2,0)$);
            \draw[webline] ($(Q4)+(0,-2)$) -- (4,0) -- ($(Q4)+(0,2)$);
            \draw[dashed] (P1) -- (P2) -- (P3) -- (P4) -- cycle;
            }\ 
        }
        =
        \mathord{
            \ \tikz[baseline=-.6ex, scale=.1]{
            \coordinate (P1) at (-6,6);
            \coordinate (P2) at (-6,-6);
            \coordinate (P3) at (6,-6);
            \coordinate (P4) at (6,6);
            \coordinate (Q1) at (0,6);
            \coordinate (Q2) at (-6,0);
            \coordinate (Q3) at (0,-6);
            \coordinate (Q4) at (6,0);
            \draw[wline] (Q2) -- (-2,0);
            \draw[wline] (Q3) -- (0,-2);
            \draw[webline] (-2,0) -- ($(Q1)+(-2,0)$);
            \draw[webline] ($(Q4)+(0,-2)$) -- (3,0) to[out=north west, in=south] ($(Q1)+(2,0)$);
            \draw[webline] ($(Q4)+(0,2)$) -- (3,0) -- (0,-2);
            \draw[webline] (-2,0) -- (0,-2);
            \draw[dashed] (P1) -- (P2) -- (P3) -- (P4) -- cycle;
            \draw[fill=pink, thick] (3,0) circle [radius=20pt];
            }\ 
        }
        +\frac{1}{[2]}
        \mathord{
            \ \tikz[baseline=-.6ex, scale=.1]{
            \coordinate (P1) at (-6,6);
            \coordinate (P2) at (-6,-6);
            \coordinate (P3) at (6,-6);
            \coordinate (P4) at (6,6);
            \coordinate (Q1) at (0,6);
            \coordinate (Q2) at (-6,0);
            \coordinate (Q3) at (0,-6);
            \coordinate (Q4) at (6,0);
            \draw[wline] (Q2) -- (-2,0);
            \draw[wline] (Q3) -- (0,-2);
            \draw[webline] ($(Q1)+(-2,0)$) -- (-2,0) -- (0,-2) -- (2,0) -- ($(Q1)+(2,0)$);
            \draw[webline] ($(Q4)+(0,-2)$) to[out=west, in=west] ($(Q4)+(0,2)$);
            \draw[dashed] (P1) -- (P2) -- (P3) -- (P4) -- cycle;
            }\ 
        }\\
        &=
        \mathord{
            \ \tikz[baseline=-.6ex, scale=.1]{
            \coordinate (P1) at (-6,6);
            \coordinate (P2) at (-6,-6);
            \coordinate (P3) at (6,-6);
            \coordinate (P4) at (6,6);
            \coordinate (Q1) at (0,6);
            \coordinate (Q2) at (-6,0);
            \coordinate (Q3) at (0,-6);
            \coordinate (Q4) at (6,0);
            \draw[wline] (Q2) -- (-2,0);
            \draw[wline] (Q3) -- (0,-2);
            \draw[webline] (-2,0) -- ($(Q1)+(-2,0)$);
            \draw[webline] ($(Q4)+(0,-2)$) -- (3,0) to[out=north west, in=south] ($(Q1)+(2,0)$);
            \draw[webline] ($(Q4)+(0,2)$) -- (3,0) -- (0,-2);
            \draw[webline] (-2,0) -- (0,-2);
            \draw[dashed] (P1) -- (P2) -- (P3) -- (P4) -- cycle;
            \draw[fill=pink, thick] (3,0) circle [radius=20pt];
            }\ 
        }
        -
        \mathord{
            \ \tikz[baseline=-.6ex, scale=.1]{
            \coordinate (P1) at (-6,6);
            \coordinate (P2) at (-6,-6);
            \coordinate (P3) at (6,-6);
            \coordinate (P4) at (6,6);
            \coordinate (Q1) at (0,6);
            \coordinate (Q2) at (-6,0);
            \coordinate (Q3) at (0,-6);
            \coordinate (Q4) at (6,0);
            \draw[wline] (Q2) to[out=east, in=north] (Q3);
            \draw[dashed] (P1) -- (P2) -- (P3) -- (P4) -- cycle;
            }\ 
        }\\
        &=
        \mathord{
            \ \tikz[baseline=-.6ex, scale=.1]{
            \coordinate (P1) at (-6,6);
            \coordinate (P2) at (-6,-6);
            \coordinate (P3) at (6,-6);
            \coordinate (P4) at (6,6);
            \coordinate (Q1) at (0,6);
            \coordinate (Q2) at (-6,0);
            \coordinate (Q3) at (0,-6);
            \coordinate (Q4) at (6,0);
            \draw[webline] ($(Q4)+(0,-2)$) -- (3,0) to[out=north west, in=south] ($(Q1)+(2,0)$);
            \draw[webline] ($(Q4)+(0,2)$) -- (3,0) to[out=south west, in=north] ($(Q3)+(2,0)$);
            \draw[webline] ($(Q2)+(0,-2)$) -- (-3,0) to[out=north east, in=south] ($(Q1)+(-2,0)$);
            \draw[webline] ($(Q2)+(0,2)$) -- (-3,0) to[out=south east, in=north] ($(Q3)+(-2,0)$);
            \draw[dashed] (P1) -- (P2) -- (P3) -- (P4) -- cycle;
            \draw[fill=pink, thick] (3,0) circle [radius=20pt];
            \draw[fill=pink, thick] (-3,0) circle [radius=20pt];
            }\ 
        }
        +\frac{1}{[2]}
        \mathord{
            \ \tikz[baseline=-.6ex, scale=.1]{
            \coordinate (P1) at (-6,6);
            \coordinate (P2) at (-6,-6);
            \coordinate (P3) at (6,-6);
            \coordinate (P4) at (6,6);
            \coordinate (Q1) at (0,6);
            \coordinate (Q2) at (-6,0);
            \coordinate (Q3) at (0,-6);
            \coordinate (Q4) at (6,0);
            \draw[webline] ($(Q4)+(0,-2)$) -- (3,0) to[out=north west, in=south] ($(Q1)+(2,0)$);
            \draw[webline] ($(Q4)+(0,2)$) -- (3,0) -- (0,-2) to[out=west, in=south] ($(Q1)+(-2,0)$);
            \draw[dashed] (P1) -- (P2) -- (P3) -- (P4) -- cycle;
            \draw[fill=pink, thick] (3,0) circle [radius=20pt];
            }\ 
        }
        -
        \mathord{
            \ \tikz[baseline=-.6ex, scale=.1]{
            \coordinate (P1) at (-6,6);
            \coordinate (P2) at (-6,-6);
            \coordinate (P3) at (6,-6);
            \coordinate (P4) at (6,6);
            \coordinate (Q1) at (0,6);
            \coordinate (Q2) at (-6,0);
            \coordinate (Q3) at (0,-6);
            \coordinate (Q4) at (6,0);
            \draw[wline] (Q2) to[out=east, in=north] (Q3);
            \draw[dashed] (P1) -- (P2) -- (P3) -- (P4) -- cycle;
            }\ 
        }\\
        &=
        \mathord{
            \ \tikz[baseline=-.6ex, scale=.1]{
            \coordinate (P1) at (-6,6);
            \coordinate (P2) at (-6,-6);
            \coordinate (P3) at (6,-6);
            \coordinate (P4) at (6,6);
            \coordinate (Q1) at (0,6);
            \coordinate (Q2) at (-6,0);
            \coordinate (Q3) at (0,-6);
            \coordinate (Q4) at (6,0);
            \draw[webline] ($(Q4)+(0,-2)$) -- (3,0) to[out=north west, in=south] ($(Q1)+(2,0)$);
            \draw[webline] ($(Q4)+(0,2)$) -- (3,0) to[out=south west, in=north] ($(Q3)+(2,0)$);
            \draw[webline] ($(Q2)+(0,-2)$) -- (-3,0) to[out=north east, in=south] ($(Q1)+(-2,0)$);
            \draw[webline] ($(Q2)+(0,2)$) -- (-3,0) to[out=south east, in=north] ($(Q3)+(-2,0)$);
            \draw[dashed] (P1) -- (P2) -- (P3) -- (P4) -- cycle;
            \draw[fill=pink, thick] (3,0) circle [radius=20pt];
            \draw[fill=pink, thick] (-3,0) circle [radius=20pt];
            }\ 
        }
        -
        \mathord{
            \ \tikz[baseline=-.6ex, scale=.1]{
            \coordinate (P1) at (-6,6);
            \coordinate (P2) at (-6,-6);
            \coordinate (P3) at (6,-6);
            \coordinate (P4) at (6,6);
            \coordinate (Q1) at (0,6);
            \coordinate (Q2) at (-6,0);
            \coordinate (Q3) at (0,-6);
            \coordinate (Q4) at (6,0);
            \draw[wline] (Q4) to[out=west, in=south] (Q1);
            \draw[dashed] (P1) -- (P2) -- (P3) -- (P4) -- cycle;
            }\ 
        }
        -
        \mathord{
            \ \tikz[baseline=-.6ex, scale=.1]{
            \coordinate (P1) at (-6,6);
            \coordinate (P2) at (-6,-6);
            \coordinate (P3) at (6,-6);
            \coordinate (P4) at (6,6);
            \coordinate (Q1) at (0,6);
            \coordinate (Q2) at (-6,0);
            \coordinate (Q3) at (0,-6);
            \coordinate (Q4) at (6,0);
            \draw[wline] (Q2) to[out=east, in=north] (Q3);
            \draw[dashed] (P1) -- (P2) -- (P3) -- (P4) -- cycle;
            }\ 
        }.
    \end{align*}
    In both cases, the resulting crossroad webs have the weighted crossing number less than $c(G)$, and their coefficients lie in $\bZ_q$.
    One can prove the same assertion for the cases where one or two of them are arcs rather than loops. 
    Then we conclude that $G$ belongs to $\Skein{\Sigma}^{\bZ_q}$ by induction on $c(G)$.
\end{proof}

For a subset $S \subset \Skein{\Sigma}$, let us denote by $\langle S \rangle_{\bZ_q}[\partial^{-1}]$ a $\bZ_q$-subalgebra generated by $S$ and boundary webs $\partial_{\Sigma}\cup\partial_{\Sigma}^{-1}$ in $\Skein{\Sigma}^{\bZ_q}[\partial^{-1}]$.

\begin{thm}\label{thm:generator-Zv-form}
    $\Skein{\Sigma}^{\bZ_q}[\partial^{-1}]=\langle \Desc{\Sigma} \rangle_{\bZ_q}[\partial^{-1}]=\langle \SimpWil{\Sigma} \rangle_{\bZ_q}[\partial^{-1}]$.
\end{thm}
\begin{proof}
    It is obvious that $\langle \SimpWil{\Sigma} \rangle_{\bZ_q}[\partial^{-1}]\subset\Skein{\Sigma}^{\bZ_q}[\partial^{-1}]=\bZ_q\Bweb{\Sigma}$ because $\SimpWil{\Sigma}\subset\Bweb{\Sigma}$.
    The inclusion $\Skein{\Sigma}^{\bZ_q}[\partial^{-1}]\subset\langle \Desc{\Sigma} \rangle_{\bZ_q}[\partial^{-1}]$ follows from \cref{prop:desc-gen} together with the argument on simple loops and arcs of type $2$ discussed in the proof of \cref{thm:localized-gen}.
    Moreover, 
    \cref{thm:localized-gen} tells us that $\langle\Desc{\Sigma} \rangle_{\bZ_q}[\partial^{-1}]\subset\langle \SimpWil{\Sigma} \rangle_{\bZ_q}[\partial^{-1}]$.
    Consequently, we obtain a sequence of inclusions
    \[
        \Skein{\Sigma}^{\bZ_q}[\partial^{-1}]\subset\langle \Desc{\Sigma} \rangle_{\bZ_q}[\partial^{-1}]\subset\langle \SimpWil{\Sigma} \rangle_{\bZ_q}[\partial^{-1}]\subset\Skein{\Sigma}^{\bZ_q}[\partial^{-1}],
    \]
    which concludes the assertion.
\end{proof}

\subsection{Laurent expressions and positivity}
The notion of ``elevation-preserving'' webs has been introduced in \cite{IYsl3} for tangled $\mathfrak{sl}_3$-webs.
An elevation-preserving web is shown to have a Laurent expression with positive coefficients, based on the cutting trick in \cref{lem:cutteing-trick}.
Let us discuss the $\mathfrak{sp}_4$-case here.

We first explain the expansion of a tangled $\mathfrak{sp}_4$-graph based on the cutting trick in \cref{lem:cutteing-trick}.
For an ideal arc $E$ connecting two special points $p$ and $q$, denote the type~$s$ simple $\mathfrak{sp}_4$-arc parallel to $E$ by $e_s$ for $s=1,2$.

Let us consider a tangled $\mathfrak{sp}_4$-graph diagram $G$ such that $E\setminus\{p,q\}$ has $n_{s}$ transverse intersection points with type~$s$ edges of $G$ for $s=1,2$.
Then, one can easily see that $G(e_1e_2)^{n_1}(e_1e_2^2)^{n_2}$ is expanded into an $\bZ_q$-linear combination of tangled graph diagrams with no intersection  with $E\setminus\{p,q\}$.
Applying this procedure to each edge of an ideal triangulation $\Delta$ of $\Sigma$, we can expand a tangled graph $G$ into a $\bZ_q$-linear combination of $\mathfrak{sp}_4$-webs contained in triangles of $\Delta$.
Recall that any such $\mathfrak{sp}_4$-web contained in triangles is expressed as a polynomial in $\cup_{T\in t(\Delta)}\Eweb{T}$ with coefficients in $\cR$ by \cref{thm:Eweb-T}.
The coefficients can be restricted to $\bZ_q$ if $G$ is a crossroad web because of \cref{lem:crossroads-in-Zv}.
Moreover, it turns out that one can write $G$ by only using the elementary webs associated with a fixed decorated triangulation:

\begin{thm}[web cluster expansions]\label{thm:Cweb-exp}
    Let $\Delta$ be an ideal triangulation of an unpunctured marked surface $\Sigma$.
    Choosing a web cluster $\cC_{T}\in\Cweb{T}$ for each $T\in t(\Delta)$, let $\cC_\Delta:=\cup_{T\in t(\Delta)}\cC_T$ be the corresponding web cluster of $\Skein{\Sigma}$.
    For any tangled $\mathfrak{sp}_4$-graph (resp.~crossroad web) $G$ in $\Skein{\Sigma}$ (resp.~$\Skein{\Sigma}^{\bZ_q}$), there exists a monomial $J$ in $\cC_{\Delta}$ such that $GJ$ is a polynomial in $\cC_{\Delta}$ with coefficients in $\cR$ (resp.~$\bZ_q$).
\end{thm}
\begin{proof}
    From the above discussion, it suffices to see that any elementary web in a triangle $T$ becomes a polynomial in $\cC_{T}$ after multiplying a monomial in $\cC_{T}$.
    Let us consider the web cluster $\cC_{T}$ consisting of
    $
    \mathord{
        \ \tikz[baseline=-.6ex, scale=.07, rotate=120]{
            \coordinate (A) at (90:8);
            \coordinate (B) at (210:8);
            \coordinate (C) at (-30:8);
            \coordinate (P) at (0,0);
            \draw[wline] (P) -- (A);
            \draw[webline] (P) -- (B);
            \draw[webline] (P) -- (C);
            \draw[blue] (A) -- (B) -- (C) -- cycle;
            \filldraw (A) circle [radius=20pt];
            \filldraw (B) circle [radius=20pt];
            \filldraw (C) circle [radius=20pt];
        }\ 
    }$,
    $
    \mathord{
        \ \tikz[baseline=-.6ex, scale=.07, rotate=-120]{
            \coordinate (A) at (90:8);
            \coordinate (B) at (210:8);
            \coordinate (C) at (-30:8);
            \coordinate (P1) at ($(A)!.8!(B)+(3,0)$);
            \coordinate (P2) at ($(A)!.8!(C)+(-3,0)$);
            \coordinate (P) at (0,0);
            \draw[wline] (P1) -- (B);
            \draw[wline] (P2) -- (C);
            \draw[webline] (A) -- (P1) -- (P2) -- (A);
            \draw[blue] (A) -- (B) -- (C) -- cycle;
            \filldraw (A) circle [radius=20pt];
            \filldraw (B) circle [radius=20pt];
            \filldraw (C) circle [radius=20pt];
        }\ 
    }$ and the $\mathfrak{sp}_4$-webs along the edges of $T$.
    The following are verified straightforwardly by using the skein relations:
    \begin{align}
        \mathord{
            \ \tikz[baseline=-.6ex, scale=.1, rotate=120]{
                \coordinate (A) at (90:8);
                \coordinate (B) at (210:8);
                \coordinate (C) at (-30:8);
                \coordinate (P1) at ($(A)!.8!(B)+(3,0)$);
                \coordinate (P2) at ($(A)!.8!(C)+(-3,0)$);
                \coordinate (P) at (0,0);
                \draw[blue] (A) -- (B) -- (C) -- cycle;
                \draw[wline] (P1) -- (B);
                \draw[wline] (P2) -- (C);
                \draw[webline] (A) -- (P1) -- (P2) -- (A);
                \filldraw (A) circle [radius=20pt];
                \filldraw (B) circle [radius=20pt];
                \filldraw (C) circle [radius=20pt];
            }\ 
        }
        \mathord{
            \ \tikz[baseline=-.6ex, scale=.1, rotate=120]{
                \coordinate (A) at (90:8);
                \coordinate (B) at (210:8);
                \coordinate (C) at (-30:8);
                \coordinate (P) at (0,0);
                \draw[wline] (P) -- (A);
                \draw[webline] (P) -- (B);
                \draw[webline] (P) -- (C);
                \draw[blue] (A) -- (B) -- (C) -- cycle;
                \filldraw (A) circle [radius=20pt];
                \filldraw (B) circle [radius=20pt];
                \filldraw (C) circle [radius=20pt];
            }\ 
        }
        &=q
        \mathord{
            \ \tikz[baseline=-.6ex, scale=.1]{
                \coordinate (A) at (90:8);
                \coordinate (B) at (210:8);
                \coordinate (C) at (-30:8);
                \coordinate (P) at (0,0);
                \draw[blue] (A) -- (B) -- (C) -- cycle;
                \draw[wline] (P) -- (A);
                \draw[webline] (P) -- (B);
                \draw[webline] (P) -- (C);
                \draw[webline] (B) to[bend left=15] (A);
                \draw[wline] (B) to[bend right=15] (C);
                \filldraw (A) circle [radius=20pt];
                \filldraw (B) circle [radius=20pt];
                \filldraw (C) circle [radius=20pt];
            }\ 
        }
        +q^{-1}
        \mathord{
            \ \tikz[baseline=-.6ex, scale=.1, rotate=-120]{
                \coordinate (A) at (90:8);
                \coordinate (B) at (210:8);
                \coordinate (C) at (-30:8);
                \coordinate (P) at (0,0);
                \draw[blue] (A) -- (B) -- (C) -- cycle;
                \draw[wline] (P) -- (A);
                \draw[webline] (P) -- (B);
                \draw[webline] (P) -- (C);
                \draw[webline] (C) to[bend right=15] (A);
                \draw[wline] (B) to[bend right=15] (C);
                \filldraw (A) circle [radius=20pt];
                \filldraw (B) circle [radius=20pt];
                \filldraw (C) circle [radius=20pt];
            }\ 
        },\label{eq:skein_rel_triangle}\\
        \mathord{
            \ \tikz[baseline=-.6ex, scale=.1]{
                \coordinate (A) at (90:8);
                \coordinate (B) at (210:8);
                \coordinate (C) at (-30:8);
                \coordinate (P) at (0,0);
                \draw[wline] (P) -- (A);
                \draw[webline] (P) -- (B);
                \draw[webline] (P) -- (C);
                \draw[blue] (A) -- (B) -- (C) -- cycle;
                \filldraw (A) circle [radius=20pt];
                \filldraw (B) circle [radius=20pt];
                \filldraw (C) circle [radius=20pt];
            }\ 
        }
        \mathord{
            \ \tikz[baseline=-.6ex, scale=.1, rotate=120]{
                \coordinate (A) at (90:8);
                \coordinate (B) at (210:8);
                \coordinate (C) at (-30:8);
                \coordinate (P) at (0,0);
                \draw[wline] (P) -- (A);
                \draw[webline] (P) -- (B);
                \draw[webline] (P) -- (C);
                \draw[blue] (A) -- (B) -- (C) -- cycle;
                \filldraw (A) circle [radius=20pt];
                \filldraw (B) circle [radius=20pt];
                \filldraw (C) circle [radius=20pt];
            }\ 
        }
        &=q^{-\frac{1}{2}}
        \mathord{
            \ \tikz[baseline=-.6ex, scale=.1]{
                \coordinate (A) at (90:8);
                \coordinate (B) at (210:8);
                \coordinate (C) at (-30:8);
                \coordinate (P) at (0,0);
                \draw[blue] (A) -- (B) -- (C) -- cycle;
                \draw[wline] (A) to[bend left=15] (B);
                \draw[webline] (B) to[bend left=15] (C);
                \draw[webline] (A) to[bend right=15] (C);
                \filldraw (A) circle [radius=20pt];
                \filldraw (B) circle [radius=20pt];
                \filldraw (C) circle [radius=20pt];
            }\ 
        }
        +q^{\frac{1}{2}}
        \mathord{
            \ \tikz[baseline=-.6ex, scale=.1, rotate=-120]{
                \coordinate (A) at (90:8);
                \coordinate (B) at (210:8);
                \coordinate (C) at (-30:8);
                \coordinate (P1) at ($(A)!.7!(B)+(3,0)$);
                \coordinate (P2) at ($(A)!.7!(C)+(-3,0)$);
                \coordinate (P) at (0,0);
                \draw[blue] (A) -- (B) -- (C) -- cycle;
                \draw[wline] (P1) -- (B);
                \draw[wline] (P2) -- (C);
                \draw[webline] (B) to[bend right=15] (C);
                \draw[webline] (A) -- (P1) -- (P2) -- (A);
                \filldraw (A) circle [radius=20pt];
                \filldraw (B) circle [radius=20pt];
                \filldraw (C) circle [radius=20pt];
            }\ 
        },\\
        \mathord{
            \ \tikz[baseline=-.6ex, scale=.1, rotate=-120]{
                \coordinate (A) at (90:8);
                \coordinate (B) at (210:8);
                \coordinate (C) at (-30:8);
                \coordinate (P) at (0,0);
                \draw[wline] (P) -- (A);
                \draw[webline] (P) -- (B);
                \draw[webline] (P) -- (C);
                \draw[blue] (A) -- (B) -- (C) -- cycle;
                \filldraw (A) circle [radius=20pt];
                \filldraw (B) circle [radius=20pt];
                \filldraw (C) circle [radius=20pt];
            }\ 
        }
        \mathord{
            \ \tikz[baseline=-.6ex, scale=.1, rotate=120]{
                \coordinate (A) at (90:8);
                \coordinate (B) at (210:8);
                \coordinate (C) at (-30:8);
                \coordinate (P) at (0,0);
                \draw[wline] (P) -- (A);
                \draw[webline] (P) -- (B);
                \draw[webline] (P) -- (C);
                \draw[blue] (A) -- (B) -- (C) -- cycle;
                \filldraw (A) circle [radius=20pt];
                \filldraw (B) circle [radius=20pt];
                \filldraw (C) circle [radius=20pt];
            }\ 
        }
        &=q^{\frac{1}{2}}
        \mathord{
            \ \tikz[baseline=-.6ex, scale=.1, rotate=120]{
                \coordinate (A) at (90:8);
                \coordinate (B) at (210:8);
                \coordinate (C) at (-30:8);
                \coordinate (P) at (0,0);
                \draw[blue] (A) -- (B) -- (C) -- cycle;
                \draw[wline] (A) to[bend left=15] (B);
                \draw[webline] (B) to[bend left=15] (C);
                \draw[webline] (A) to[bend right=15] (C);
                \filldraw (A) circle [radius=20pt];
                \filldraw (B) circle [radius=20pt];
                \filldraw (C) circle [radius=20pt];
            }\ 
        }
        +q^{-\frac{1}{2}}
        \mathord{
            \ \tikz[baseline=-.6ex, scale=.1]{
                \coordinate (A) at (90:8);
                \coordinate (B) at (210:8);
                \coordinate (C) at (-30:8);
                \coordinate (P1) at ($(A)!.7!(B)+(3,0)$);
                \coordinate (P2) at ($(A)!.7!(C)+(-3,0)$);
                \coordinate (P) at (0,0);
                \draw[blue] (A) -- (B) -- (C) -- cycle;
                \draw[wline] (P1) -- (B);
                \draw[wline] (P2) -- (C);
                \draw[webline] (B) to[bend right=15] (C);
                \draw[webline] (A) -- (P1) -- (P2) -- (A);
                \filldraw (A) circle [radius=20pt];
                \filldraw (B) circle [radius=20pt];
                \filldraw (C) circle [radius=20pt];
            }\ 
        },\label{eq:skein_rel_rot_1}\\
        \mathord{
            \ \tikz[baseline=-.6ex, scale=.1]{
                \coordinate (A) at (90:8);
                \coordinate (B) at (210:8);
                \coordinate (C) at (-30:8);
                \coordinate (P1) at ($(A)!.8!(B)+(3,0)$);
                \coordinate (P2) at ($(A)!.8!(C)+(-3,0)$);
                \coordinate (P) at (0,0);
                \draw[wline] (P1) -- (B);
                \draw[wline] (P2) -- (C);
                \draw[webline] (A) -- (P1) -- (P2) -- (A);
                \draw[blue] (A) -- (B) -- (C) -- cycle;
                \filldraw (A) circle [radius=20pt];
                \filldraw (B) circle [radius=20pt];
                \filldraw (C) circle [radius=20pt];
            }\ 
        }
        \mathord{
            \ \tikz[baseline=-.6ex, scale=.1, rotate=-120]{
                \coordinate (A) at (90:8);
                \coordinate (B) at (210:8);
                \coordinate (C) at (-30:8);
                \coordinate (P1) at ($(A)!.8!(B)+(3,0)$);
                \coordinate (P2) at ($(A)!.8!(C)+(-3,0)$);
                \coordinate (P) at (0,0);
                \draw[wline] (P1) -- (B);
                \draw[wline] (P2) -- (C);
                \draw[webline] (A) -- (P1) -- (P2) -- (A);
                \draw[blue] (A) -- (B) -- (C) -- cycle;
                \filldraw (A) circle [radius=20pt];
                \filldraw (B) circle [radius=20pt];
                \filldraw (C) circle [radius=20pt];
            }\ 
        }
        &=q
        \mathord{
            \ \tikz[baseline=-.6ex, scale=.1]{
                \coordinate (A) at (90:8);
                \coordinate (B) at (210:8);
                \coordinate (C) at (-30:8);
                \coordinate (P) at (0,0);
                \draw[blue] (A) -- (B) -- (C) -- cycle;
                \draw[wline] (A) to[bend left=15] (B);
                \draw[wline] (B) to[bend left=15] (C);
                \draw[webline] (A) to[bend right=15] (C);
                \draw[webline] (A) to[bend left=20] (C);
                \filldraw (A) circle [radius=20pt];
                \filldraw (B) circle [radius=20pt];
                \filldraw (C) circle [radius=20pt];
            }\ 
        }
        +q^{-1}
        \mathord{
            \ \tikz[baseline=-.6ex, scale=.1]{
                \coordinate (A) at (90:8);
                \coordinate (B) at (210:8);
                \coordinate (C) at (-30:8);
                \coordinate (P) at (0,0);
                \coordinate (P1) at ($(P)+(150:3)$);
                \coordinate (P2) at ($(P)+(-90:3)$);
                \draw[blue] (A) -- (B) -- (C) -- cycle;
                \draw[wline] (P2) -- (B);
                \draw[webline] (P2) -- (A);
                \draw[overarc] (P1) -- (C);
                \draw[wline] (P1) -- (B);
                \draw[webline] (P1) -- (A);
                \draw[webline] (P2) -- (C);
                \draw[wline] (A) to[bend left=20] (C);
                \filldraw (A) circle [radius=20pt];
                \filldraw (B) circle [radius=20pt];
                \filldraw (C) circle [radius=20pt];
            }\ 
        }.\label{eq:skein_rel_rot_2}
    \end{align}
    Using these relations, we see that each elementary web can be expanded into a polynomial in $\cC_T$. Thus the assertion is proved.
\end{proof}
One may notice that the coefficients in \eqref{eq:s-cutting},\eqref{eq:w-cutting}, \eqref{eq:skein_rel_triangle}--\eqref{eq:skein_rel_rot_2} are all positive, namely in the monoid $\bZ_{+}[q^{\pm 1/2}] \subset \bZ[q^{\pm 1/2}]$ of Laurent polynomials with non-negative coefficients. Globally, however, it is possible that we have to use skein relations which produce negative signs (such as \eqref{rel:bigon}) in expanding a tangled $\mathfrak{sp}_4$-graph. 
Let us introduce a class of tangled $\mathfrak{sp}_4$-graphs whose expansions are ensured to have only positive coefficients.

\begin{figure}
	\begin{tikzpicture}[scale=.1]
		\coordinate (P1) at (90:15);
		\coordinate (P2) at (180:20);
		\coordinate (P3) at (-90:15);
		\coordinate (P4) at (0:20);
		\draw[blue] (P1) -- (P2) -- (P3) -- (P4) -- (P1) -- cycle;
		\draw[blue] (P1) -- (P3);
		\node at ($(P2)!.5!(P4)+(0,5)$) [right]{\scriptsize $E$};
		\node at ($(P2)!.3!(P4)$) {$T$};
		\node at ($(P4)!.3!(P2)$) {$T'$};
		\node at (P3) [below=10pt]{$\Delta$};
	\end{tikzpicture}
	\hspace{1em}
	\begin{tikzpicture}[scale=.1]
		\coordinate (P1) at (90:15);
		\coordinate (P2) at (180:20);
		\coordinate (P3) at (-90:15);
		\coordinate (P4) at (0:20);
		\draw[blue!50] (P1) -- (P2) -- (P3) -- (P4) -- (P1) -- cycle;
		\draw[blue!50] (P1) -- (P3);
		\draw[blue, thick] (P1) to[bend right=20] (P3);
		\draw[blue, thick] (P1) to[bend left=20] (P3);
		\draw[blue, thick] (P1) to[bend right=20] (P2);
		\draw[blue, thick] (P1) to[bend left=20] (P2);
		\draw[blue, thick, rounded corners] (P3) to[bend right=20] (P4);
		\draw[blue, thick, rounded corners] (P3) to[bend left=20] (P4);
		\draw[blue, thick, rounded corners] (P2) to[bend right=20] (P3);
		\draw[blue, thick, rounded corners] (P2) to[bend left=20] (P3);
		\draw[blue, thick, rounded corners] (P4) to[bend right=20] (P1);
		\draw[blue, thick, rounded corners] (P4) to[bend left=20] (P1);
		\node at (0,0) {\scriptsize $B_{E}$};
		\node at (-0.2,1) [above left]{\scriptsize $E_T$};
		\node at (0,1) [above right]{\scriptsize $E_{T'}$};
		\node at ($(P2)!.25!(P4)$) {$T$};
		\node at ($(P4)!.25!(P2)$) {$T'$};
		\node at (P3) [below=10pt]{$\Delta^{\mathrm{split}}$};
	\end{tikzpicture}
	\caption{The split triangulation $\Delta^{\mathrm{split}}$ associated with $\Delta$. Here $B_{E}$ is the biangle bounded by the two edges $E_T$ and $E_{T'}$.}
	\label{fig:split-tri}
\end{figure}
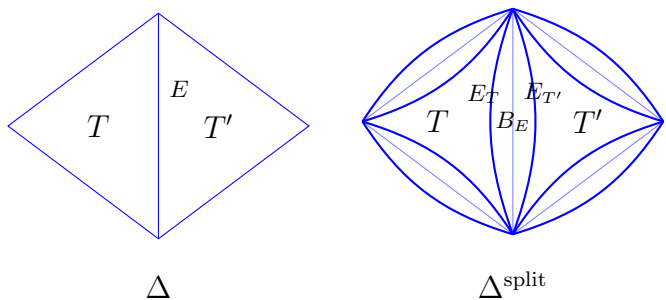

For an ideal triangulation $\Delta$, the associated \emph{splitting triangulation $\Delta^{\mathrm{split}}$} is obtained by replacing each edge $E$ of $\Delta$ with two parallel edges, as shown in \cref{fig:split-tri}.
The set of connected components of $\Sigma\setminus\Delta^{\mathrm{split}}$ is divided into two subsets: the set $t(\Delta^{\mathrm{split}})$ of triangles and the set $b(\Delta^{\mathrm{split}})$ of biangles.
We denote a triangle in $t(\Delta^{\mathrm{split}})$ by the same symbol as the corresponding triangle in $t(\Delta)$, while the biangle corresponding to an edge $E\in e(\Delta)$ is denoted by $B_{E}\in b(\Delta^{\mathrm{split}})$. 
For an edge $E\in e(\Delta)$ and a triangle $T\in t(\Delta)$ adjacent to $E$, let $E_{T}\in e(\Delta^{\mathrm{split}})$ denote the edge shared by $B_{E}$ and $T$.
\begin{dfn}[elevation-preserving $\mathfrak{sp}_{4}$-webs]\label{def:elevation-pres}
    A tangled $\mathfrak{sp}_4$-graph is said to be \emph{elevation-preserving with respect to $\Delta$} if it is represented by a diagram $G$ on $\Sigma$ satisfying the following conditions:
    \begin{itemize}
        \item For $B\in b(\Delta^{\mathrm{split}})$ corresponding to an interior edge, $G\cap B$ consists of a braid between the two sides of $B$ obtained by a superposition of arcs of either type. For $B\in b(\Delta^{\mathrm{split}})$ corresponding to a boundary interval, $G \cap B$ consists of a braid between the interior side and special points. For any pair of strands $\beta_1$ and $\beta_2$ in the braid that intersects each other, one of them is required to have a higher elevation than the other throughout the way: we denote by $\beta_1<\beta_2$ if $\beta_2$ passes over $\beta_1$.
        \item For any $T\in t(\Delta^{\mathrm{split}})$, $G\cap T$ consists of a superposition of the diagrams in
        \[
            \left\{
            \mathord{
                \ \tikz[baseline=-.6ex, scale=.07]{
                    \coordinate (A) at (90:8);
                    \coordinate (B) at (210:8);
                    \coordinate (C) at (-30:8);
                    \coordinate (AB) at ($(A)!.5!(B)$);
                    \coordinate (BC) at ($(B)!.5!(C)$);
                    \coordinate (CA) at ($(C)!.5!(A)$);
                    \coordinate (P) at (0,0);
                    \draw[blue] (A) -- (B) -- (C) -- cycle;
                    \begin{scope}
                        \clip (A) -- (B) -- (C) -- cycle;
                        \draw[wline] (P) -- (AB);
                        \draw[webline] (P) -- (BC);
                        \draw[webline] (P) -- (CA);                        
                    \end{scope}
                    \filldraw (A) circle [radius=20pt];
                    \filldraw (B) circle [radius=20pt];
                    \filldraw (C) circle [radius=20pt];
                }\ 
            },
            \mathord{
                \ \tikz[baseline=-.6ex, scale=.07]{
                    \coordinate (A) at (90:8);
                    \coordinate (B) at (210:8);
                    \coordinate (C) at (-30:8);
                    \coordinate (AB) at ($(A)!.5!(B)$);
                    \coordinate (BC) at ($(B)!.5!(C)$);
                    \coordinate (CA) at ($(C)!.5!(A)$);
                    \coordinate (P) at (0,0);
                    \draw[blue] (A) -- (B) -- (C) -- cycle;
                    \begin{scope}
                        \clip (A) -- (B) -- (C) -- cycle;
                        \draw[webline] (AB) to[bend right] (CA);
                    \end{scope}
                    \filldraw (A) circle [radius=20pt];
                    \filldraw (B) circle [radius=20pt];
                    \filldraw (C) circle [radius=20pt];
                }\ 
            },
            \mathord{
                \ \tikz[baseline=-.6ex, scale=.07]{
                    \coordinate (A) at (90:8);
                    \coordinate (B) at (210:8);
                    \coordinate (C) at (-30:8);
                    \coordinate (AB) at ($(A)!.5!(B)$);
                    \coordinate (BC) at ($(B)!.5!(C)$);
                    \coordinate (CA) at ($(C)!.5!(A)$);
                    \coordinate (P) at (0,0);
                    \draw[blue] (A) -- (B) -- (C) -- cycle;
                    \begin{scope}
                        \clip (A) -- (B) -- (C) -- cycle;
                        \draw[wline] (AB) to[bend right] (CA);
                    \end{scope}
                    \filldraw (A) circle [radius=20pt];
                    \filldraw (B) circle [radius=20pt];
                    \filldraw (C) circle [radius=20pt];
                }\ 
            },
            \text{ and their rotations}\right\}.
        \]
        Similarly, the relative elevations between any two components are required to be fixed throughout the way if they have a crossing.
        \item For any $T\in t(\Delta^{\mathrm{split}})$ and its adjacent biangle $B\in b(\Delta^{\mathrm{split}})$, two components $G_1,G_2$ of the diagram in $T$ satisfy $G_1<G_2$ if $\beta_1<\beta_2$, where $\beta_i$ are strands of the braid in $B$ connected to $G_i$ for $i=1,2$.
    \end{itemize}
\end{dfn}

\begin{ex}\label{ex:Desc}
    Elements in $\Desc{\Sigma}$, especially geometric $n$-bracelets and $n$-bangles in \cref{fig:bracelet}, are elevation-preserving for any ideal triangulation $\Delta$. Note that the legs of a descending loop or arc can avoid intersecting with the edges of $\Delta$.
\end{ex}

\begin{thm}\label{thm:positivity-web}
    Let $\Sigma$ be any unpunctured marked surface, $\Delta$ its ideal triangulation, and $\cC_{\bD}$ the web cluster associated with a decorated triangulation $\bD$ over $\Delta$.
    For any elevation-preserving $\mathfrak{sp}_4$-web $G$ with respect to $\Delta$, there exists a monomial $J$ in $\cC_{\bD}$ such that $[2]^{k}GJ$ is a polynomial in $\cC_{\bD}$ with coefficients in $\bZ_{+}[q^{\pm 1/2}]$, where $k$ is the total number of intersection points between type $2$ edges of $G$ and all the edges of the triangulation.
\end{thm}
\begin{proof}
    Let $G$ be an elevation-preserving web with respect to an ideal triangulation $\Delta$.
    We first expand $G$ into $\mathfrak{sp}_4$-webs in biangles and triangles in $\Delta^{\mathrm{split}}$ by the cutting trick in \cref{lem:cutteing-trick}, only producing positive coefficients by the discussion above. 
    Here notice that the expansion of $G$ can be obtained by expanding $G\cap T$ and $G\cap B$ for each $T\in t(\Delta^{\mathrm{split}})$ and $B\in b(\Delta^{\mathrm{split}})$ in the order from the webs in lower elevation to the higher, thanks to the elevation-preserving property. 
    Then it remains to see how the local pieces shown below are glued in biangles $B$ or triangles $T$:
    \begin{align*}
        X_1&:=\mathord{
            \ \tikz[baseline=-.6ex, scale=.07, yshift=-5cm]{
                \coordinate (A1) at (-10,0);
                \coordinate (A2) at (10,0);
                \coordinate (C0) at ($(A1)!0.5!(A2)+(0,12)$);
                \coordinate (B0) at ($(A1)!0.5!(A2)+(0,5)$);
                \draw[dashed] (A1) -- (C0) -- (A2);
                \draw[webline, rounded corners] (C0) -- (B0) -- (A2);
                \draw[blue, thick] (A1) -- (A2);
                \draw[fill] (A1) circle (30pt);
                \draw[fill] (A2) circle (30pt);
            }\ 
        },&
        X_2&:=\mathord{
            \ \tikz[baseline=-.6ex, scale=.07, yshift=-5cm]{
                \coordinate (A1) at (-10,0);
                \coordinate (A2) at (10,0);
                \coordinate (C0) at ($(A1)!0.5!(A2)+(0,12)$);
                \coordinate (B0) at ($(A1)!0.5!(A2)+(0,5)$);
                \draw[dashed] (A1) -- (C0) -- (A2);
                \draw[webline] (C0) -- (B0);
                \draw[webline] (A1) -- (B0);
                \draw[wline] (A2) -- (B0);
                \draw[blue, thick] (A1) -- (A2);
                \draw[fill] (A1) circle (30pt);
                \draw[fill] (A2) circle (30pt);
            }\ 
        },&
        X_3&:=\mathord{
            \ \tikz[baseline=-.6ex, scale=.07, yshift=-5cm]{
                \coordinate (A1) at (-10,0);
                \coordinate (A2) at (10,0);
                \coordinate (C0) at ($(A1)!0.5!(A2)+(0,12)$);
                \coordinate (B0) at ($(A1)!0.5!(A2)+(0,5)$);
                \draw[dashed] (A1) -- (C0) -- (A2);
                \draw[webline] (C0) -- (B0);
                \draw[wline] (A1) -- (B0);
                \draw[webline] (A2) -- (B0);
                \draw[blue, thick] (A1) -- (A2);
                \draw[fill] (A1) circle (30pt);
                \draw[fill] (A2) circle (30pt);
            }\ 
        },&
        X_4&:=\mathord{
            \ \tikz[baseline=-.6ex, scale=.07, yshift=-5cm]{
                \coordinate (A1) at (-10,0);
                \coordinate (A2) at (10,0);
                \coordinate (C0) at ($(A1)!0.5!(A2)+(0,12)$);
                \coordinate (B0) at ($(A1)!0.5!(A2)+(0,5)$);
                \draw[dashed] (A1) -- (C0) -- (A2);
                \draw[webline, rounded corners] (C0) -- (B0) -- (A1);
                \draw[blue, thick] (A1) -- (A2);
                \draw[fill] (A1) circle (30pt);
                \draw[fill] (A2) circle (30pt);
            }\ 
        },&&\\
        Y_1&:=\mathord{
            \ \tikz[baseline=-.6ex, scale=.07, yshift=-5cm]{
                \coordinate (A1) at (-10,0);
                \coordinate (A2) at (10,0);
                \coordinate (C0) at ($(A1)!0.5!(A2)+(0,12)$);
                \coordinate (B0) at ($(A1)!0.5!(A2)+(0,5)$);
                \draw[dashed] (A1) -- (C0) -- (A2);
                \draw[wline, rounded corners] (C0) -- (B0) -- (A2);
                \draw[blue, thick] (A1) -- (A2);
                \draw[fill] (A1) circle (30pt);
                \draw[fill] (A2) circle (30pt);
            }\ 
        },&
        Y_2&:=\mathord{
            \ \tikz[baseline=-.6ex, scale=.07, yshift=-5cm]{
                \coordinate (A1) at (-10,0);
                \coordinate (A2) at (10,0);
                \coordinate (C0) at ($(A1)!0.5!(A2)+(0,12)$);
                \coordinate (B1) at ($(A1)!0.5!(A2)+(0,3)$);
                \coordinate (B2) at ($(A1)!0.5!(A2)+(0,8)$);
                \draw[dashed] (A1) -- (C0) -- (A2);
                \draw[wline] (C0) -- (B2);
                \draw[wline] (A1) -- (B1);
                \draw[webline] (A2) -- (B1) -- (B2) -- cycle;
                \draw[blue, thick] (A1) -- (A2);
                \draw[fill] (A1) circle (30pt);
                \draw[fill] (A2) circle (30pt);
            }\ 
        },&
        Y_3&:=\mathord{
            \ \tikz[baseline=-.6ex, scale=.07, yshift=-5cm]{
                \coordinate (A1) at (-10,0);
                \coordinate (A2) at (10,0);
                \coordinate (C0) at ($(A1)!0.5!(A2)+(0,12)$);
                \coordinate (B0) at ($(A1)!0.5!(A2)+(0,5)$);
                \draw[dashed] (A1) -- (C0) -- (A2);
                \draw[wline] (C0) -- (B0);
                \draw[webline] (A1) -- (B0);
                \draw[webline] (A2) -- (B0);
                \draw[blue, thick] (A1) -- (A2);
                \draw[fill] (A1) circle (30pt);
                \draw[fill] (A2) circle (30pt);
            }\ 
        },&
        Y_4&:=\mathord{
            \ \tikz[baseline=-.6ex, scale=.07, yshift=-5cm]{
                \coordinate (A1) at (-10,0);
                \coordinate (A2) at (10,0);
                \coordinate (C0) at ($(A1)!0.5!(A2)+(0,12)$);
                \coordinate (B1) at ($(A1)!0.5!(A2)+(0,3)$);
                \coordinate (B2) at ($(A1)!0.5!(A2)+(0,8)$);
                \draw[dashed] (A1) -- (C0) -- (A2);
                \draw[wline] (C0) -- (B2);
                \draw[wline] (A2) -- (B1);
                \draw[webline] (A1) -- (B1) -- (B2) -- cycle;
                \draw[blue, thick] (A1) -- (A2);
                \draw[fill] (A1) circle (30pt);
                \draw[fill] (A2) circle (30pt);
            }\ 
        },&
        Y_5&:=\mathord{
            \ \tikz[baseline=-.6ex, scale=.07, yshift=-5cm]{
                \coordinate (A1) at (-10,0);
                \coordinate (A2) at (10,0);
                \coordinate (C0) at ($(A1)!0.5!(A2)+(0,12)$);
                \coordinate (B0) at ($(A1)!0.5!(A2)+(0,5)$);
                \draw[dashed] (A1) -- (C0) -- (A2);
                \draw[wline, rounded corners] (C0) -- (B0) -- (A1);
                \draw[blue, thick] (A1) -- (A2);
                \draw[fill] (A1) circle (30pt);
                \draw[fill] (A2) circle (30pt);
            }\ .
        }
    \end{align*}
    The gluing patterns are
    \begin{align*}
        B(X_i,X_j)&:=
        \mathord{
            \ \tikz[baseline=-.6ex, scale=.08]{
                \coordinate (A1) at (-12,0);
                \coordinate (A2) at (12,0);
                \coordinate (P) at (0,0);
                \draw[dashed] (A1) -- (A2);
                \draw[blue, thick] (A1) to[bend left=60] (A2);
                \draw[blue, thick] (A1) to[bend right=60] (A2);
                \node at ($(P)-(0,3)$) [rotate=0]{\small $X_i$};
                \node at ($(P)+(0,3)$) [rotate=180]{\small $X_j$};
                \draw[fill] (A1) circle (30pt);
                \draw[fill] (A2) circle (30pt);
            }\ 
        },&
        B(Y_i,Y_j)&:=
        \mathord{
            \ \tikz[baseline=-.6ex, scale=.08]{
                \coordinate (A1) at (-12,0);
                \coordinate (A2) at (12,0);
                \coordinate (P) at (0,0);
                \draw[dashed] (A1) -- (A2);
                \draw[blue, thick] (A1) to[bend left=60] (A2);
                \draw[blue, thick] (A1) to[bend right=60] (A2);
                \node at ($(P)-(0,3)$) [rotate=0]{\small $Y_i$};
                \node at ($(P)+(0,3)$) [rotate=180]{\small $Y_j$};
                \draw[fill] (A1) circle (30pt);
                \draw[fill] (A2) circle (30pt);
            }\ 
        },&&\\
        T(X_i,X_j,\emptyset)&:=
        \mathord{
            \ \tikz[baseline=-.6ex, scale=.08]{
                \coordinate (A) at (90:15);
                \coordinate (B) at (210:15);
                \coordinate (C) at (330:15);
                \coordinate (AB) at ($(A)!.5!(B)$);
                \coordinate (BC) at ($(B)!.5!(C)$);
                \coordinate (CA) at ($(C)!.5!(A)$);
                \coordinate (P) at (0,0);
                \draw[dashed] (P) -- (A);
                \draw[dashed] (P) -- (B);
                \draw[dashed] (P) -- (C);
                \draw[blue, thick] (A) -- (B) -- (C) -- cycle;
                \node at ($(BC)!.4!(P)$) [rotate=0]{\small $X_i$};
                \node at ($(CA)!.4!(P)$) [rotate=120]{\small $X_j$};
                \node at ($(AB)!.4!(P)$) [rotate=240]{};
                \draw[fill] (A) circle (30pt);
                \draw[fill] (B) circle (30pt);
                \draw[fill] (C) circle (30pt);
            }\ 
        },&
        T(Y_i,Y_j,\emptyset)&:=
        \mathord{
            \ \tikz[baseline=-.6ex, scale=.08]{
                \coordinate (A) at (90:15);
                \coordinate (B) at (210:15);
                \coordinate (C) at (330:15);
                \coordinate (AB) at ($(A)!.5!(B)$);
                \coordinate (BC) at ($(B)!.5!(C)$);
                \coordinate (CA) at ($(C)!.5!(A)$);
                \coordinate (P) at (0,0);
                \draw[dashed] (P) -- (A);
                \draw[dashed] (P) -- (B);
                \draw[dashed] (P) -- (C);
                \draw[blue, thick] (A) -- (B) -- (C) -- cycle;
                \node at ($(BC)!.4!(P)$) [rotate=0]{\small $Y_i$};
                \node at ($(CA)!.4!(P)$) [rotate=120]{\small $Y_j$};
                \node at ($(AB)!.4!(P)$) [rotate=240]{};
                \draw[fill] (A) circle (30pt);
                \draw[fill] (B) circle (30pt);
                \draw[fill] (C) circle (30pt);
            }\ 
        },&
        T(Y_i,X_j,X_k)&:=
        \mathord{
            \ \tikz[baseline=-.6ex, scale=.08]{
                \coordinate (A) at (90:15);
                \coordinate (B) at (210:15);
                \coordinate (C) at (330:15);
                \coordinate (AB) at ($(A)!.5!(B)$);
                \coordinate (BC) at ($(B)!.5!(C)$);
                \coordinate (CA) at ($(C)!.5!(A)$);
                \coordinate (P) at (0,0);
                \draw[dashed] (P) -- (A);
                \draw[dashed] (P) -- (B);
                \draw[dashed] (P) -- (C);
                \draw[blue, thick] (A) -- (B) -- (C) -- cycle;
                \node at ($(BC)!.4!(P)$) [rotate=0]{\small $Y_i$};
                \node at ($(CA)!.4!(P)$) [rotate=120]{\small $X_j$};
                \node at ($(AB)!.4!(P)$) [rotate=240]{\small $X_k$};
                \draw[fill] (A) circle (30pt);
                \draw[fill] (B) circle (30pt);
                \draw[fill] (C) circle (30pt);
            }\ 
        }
    \end{align*}
    up to symmetries of the bigon or the triangle.
    In the case of bigon, it is easy to see that
    \begin{align*}
        &B(X_1,X_1)=B(X_4,X_4)=
        \mathord{
            \ \tikz[baseline=-.6ex, scale=.06]{
                \coordinate (A1) at (-12,0);
                \coordinate (A2) at (12,0);
                \coordinate (P) at (0,0);
                \draw[blue, thick] (A1) to[bend left=60] (A2);
                \draw[blue, thick] (A1) to[bend right=60] (A2);
                \draw[webline] (A1) -- (A2);
                \draw[fill] (A1) circle (30pt);
                \draw[fill] (A2) circle (30pt);
            }\ 
        },&
        &B(X_2,X_2)=B(X_3,X_3)=
        \mathord{
            \ \tikz[baseline=-.6ex, scale=.06]{
                \coordinate (A1) at (-12,0);
                \coordinate (A2) at (12,0);
                \coordinate (P) at (0,0);
                \draw[blue, thick] (A1) to[bend left=60] (A2);
                \draw[blue, thick] (A1) to[bend right=60] (A2);
                \draw[webline] (A1) to[bend right=15] (A2);
                \draw[wline] (A1) to[bend left=15] (A2);
                \draw[fill] (A1) circle (30pt);
                \draw[fill] (A2) circle (30pt);
            }\ 
        },\\
        &B(Y_1,Y_1)=B(Y_4,Y_4)=
        \mathord{
            \ \tikz[baseline=-.6ex, scale=.06]{
                \coordinate (A1) at (-12,0);
                \coordinate (A2) at (12,0);
                \coordinate (P) at (0,0);
                \draw[blue, thick] (A1) to[bend left=60] (A2);
                \draw[blue, thick] (A1) to[bend right=60] (A2);
                \draw[wline] (A1) -- (A2);
                \draw[fill] (A1) circle (30pt);
                \draw[fill] (A2) circle (30pt);
            }\ 
        },&
        &B(Y_2,Y_2)=B(Y_3,Y_3)=
        \mathord{
            \ \tikz[baseline=-.6ex, scale=.06]{
                \coordinate (A1) at (-12,0);
                \coordinate (A2) at (12,0);
                \coordinate (P) at (0,0);
                \draw[blue, thick] (A1) to[bend left=60] (A2);
                \draw[blue, thick] (A1) to[bend right=60] (A2);
                \draw[webline] (A1) to[bend right=10] (A2);
                \draw[webline] (A1) to[bend right=30] (A2);
                \draw[wline] (A1) to[bend left=15] (A2);
                \draw[fill] (A1) circle (30pt);
                \draw[fill] (A2) circle (30pt);
            }\ 
        },
    \end{align*}
    and otherwise zero.
    
    In the case of triangle, it is easy to confirm that 
    \begin{align*}
        T(X_i,X_j,\emptyset)=T(Y_i,Y_j,\emptyset)=0
    \end{align*}
    for $i<j$, and 
    \begin{align*}
        T(Y_{i},X_{1},X_{4})=T(Y_{1},X_{3},X_{k})=T(Y_{1},X_{4},X_{k})=T(Y_{2},X_{4},X_{k})=0
    \end{align*}
    for any $i\in\{1,\dots,5\}$ and $k\in\{1,\dots,4\}$.
    One can also see that $X_i$ (resp.~$Y_i$) coincides with $X_{5-i}$ (resp.~$Y_{6-i}$).
    Thus, $T(X_{5-j},X_{5-i},\emptyset)$, $T(Y_{6-j},Y_{6-i},\emptyset)$, and $T(Y_{6-i},X_{5-j},X_{5-k})$ are obtained from computations of $T(X_{i},X_{j},\emptyset)$, $T(Y_{i},Y_{j},\emptyset)$, and $T(Y_{i},X_{j},X_{k})$, respectively.
    Then it remains to compute the following patterns, which are verified by straightforward calculation:


    Observe that all the coefficients are in $\bZ^{+}[q^{\pm 1/2}]$, up to the appearance of $[2]^{-1}$ related to $Y_3$.
    We remark that $Y_3$ appears only when we apply the cutting lemma to a type $2$ edges. 
    Hence, we see that $[2]^{k}G$ is expanded as a polynomial in $\cup_{T\in t(\Delta)}\Eweb{T}$ with coefficients in $\bZ^{+}[q^{\pm 1/2}]$, where $k$ is as in the statement.
    Finally, we can further expand the polynomial in $\cup_{T\in t(\Delta)}\Eweb{T}$ into a polynomial in $\cC_{\Delta}$ with positive coefficients in the same way as in the proof of \cref{thm:Cweb-exp}. 
\end{proof}
\section{Quantum cluster algebras}\label{sec:cluster}

\subsection{Quantum cluster algebra}\label{subsub:CA}
Recall that for a skew-symmetric form $\Pi$ on a lattice $L$, the associated \emph{based quantum torus} is the associative $\bZ_q$-algebra $T_\Pi$ such that
\begin{itemize}
    \item $T_\Pi$ has a free $\bZ_q$-basis $M^\alpha$ parametrized by $\alpha \in L$, and 
    \item the product of these basis elements is given by $M^\alpha\cdot M^\beta = q^{\Pi(\alpha,\beta)/2} M^{\alpha+\beta}$.
\end{itemize}
Fix a datum $(I,I_\uf,D)$, where $I=\{1,\dots,N\}$ is a finite set of indices and $I_\uf \subset I$ is a subset; $D=\mathrm{diag}(d_i \mid i \in I)$ is a positive integral diagonal matrix. The indices in $I_\uf$ are called \emph{unfrozen indices}, while those in the complement $I_\f:=I \setminus I_\uf$ are called \emph{frozen indices}. 
A \emph{quantum seed} in $\cF$ is a quadruple $(B,\Pi,\accentset{\circ}{\Lambda},M)$, where
\begin{itemize}
    \item $B=(b_{ij})_{i,j \in I}$ is a matrix with half-integral entries such that $DB$ is skew-symmetric, and $b_{ij} \in \bZ$ unless $(i,j) \in I_\f \times I_\f$;
    \item $\Pi=(\pi_{ij})_{i,j \in I}$ is a skew-symmetric matrix with integral entries satisfying the \emph{compatibility relation}
    \begin{align*}
        \sum_{k \in I} b_{ki}\pi_{kj} =\delta_{ij}\widehat{d}_j
    \end{align*}
    for all $i \in I_\uf$ and $j \in I$, where $\widehat{d}_i$ is a positive integer for $i \in I_\uf$. 
    \item $\accentset{\circ}{\Lambda}=\bigoplus_{i \in I}\bZ\sff_i$ is a lattice, on which the matrix $\Pi$ defines a skew-symmetric form by $\Pi(\sff_i,\sff_j):=\pi_{ij}$;
    \item $M: \accentset{\circ}{\Lambda} \to \cF\setminus \{0\}$ is a function such that 
    \begin{align*}
        M(\alpha)M(\beta) = q^{\Pi(\alpha,\beta)/2} M(\alpha+\beta)
    \end{align*}
    for $\alpha,\beta \in \accentset{\circ}{\Lambda}$, 
    and the $\bZ_q$-span of $M(\accentset{\circ}{\Lambda}) \subset \cF$ is the based quantum torus of the form $\Pi$ whose skew-field of fractions coincides with $\cF$. 
\end{itemize}
We call $B$ the \emph{exchange matrix}, $\Pi$ the \emph{compatibility matrix}, and $M$ the \emph{toric frame} of the quantum seed. When no confusion can occur, we omit the lattice $\accentset{\circ}{\Lambda}$ from the notation and call the triple $(B,\Pi,M)$ a quantum seed. 
The compatibility relation can be written as
\begin{align}\label{eq:compatibility_definition}
    \varepsilon\Pi = B^{\mathsf{T}} \Pi = (\widehat{D},0),
\end{align}
where $\widehat{D}:=\mathrm{diag}(\widehat{d}_i\mid i \in I_\uf)$ and $0$ denotes the $I_\uf \times I_\f$-zero matrix. By \cite[Lemma 4.4]{BZ}, a toric frame $M$ is uniquely determined by the values $A_i:=M(\sff_i)$, which we call the \emph{(quantum) cluster variable}, on the basis vectors $\sff_i$ for $i \in I$. Indeed, 
we have
\begin{align}\label{eq:extension_toric_frame}
    M\bigg(\sum_{i \in I} x_i\sff_i\bigg) = q^{\frac{1}{2}\sum_{l < k}x_k x_l \pi_{kl} } A_1^{x_1}\dots A_N^{x_N}
\end{align}
for all $(x_1,\dots,x_N) \in \bZ^N$.
Note that both sides are invariant under permutations of indices. 
Elements of the form $M(\alpha)$ for $\alpha \in \accentset{\circ}{\Lambda}$ are called \emph{cluster monomials}. 

Given a quantum seed $(B,\Pi,M)$ in $\cF$ and an unfrozen index $k \in I_\uf$, the \emph{quantum seed mutation} produces a new quantum seed $(B',\Pi',M')=\mu_k(B,\Pi,M)$ according to the following rule. Let $E_{k,\epsilon}= (e_{ij})_{i,j \in I}$ and $F_{k,\epsilon}= (f_{ij})_{i,j \in I}$ be the matrices defined by
\begin{align*}
    e_{ij}:=\begin{cases}
    \delta_{ij} & \mbox{if $j \neq k$},\\
    -1 & \mbox{if $i=k=j$},\\
    [-\epsilon b_{ik}]_+ & \mbox{if $i\neq k=j$},
    \end{cases}
\end{align*}
and 
\begin{align*}
    f_{ij}:=\begin{cases}
    \delta_{ij} & \mbox{if $i \neq k$},\\
    -1 & \mbox{if $i=k=j$},\\
    [\epsilon b_{kj}]_+ & \mbox{if $i=k \neq j$},
    \end{cases}
\end{align*}
respectively for $\epsilon \in \{+,-\}$. Then we define
\begin{align}
    B' &= E_{k,\epsilon} B F_{k,\epsilon}, \label{eq:matrix-mutation}\\
    \Pi' &= E_{k,\epsilon}^{\mathsf{T}} \Pi E_{k,\epsilon}, \label{eq:compatible-mutation}\\
    M'(\sff_i') &= \begin{cases}
    M(\sff_i) & \mbox{if $i \neq k$}, \\
    M(-\sff_k + \sum_{j \in I} [b_{jk}]_+ \sff_j) + M(-\sff_k + \sum_{j \in I} [-b_{jk}]_+ \sff_j) & \mbox{if $i=k$}.
    \end{cases}\label{eq:q-mutation}
\end{align}
Here $(\sff_i)_{i \in I}$ and $(\sff'_i)_{i \in I}$ denote the basis vectors of the underlying lattices. The relation \eqref{eq:q-mutation} is called the \emph{quantum exchange relation}. It can be verified that \eqref{eq:matrix-mutation} and \eqref{eq:compatible-mutation} do not depend on the sign $\epsilon$ \cite[Proposition 3.4]{BZ}. The mutation rule \eqref{eq:matrix-mutation} is the same one as the well-known matrix mutation formula.
In this paper, the quantum exchange relation is often used in the following form:

\begin{lem}
The quantum exchange relation \eqref{eq:q-mutation} is equivalent to
\begin{align}\label{eq:q-mutation_asymmetric}
    A_k A'_k = q^{\frac{1}{2}\sum_{j \in I}[b_{jk}]_+\pi_{kj} } \left( M\bigg(\sum_{j \in I} [b_{jk}]_+ \sff_j\bigg) + q^{\frac{1}{2}\widehat{d}_k}M\bigg(\sum_{j \in I} [-b_{jk}]_+ \sff_j\bigg) \right)
\end{align}
and $A'_i = A_i$ for $i \neq k$, 
where $A_j:=M(\sff_j)$ and $A'_j:=M'(\sff'_j)$ for $j \in I$.
\end{lem}
The verification of the following lemma is also straightforward:

\begin{lem}\label{lem:compatibility_check}
Let $(B,\Pi,M)$ be a quantum seed in $\cF$, $k \in I_\uf$, and consider the exchange matrix $B' := E_{k,\epsilon} B F_{k,\epsilon}$ and the toric frame $M'$ determined by \eqref{eq:q-mutation}. Let $\Pi'=(\pi'_{ij})_{i,j \in I}$ be the skew-symmetric matrix associated with $M'$, which is uniquely determined by the condition 
\begin{align*}
    A'_i A'_j = q^{\pi'_{ij}} A'_j A'_i
\end{align*}
for $i,j \in I$ with $A'_i:=M'(\sff'_i)$. Then the pair $(B',\Pi')$ satisfies the compatibility relation.
\end{lem}


Let 
\begin{align*}
    \mathfrak{S}_I^{\mathrm{cl}}:= \{ \sigma \in \mathfrak{S}_I \mid \sigma(I_\uf) = I_\uf,~d_{\sigma^{-1}(i)}=d_i~\mbox{for all $i \in I$}\}
\end{align*}
denote the group of permutations that do not mix the unfrozen/frozen indices and that preserve the weights. For $\sigma \in \mathfrak{S}_I^\mathrm{cl}$, a new quantum seed $(B',\Pi',M')=\sigma(B,\Pi,M)$ is defined by 
\begin{align*}
    b'_{ij} = b_{\sigma^{-1}(i),\sigma^{-1}(j)}, \quad \pi'_{ij} = \pi_{\sigma^{-1}(i),\sigma^{-1}(j)}, \quad A'_i = A_{\sigma^{-1}(i)}.
\end{align*}
An $\mathfrak{S}_{I}^\mathrm{cl}$-orbit of quantum seeds is called a \emph{quantum unlabeled seed}. 
Two quantum seeds in $\cF$ are said to be \emph{mutation-equivalent} if they are transformed to each other by a finite sequence of seed mutations and permutations. An equivalence class of quantum seeds is called a \emph{mutation class}. The relations among the quantum seeds in a given mutation class $\sfs_q$ can be encoded in the \emph{(labeled) exchange graphs}:
\begin{dfn}
The \emph{labeled exchange graph} is a graph $\bExch_{\sfs}$ with vertices $v$ corresponding to the quantum seeds $\sfs_q^{(v)}$ in $\sfs_q$, together with labeled edges of the following two types:
\begin{itemize}
    \item edges of the form $v \overbar{k} v'$ whenever the quantum seeds $\sfs_q^{(v)}$ and $\sfs_q^{(v')}$ are related by the mutation $\mu_k$ for $k \in I_\uf$;
    \item edges of the form $v \overbar{\sigma} v'$ whenever the quantum seeds $\sfs_q^{(v)}$ and $\sfs_q^{(v')}$ are related by the transposition $\sigma=(j\ k)$ in $\mathfrak{S}_I^\mathrm{cl}$.
\end{itemize}
The \emph{exchange graph} is a graph $\Exch_\sfs$ with vertices $\omega$ corresponding to the quantum unlabeled seeds $\sfs_q^{(\omega)}$ in $\sfs_q$, together with (unlabeled, horizontal) edges corresponding to the mutations. There is a graph projection $\pi_\sfs: \bExch_\sfs \to \Exch_\sfs$. 
\end{dfn}
When no confusion can occur, we simply denote a vertex of the labeled exchange graph by $v \in \bExch_\sfs$ instead of $v \in V(\bExch_\sfs)$, and similarly for the exchange graph. We remark that the (labeled) exchange graph depend only on the mutation class of the underlying exchange matrices. The absense of $q$ in the notation will be explained in \cref{rem:classical_CA}. 

\begin{rem}[labelings]\label{rem:labelings}
For a vertex $\omega \in \Exch_\sfs$, picking up a lift $v \in \pi_\sfs^{-1}(\omega)$ amounts to fixing a labeling of the unlabeled seed $\sfs_q^{(\omega)}$. More generally, $\pi_\sfs$ can be trivialized over any subgraph $\Gamma \subset \Exch_{\sfs}$ only with square faces (corresponding to the commutativity $\mu_i\mu_j=\mu_j\mu_i$ for $b_{ij}=0$). Such a graph $\Gamma$ can be uniquely lifted to a subgraph $\widetilde{\Gamma} \subset \bExch_{\sfs}$ if one fixes a labeling at a single vertex.
\end{rem}

To each vertex $v \in \bExch_{\sfs}$, associated is a based quantum torus 
\begin{align*}
    T_{(v)}=\mathrm{span}_{\bZ_q} M^{(v)}(\accentset{\circ}{\Lambda}^{(v)}) \subset \cF.
\end{align*}
We also have the unlabeled version $T_{(\omega)}=\mathrm{span}_{\bZ_q} M^{(\omega)}(\accentset{\circ}{\Lambda}^{(\omega)})$ for $\omega \in \Exch_\sfs$, where the basis of $\Lambda^{(\omega)}$ is unordered. The unlabeled collection $\mathbf{A}_{(\omega)}:=\{A_i^{(v)}\}_{i \in I}$ with $v \in \pi_\sfs^{-1}(\omega)$ is called a \emph{quantum cluster}. We also have the subcollection $\mathbf{A}_{(\omega)}^{\f}:=\{A_i^{(v)}\}_{i \in I_\f}$ of frozen variables. 

\begin{dfn}
The \emph{quantum cluster algebra} associated with a mutation class $\sfs_q$ of quantum seeds is the $\bZ_q$-subalgebra $\CA_{\sfs_q} \subset \cF$ generated by the union of the quantum clusters $\mathbf{A}_{(\omega)}$ and the inverses of frozen variables in $\mathbf{A}_{(\omega)}^{\f}$ for $\omega \in \Exch_{\sfs}$. 
The \emph{quantum upper cluster algebra} is defined to be
\begin{align*}
    \UCA_{\sfs_q} :=\bigcap_{\omega \in \Exch_{\sfs}} T_{(\omega)} \subset \cF.
\end{align*}
\end{dfn}
For each vertex $\omega \in \Exch_\sfs$, the \emph{upper bound} at $\omega$ is defined to be 
\begin{align*}
    \UCA_{\sfs_q}(\omega):=T_{(\omega)}\cap \bigcap_{\omega'} T_{(\omega')},
\end{align*}
where $\omega' \in \Exch_\bs$ runs over the vertices adjacent to $\omega$. 

\begin{thm}[Quantum upper bound theorem {\cite[Theorem 5.1]{BZ}}]\label{thm:q-Laurent}
For any vertices $\omega,\omega' \in \Exch_\sfs$, we have $\UCA_{\sfs_q}(\omega) = \UCA_{\sfs_q}(\omega')$. In particular, we have
\begin{align*}
    \UCA_{\sfs_q} = \UCA_{\sfs_q}(\omega)
\end{align*}
for any $\omega \in \Exch_\sfs$.\footnote{We remark here that the coprimality condition required in the classical setting (\cite[Corollary 1.7]{BFZ}) is automatically satisfied in the quantum setting, since the existence of the compatibility matrix forces the exchange matrix to be full-rank.}
\end{thm}
It in particular implies the inclusion $\CA_{\sfs_q} \subset \UCA_{\sfs_q}$, which is called the \emph{quantum Laurent phenomenon}. Again we remark that the quantum (upper) cluster algebra depends only on the mutation class of the compatibility pairs $(B,\Pi)$, up to automorphisms of the ambient skew-field. In other words, the choice of toric frames determines the way of realization of these algebras in some skew-field.

\begin{rem}[cluster algebras]\label{rem:classical_CA}
Let $\overline{\cF}$ be a field isomorphic to the field of rational functions on $N$ variables with complex coefficients. 
A seed in $\overline{\cF}$ is a pair $(B,\mathbf{A})$, where $B$ is an exchange matrix as above, and $\mathbf{A}=(A_i)_{i \in I}$ is a tuple of algebraically independent elements of $\overline{\cF}$. The mutation of a seed for $k \in I_\uf$ is defined by the matrix mutation rule \eqref{eq:matrix-mutation} and 
\begin{align*}
    A'=\begin{cases}
    \displaystyle{ A_k^{-1} \left(\prod_{j \in I}A_j^{[ b_{jk}]_+} + \prod_{j \in I}A_j^{[-b_{jk}]_+}\right)} & \mbox{if $i=k$}, \\
    A_i & \mbox{if $i \neq k$},
    \end{cases} 
\end{align*}
which is the classical counterpart of \eqref{eq:q-mutation}. A mutation class $\sfs$ of seeds is similarly defined, and the \emph{cluster algebra} is the $\bC$-subalgebra $\CA_\sfs \subset \overline{\cF}$ generated by all the cluster variables in $\sfs$. The \emph{upper cluster algebra} is defined to be the intersection $\UCA_\sfs:=\bigcap_{\omega \in \Exch_\sfs}\bC[\mathbf{A}^{(\omega)}] \subset \overline{\cF}$ of Laurent polynomial rings. 

When a mutation class $\sfs_q$ of quantum seeds and a mutation class $\sfs$ of seeds contain the same mutation class of exchange matrices, there is a graph isomorphism between the labeled exchange graph $\bExch_{\sfs}$ associated with $\sfs_q$ and a similar graph parametrizing the seeds in $\sfs$. In particular the quantum cluster variables in $\sfs_q$ bijectively corresponds to the cluster variables in $\sfs$. 
We say that $\sfs_q$ (resp. $\CA_{\sfs_q}$) is a \emph{quantization} of $\sfs$ (resp. $\CA_\sfs$). 
\end{rem}

\paragraph{\textbf{Bar-involution}}
For each $\omega \in \Exch_{\sfs}$, define a $\bZ$-linear
involution $\dagger: T_{(\omega)} \to T_{(\omega)}$ by
\begin{align*}
    (q^{r/2}M^{(\omega)}(\alpha))^\dagger := q^{-r/2}M^{(\omega)}(\alpha)
\end{align*}
for $r \in \bZ$ and $\alpha \in \accentset{\circ}{\Lambda}{}^{(\omega)}$. Then $\dagger$ preserves the subalgebras $\CA_{\sfs_q} \subset \UCA_{\sfs_q} \subset T_{(\omega)}$, and the induced involution does not depend on the choice of $\omega$ \cite[Proposition 6.2]{BZ}. Following \cite{BZ}, we call this anti-involution $\dagger: \UCA_{\sfs_q} \to \UCA_{\sfs_q}$ the \emph{bar-involution}. Each quantum cluster variable is invariant under the bar-involution.

\bigskip
\paragraph{\textbf{Exchange matrices and weighted quivers}}
It is useful to represent an exchange matrix $B=(b_{ij})_{i,j \in I}$ by a weighted quiver $Q$. 
For the correspondence, we use the convention in \cite{IIO21}. 
Let us define the \emph{Fock--Goncharov exchange matrix} $\varepsilon=(\varepsilon_{ij})_{i,j \in I}$ by $\varepsilon_{ij}:=b_{ji}$. 
Then the weighted quiver $Q$ corresponding to $B$ has vertices parametrized by the set $I$, and each vertex $i \in I$ is assigned the integer weight $d_i$. The structure matrix $\sigma=(\sigma_{ij})_{i,j \in I}$ of $Q$, whose $(i,j)$-entry indicates the number of arrows from $i$ to $j$ minus the number of arrows from $j$ to $i$, is defined to be
\begin{align*}
    \sigma_{ij}:=d_i^{-1}\varepsilon_{ij}\gcd(d_i,d_j).
\end{align*}
In figures, we draw $n$ dashed arrows from $i$ to $j$ if $\sigma_{ij}=n/2$ for $n \in \bZ$, where a pair of dashed arrows is replaced with a solid arrow. In this paper, we only deal with the weighted quivers whose vertices have weights $1$ or $2$. A vertex of weight $1$ (resp. $2$) is shown by a small circle (resp. a doubled circle).
\begin{ex}
For the matrices $B=\begin{pmatrix} 0 & -1 \\ 2 & 0\end{pmatrix}$ and $D=\mathrm{diag}(2,1)$, the matrix $DB=\begin{pmatrix} 0 & -2 \\ 2 & 0\end{pmatrix}$ is skew-symmetric. 
We have $\varepsilon=B^\mathsf{T}=\begin{pmatrix} 0 & 2 \\ -1 & 0\end{pmatrix}$ and $\sigma=\begin{pmatrix} 0 & 1 \\ -1 & 0\end{pmatrix}$. The corresponding weighted quiver is given by
\begin{align*}
    \tikz[>=latex]{\dnode{0,0}{black} node[above]{$1$}; \draw(2,0) circle(2pt) node[above]{$2$}; \qstarrow{0,0}{2,0};}.
\end{align*}
\end{ex}

The following lemma is useful to compute the mutations in terms of the weighted quivers:
\begin{lem}[{\cite[Lemma 2.3]{IIO21}}]\label{lem:weighted_mutation}
Assume that the weights $d_i$ take only two values: $\{d_i\}_{i \in I}=\{1,d\}$ for some integer $d\geq 2$. 
Then for $B'=\mu_k B$, the corresponding mutation of
the structure matrix $\sigma' = \mu_k\sigma$ is given by
\begin{align*}
  \sigma_{ij}' = 
  \begin{cases}
    -\sigma_{ij} & i=k \text{ or } j=k, \\
    \sigma_{ij} + ([\sigma_{ik}]_+ [\sigma_{kj}]_+ - [-\sigma_{ik}]_+ [-\sigma_{kj}]_+)\alpha_{ij}^k & \text{otherwise},
  \end{cases} 
\end{align*}
where 
\begin{align*}
    \alpha_{ij}^k = \begin{cases}
    d & \mbox{if $d_i=d_j \neq d_k$},\\
    1 & \mbox{otherwise}.
    \end{cases}
\end{align*}
\end{lem}

\subsection{The cluster algebras related to the moduli space \texorpdfstring{$\A_{Sp_4,\Sigma}$}{A(sp4,S)}}\label{subsec:cluster_sp4}

Let $\Sigma$ be a marked surface. Although the classical construction presented here works for any marked surface, we assume that $\Sigma$ is unpunctured for simplicity. 

Recall that a decorated triangulation $\bD=(\Delta,m_\Delta,\bs_\Delta)$ consists of an ideal triangulation $\Delta$ of $\Sigma$, together with a choice of a vertex $m_\Delta(T)$ and a sign $\bs_\Delta(T)$ for each triangle $T \in t(\Delta)$. 
Given a decorated triangulation $\bD$, we define a weighted quiver $Q^{\bD}$ as follows. Let $Q_{m,+}$ and $Q_{m,-}$ be the weighted quivers on a triangle shown in the left and right of \cref{fig:quiver_triangle}, respectively. Here notice that these weighed quivers are not symmetric for the rotations of the triangle, and depend on a chosen special point $m$. By convention, $Q_{m,\pm}$ are considered up to isotopy on $T$ which preserves each edge set-wisely. In particular, we are allowed to move an interior vertex inside the triangle; move and swap the two vertices on a common edge. 

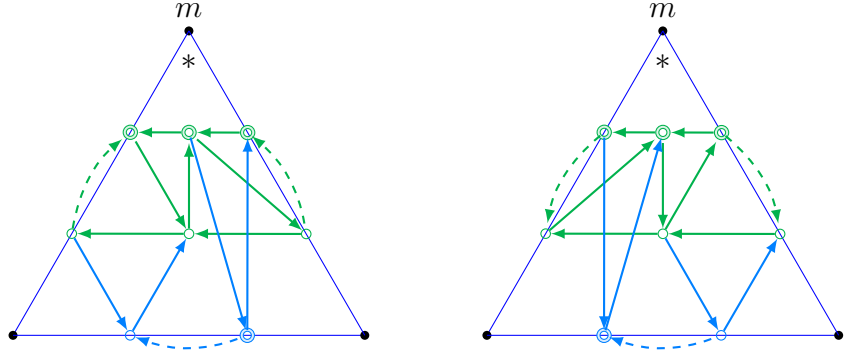
\begin{figure}[ht]
\begin{tikzpicture}[scale=0.9]
    \foreach \i in {90,210,330}
    {
    \markedpt{\i:3};
    \draw[blue] (\i:3) -- (\i+120:3);
    }
\quiverplusC{90:3}{210:3}{330:3};
{\color{mygreen}
\uniarrow{x122}{x121}{dashed,shorten >=4pt, shorten <=2pt,bend left=20}
\uniarrow{x311}{x312}{dashed,shorten >=2pt, shorten <=4pt,bend right=20}
}
\uniarrow{x232}{x231}{myblue,dashed,shorten >=2pt, shorten <=4pt,bend left=20}
\draw(90:3.3) node{$m$};

\begin{scope}[xshift=7cm]
    \foreach \i in {90,210,330}
    {
    \markedpt{\i:3};
    \draw[blue] (\i:3) -- (\i+120:3);
    }
\quiverminusC{90:3}{210:3}{330:3};
{\color{mygreen}
\uniarrow{x121}{x122}{dashed,shorten >=4pt, shorten <=2pt,bend right=20}
\uniarrow{x312}{x311}{dashed,shorten >=4pt, shorten <=2pt,bend left=20}
}
\uniarrow{x232}{x231}{myblue,dashed,shorten >=2pt, shorten <=4pt,bend left=20}
\draw(90:3.3) node{$m$};
\end{scope}
\end{tikzpicture}
    \caption{The quivers $Q_{(\Delta,m,+)}$ (left) and $Q_{(\Delta,m,-)}$ (right) placed on a triangle with a fixed special point $m$. Here the vertices on the opposite side of $m$ and the arrows incident to them are colored blue for visibility.}
    \label{fig:quiver_triangle}
\end{figure}

For each triangle $T \in t(\Delta)$, we draw the quiver $Q_{m_\Delta(T),\bs_\Delta(T)}$, and glue them via the \emph{amalgamation} procedure \cite{FG06} to get a weighted quiver $Q^{\bD}$ drawn on $\Sigma$. Here two vertices on a common interior edge with the same weight are identified; opposite half-arrows cancel together, and parallel half-arrows combine to give a solid arrow. Some examples are shown in \cref{fig:amalgamation}. By convention, $Q^{\bD}$ is considered up to isotopy on $\Sigma$ which preserves each boundary interval set-wisely. 

The vertex set of $Q^{\bD}$ is denoted by $I(\Delta)=I_{\mathfrak{sp}_4}(\Delta)$. Note that $\# I(\Delta) = 2\# e(\Delta)+2 \# t(\Delta)$, which is the number appeared in \cref{def:web-cluster}. 
It admits the following decompositions according to the properties of vertices:
\begin{itemize}
    \item Let $I^{\mathrm{edge}}(\Delta)$ (resp. $I^{\mathrm{tri}}(\Delta)$) denote the subset of vertices on edges (resp. faces of triangles), so that $I(\Delta)=I^{\mathrm{edge}}(\Delta) \sqcup I^{\mathrm{tri}}(\Delta)$.
    \item Let $I(\Delta)_\f \subset I^\mathrm{edge}(\Delta)$ be the subset of the vertices on $\partial\Sigma$, and $I(\Delta)_\uf$ its complement, so that $I(\Delta)=I(\Delta)_\uf \sqcup I(\Delta)_\f$.
    \item Let $I_s(\Delta) \subset I(\Delta)$ be the subset of vertices of weight $s \in \{1,2\}$, so that $I(\Delta)=I_1(\Delta) \sqcup I_2(\Delta)$.
\end{itemize}
Let $i_s(E) \in I(\Delta)$ denote the unique vertex on an edge $E \in e(\Delta)$ with weight $s \in \{1,2\}$. Similarly, let $i_s(T) \in I(\Delta)$ denote the unique vertex on a triangle $T \in t(\Delta)$ with weight $s$.

\begin{figure}[ht]
\begin{tikzpicture}[scale=0.9]
\draw[blue] (0,0) -- (4,0) -- (4,4) -- (0,4) --cycle;
\draw[blue] (4,0) -- (0,4);
\markedpt{0,0};\markedpt{4,0};\markedpt{0,4};\markedpt{4,4};
\quiverplusC{0,4}{0,0}{4,0}
\uniarrow{x122}{x121}{mygreen,dashed,shorten >=4pt, shorten <=2pt,bend left=15}
\uniarrow{x232}{x231}{myblue,dashed,shorten >=2pt, shorten <=4pt,bend left=15}
\quiverminusC{0,4}{4,0}{4,4}
\uniarrow{x312}{x311}{mygreen,dashed,shorten >=2pt, shorten <=4pt,bend left=15}
\uniarrow{x232}{x231}{myblue,dashed,shorten >=4pt, shorten <=2pt,bend left=15}

\draw[blue] (-2.5,1.5) --++(1,0) --++(0,1)--++(-1,0) --cycle;
\draw[blue] (-2.5,2.5) --++(1,-1);
\node[scale=0.9] at (-2.4,2.2) {$\ast$};
\node[scale=0.9] at (-2.2,2.4) {$\ast$};
\node[scale=0.8] at (-2.2,1.8) {$+$};
\node[scale=0.8] at (-1.9,2.2) {$-$};
\draw[thick,double distance=0.15em] (-1,2) -- (-0.5,2);

\begin{scope}[xshift=9cm]
\draw[blue] (0,0) -- (4,0) -- (4,4) -- (0,4) --cycle;
\draw[blue] (4,0) -- (0,4);
\markedpt{0,0};\markedpt{4,0};\markedpt{0,4};\markedpt{4,4};
\quiverplusC{0,4}{0,0}{4,0}
\uniarrow{x122}{x121}{mygreen,dashed,shorten >=4pt, shorten <=2pt,bend left=15}
\uniarrow{x232}{x231}{myblue,dashed,shorten >=2pt, shorten <=4pt,bend left=15}
\uniarrow{x311}{x312}{mygreen,shorten >=4pt, shorten <=4pt,bend right=15}
\quiverminusC{4,4}{0,4}{4,0}
\uniarrow{x122}{x121}{mygreen,dashed,shorten >=2pt, shorten <=4pt,bend left=15}
\uniarrow{x311}{x312}{mygreen,dashed,shorten >=2pt, shorten <=4pt,bend right=15}

\draw[blue] (-2.5,1.5) --++(1,0) --++(0,1)--++(-1,0) --cycle;
\draw[blue] (-2.5,2.5) --++(1,-1);
\node[scale=0.9] at (-2.4,2.2) {$\ast$};
\node[scale=0.9] at (-1.6,2.4) {$\ast$};
\node[scale=0.8] at (-2.2,1.8) {$+$};
\node[scale=0.8] at (-1.9,2.2) {$-$};
\draw[thick,double distance=0.15em] (-1,2) -- (-0.5,2);
\end{scope}
\end{tikzpicture}
    \caption{Two examples of weighted quivers associated with decorated triangulations of a quadrilateral. Here the pair of dashed arrows along the diagonal edge is canceled in the first example, while it is combined to a solid arrow in the second example.}
    \label{fig:amalgamation}
\end{figure}
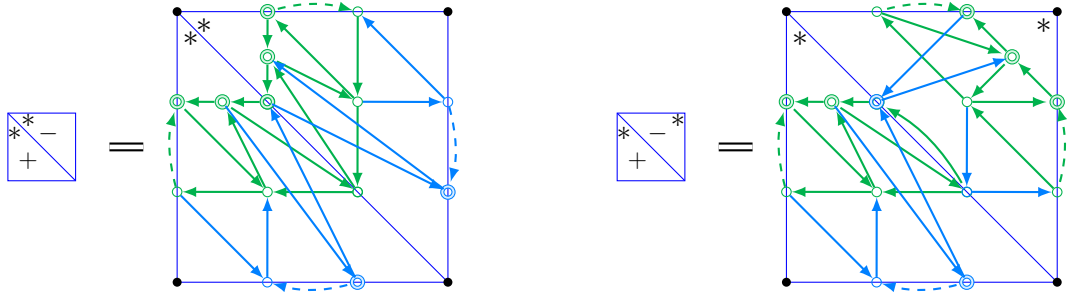

\begin{thm}[cf.~Goncharov--Shen {\cite[Section 8.5]{GS19}}]\label{thm:classical_mutation_equivalence}
For any two decorated triangulations $\bD$, $\bD'$ of $\Sigma$, the two weighted quivers $Q^{\bD}$, $Q^{\bD'}$ are mutation-equivalent. 
\end{thm}

\begin{proof}
Here we give an explicit mutation equivalence. Let us first consider the case $\Sigma=T$. In this case, we have six decorated triangulations. The associated weighted quivers are related as shown in \cref{fig:exch_triangle}. Thus the assertion for the triangle case is proved. We remark here that $\mu_1\mu_2$ and $\mu_2\mu_1$ amount to the rotations of the distinguished vertex, and $\mu_1\mu_2\mu_1=\mu_2\mu_1\mu_2$ amounts to the change of sign.

In the general case, a consequence of the previous paragraph is that the weighted quivers $Q^{\bD}$ and $Q^{\bD'}$ are mutation-equivalent if the underlying triangulations of $\bD$ and $\bD'$ are the same. It remains to consider the flips of ideal triangulations. Again by the previous paragraph, we can choose the decorations as shown in the left-most and right-most pictures in \cref{fig:flip_sequence}. Then it is easily verified that the flip can be realized by $2+4+2$ mutations as shown there. Since any two ideal triangulations are transformed into each other by a finite sequence of flips, the assertion is proved. 
\end{proof}

\begin{figure}
\begin{tikzpicture}[scale=0.8]
\begin{scope}[xshift=3cm]
    \foreach \i in {90,210,330}
    {
    \markedpt{\i:1};
    \draw[blue] (\i:1) -- (\i+120:1);
    }
\node at (0,0) {$+$};
\node at (210:0.8) {$\ast$};
\end{scope}

\begin{scope}[xshift=1.5cm, yshift=1.5*1.732cm]
    \foreach \i in {90,210,330}
    {
    \markedpt{\i:1};
    \draw[blue] (\i:1) -- (\i+120:1);
    }
\node at (0,0) {$-$};
\node at (330:0.8) {$\ast$};
\end{scope}

\begin{scope}[xshift=-1.5cm, yshift=1.5*1.732cm]
    \foreach \i in {90,210,330}
    {
    \markedpt{\i:1};
    \draw[blue] (\i:1) -- (\i+120:1);
    }
\node at (0,0) {$+$};
\node at (90:0.8) {$\ast$};
\end{scope}

\begin{scope}[xshift=-3cm]
    \foreach \i in {90,210,330}
    {
    \markedpt{\i:1};
    \draw[blue] (\i:1) -- (\i+120:1);
    }
\node at (0,0) {$-$};
\node at (210:0.8) {$\ast$};
\end{scope}

\begin{scope}[xshift=-1.5cm, yshift=-1.5*1.732cm]
    \foreach \i in {90,210,330}
    {
    \markedpt{\i:1};
    \draw[blue] (\i:1) -- (\i+120:1);
    }
\node at (0,0) {$+$};
\node at (330:0.8) {$\ast$};
\end{scope}

\begin{scope}[xshift=1.5cm, yshift=-1.5*1.732cm]
    \foreach \i in {90,210,330}
    {
    \markedpt{\i:1};
    \draw[blue] (\i:1) -- (\i+120:1);
    }
\node at (0,0) {$-$};
\node at (90:0.8) {$\ast$};
\end{scope}

\foreach \k in {0,60,120,180,240,300}
\draw[thick] (\k+15:5.5) -- (\k+45:5.5);
\foreach \k in {30,150,270}
\node at (\k:5.8) {$\mu_1$};
\foreach \k in {90,210,-30}
\node at (\k:5.8) {$\mu_2$};
\begin{scope}[xshift=6cm]
    \foreach \i in {90,210,330}
    {
    \markedpt{\i:2};
    \draw[blue] (\i:2) -- (\i+120:2);
    }
\quiverplus{210:2}{330:2}{90:2};
\end{scope}

\begin{scope}[xshift=3cm, yshift=3*1.732cm]
    \foreach \i in {90,210,330}
    {
    \markedpt{\i:2};
    \draw[blue] (\i:2) -- (\i+120:2);
    }
\quiverminus{330:2}{90:2}{210:2};
\end{scope}

\begin{scope}[xshift=-3cm, yshift=3*1.732cm]
    \foreach \i in {90,210,330}
    {
    \markedpt{\i:2};
    \draw[blue] (\i:2) -- (\i+120:2);
    }
\quiverplus{90:2}{210:2}{330:2};
\end{scope}

\begin{scope}[xshift=-6cm]
    \foreach \i in {90,210,330}
    {
    \markedpt{\i:2};
    \draw[blue] (\i:2) -- (\i+120:2);
    }
\quiverminus{210:2}{330:2}{90:2};
\end{scope}

\begin{scope}[xshift=-3cm, yshift=-3*1.732cm]
    \foreach \i in {90,210,330}
    {
    \markedpt{\i:2};
    \draw[blue] (\i:2) -- (\i+120:2);
    }
\quiverplus{330:2}{90:2}{210:2};
\end{scope}

\begin{scope}[xshift=3cm, yshift=-3*1.732cm]
    \foreach \i in {90,210,330}
    {
    \markedpt{\i:2};
    \draw[blue] (\i:2) -- (\i+120:2);
    }
\quiverminus{90:2}{210:2}{330:2};
\end{scope}

\end{tikzpicture}
    \caption{The exchange graph $\Exch_{\sfs(\mathfrak{sp}_4,T)}$ for a triangle $T$. Here $\mu_d$ denotes the mutation at the unique unfrozen vertex with weight $d \in \{1,2\}$.}
    \label{fig:exch_triangle}
\end{figure}
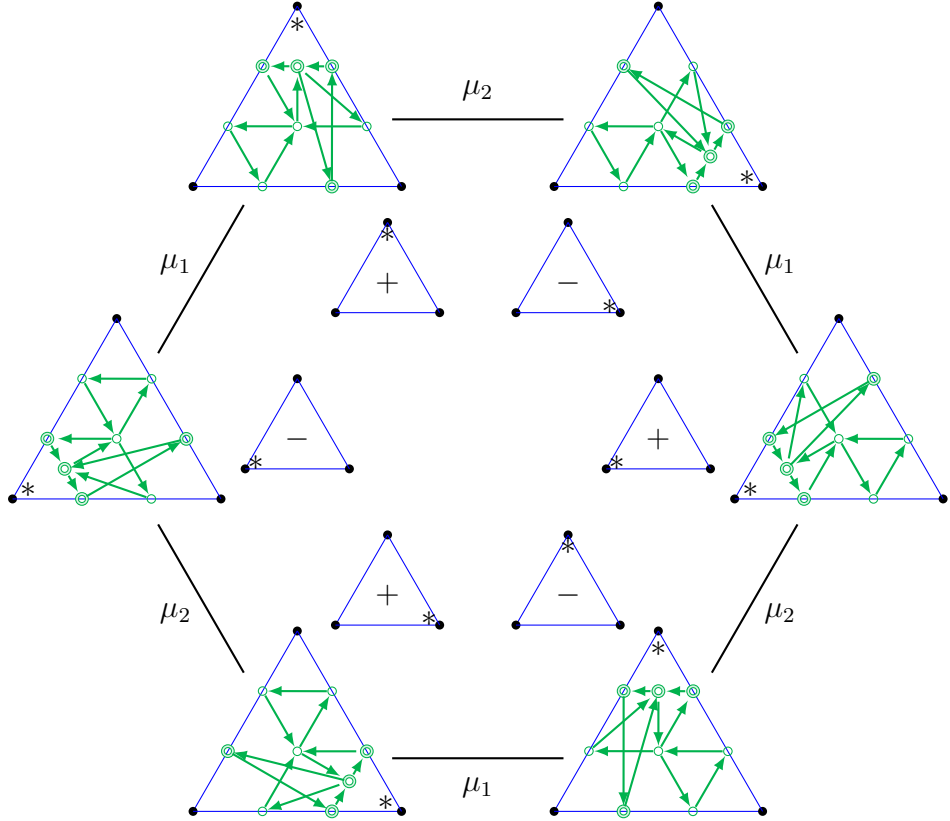

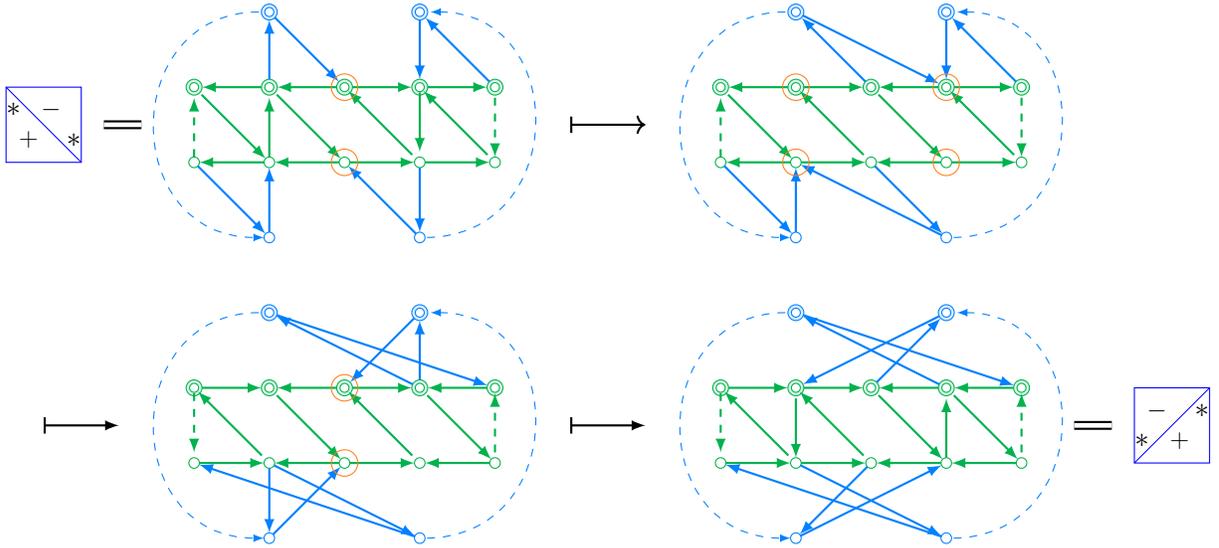
\begin{figure}[ht]
\begin{tikzpicture}
\quiversquare{0,0}{4,0}{4,3}{0,3};
\begin{scope}[>=latex]
{\color{mygreen}
\qsarrow{v21}{v20};
\qsarrow{v22}{v21};
\qsarrow{v22}{v23};
\qsarrow{v23}{v24};
\qarrow{v11}{v10};
\qarrow{v12}{v11};
\qarrow{v12}{v13};
\qarrow{v13}{v14};
\qstarrow{v20}{v11};
\qsharrow{v11}{v21};
\qstarrow{v21}{v12};
\qstarrow{v14}{v23};
\qsharrow{v23}{v13};
\qstarrow{v13}{v22};
\qshdarrow{v10}{v20};
\qstdarrow{v24}{v14};
}
{\color{myblue}
\qarrow{v10}{yl};
\qarrow{yl}{v11};
\qarrow{v13}{yr};
\qarrow{yr}{v12};
\qsarrow{v21}{zl};
\qsarrow{zl}{v22};
\qsarrow{v24}{zr};
\qsarrow{zr}{v23};
\draw[dashed,->,shorten >=2pt,shorten <=4pt] (zl) ..controls ++(-2,0) and ($(yl)+(-2,0)$).. (yl);
\draw[dashed,<-,shorten >=2pt,shorten <=4pt] (zr) ..controls ++(2,0) and ($(yr)+(2,0)$).. (yr);
}
\draw[myorange] (v12) circle(5pt);
\draw[myorange] (v22) circle(5pt);
\end{scope}
\draw[thick,|->] (5,1.5) -- (6,1.5);
\draw[blue] (-2.5,1) --++(1,0) --++(0,1)--++(-1,0) --cycle;
\draw[blue] (-2.5,2) --++(1,-1);
\node[scale=0.9] at (-2.4,1.7) {$\ast$};
\node[scale=0.9] at (-1.6,1.3) {$\ast$};
\node[scale=0.8] at (-2.2,1.3) {$+$};
\node[scale=0.8] at (-1.9,1.7) {$-$};
\draw[thick,double distance=0.15em] (-1.2,1.5) -- (-0.7,1.5);

\begin{scope}[xshift=7cm,>=latex]
\quiversquare{0,0}{4,0}{4,3}{0,3};
{\color{mygreen}
\qsarrow{v21}{v20};
\qsarrow{v21}{v22};
\qsarrow{v23}{v22};
\qsarrow{v23}{v24};
\qarrow{v11}{v10};
\qarrow{v11}{v12};
\qarrow{v13}{v12};
\qarrow{v13}{v14};
\qstarrow{v20}{v11};
\qstarrow{v12}{v21};
\qstarrow{v14}{v23};
\qstarrow{v22}{v13};
\qshdarrow{v10}{v20};
\qstdarrow{v24}{v14};
}
{\color{myblue}
\qarrow{v10}{yl};
\qarrow{yl}{v11};
\qarrow{v12}{yr};
\qarrow{yr}{v11};
\qsarrow{v22}{zl};
\qsarrow{v24}{zr};
\qsarrow{zr}{v23};
\qsarrow{zl}{v23};
\draw[dashed,->,shorten >=2pt,shorten <=4pt] (zl) ..controls ++(-2,0) and ($(yl)+(-2,0)$).. (yl);
\draw[dashed,<-,shorten >=2pt,shorten <=4pt] (zr) ..controls ++(2,0) and ($(yr)+(2,0)$).. (yr);
}
\draw[myorange] (v11) circle(5pt);
\draw[myorange] (v21) circle(5pt);
\draw[myorange] (v13) circle(5pt);
\draw[myorange] (v23) circle(5pt);
\end{scope}

\begin{scope}[yshift=-4cm,>=latex]
\quiversquare{0,0}{4,0}{4,3}{0,3};
{\color{mygreen}
\qsarrow{v20}{v21};
\qsarrow{v22}{v21};
\qsarrow{v22}{v23};
\qsarrow{v24}{v23};
\qarrow{v10}{v11};
\qarrow{v12}{v11};
\qarrow{v12}{v13};
\qarrow{v14}{v13};
\qstarrow{v11}{v20};
\qstarrow{v21}{v12};
\qstarrow{v23}{v14};
\qstarrow{v13}{v22};
\qshdarrow{v20}{v10};
\qstdarrow{v14}{v24};
}
{\color{myblue}
\qarrow{v11}{yl};
\qarrow{yl}{v12};
\qarrow{v11}{yr};
\qarrow{yr}{v10};
\qsarrow{v23}{zr};
\qsarrow{zr}{v22};
\qsarrow{v23}{zl};
\qsarrow{zl}{v24};
\draw[dashed,->,shorten >=2pt,shorten <=4pt] (zl) ..controls ++(-2,0) and ($(yl)+(-2,0)$).. (yl);
\draw[dashed,<-,shorten >=2pt,shorten <=4pt] (zr) ..controls ++(2,0) and ($(yr)+(2,0)$).. (yr);
}
\draw[myorange] (v12) circle(5pt);
\draw[myorange] (v22) circle(5pt);
\draw[thick,|->] (5,1.5) -- (6,1.5);
\draw[thick,|->] (-2,1.5) -- (-1,1.5);
\end{scope}

\begin{scope}[xshift=7cm,yshift=-4cm,>=latex]
\quiversquare{0,0}{4,0}{4,3}{0,3};
{\color{mygreen}
\qsarrow{v20}{v21};
\qsarrow{v21}{v22};
\qsarrow{v23}{v22};
\qsarrow{v24}{v23};
\qarrow{v10}{v11};
\qarrow{v11}{v12};
\qarrow{v13}{v12};
\qarrow{v14}{v13};
\qstarrow{v11}{v20};
\qstarrow{v12}{v21};
\qstarrow{v23}{v14};
\qstarrow{v22}{v13};
\qshdarrow{v20}{v10};
\qstdarrow{v14}{v24};
\qstarrow{v21}{v11};
\qsharrow{v13}{v23};
}
{\color{myblue}
\qarrow{v12}{yl};
\qarrow{yl}{v13};
\qarrow{v11}{yr};
\qarrow{yr}{v10};
\qsarrow{v22}{zr};
\qsarrow{v23}{zl};
\qsarrow{zl}{v24};
\qsarrow{zr}{v21};
\draw[dashed,->,shorten >=2pt,shorten <=4pt] (zl) ..controls ++(-2,0) and ($(yl)+(-2,0)$).. (yl);
\draw[dashed,<-,shorten >=2pt,shorten <=4pt] (zr) ..controls ++(2,0) and ($(yr)+(2,0)$).. (yr);
}
\draw[blue] (5.5,1) --++(1,0) --++(0,1)--++(-1,0) --cycle;
\draw[blue] (6.5,2) --++(-1,-1);
\node[scale=0.9] at (5.6,1.3) {$\ast$};
\node[scale=0.9] at (6.4,1.7) {$\ast$};
\node[scale=0.8] at (5.8,1.7) {$-$};
\node[scale=0.8] at (6.1,1.3) {$+$};
\draw[thick,double distance=0.15em] (5.2,1.5) -- (4.7,1.5);
\end{scope}
\end{tikzpicture}
    \caption{Mutation sequences that realize a flip of triangulation. Here we forget the boundary framing of the weighted quivers so that most parts are drawn planar. The vertices at which we perform mutations are shown in orange circles. Notice that the mutations at two vertices without arrows between them commute with each other.}
    \label{fig:flip_sequence}
\end{figure}
For the construction of a seed $(Q^{\bD},\mathbf{A}^{\bD})$ in the field $\mathcal{K}(\A_{Sp_4,\Sigma})$ of rational functions on the moduli space $\A_{Sp_4,\Sigma}$, see \cite{GS19} or \cref{rem:GS geometry} below. 

By \cref{thm:classical_mutation_equivalence} (and \cref{rem:GS geometry}), the mutation class $\sfs(\mathfrak{sp}_4,\Sigma)$ of the weighted quivers $Q^{\bD}$ (or the seeds $(Q^{\bD},\mathbf{A}^{\bD})$) defines a canonical cluster algebra $\CA_{\sfs(\mathfrak{sp}_4,\Sigma)}$ (inside the field $\mathcal{K}(\A_{Sp_4,\Sigma})$). We will identify a decorated triangulation $\bD$ with the corresponding vertex of the exchange graph $\Exch_{\sfs(\mathfrak{sp}_4,\Sigma)}$. 
Let us simplify the notation as 
\begin{align*}
    \CA_{\mathfrak{sp}_4,\Sigma}:=\CA_{\sfs(\mathfrak{sp}_4,\Sigma)},\quad \UCA_{\mathfrak{sp}_4,\Sigma}:=\UCA_{\sfs(\mathfrak{sp}_4,\Sigma)}, \mbox{ and }\Exch_{\mathfrak{sp}_4,\Sigma}:=\Exch_{\sfs(\mathfrak{sp}_4,\Sigma)}
\end{align*}

A \emph{decorated cell decomposition (with deficiency $1$)} consists of the following data:
\begin{itemize}
    \item An ideal cell decomposition $(\Delta;E)$ of deficiency $1$ with the unique quadrilateral $Q_E$ having $E$ as a diagonal.
    \item A decoration $(m(T),\bs(T))$ for each triangle of $(\Delta;E)$.
    \item A choice of a weighed quiver among those appearing in the sequence \cref{fig:flip_sequence}. 
\end{itemize}
In particular, a decorated triangulation is a decorated cell decomposition. 

\begin{dfn}
Define the \emph{surface subgraph} to be the subgraph $\Exch'_{\mathfrak{sp}_4,\Sigma} \subset \Exch_{\mathfrak{sp}_4,\Sigma}$ such that 
\begin{itemize}
    \item the vertices are the seeds corresponding to the decorated cell decompositions;
    \item the edges are mutations appearing in \cref{fig:exch_triangle} or \cref{fig:flip_sequence}.
\end{itemize}
The exchange matrix $B^{(\omega)}$ for any vertex $\omega \in \Exch'_{\mathfrak{sp}_4,\Sigma}$ is determined by the corresponding weighted quiver $Q^{(\omega)}$. 
\end{dfn}
By the proof of \cref{thm:classical_mutation_equivalence}, the surface subgraph is a connected graph on which the associated seeds are explicitly understood. 

Here is a remark on the labeling. If we fix a labeling $\ell: I(\Delta) \xrightarrow{\sim} \{1,\dots,N\}$, then by \cref{rem:labelings}, 
the part of the exchange graph shown in \cref{fig:exch_triangle} or the part corresponding to \cref{fig:flip_sequence} can be lifted to a subgraph of $\bExch_{\mathfrak{sp}_4,\Sigma}$. We call the pair $(\bD,\ell)$ a \emph{labeled decorated triangulation}. 
In other words, the index set $I(\Delta)$ can be commonly used for the seeds assigned to these parts when we locally discuss rotations in a triangle or a single flip in a quadrilateral. 

\begin{rem}\label{rem:GS geometry}
Here are remarks on the geometry of the moduli space $\A_{Sp_4,\Sigma}$ of decorated twisted $Sp_4$-local systems on $\Sigma$ \cite{FG03,GS19} behind the constructions above. Let $G:=Sp_4$ and $\mathfrak{g}:=\mathfrak{sp}_4$.
\begin{enumerate}
    \item For $\Sigma=T$, the moduli space $\A_{G,T}$ can be identified with the configuration space $\Conf_3 \A_G$ of three decorated flags, where $\A_G:=G/U^+$. Such an identification 
    \begin{align*}
        f_m:\A_{G,T} \xrightarrow{\sim} \Conf_3 \A_G
    \end{align*}
    is determined by a choice of distinguished special point $m$. Moreover, a reduced word $\bs$ of the longest element $w_0 \in W(\mathfrak{g})$ determines a birational chart on $\Conf_3 \A_G$. We have only two reduced words $\bs_+:=(1,2,1,2)$, $\bs_-:=(2,1,2,1)$ in the $\mathfrak{sp}_4$-case. Let $\mathbf{A}_+$, $\mathbf{A}_-$ denote the corresponding charts (tuples of regular functions). Given a datum $(m,\epsilon)$, we get a birational chart
    \begin{align*}
        \mathbf{A}_{m,\epsilon}:=f_m^\ast \mathbf{A}_\epsilon: \A_{G,T} \to (\bC^\ast)^N,
    \end{align*}
    which gives rise to a seed $(Q_{m,\epsilon},\mathbf{A}_{m,\epsilon})$ in the field $\mathcal{K}(\A_{G,T})$ of rational functions. 
    \item Given a decorated triangulation $\bD$ of $\Sigma$, these coordinates collectively give a birational chart $\mathbf{A}^{\bD}$ on $\A_{G,\Sigma}$ via restrictions to the triangle moduli spaces, which gives rise to a seed $(Q^{\bD},\mathbf{A}^{\bD})$ in the field $\mathcal{K}(\A_{G,\Sigma})$. It is known that these seeds are mutation-equivalent to each other \cite[Section 8.5]{GS19}, where the equivalence corresponding to a flip is described by transformations of \emph{double reduced words}. The mutation sequences in \cref{fig:flip_sequence} corresponds to specific transformations between the double reduced words 
    \begin{align*}
        (1,2,1,2,\overline{1},\overline{2},\overline{1},\overline{2})\quad \mbox{and} \quad (\overline{1},\overline{2},\overline{1},\overline{2},1,2,1,2)
    \end{align*}
    of $(w_0,w_0)$. 
    \item We have $\CA_{\mathfrak{sp}_4,\Sigma}=\UCA_{\mathfrak{sp}_4,\Sigma}=\cO(\A_{Sp_4,\Sigma}^\times)$ if $\Sigma$ has at least two special points \cite{IOS}, where $\A_{Sp_4,\Sigma}^\times \subset \A_{Sp_4,\Sigma}$ is the open subspace obtained by imposing the genericity on consecutive pairs of decorated flags, and $\cO(\A_{Sp_4,\Sigma}^\times)$ denotes the $\bC$-algebra of regular functions on that space.
\end{enumerate}
\end{rem}

\paragraph{\textbf{Ensemble grading of cluster algebras.}}
Recall from \cite[Section 2.3]{FG09} that there is a natural $H_\A$-action on the cluster variety $\A_\sfs$. Here $H_\A$ is a split algebraic torus. Let $X^\ast(H_\A)$ be its character lattice. It induces $X^\ast(H_\A)$-gradings of the (upper) cluster algebras $\CA_\sfs \subset \UCA_\sfs=\cO(\A_\sfs)$, which can be lifted to gradings of any quantization $\CA_{\sfs_q} \subset \UCA_{\sfs_q}$ such that each (quantum) cluster variable is homogeneous. We call these gradings the \emph{ensemble gradings}, and denote by $\mathrm{deg}_\mathrm{cl}(A_i^{(\omega)}) \in X^\ast(H_\A)$ the degree of a (quantum) cluster variable $A_i^{(\omega)}$. 
See also \cite[Lemma-Definition 4.7]{IYsl3}. 

In the case $\sfs=\sfs(\mathfrak{g},\Sigma)$ related to the moduli space $\A_{G,\Sigma}$, the torus $H_\A$ covers the product $H^{\bM}$ of Cartan subgroups $H \subset G$, one for each special point. Thus the grading lattice $X^\ast(H_\A)$ contains the direct sum $\bigoplus_{\bM} \mathsf{P}$ of weight lattices. 
For the triangle case $\Sigma=T$, we have the Peter--Weyl isomorphism
\begin{align*}
    \cO(\A_{G,T})=\cO(\Conf_3 \A_G) \cong \bigoplus_{\lambda,\mu,\nu} (V(\lambda)\otimes V(\mu) \otimes V(\nu))^G,
\end{align*}
where $\lambda,\mu,\nu \in \mathsf{P}_+$ are dominant integral weights. The ensemble degree is nothing but the triple $(\lambda,\mu,\nu) \in \mathsf{P}_+^{\oplus 3}$ in this case. The ensemble degrees of the cluster variables arising from Goncharov--Shen's construction can be computed from \cite[Section 6.3]{GS19}, which we show in \cref{fig:grading}. In particular, the ensemble degree of each cluster variable associated with any decorated triangulation belongs to the submonoid $\bigoplus_{\bM} \mathsf{P}_+$ of dominant integral weights. 

Actually, the ensemble grading is one of our basic tools to find appropriate web clusters in the skein algebra. See \cref{lem:comparison_grading}.

\begin{figure}[ht]
\begin{tikzpicture}

    \foreach \i in {90,210,330}
    {
    \markedpt{\i:3};
    \draw[blue] (\i:3) -- (\i+120:3);
    }
\quiverplusC{90:3}{210:3}{330:3};
{\color{mygreen}
\uniarrow{x122}{x121}{dashed,shorten >=4pt, shorten <=2pt,bend left=20}
\uniarrow{x311}{x312}{dashed,shorten >=2pt, shorten <=4pt,bend right=20}
}
\uniarrow{x232}{x231}{myblue,dashed,shorten >=2pt, shorten <=4pt,bend left=20}
\draw(90:3.3) node{$\lambda$};
\draw(210:3.3) node{$\mu$};
\draw(-30:3.3) node{$\nu$};
\node[red,scale=0.9] at (0,-0.5) {$(\varpi_1,\varpi_2,\varpi_1)$};
\node[red,scale=0.9] at (0,2) {$(\varpi_2,\varpi_2,2\varpi_1)$};
\node[red,scale=0.9,right=0.2em] at (x311) {$(\varpi_1,0,\varpi_1)$};
\node[red,scale=0.9,right=0.2em] at (x312) {$(\varpi_2,0,\varpi_2)$};
\node[red,scale=0.9,left=0.2em] at (x121) {$(\varpi_2,\varpi_2,0)$};
\node[red,scale=0.9,left=0.2em] at (x122) {$(\varpi_1,\varpi_1,0)$};
\node[red,scale=0.9,below=0.2em] at ($(x231)+(-0.5,0)$) {$(0,\varpi_1,\varpi_1)$};
\node[red,scale=0.9,below=0.2em] at ($(x232)+(0.5,0)$) {$(0,\varpi_2,\varpi_2)$};

\begin{scope}[xshift=8cm]
    \foreach \i in {90,210,330}
    {
    \markedpt{\i:3};
    \draw[blue] (\i:3) -- (\i+120:3);
    }
\quiverminusC{90:3}{210:3}{330:3};
{\color{mygreen}
\uniarrow{x121}{x122}{dashed,shorten >=4pt, shorten <=2pt,bend right=20}
\uniarrow{x312}{x311}{dashed,shorten >=4pt, shorten <=2pt,bend left=20}
}
\uniarrow{x232}{x231}{myblue,dashed,shorten >=2pt, shorten <=4pt,bend left=20}
\draw(90:3.3) node{$\lambda$};
\draw(210:3.3) node{$\mu$};
\draw(-30:3.3) node{$\nu$};
\node[red,scale=0.9] at (0,-0.5) {$(\varpi_1,\varpi_1,\varpi_2)$};
\node[red,scale=0.9] at (0,2) {$(\varpi_2,2\varpi_1,\varpi_2)$};
\node[red,scale=0.9,right=0.2em] at (x311) {$(\varpi_1,0,\varpi_1)$};
\node[red,scale=0.9,right=0.2em] at (x312) {$(\varpi_2,0,\varpi_2)$};
\node[red,scale=0.9,left=0.2em] at (x121) {$(\varpi_2,\varpi_2,0)$};
\node[red,scale=0.9,left=0.2em] at (x122) {$(\varpi_1,\varpi_1,0)$};
\node[red,scale=0.9,below=0.2em] at ($(x231)+(-0.5,0)$) {$(0,\varpi_2,\varpi_2)$};
\node[red,scale=0.9,below=0.2em] at ($(x232)+(0.5,0)$) {$(0,\varpi_1,\varpi_1)$};
\end{scope}
\end{tikzpicture}
    \caption{The ensemble degrees of cluster variables on a triangle.}
    \label{fig:grading}
\end{figure}
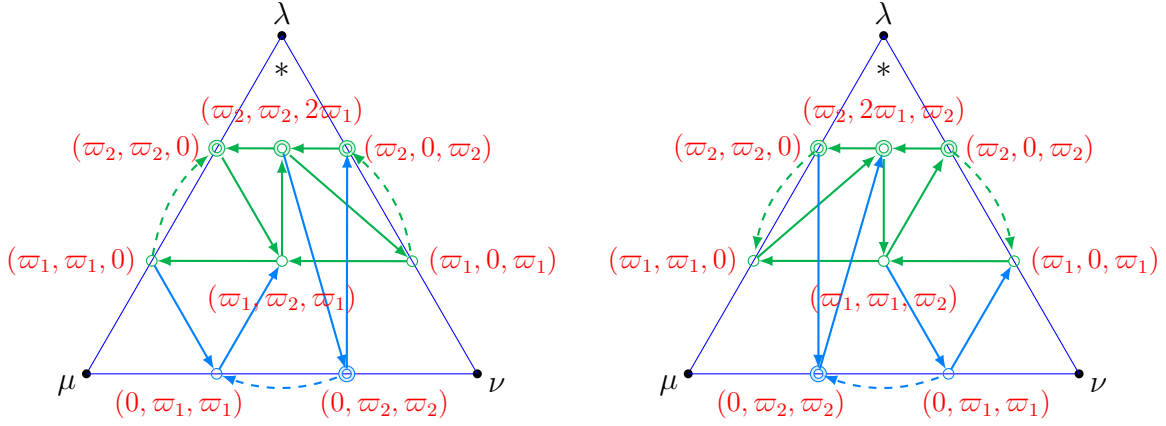
\section{Quantum cluster algebras and skein algebras}\label{sec:realization}

\subsection{Realization of the quantum cluster algebra inside the fraction of the skein algebra}\label{sect:correspondance}
In this section, we construct a mutation class $\sfs_q(\mathfrak{sp}_4,\Sigma)$ of quantum seeds in the skew-field $\mathrm{Frac}\Skein{\Sigma}^q$ of fractions of the skein algebra $\Skein{\Sigma}^q$, which defines a quantum cluster algebra $\CA^q_{\mathfrak{sp}_4,\Sigma} \subset \mathrm{Frac}\Skein{\Sigma}^q$.
In what follows, we identify the quantum parameters as $q=v$. 

For any vertex $\omega \in \Exch'_{\mathfrak{sp}_4,\Sigma}$ of the surface subgraph, we are going to define a quantum seed $(B^{(\omega)},\Pi^{(\omega)},M^{(\omega)})$ in $\mathrm{Frac}\Skein{\Sigma}^q$. The exchange matrix $B^{(\omega)}$ is the one already defined in \cref{subsec:cluster_sp4}. 
In order to define the remaining data, we consider a web cluster  $C_{(\omega)}=\{e^{(\omega)}_i \mid i \in I(\omega)\}$ defined as follows (recall \cref{def:elementary-web}). 
First consider the case where $\omega=\bD$ is a decorated triangulation. 
\begin{itemize}
    \item For each edge $E$ and $s \in \{1,2\}$, let $e_{i_s(E)}^{\bD}$ be the elementary web of weight $s$ given by the type $s$ arc along $E$.
    \item For each triangle $T$ and $s \in \{1,2\}$, let $e_{i_s(T)}^{\bD}$ be the elementary web of weight $s$ given in \cref{fig:web_cluster_triangle}, depending on the decoration data $(m_T,\bs_T)$ on $T$. 
\end{itemize}
See also \cref{fig:exch_tri_web}. 
By \cref{cor:Cweb-T}, these elementary webs indeed form a web cluster in $\mathscr{S}_{\mathfrak{sp}_4,\Sigma}^q$. 
If $\omega$ is a decorated cell decomposition over $(\Delta;E)$, then $C_{(\omega)}=\{e^{(\omega)}_i \mid i \in I(\omega)\}$ is defined as follows.
\begin{itemize}
    \item The assignment of an elementary web $e_i^{(\omega)}$ for a vertex $i \in I(\omega)$ is the same as before unless $i$ is contained in the interior of the unique quadrilateral $Q_E$.
    \item Recalling that the weighted quiver on $Q_E$ is one of those appearing in the mutation sequence in \cref{fig:flip_sequence}, let us assign the elementary webs $e_i^{(\omega)}$ to the six interior vertices as shown in \cref{fig:flip_sequence_web}. Here they are auxiliary labeled so that one of weight $s=1$ (resp. $s=2$) is labeled by an odd (resp. even) number. 
\end{itemize}
Then it is not hard to verify that they form a web cluster. Notice that in each of these cases, the weight of the elementary web $e_i^{(\omega)}$ in the sense of \cref{def:elementary-web} coincides with the weight $s$ of the corresponding vertex of the weighted quiver.

\begin{figure}
\begin{tikzpicture}
    \foreach \i in {90,210,330}
    {
        \markedpt{\i:2};
        \draw[blue] (\i:2) -- (\i+120:2);
    }
    \draw(90:2.3) node{$m_T$};
    \draw(90:1.7) node{$\ast$};
    \draw(0,0) node[scale=1.2]{$+$};
    \begin{scope}[xshift=7cm]
        \foreach \i in {90,210,330}
        {
        \markedpt{\i:2};
        \draw[blue] (\i:2) -- (\i+120:2);
        }
    \draw(90:2.3) node{$m_T$};
    \draw(90:1.7) node{$\ast$};
    \draw(0,0) node[scale=1.2]{$-$};
    \end{scope}
    \node[draw,rectangle,rounded corners=0.7em,inner sep=0.6em,scale=0.9] at (0,-3.5) {\tikz{
	\begin{scope}[yshift=0.5cm]
	\draw[wline] (0,0) -- (210:1);
	\draw[webline] (0,0) -- (90:1);
	\draw[webline] (0,0) -- (-30:1);
	\foreach \i in {90,210,330}
	{
	\markedpt{\i:1};\draw[blue] (\i:1) -- (\i+120:1);
	}
	\node[inner sep=0] at (0,-1) {$e_{i_1(T)}^{\bD}$};
	\end{scope}
	\begin{scope}[xshift=2.5cm,yshift=0.5cm]
	\draw[wline] (90:0.5) -- (90:1);
	\draw[wline] (210:0.5) -- (210:1);
	\draw[webline] (90:0.5) -- (210:0.5);
	\draw[webline] (90:0.5) -- (-30:1); 
	\draw[webline] (210:0.5) -- (-30:1);
	\foreach \i in {90,210,330}
	{
	\markedpt{\i:1};\draw[blue] (\i:1) -- (\i+120:1);
	}
    \node[inner sep=0] at (0,-1) {$e_{i_2(T)}^{\bD}$};
	\end{scope}
	}};
    \node[draw,rectangle,rounded corners=0.7em,inner sep=0.6em,scale=0.9] at (7,-3.5) {\tikz{
	\begin{scope}[yshift=0.5cm]
	\draw[wline] (0,0) -- (-30:1);
	\draw[webline] (0,0) -- (90:1);
	\draw[webline] (0,0) -- (210:1);
	\foreach \i in {90,210,330}
	{
	\markedpt{\i:1};
    \draw[blue] (\i:1) -- (\i+120:1);
	}
	\node[inner sep=0] at (0,-1) {$e_{i_1(T)}^{\bD}$};
	\end{scope}
	\begin{scope}[xshift=2.5cm,yshift=0.5cm]
	\draw[wline] (-30:0.5) -- (-30:1);
	\draw[wline] (90:0.5) -- (90:1);
	\draw[webline] (-30:0.5) -- (90:0.5);
	\draw[webline] (-30:0.5) -- (210:1); 
	\draw[webline] (90:0.5) -- (210:1);
	\foreach \i in {90,210,330}
	{
	\markedpt{\i:1};
    \draw[blue] (\i:1) -- (\i+120:1);
	}
	\node[inner sep=0] at (0,-1) {$e_{i_2(T)}^{\bD}$};
	\end{scope}
	}};
\end{tikzpicture}
    \caption{Assignment of a web cluster to a decorated triangle $(T,m_T,\bs_T)$.}
    \label{fig:web_cluster_triangle}
\end{figure}
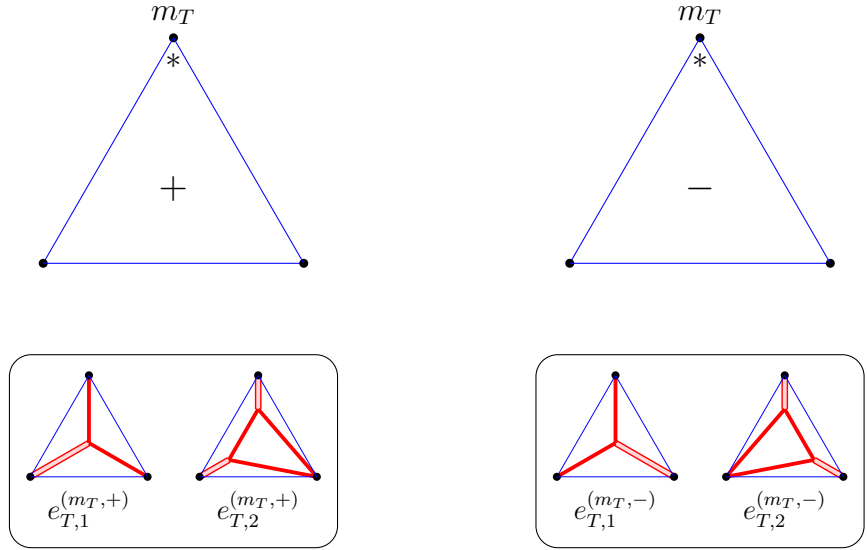

Define the compatibility matrix  $\Pi^{(\omega)}=(\pi^{(\omega)}_{ij})_{i,j \in I(\omega)}$ by
\begin{align*}
    \pi^{(\omega)}_{ij}:= \Pi( e^{(\omega)}_i,e^{(\omega)}_j ).
\end{align*}
Here recall \cref{def:web-exchange}. 
Then $\Pi^{(\omega)}$ is evidently skew-symmetric.

We first verify that their degrees are correct. We identify the monoid $\mathsf{P}_+$ of dominant integral weights with $\bN \times \bN$ via 
\begin{align}\label{eq:lattice_isomorphism}
    \mathsf{P}_+ \xrightarrow{\sim} \bN \times \bN,\quad c_1\varpi_1 + c_2\varpi_2 \mapsto (c_1,c_2).
\end{align}

\begin{lem}\label{lem:comparison_grading}
For any decorated triangulation $\bD$ and $i \in I(\Delta)$, the endpoint degree of the elementary web $e_i^{\bD}$ coincides with the ensemble degree of the corresponding cluster variable $A_i^{\bD}$ via \eqref{eq:lattice_isomorphism}. Namely, we have
\begin{align*}
    \mathbf{\mathrm{deg}}_\mathrm{cl}(A_i^{\bD}) = \mathrm{deg}(e_i^{\bD}).
\end{align*}
\end{lem}

\begin{proof}
It is verified by comparing \cref{fig:grading} with \cref{fig:web_cluster_triangle}. 
\end{proof}
Then the coincidence of the two degrees for the (web) clusters associated to general $\omega \in \Exch'_{\mathfrak{sp}_4,\Sigma}$ will follow from the mutation-equivalence shown in \cref{thm:q-mutation-equivalence} below.

\begin{prop}\label{prop:compatibility}
For any decorated triangulation $\bD=(\Delta,m_\Delta,\bs_\Delta)$ with $\bs_\Delta(T)={+}$ for all $T \in t(\Delta)$, the pair $(B^{\bD},\Pi^{\bD})$ satisfies the compatibility relation
\begin{align*}
    (B^{\bD})^{\mathsf{T}}\Pi^{\bD} = (2D, 0).
\end{align*}
Here $D=\mathrm{diag}(d_i \mid i \in I(\bD))$ is the diagonal matrix encoding the weights of vertices.
\end{prop}

\begin{figure}[ht]
\begin{tikzpicture}[scale=0.9]
\foreach \i in {90,210,330}
    {
    \markedpt{\i:3};
    \draw[blue] (\i:3) -- (\i+120:3);
    }
\quiverplusC{90:3}{210:3}{330:3};
{\color{mygreen}
\uniarrow{x122}{x121}{dashed,shorten >=4pt, shorten <=2pt,bend left=20}
\uniarrow{x311}{x312}{dashed,shorten >=2pt, shorten <=4pt,bend right=20}
}
\uniarrow{x232}{x231}{myblue,dashed,shorten >=2pt, shorten <=4pt,bend left=20}
\draw (x121) node[above left]{$4$};
\draw (x122) node[above left]{$3$};
\draw (x231) node[below]{$5$};
\draw (x232) node[below]{$6$};
\draw (x311) node[above right]{$7$};
\draw (x312) node[above right]{$8$};
\draw (G1) node[below right]{$1$};
\draw (G2) node[above=0.2em]{$2$};

{\begin{scope}[xshift=8cm,yshift=0.5cm]
\begin{scope}
\draw[wline] (0,0) -- (210:1);
	\draw[webline] (0,0) -- (90:1);
	\draw[webline] (0,0) -- (-30:1);
	\foreach \i in {90,210,330}
	{
	\markedpt{\i:1};\draw[blue] (\i:1) -- (\i+120:1);
	}
    \node at (0,-1) {$e_1$};
\end{scope}
\begin{scope}[yshift=2.5cm]
	\draw[wline] (90:0.5) -- (90:1);
	\draw[wline] (210:0.5) -- (210:1);
	\draw[webline] (90:0.5) -- (210:0.5);
	\draw[webline] (90:0.5) -- (-30:1); 
	\draw[webline] (210:0.5) -- (-30:1);
	\foreach \i in {90,210,330}
	{
	\markedpt{\i:1};\draw[blue] (\i:1) -- (\i+120:1);
	}
	\node at (0,-1) {$e_2$};
\end{scope}
	
\begin{scope}[xshift=-2.5cm]
	\draw[webline] (210:1) to[bend right=15] (90:1);
	\foreach \i in {90,210,330}
	{
	\markedpt{\i:1};\draw[blue] (\i:1) -- (\i+120:1);
	}
    \node at (0,-1) {$e_3$};
\end{scope}
	
\begin{scope}[xshift=-2.5cm,yshift=2.5cm]
	\draw[wline] (210:1) to[bend right=15] (90:1);
	\foreach \i in {90,210,330}
	{
	\markedpt{\i:1};\draw[blue] (\i:1) -- (\i+120:1);
	}
    \node at (0,-1) {$e_4$};
\end{scope}
	
\begin{scope}[xshift=2.5cm]
	\draw[webline] (90:1) to[bend right=15] (-30:1);
	\foreach \i in {90,210,330}
	{
	\markedpt{\i:1};\draw[blue] (\i:1) -- (\i+120:1);
	}
    \node at (0,-1) {$e_7$};
\end{scope}
	
\begin{scope}[xshift=2.5cm,yshift=2.5cm]
	\draw[wline] (90:1) to[bend right=15] (-30:1);
	\foreach \i in {90,210,330}
	{
	\markedpt{\i:1};\draw[blue] (\i:1) -- (\i+120:1);
	}
    \node at (0,-1) {$e_8$};
\end{scope}
	
\begin{scope}[xshift=1.25cm,yshift=-2.5cm]
	\draw[wline] (-30:1) to[bend right=15] (210:1);
	\foreach \i in {90,210,330}
	{
	\markedpt{\i:1};\draw[blue] (\i:1) -- (\i+120:1);
	}
    \node at (0,-1) {$e_6$};
\end{scope}
	
\begin{scope}[xshift=-1.25cm,yshift=-2.5cm]
	\draw[webline] (-30:1) to[bend right=15] (210:1);
	\foreach \i in {90,210,330}
	{
	\markedpt{\i:1};\draw[blue] (\i:1) -- (\i+120:1);
	}
    \node at (0,-1) {$e_5$};
	\end{scope}
\end{scope}}
\end{tikzpicture}
    \caption{Labeling of quiver vertices and the corresponding elementary webs on a triangle.}
    \label{fig:labeling_triangle_case}
\end{figure}
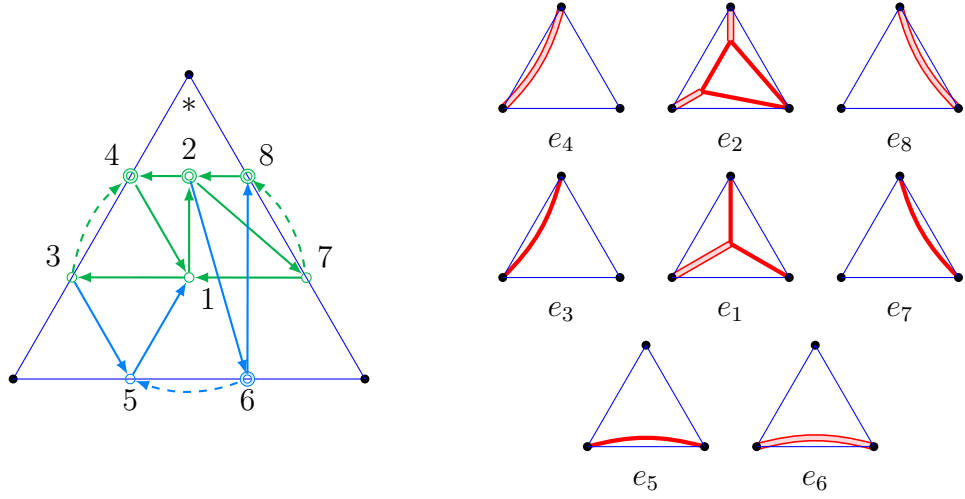

\begin{figure}[ht]
\begin{tikzpicture}[scale=0.9]
\draw[blue] (0,0) -- (4,0) -- (4,4) -- (0,4) --cycle;
\draw[blue] (4,0) -- (0,4);
\foreach \i in {0,4} \foreach \j in {0,4} \markedpt{\i,\j};
{\color{mygreen}
\begin{scope}[scale=4/3]
\dnode{1,0}{mygreen} node[black,below=0.2em]{$10$};
\dnode{0,2}{mygreen} node[black,left=0.2em]{$8$};
\dnode{3,1}{mygreen} node[black,right=0.2em]{$12$};
\dnode{2,3}{mygreen} node[black,above=0.2em]{$14$};
\dnode{1,2}{mygreen} node[black,right=0.2em]{$4$};
\dnode{0.5,1.5}{mygreen} node[black,below=0.2em]{$2$};
\dnode{1.5,2.5}{mygreen} node[black,right=0.2em]{$6$};
\draw(2,0) circle(2pt) node[black,below=0.2em]{$9$};
\draw(0,1) circle(2pt) node[black,left=0.2em]{$7$};
\draw(3,2) circle(2pt) node[black,right=0.2em]{$11$};
\draw(1,3) circle(2pt) node[black,above=0.2em]{$13$};
\draw(2,1) circle(2pt) node[black,left=0.2em]{$3$};
\draw(1.5,0.5) circle(2pt) node[black,left=0.2em]{$1$};
\draw(2.5,1.5) circle(2pt) node[black,above=0.2em]{$5$};
\end{scope}
}
\begin{scope}[xshift=6cm]
\draw[webline] (0,1.5) to (1.5,0);
\draw[blue] (0,0) -- (1.5,0) -- (1.5,1.5) -- (0,1.5) --cycle;
\foreach \i in {0,1.5} \foreach \j in {0,1.5} \markedpt{\i,\j};
\node at (0.75,-0.5) {$e_3$};
\end{scope}

\begin{scope}[xshift=6cm,yshift=2.5cm]
\draw[wline] (0,1.5) -- (1.5,0);
\draw[blue] (0,0) -- (1.5,0) -- (1.5,1.5) -- (0,1.5) --cycle;
\foreach \i in {0,1.5} \foreach \j in {0,1.5} \markedpt{\i,\j};
\node at (0.75,-0.5) {$e_4$};
\end{scope}

\begin{scope}[xshift=8.5cm]
\draw[webline] (0,0) to[bend right=15] (0,1.5);
\draw[blue] (0,0) -- (1.5,0) -- (1.5,1.5) -- (0,1.5) --cycle;
\foreach \i in {0,1.5} \foreach \j in {0,1.5} \markedpt{\i,\j};
\node at (0.75,-0.5) {$e_7$};
\end{scope}

\begin{scope}[xshift=8.5cm,yshift=2.5cm]
\draw[wline] (0,0) to[bend right=15] (0,1.5);
\draw[blue] (0,0) -- (1.5,0) -- (1.5,1.5) -- (0,1.5) --cycle;
\foreach \i in {0,1.5} \foreach \j in {0,1.5} \markedpt{\i,\j};
\node at (0.75,-0.5) {$e_8$};
\end{scope}

\begin{scope}[xshift=11cm]
\draw[webline] (1.5,0) to[bend right=15] (0,0);
\draw[blue] (0,0) -- (1.5,0) -- (1.5,1.5) -- (0,1.5) --cycle;
\foreach \i in {0,1.5} \foreach \j in {0,1.5} \markedpt{\i,\j};
\node at (0.75,-0.5) {$e_9$};
\end{scope}

\begin{scope}[xshift=11cm,yshift=2.5cm]
\draw[wline] (1.5,0) to[bend right=15] (0,0);
\draw[blue] (0,0) -- (1.5,0) -- (1.5,1.5) -- (0,1.5) --cycle;
\foreach \i in {0,1.5} \foreach \j in {0,1.5} \markedpt{\i,\j};
\node at (0.75,-0.5) {$e_{10}$};
\end{scope}

\begin{scope}[xshift=13.5cm]
\draw[webline] (1.5,0) to[bend left=15] (1.5,1.5);
\draw[blue] (0,0) -- (1.5,0) -- (1.5,1.5) -- (0,1.5) --cycle;
\foreach \i in {0,1.5} \foreach \j in {0,1.5} \markedpt{\i,\j};
\node at (0.75,-0.5) {$e_{11}$};
\end{scope}

\begin{scope}[xshift=13.5cm,yshift=2.5cm]
\draw[wline] (1.5,0) to[bend left=15] (1.5,1.5);
\draw[blue] (0,0) -- (1.5,0) -- (1.5,1.5) -- (0,1.5) --cycle;
\foreach \i in {0,1.5} \foreach \j in {0,1.5} \markedpt{\i,\j};
\node at (0.75,-0.5) {$e_{12}$};
\end{scope}

\begin{scope}[xshift=16cm]
\draw[webline] (0,1.5) to[bend right=15] (1.5,1.5);
\draw[blue] (0,0) -- (1.5,0) -- (1.5,1.5) -- (0,1.5) --cycle;
\foreach \i in {0,1.5} \foreach \j in {0,1.5} \markedpt{\i,\j};
\node at (0.75,-0.5) {$e_{13}$};
\end{scope}

\begin{scope}[xshift=16cm,yshift=2.5cm]
\draw[wline] (0,1.5) to[bend right=15] (1.5,1.5);
\draw[blue] (0,0) -- (1.5,0) -- (1.5,1.5) -- (0,1.5) --cycle;
\foreach \i in {0,1.5} \foreach \j in {0,1.5} \markedpt{\i,\j};
\node at (0.75,-0.5) {$e_{14}$};
\end{scope}
\end{tikzpicture}
    \caption{Labeling of quiver vertices and the corresponding elementary webs. }
    \label{fig:labeling_square_case}
\end{figure}

\begin{proof}
During the proof, we fix a decorated triangulation $\bD$ and omit the superscript $\bD$. 
Let $\varepsilon:=B^{\mathsf{T}}$ denote the Fock--Goncharov exchange matrix associated with $\bD$. 
For $i \in I(\bD)_\uf$ and $j \in I(\bD)$, we are going to compute $(\varepsilon\Pi)_{ij}=\sum_{k \in I(\bD)} \varepsilon_{ik}\pi_{kj}$. 
For a fixed $i \in I(\bD)_\uf$, let us consider the Weyl-ordered Laurent monomial
\begin{align*}
    X_i:= \left[\prod_{k \in I(\bD)} e_k^{\varepsilon_{ik}}\right].
\end{align*}
Then we have
\begin{align*}
    (\varepsilon\Pi)_{ij}=\sum_{k \in I(\bD)} \varepsilon_{ik}\pi_{kj} = \sum_{k \in I(\bD)} \varepsilon_{ik}\Pi(e_k,e_j) = \Pi(X_i, e_j).
\end{align*}
From the very definition, we see that the ensemble degree of $X_i$ is $0$. Then by \cref{lem:comparison_grading}, its endpoint degree is also $0$. 
On the other hand, observe that 
\begin{itemize}
    \item $X_i$ only involves the elementary webs on the triangle $T$ if $i=i_s(T) \in I^\mathrm{tri}(\bD)$, and 
    \item $X_i$ only involves the elementary webs on the quadrilateral $Q_E$ with diagonal $E$ if $i=i_s(E) \in I^\mathrm{edge}(\bD)$.
\end{itemize}
From these facts, we immediately get $(\varepsilon\Pi)_{ij}=\Pi(X_i,e_j)=0$ for the vertices $j$ not lying in the interior of $T$ if $i=i_s(T)$, and for $j$ not lying in the interior of $Q_E$ if $i=i_s(E)$. Thus the problem is reduced to the triangle and quadrilateral cases.

\begin{description}
\item[Triangle case]
Let us label the vertices of the weighted quiver and the corresponding elementary webs as shown in \cref{fig:labeling_triangle_case}. Then we have
\begin{align*}
    \pi_{21}&=1, \\
    \pi_{13}&=\pi_{15}=\pi_{17}=\pi_{18}=0, \quad \pi_{14}=1, \quad \pi_{16}=-1, \\
    \pi_{23}&=\pi_{24}=\pi_{26}=\pi_{28}=0, \quad \pi_{25}=1, \quad \pi_{27}=-1, \\
    X_1 &= \left[ \frac{e_2 e_3}{e_4 e_5 e_7} \right], 
    X_2 = \left[ \frac{e_4 e_6 e_7^2}{e_1^2 e_8} \right].
\end{align*}
Then one can easily verify that $\Pi(X_1,e_2)=\Pi(X_2,e_1)=0$, $\Pi(X_1,e_1)=2$ and $\Pi(X_2,e_2)=4$. Thus we have $(\varepsilon\Pi)_{ij}=2d_i\delta_{ij}$ for $i=i_s(T)$. 
\item[Quadrilateral case]
Up to symmetry, we need to consider the following five cases:
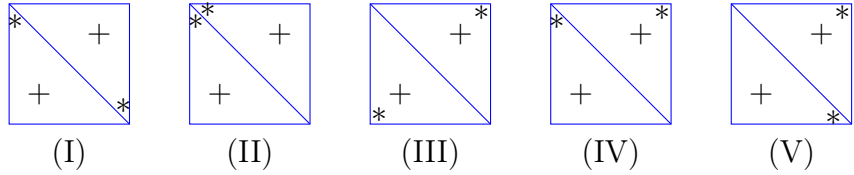
\begin{align*}
\begin{tikzpicture}[scale=0.8]
\draw[blue] (0,0) -- (2,0) -- (2,2) -- (0,2) --cycle;
\draw[blue] (2,0) -- (0,2);
\node at (0.5,0.5) {$+$};
\node at (1.5,1.5) {$+$};
\node at (0.1,1.7) {$\ast$};
\node at (1.9,0.3) {$\ast$};
\node at (1,-0.5) {(I)};
\begin{scope}[xshift=3cm]
\draw[blue] (0,0) -- (2,0) -- (2,2) -- (0,2) --cycle;
\draw[blue] (2,0) -- (0,2);
\node at (0.5,0.5) {$+$};
\node at (1.5,1.5) {$+$};
\node at (0.1,1.7) {$\ast$};
\node at (0.3,1.9) {$\ast$};
\node at (1,-0.5) {(II)};
\end{scope}
\begin{scope}[xshift=6cm]
\draw[blue] (0,0) -- (2,0) -- (2,2) -- (0,2) --cycle;
\draw[blue] (2,0) -- (0,2);
\node at (0.5,0.5) {$+$};
\node at (1.5,1.5) {$+$};
\node at (0.15,0.15) {$\ast$};
\node at (2-0.15,2-0.15) {$\ast$};
\node at (1,-0.5) {(III)};
\end{scope}
\begin{scope}[xshift=9cm]
\draw[blue] (0,0) -- (2,0) -- (2,2) -- (0,2) --cycle;
\draw[blue] (2,0) -- (0,2);
\node at (0.5,0.5) {$+$};
\node at (1.5,1.5) {$+$};
\node at (0.1,1.7) {$\ast$};
\node at (2-0.15,2-0.15) {$\ast$};
\node at (1,-0.5) {(IV)};
\end{scope}
\begin{scope}[xshift=12cm]
\draw[blue] (0,0) -- (2,0) -- (2,2) -- (0,2) --cycle;
\draw[blue] (2,0) -- (0,2);
\node at (0.5,0.5) {$+$};
\node at (1.5,1.5) {$+$};
\node at (1.7,0.1) {$\ast$};
\node at (2-0.15,2-0.15) {$\ast$};
\node at (1,-0.5) {(V)};
\end{scope}
\end{tikzpicture}
\end{align*}
We use the common labeling as shown in \cref{fig:labeling_square_case}. The structure of the quiver depends on the cases (I)--(V), as well as the elementary webs $e_1,e_2,e_5,e_6$. In particular, the Laurent monomials $X_3,X_4$ depend on the cases. They are listed as follows:
    \begin{align*}
        \mbox{(I):}\quad 
        e_1 &= \tikz[scale=0.9,baseline=0.5cm]{\smallsq{\triv{B}{D}{A}}}\ , \quad 
        e_2 = \tikz[scale=0.9,baseline=0.5cm]{\smallsq{\draw[wline] (0,0) -- (0.15*1.2,0.25*1.2);
        \draw[wline] (0,1.2) -- (0.15*1.2,0.75*1.2);
        \draw[webline] (0.15*1.2,0.75*1.2) --(0.15*1.2,0.25*1.2);
        \draw[webline] (1.2,0) --(0.15*1.2,0.25*1.2);
        \draw[webline] (1.2,0) --(0.15*1.2,0.75*1.2);}}\ , \quad 
        e_5 = \tikz[scale=0.9,baseline=-0.5cm,rotate=180]{\smallsq{\triv{B}{D}{A}}}\ , \quad 
        e_6 = \tikz[scale=0.9,baseline=-0.5cm,rotate=180]{\smallsq{\draw[wline] (0,0) -- (0.15*1.2,0.25*1.2);
        \draw[wline] (0,1.2) -- (0.15*1.2,0.75*1.2);
        \draw[webline] (0.15*1.2,0.75*1.2) --(0.15*1.2,0.25*1.2);
        \draw[webline] (1.2,0) --(0.15*1.2,0.25*1.2);
        \draw[webline] (1.2,0) --(0.15*1.2,0.75*1.2);}}\ ,\\
        X_3 &= \left[ \frac{e_1 e_4 e_5}{e_2 e_6} \right], \quad
        X_4 = \left[ \frac{e_2 e_6}{e_3^2 e_{10} e_{14}} \right]. \\[2mm]
        \mbox{(II):}\quad 
        e_1 &= \tikz[scale=0.9,baseline=0.5cm]{\smallsq{\triv{B}{D}{A}}}\ , \quad 
        e_2 = \tikz[scale=0.9,baseline=0.5cm]{\smallsq{\draw[wline] (0,0) -- (0.15*1.2,0.25*1.2);
        \draw[wline] (0,1.2) -- (0.15*1.2,0.75*1.2);
        \draw[webline] (0.15*1.2,0.75*1.2) --(0.15*1.2,0.25*1.2);
        \draw[webline] (1.2,0) --(0.15*1.2,0.25*1.2);
        \draw[webline] (1.2,0) --(0.15*1.2,0.75*1.2);}}\ , \quad 
        e_5 = \tikz[scale=0.9,baseline=0.5cm]{\smallsq{\triv{B}{C}{D}}}\ , \quad 
        e_6 = \tikz[scale=0.9,baseline=0.5cm]{\smallsq{\draw[wline] (0,1.2) -- (0.5*1.2,0.75*1.2);
        \draw[wline] (1.2,0) -- (0.75*1.2,0.5*1.2);
        \draw[webline] (0.5*1.2,0.75*1.2) -- (0.75*1.2,0.5*1.2);
        \draw[webline] (1.2,1.2) -- (0.75*1.2,0.5*1.2);
        \draw[webline] (0.5*1.2,0.75*1.2) -- (1.2,1.2);}}\ ,\\
        X_3 &= \left[ \frac{e_1 e_4 e_{11}}{e_2 e_5} \right], \quad
        X_4 = \left[ \frac{e_2 e_5^2}{e_3^2 e_{6} e_{10}} \right]. \\[2mm]
        \mbox{(III):}\quad 
        e_1 &= \tikz[scale=0.9,baseline=0.5cm]{\smallsq{\triv{B}{A}{D}}}\ , \quad 
        e_2 = \tikz[scale=0.9,baseline=0.5cm,rotate=90]{\smallsq{\draw[wline] (0,0) -- (0.15*1.2,0.25*1.2);
        \draw[wline] (0,1.2) -- (0.15*1.2,0.75*1.2);
        \draw[webline] (0.15*1.2,0.75*1.2) --(0.15*1.2,0.25*1.2);
        \draw[webline] (1.2,1.2) --(0.15*1.2,0.25*1.2);
        \draw[webline] (1.2,1.2) --(0.15*1.2,0.75*1.2);}}\ , \quad 
        e_5 = \tikz[scale=0.9,baseline=0.5cm]{\smallsq{\triv{D}{C}{B}}}\ , \quad 
        e_6 = \tikz[scale=0.9,baseline=-0.5cm,rotate=-90]{\smallsq{\draw[wline] (0,0) -- (0.15*1.2,0.25*1.2);
        \draw[wline] (0,1.2) -- (0.15*1.2,0.75*1.2);
        \draw[webline] (0.15*1.2,0.75*1.2) --(0.15*1.2,0.25*1.2);
        \draw[webline] (1.2,1.2) --(0.15*1.2,0.25*1.2);
        \draw[webline] (1.2,1.2) --(0.15*1.2,0.75*1.2);}}\ ,\\
        X_3 &= \left[ \frac{e_1 e_5}{e_4 e_9 e_{13}} \right], \quad
        X_4 = \left[ \frac{e_3^2 e_8 e_{12}}{e_2 e_{6}} \right]. \\[2mm]
        \mbox{(IV):}\quad 
        e_1 &= \tikz[scale=0.9,baseline=0.5cm]{\smallsq{\triv{D}{A}{B}}}\ , \quad 
        e_2 = \tikz[scale=0.9,baseline=-0.5cm,rotate=180]{\smallsq{\draw[wline] (0,1.2) -- (0.5*1.2,0.75*1.2);
        \draw[wline] (1.2,0) -- (0.75*1.2,0.5*1.2);
        \draw[webline] (0.5*1.2,0.75*1.2) -- (0.75*1.2,0.5*1.2);
        \draw[webline] (1.2,1.2) -- (0.75*1.2,0.5*1.2);
        \draw[webline] (0.5*1.2,0.75*1.2) -- (1.2,1.2);}}\ , \quad 
        e_5 = \tikz[scale=0.9,baseline=0.5cm]{\smallsq{\triv{D}{C}{B}}}\ , \quad 
        e_6 = \tikz[scale=0.9,baseline=-0.5cm,rotate=-90]{\smallsq{\draw[wline] (0,0) -- (0.15*1.2,0.25*1.2);
        \draw[wline] (0,1.2) -- (0.15*1.2,0.75*1.2);
        \draw[webline] (0.15*1.2,0.75*1.2) --(0.15*1.2,0.25*1.2);
        \draw[webline] (1.2,1.2) --(0.15*1.2,0.25*1.2);
        \draw[webline] (1.2,1.2) --(0.15*1.2,0.75*1.2);}}\ ,\\
        X_3 &= \left[ \frac{e_5 e_7}{e_1 e_{13}} \right], \quad
        X_4 = \left[ \frac{e_1^2 e_{12}}{e_2 e_{6}} \right]. \\[2mm]
        \mbox{(V):}\quad 
        e_1 &= \tikz[scale=0.9,baseline=0.5cm]{\smallsq{\triv{B}{D}{A}}}\ , \quad 
        e_2 = \tikz[scale=0.9,baseline=0.5cm]{\smallsq{\draw[wline] (0,0) -- (0.15*1.2,0.25*1.2);
        \draw[wline] (0,1.2) -- (0.15*1.2,0.75*1.2);
        \draw[webline] (0.15*1.2,0.75*1.2) --(0.15*1.2,0.25*1.2);
        \draw[webline] (1.2,0) --(0.15*1.2,0.25*1.2);
        \draw[webline] (1.2,0) --(0.15*1.2,0.75*1.2);}}\ , \quad 
        e_5 = \tikz[scale=0.9,baseline=0.5cm]{\smallsq{\triv{D}{C}{B}}}\ , \quad 
        e_6 = \tikz[scale=0.9,baseline=-0.5cm,rotate=-90]{\smallsq{\draw[wline] (0,0) -- (0.15*1.2,0.25*1.2);
        \draw[wline] (0,1.2) -- (0.15*1.2,0.75*1.2);
        \draw[webline] (0.15*1.2,0.75*1.2) --(0.15*1.2,0.25*1.2);
        \draw[webline] (1.2,1.2) --(0.15*1.2,0.25*1.2);
        \draw[webline] (1.2,1.2) --(0.15*1.2,0.75*1.2);}}\ ,\\
        X_3 &= \left[ \frac{e_1 e_5}{e_2 e_{13}} \right], \quad
        X_4 = \left[ \frac{e_2 e_{12}}{e_6 e_{10}} \right]. 
    \end{align*}
We have $\pi_{21}=\pi_{65}=1$ in each case. 
The other $q$-commutation relations among the variables appearing above can be easily computed from the clasped skein relations (\cref{def:bdry-skeinrel}). For instance, $\pi_{34}=0$, $\pi_{39}=1$, $\pi_{4,10}=2$. Then it is not hard to verify that $\Pi(X_i,e_j)=2d_i \delta_{ij}$ holds for $i=3,4$ and $j=1,2,3,4,5,6$ in each case. 
\end{description}
\end{proof}

The check of the compatibility relation for a general  $\omega \in \Exch'_{\mathfrak{sp}_4,\Sigma}$ is postponed until the proof of \cref{thm:q-mutation-equivalence} below. 
For any vertex $\omega \in \Exch'_{\mathfrak{sp}_4,\Sigma}$, define a toric frame
\begin{align*}
    M^{(\omega)}: \accentset{\circ}{\Lambda}{}^{(\omega)} \to \mathrm{Frac}\Skein{\Sigma}^q
\end{align*}
by the formula
\begin{align*}
    M^{(\omega)}\bigg(\sum_{i \in I} x_i\sff^{(\omega)}_i\bigg):=\left[\prod_{i \in I}(e^{(\omega)}_{i})^{x_i}\right],
\end{align*}
where we use the Weyl ordering (\cref{def:simul-crossing}). Observe that this is the same extension rule as \eqref{eq:extension_toric_frame}. As a consequence, we get: 

\begin{lem}\label{prop:q-seed condition}
For any vertex $\omega \in \Exch'_{\mathfrak{sp}_4,\Sigma}$, the pair $(\Pi^{(\omega)},M^{(\omega)})$ satisfies
\begin{align*}
    M^{(\omega)}(\alpha)M^{(\omega)}(\beta) = q^{\Pi^{(\omega)}(\alpha,\beta)/2} M^{(\omega)}(\alpha+\beta)
\end{align*}
for $\alpha,\beta \in \accentset{\circ}{\Lambda}{}^{(\omega)}$. 
\end{lem}

\begin{thm}\label{thm:q-mutation-equivalence}
For any vertex $\omega \in \Exch'_{\mathfrak{sp}_4,\Sigma}$, the triple  $(B^{(\omega)},\Pi^{(\omega)},M^{(\omega)})$ is a quantum seed in $\mathrm{Frac}\Skein{\Sigma}^q$. These quantum seeds are mutation-equivalent to each other. 
\end{thm}

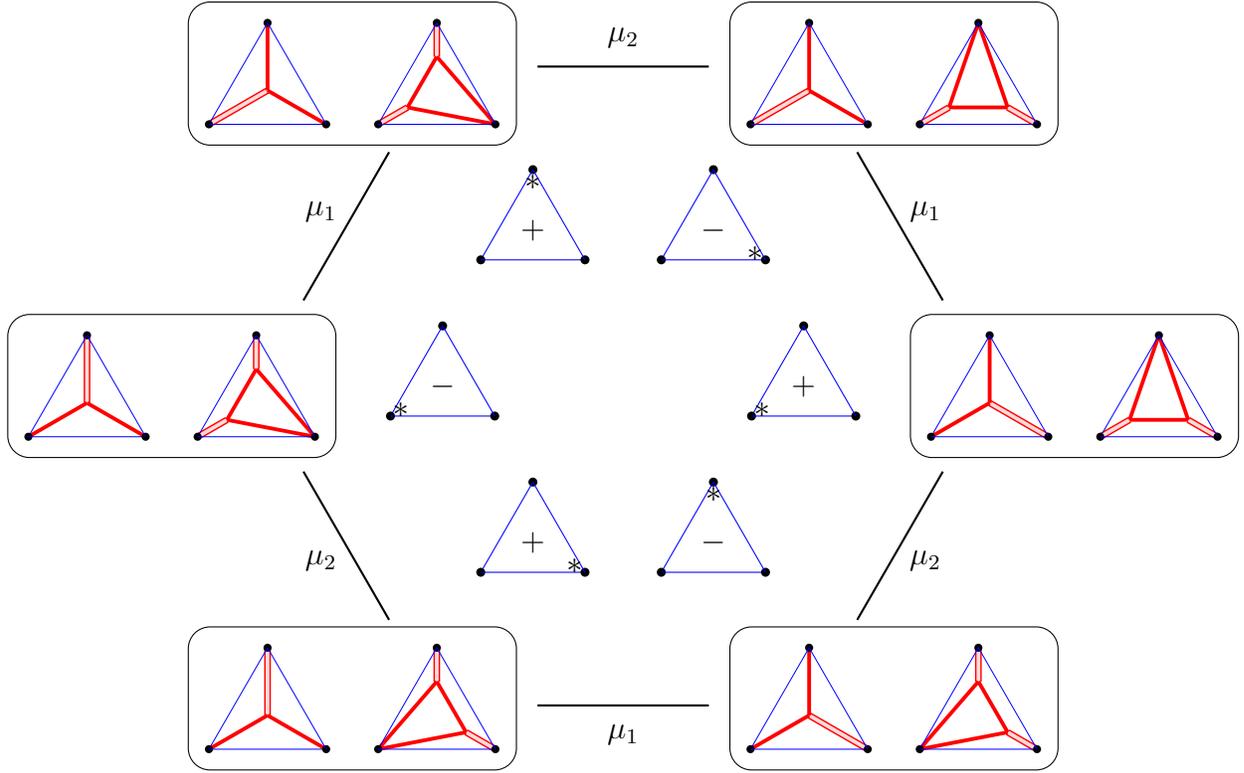
\begin{figure}[ht]
\begin{tikzpicture}[scale=0.8]
\begin{scope}[xshift=3cm]
    \foreach \i in {90,210,330}
    {
    \markedpt{\i:1};
    \draw[blue] (\i:1) -- (\i+120:1);
    }
    \node at (0,0) {$+$};
    \node at (210:0.8) {$\ast$};
\end{scope}
\begin{scope}[xshift=1.5cm, yshift=1.5*1.732cm]
    \foreach \i in {90,210,330}
    {
    \markedpt{\i:1};
    \draw[blue] (\i:1) -- (\i+120:1);
    }
    \node at (0,0) {$-$};
    \node at (330:0.8) {$\ast$};
\end{scope}
\begin{scope}[xshift=-1.5cm, yshift=1.5*1.732cm]
    \foreach \i in {90,210,330}
    {
    \markedpt{\i:1};
    \draw[blue] (\i:1) -- (\i+120:1);
    }
    \node at (0,0) {$+$};
    \node at (90:0.8) {$\ast$};
\end{scope}
\begin{scope}[xshift=-3cm]
    \foreach \i in {90,210,330}
    {
    \markedpt{\i:1};
    \draw[blue] (\i:1) -- (\i+120:1);
    }
    \node at (0,0) {$-$};
    \node at (210:0.8) {$\ast$};
\end{scope}
\begin{scope}[xshift=-1.5cm, yshift=-1.5*1.732cm]
    \foreach \i in {90,210,330}
    {
    \markedpt{\i:1};
    \draw[blue] (\i:1) -- (\i+120:1);
    }
    \node at (0,0) {$+$};
    \node at (330:0.8) {$\ast$};
\end{scope}
\begin{scope}[xshift=1.5cm, yshift=-1.5*1.732cm]
    \foreach \i in {90,210,330}
    {
    \markedpt{\i:1};
    \draw[blue] (\i:1) -- (\i+120:1);
    }
    \node at (0,0) {$-$};
    \node at (90:0.8) {$\ast$};
\end{scope}
\foreach \k in {0,60,120,180,240,300}
\draw[thick] (\k+15:5.5) -- (\k+45:5.5);
\foreach \k in {30,150,270}
\node at (\k:5.8) {$\mu_1$};
\foreach \k in {90,210,-30}
\node at (\k:5.8) {$\mu_2$};
\node[draw,rectangle,rounded corners=0.7em,inner sep=0.6em,scale=0.9] at (7.5,0) {\tikz{
    \begin{scope}[yshift=0.5cm]
	\draw[wline] (0,0) -- (-30:1);
	\draw[webline] (0,0) -- (90:1);
	\draw[webline] (0,0) -- (210:1);
	\foreach \i in {90,210,330}
	{
	\markedpt{\i:1};
    \draw[blue] (\i:1) -- (\i+120:1);
	}
	\end{scope}
	\begin{scope}[xshift=2.5cm,yshift=0.5cm]
	\draw[wline] (-30:0.5) -- (-30:1);
	\draw[wline] (210:0.5) -- (210:1);
	\draw[webline] (-30:0.5) -- (210:0.5);
	\draw[webline] (-30:0.5) -- (90:1); 
	\draw[webline] (210:0.5) -- (90:1);
	\foreach \i in {90,210,330}
	{
	\markedpt{\i:1};
    \draw[blue] (\i:1) -- (\i+120:1);
	}
	\end{scope}
	}};
\node[draw,rectangle,rounded corners=0.7em,inner sep=0.6em,scale=0.9] at (4.5,3*1.732) {\tikz{
	\begin{scope}[yshift=0.5cm]
	\draw[wline] (0,0) -- (210:1);
	\draw[webline] (0,0) -- (90:1);
	\draw[webline] (0,0) -- (-30:1);
	\foreach \i in {90,210,330}
	{
	\markedpt{\i:1};
    \draw[blue] (\i:1) -- (\i+120:1);
	}
	\end{scope}
	\begin{scope}[xshift=2.5cm,yshift=0.5cm]
	\draw[wline] (-30:0.5) -- (-30:1);
	\draw[wline] (210:0.5) -- (210:1);
	\draw[webline] (-30:0.5) -- (210:0.5);
	\draw[webline] (-30:0.5) -- (90:1); 
	\draw[webline] (210:0.5) -- (90:1);
	\foreach \i in {90,210,330}
	{
	\markedpt{\i:1};
    \draw[blue] (\i:1) -- (\i+120:1);
	}
	\end{scope}
	}};
\node[draw,rectangle,rounded corners=0.7em,inner sep=0.6em,scale=0.9] at (-4.5,3*1.732) {\tikz{
	\begin{scope}[yshift=0.5cm]
	\draw[wline] (0,0) -- (210:1);
	\draw[webline] (0,0) -- (90:1);
	\draw[webline] (0,0) -- (-30:1);
	\foreach \i in {90,210,330}
	{
	\markedpt{\i:1};\draw[blue] (\i:1) -- (\i+120:1);
	}
	\end{scope}
	\begin{scope}[xshift=2.5cm,yshift=0.5cm]
	\draw[wline] (90:0.5) -- (90:1);
	\draw[wline] (210:0.5) -- (210:1);
	\draw[webline] (90:0.5) -- (210:0.5);
	\draw[webline] (90:0.5) -- (-30:1); 
	\draw[webline] (210:0.5) -- (-30:1);
	\foreach \i in {90,210,330}
	{
	\markedpt{\i:1};\draw[blue] (\i:1) -- (\i+120:1);
	}
	\end{scope}
	}};
\node[draw,rectangle,rounded corners=0.7em,inner sep=0.6em,scale=0.9] at (-7.5,0) {\tikz{
	\begin{scope}[yshift=0.5cm]
	\draw[wline] (0,0) -- (90:1);
	\draw[webline] (0,0) -- (-30:1);
	\draw[webline] (0,0) -- (210:1);
	\foreach \i in {90,210,330}
	{
	\markedpt{\i:1};
    \draw[blue] (\i:1) -- (\i+120:1);
	}
	\end{scope}
	\begin{scope}[xshift=2.5cm,yshift=0.5cm]
	\draw[wline] (90:0.5) -- (90:1);
	\draw[wline] (210:0.5) -- (210:1);
	\draw[webline] (90:0.5) -- (210:0.5);
	\draw[webline] (90:0.5) -- (-30:1); 
	\draw[webline] (210:0.5) -- (-30:1);
	\foreach \i in {90,210,330}
	{
	\markedpt{\i:1};
    \draw[blue] (\i:1) -- (\i+120:1);
	}
	\end{scope}
	}};
\node[draw,rectangle,rounded corners=0.7em,inner sep=0.6em,scale=0.9] at (-4.5,-3*1.732) {\tikz{
	\begin{scope}[yshift=0.5cm]
	\draw[wline] (0,0) -- (90:1);
	\draw[webline] (0,0) -- (-30:1);
	\draw[webline] (0,0) -- (210:1);
	\foreach \i in {90,210,330}
	{
	\markedpt{\i:1};
    \draw[blue] (\i:1) -- (\i+120:1);
	}
	\end{scope}
	\begin{scope}[xshift=2.5cm,yshift=0.5cm]
	\draw[wline] (-30:0.5) -- (-30:1);
	\draw[wline] (90:0.5) -- (90:1);
	\draw[webline] (-30:0.5) -- (90:0.5);
	\draw[webline] (-30:0.5) -- (210:1); 
	\draw[webline] (90:0.5) -- (210:1);
	\foreach \i in {90,210,330}
	{
	\markedpt{\i:1};
    \draw[blue] (\i:1) -- (\i+120:1);
	}
	\end{scope}
	}};
\node[draw,rectangle,rounded corners=0.7em,inner sep=0.6em,scale=0.9] at (4.5,-3*1.732) {\tikz{
	\begin{scope}[yshift=0.5cm]
	\draw[wline] (0,0) -- (-30:1);
	\draw[webline] (0,0) -- (90:1);
	\draw[webline] (0,0) -- (210:1);
	\foreach \i in {90,210,330}
	{
	\markedpt{\i:1};
    \draw[blue] (\i:1) -- (\i+120:1);
	}
	\end{scope}
	\begin{scope}[xshift=2.5cm,yshift=0.5cm]
	\draw[wline] (-30:0.5) -- (-30:1);
	\draw[wline] (90:0.5) -- (90:1);
	\draw[webline] (-30:0.5) -- (90:0.5);
	\draw[webline] (-30:0.5) -- (210:1); 
	\draw[webline] (90:0.5) -- (210:1);
	\foreach \i in {90,210,330}
	{
	\markedpt{\i:1};
    \draw[blue] (\i:1) -- (\i+120:1);
	}
	\end{scope}
	}};
\end{tikzpicture}
    \caption{The six web clusters for a triangle $T$. Here $\mu_d$ denotes the mutation at the unique unfrozen vertex with weight $s \in \{1,2\}$. The frozen variables (boundary webs) are omitted. Compare with \cref{fig:exch_triangle}.}
    \label{fig:exch_tri_web}
\end{figure}

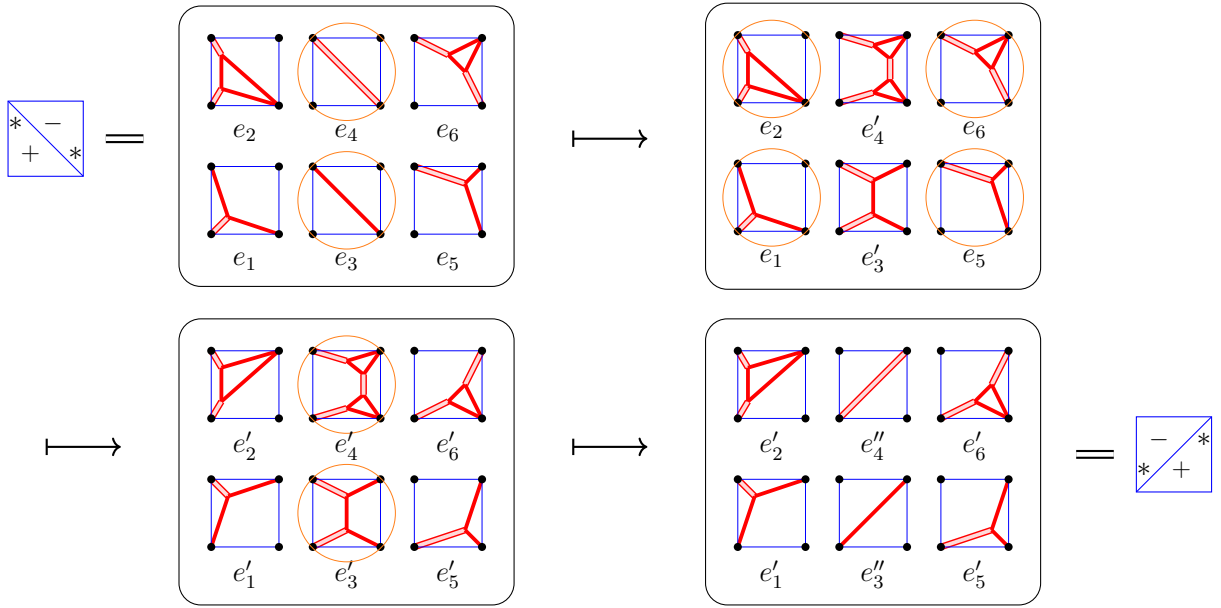
\begin{figure}[ht]
\begin{tikzpicture}
\node[draw,rectangle,rounded corners=0.7em,inner sep=0.6em,scale=0.9] at (2,1.4) {\tikz{
\draw[white] (0.5,0.5) circle(0.72cm);
\draw[wline] (0,0) -- (0.25,0.25);
\draw[webline] (1,0) --(0.25,0.25);
\draw[webline] (0,1) --(0.25,0.25);
\draw[blue,rounded corners=0] (0,0) -- (1,0) -- (1,1) -- (0,1)--cycle;
\node[inner sep=0] at (0.5,-0.4) {$e_1$};
\foreach \i in {0,1} \foreach \j in {0,1} \markedpt{\i,\j};
\begin{scope}[xshift=1.5cm]
\draw[webline] (1,0) -- (0,1);
\draw[blue,rounded corners=0] (0,0) -- (1,0) -- (1,1) -- (0,1)--cycle;
\foreach \i in {0,1} \foreach \j in {0,1} \markedpt{\i,\j};
\draw[myorange] (0.5,0.5) circle(0.72cm);
\node[inner sep=0] at (0.5,-0.4) {$e_3$};
\end{scope}
\begin{scope}[xshift=3cm]
\draw[wline] (0,1) -- (0.75,0.75);
\draw[webline] (1,0) -- (0.75,0.75);
\draw[webline] (1,1) -- (0.75,0.75);
\draw[blue,rounded corners=0] (0,0) -- (1,0) -- (1,1) -- (0,1)--cycle;
\foreach \i in {0,1} \foreach \j in {0,1} \markedpt{\i,\j};
\node[inner sep=0] at (0.5,-0.4) {$e_5$};
\end{scope}
\begin{scope}[yshift=1.9cm]
\draw[wline] (0,0) -- (0.15,0.25);
\draw[wline] (0,1) -- (0.15,0.75);
\draw[webline] (0.15,0.75) --(0.15,0.25);
\draw[webline] (1,0) --(0.15,0.25);
\draw[webline] (1,0) --(0.15,0.75);
\draw[blue,rounded corners=0] (0,0) -- (1,0) -- (1,1) -- (0,1)--cycle;
\foreach \i in {0,1} \foreach \j in {0,1} \markedpt{\i,\j};
\node[inner sep=0] at (0.5,-0.4) {$e_2$};
\end{scope}
\begin{scope}[xshift=1.5cm,yshift=1.9cm]
\draw[wline] (1,0) -- (0,1);
\draw[blue,rounded corners=0] (0,0) -- (1,0) -- (1,1) -- (0,1)--cycle;
\foreach \i in {0,1} \foreach \j in {0,1} \markedpt{\i,\j};
\draw[myorange] (0.5,0.5) circle(0.72cm);
\node[inner sep=0] at (0.5,-0.4) {$e_4$};
\end{scope}
\begin{scope}[xshift=3cm,yshift=1.9cm]
\draw[white] (0.5,0.5) circle(0.72cm);
\draw[wline] (0,1) -- (0.5,0.75);
\draw[wline] (1,0) -- (0.75,0.5);
\draw[webline] (0.5,0.75) -- (0.75,0.5);
\draw[webline] (1,1) -- (0.75,0.5);
\draw[webline] (0.5,0.75) -- (1,1);
\draw[blue,rounded corners=0] (0,0) -- (1,0) -- (1,1) -- (0,1)--cycle;
\foreach \i in {0,1} \foreach \j in {0,1} \markedpt{\i,\j};
\node[inner sep=0] at (0.5,-0.4) {$e_6$};
\end{scope}
}};
\draw[thick,|->] (5,1.5) -- (6,1.5);
\draw[blue] (-2.5,1) --++(1,0) --++(0,1)--++(-1,0) --cycle;
\draw[blue] (-2.5,2) --++(1,-1);
\node[scale=0.9] at (-2.4,1.7) {$\ast$};
\node[scale=0.9] at (-1.6,1.3) {$\ast$};
\node[scale=0.8] at (-2.2,1.3) {$+$};
\node[scale=0.8] at (-1.9,1.7) {$-$};
\draw[thick,double distance=0.15em] (-1.2,1.5) -- (-0.7,1.5);

\node[draw,rectangle,rounded corners=0.7em,inner sep=0.6em,scale=0.9] at (9,1.4) {\tikz{
\draw[wline] (0,0) -- (0.25,0.25);
\draw[webline] (1,0) --(0.25,0.25);
\draw[webline] (0,1) --(0.25,0.25);
\draw[blue,rounded corners=0] (0,0) -- (1,0) -- (1,1) -- (0,1)--cycle;
\foreach \i in {0,1} \foreach \j in {0,1} \markedpt{\i,\j};
\draw[myorange] (0.5,0.5) circle(0.72cm);
\node[inner sep=0] at (0.5,-0.4) {$e_1$};
\begin{scope}[xshift=1.5cm]
\draw[wline] (0,1) -- (0.5,0.75);
\draw[wline] (0,0) -- (0.5,0.25);
\draw[webline] (1,0) -- (0.5,0.25);
\draw[webline] (1,1) -- (0.5,0.75);
\draw[webline] (0.5,0.25) -- (0.5,0.75);
\draw[blue,rounded corners=0] (0,0) -- (1,0) -- (1,1) -- (0,1)--cycle;
\foreach \i in {0,1} \foreach \j in {0,1} \markedpt{\i,\j};
\node[inner sep=0] at (0.5,-0.4) {$e'_3$};
\end{scope}
\begin{scope}[xshift=3cm]
\draw[wline] (0,1) -- (0.75,0.75);
\draw[webline] (1,0) -- (0.75,0.75);
\draw[webline] (1,1) -- (0.75,0.75);
\draw[blue,rounded corners=0] (0,0) -- (1,0) -- (1,1) -- (0,1)--cycle;
\foreach \i in {0,1} \foreach \j in {0,1} \markedpt{\i,\j};
\draw[myorange] (0.5,0.5) circle(0.72cm);
\node[inner sep=0] at (0.5,-0.4) {$e_5$};
\end{scope}
\begin{scope}[yshift=1.9cm]
\draw[wline] (0,0) -- (0.15,0.25);
\draw[wline] (0,1) -- (0.15,0.75);
\draw[webline] (0.15,0.75) --(0.15,0.25);
\draw[webline] (1,0) --(0.15,0.25);
\draw[webline] (1,0) --(0.15,0.75);
\draw[blue,rounded corners=0] (0,0) -- (1,0) -- (1,1) -- (0,1)--cycle;
\foreach \i in {0,1} \foreach \j in {0,1} \markedpt{\i,\j};
\draw[myorange] (0.5,0.5) circle(0.72cm);
\node[inner sep=0] at (0.5,-0.4) {$e_2$};
\end{scope}
\begin{scope}[xshift=1.5cm,yshift=1.9cm]
\draw[wline] (0,0) -- (0.5,0.15);
\draw[wline] (0,1) -- (0.5,0.85);
\draw[webline] (0.5,0.85) -- (1,1);
\draw[webline] (0.5,0.85) -- (0.75,0.65);
\draw[webline] (0.75,0.65) -- (1,1);
\draw[webline] (0.5,0.15) -- (1,0);
\draw[webline] (0.5,0.15) -- (0.75,0.35);
\draw[webline] (0.75,0.35) -- (1,0);
\draw[wline] (0.75,0.35) -- (0.75,0.65);
\draw[blue,rounded corners=0] (0,0) -- (1,0) -- (1,1) -- (0,1)--cycle;
\foreach \i in {0,1} \foreach \j in {0,1} \markedpt{\i,\j};
\node[inner sep=0] at (0.5,-0.4) {$e'_4$};
\end{scope}
\begin{scope}[xshift=3cm,yshift=1.9cm]
\draw[wline] (0,1) -- (0.5,0.75);
\draw[wline] (1,0) -- (0.75,0.5);
\draw[webline] (0.5,0.75) -- (0.75,0.5);
\draw[webline] (1,1) -- (0.75,0.5);
\draw[webline] (0.5,0.75) -- (1,1);
\draw[blue,rounded corners=0] (0,0) -- (1,0) -- (1,1) -- (0,1)--cycle;
\foreach \i in {0,1} \foreach \j in {0,1} \markedpt{\i,\j};
\draw[myorange] (0.5,0.5) circle(0.72cm);
\node[inner sep=0] at (0.5,-0.4) {$e_6$};
\end{scope}
}};

\node[draw,rectangle,rounded corners=0.7em,inner sep=0.6em,scale=0.9] at (2,-2.8) {\tikz{
\draw[white] (0.5,0.5) circle(0.72cm);
\draw[wline] (0,1) -- (0.25,0.75);
\draw[webline] (1,1) --(0.25,0.75);
\draw[webline] (0,0) --(0.25,0.75);
\draw[blue,rounded corners=0] (0,0) -- (1,0) -- (1,1) -- (0,1)--cycle;
\foreach \i in {0,1} \foreach \j in {0,1} \markedpt{\i,\j};
\node[inner sep=0] at (0.5,-0.4) {$e'_1$};
\begin{scope}[xshift=1.5cm]
\draw[wline] (0,1) -- (0.5,0.75);
\draw[wline] (0,0) -- (0.5,0.25);
\draw[webline] (1,0) -- (0.5,0.25);
\draw[webline] (1,1) -- (0.5,0.75);
\draw[webline] (0.5,0.25) -- (0.5,0.75);
\draw[blue,rounded corners=0] (0,0) -- (1,0) -- (1,1) -- (0,1)--cycle;
\foreach \i in {0,1} \foreach \j in {0,1} \markedpt{\i,\j};
\node[inner sep=0] at (0.5,-0.4) {$e'_3$};
\draw[myorange] (0.5,0.5) circle(0.72cm);
\end{scope}
\begin{scope}[xshift=3cm]
\draw[wline] (0,0) -- (0.75,0.25);
\draw[webline] (1,0) -- (0.75,0.25);
\draw[webline] (1,1) -- (0.75,0.25);
\draw[blue,rounded corners=0] (0,0) -- (1,0) -- (1,1) -- (0,1)--cycle;
\foreach \i in {0,1} \foreach \j in {0,1} \markedpt{\i,\j};
\node[inner sep=0] at (0.5,-0.4) {$e'_5$};
\end{scope}
\begin{scope}[yshift=1.9cm]
\draw[wline] (0,0) -- (0.15,0.25);
\draw[wline] (0,1) -- (0.15,0.75);
\draw[webline] (0.15,0.75) --(0.15,0.25);
\draw[webline] (1,1) --(0.15,0.25);
\draw[webline] (1,1) --(0.15,0.75);
\draw[blue,rounded corners=0] (0,0) -- (1,0) -- (1,1) -- (0,1)--cycle;
\foreach \i in {0,1} \foreach \j in {0,1} \markedpt{\i,\j};
\node[inner sep=0] at (0.5,-0.4) {$e'_2$};
\end{scope}
\begin{scope}[xshift=1.5cm,yshift=1.9cm]
\draw[wline] (0,0) -- (0.5,0.15);
\draw[wline] (0,1) -- (0.5,0.85);
\draw[webline] (0.5,0.85) -- (1,1);
\draw[webline] (0.5,0.85) -- (0.75,0.65);
\draw[webline] (0.75,0.65) -- (1,1);
\draw[webline] (0.5,0.15) -- (1,0);
\draw[webline] (0.5,0.15) -- (0.75,0.35);
\draw[webline] (0.75,0.35) -- (1,0);
\draw[wline] (0.75,0.35) -- (0.75,0.65);
\draw[blue,rounded corners=0] (0,0) -- (1,0) -- (1,1) -- (0,1)--cycle;
\foreach \i in {0,1} \foreach \j in {0,1} \markedpt{\i,\j};
\draw[myorange] (0.5,0.5) circle(0.72cm);
\node[inner sep=0] at (0.5,-0.4) {$e'_4$};
\end{scope}
\begin{scope}[xshift=3cm,yshift=1.9cm]
\draw[white] (0.5,0.5) circle(0.72cm);
\draw[wline] (1,1) -- (0.75,0.5);
\draw[wline] (0,0) -- (0.5,0.25);
\draw[webline] (0.5,0.25) -- (0.75,0.5);
\draw[webline] (1,0) -- (0.75,0.5);
\draw[webline] (0.5,0.25) -- (1,0);
\draw[blue,rounded corners=0] (0,0) -- (1,0) -- (1,1) -- (0,1)--cycle;
\foreach \i in {0,1} \foreach \j in {0,1} \markedpt{\i,\j};
\node[inner sep=0] at (0.5,-0.4) {$e'_6$};
\end{scope}
}};

\draw[thick,|->] (5,-2.6) -- (6,-2.6);
\draw[thick,|->] (-2,-2.6) -- (-1,-2.6);

\node[draw,rectangle,rounded corners=0.7em,inner sep=0.6em,scale=0.9] at (9,-2.8) {\tikz{
\draw[white] (0.5,0.5) circle(0.72cm);
\draw[wline] (0,1) -- (0.25,0.75);
\draw[webline] (1,1) --(0.25,0.75);
\draw[webline] (0,0) --(0.25,0.75);
\draw[blue,rounded corners=0] (0,0) -- (1,0) -- (1,1) -- (0,1)--cycle;
\foreach \i in {0,1} \foreach \j in {0,1} \markedpt{\i,\j};
\node[inner sep=0] at (0.5,-0.4) {$e'_1$};
\begin{scope}[xshift=1.5cm]
\draw[webline] (0,0) -- (1,1);
\draw[blue,rounded corners=0] (0,0) -- (1,0) -- (1,1) -- (0,1)--cycle;
\foreach \i in {0,1} \foreach \j in {0,1} \markedpt{\i,\j};
\node[inner sep=0] at (0.5,-0.4) {$e''_3$};
\end{scope}
\begin{scope}[xshift=3cm]
\draw[wline] (0,0) -- (0.75,0.25);
\draw[webline] (1,0) -- (0.75,0.25);
\draw[webline] (1,1) -- (0.75,0.25);
\draw[blue,rounded corners=0] (0,0) -- (1,0) -- (1,1) -- (0,1)--cycle;
\foreach \i in {0,1} \foreach \j in {0,1} \markedpt{\i,\j};
\node[inner sep=0] at (0.5,-0.4) {$e'_5$};
\end{scope}
\begin{scope}[yshift=1.9cm]
\draw[wline] (0,0) -- (0.15,0.25);
\draw[wline] (0,1) -- (0.15,0.75);
\draw[webline] (0.15,0.75) --(0.15,0.25);
\draw[webline] (1,1) --(0.15,0.25);
\draw[webline] (1,1) --(0.15,0.75);
\draw[blue,rounded corners=0] (0,0) -- (1,0) -- (1,1) -- (0,1)--cycle;
\foreach \i in {0,1} \foreach \j in {0,1} \markedpt{\i,\j};
\node[inner sep=0] at (0.5,-0.4) {$e'_2$};
\end{scope}
\begin{scope}[xshift=1.5cm,yshift=1.9cm]
\draw[wline] (0,0) -- (1,1);
\draw[blue,rounded corners=0] (0,0) -- (1,0) -- (1,1) -- (0,1)--cycle;
\foreach \i in {0,1} \foreach \j in {0,1} \markedpt{\i,\j};
\node[inner sep=0] at (0.5,-0.4) {$e''_4$};
\end{scope}
\begin{scope}[xshift=3cm,yshift=1.9cm]
\draw[white] (0.5,0.5) circle(0.72cm);
\draw[wline] (1,1) -- (0.75,0.5);
\draw[wline] (0,0) -- (0.5,0.25);
\draw[webline] (0.5,0.25) -- (0.75,0.5);
\draw[webline] (1,0) -- (0.75,0.5);
\draw[webline] (0.5,0.25) -- (1,0);
\draw[blue,rounded corners=0] (0,0) -- (1,0) -- (1,1) -- (0,1)--cycle;
\foreach \i in {0,1} \foreach \j in {0,1} \markedpt{\i,\j};
\node[inner sep=0] at (0.5,-0.4) {$e'_6$};
\end{scope}
}};
\begin{scope}[xshift=7cm,yshift=-4.2cm]
\draw[blue] (5.5,1) --++(1,0) --++(0,1)--++(-1,0) --cycle;
\draw[blue] (6.5,2) --++(-1,-1);
\node[scale=0.9] at (5.6,1.3) {$\ast$};
\node[scale=0.9] at (6.4,1.7) {$\ast$};
\node[scale=0.8] at (5.8,1.7) {$-$};
\node[scale=0.8] at (6.1,1.3) {$+$};
\draw[thick,double distance=0.15em] (5.2,1.5) -- (4.7,1.5);
\end{scope}
\end{tikzpicture}
    \caption{The web clusters along a flip. Compare with \cref{fig:flip_sequence}. Here in each (web) cluster, the six elementary webs are assigned to the vertices of the weighted quiver at the corresponding position.}
    \label{fig:flip_sequence_web}
\end{figure}

\begin{proof}
By \cref{prop:compatibility} and \cref{prop:q-seed condition}, the triple $(B^{\bD},\Pi^{\bD},M^{\bD})$ associated with a decorated triangulation $\omega=\bD$ is a quantum seed. Here the condition $\mathrm{Frac}T_\bD=\mathrm{Frac} \Skein{\Sigma}^q$ follows from \cref{cor:embed_into_frac}. 
We have also seen that the exchange matrices $B^{(\omega)}$ are related to each other by matrix mutations. 

We are going to first show that the toric frames $M^{(\omega)}$ are related to each other by the quantum exchange relations \eqref{eq:q-mutation}. Since the surface subgraph is connected, we only need to consider the mutation sequences lying in the exchange graph $\Exch_{\sfs(\mathfrak{sp}_4,T)}$ for a triangle $T$ and those giving rise to a flip of triangulation. See \cref{fig:exch_tri_web,fig:flip_sequence_web}, respectively. Therefore the computation reduces to the triangle and quadrilateral cases:
\begin{description}
\item[Triangle case] 
Let us consider the web cluster and the quiver shown at the top left in \cref{fig:exch_triangle,fig:exch_tri_web}, respectively, and the mutations $\mu_1,\mu_2$ from there. We also use the labeling in \cref{fig:labeling_triangle_case}. 
From the quantum exchange relation \eqref{eq:q-mutation_asymmetric}, we expect the relations
\begin{align*}
    e_1 e'_1 &= q^{-1/2}([e_2 e_3] + q [e_4 e_5 e_7]), \\
    e_2 e'_2 &= q^{-1} ([e_4 e_6 e_7^2] + q^2[e_1^2 e_8])
\end{align*}
to hold inside the skein algebra $\Skein{\Sigma}^q$.
Indeed, the first (resp.~second) one is exactly the skein relation \eqref{eq:skein_rel_rot_1} being rotated by $120^\circ$ (resp.~\eqref{eq:skein_rel_rot_2} with mirror reflection). 
By symmetry, it implies that all of the six quantum clusters on $T$ are mutation-equivalent to each other. 

\item[Quadrilateral case (a flip)] Let us use the labeling in \cref{fig:labeling_square_case}. Then the mutation sequence under consideration is written as $(\mu_3\mu_4)(\mu_1\mu_2\mu_5\mu_6)(\mu_3\mu_4)$. Here each of the parenthesized subsequence does not depend on the order of mutations. In the first parenthesized step $\mu_3\mu_4$, we expect:
\begin{align*}
    e_3 e'_3 &= q^0 ([e_1 e_5] + q [e_2 e_{11}]), \\
    e_4 e'_4 &= q^0 ([e_2 e_6] + q^2 [e_5^2 e_{10}])
\end{align*}
inside the skein algebra $\Skein{\Sigma}^q$. 
In the second parenthesized step $\mu_1\mu_2\mu_5\mu_6$, 
\begin{align*}
    e_1 e'_1 &= q^0 ([e'_3 e_7] + q^1 [e_8 e_9 e_{10}]), \\
    e_2 e'_2 &= q^0 ([e'_4 e_8] + q^2 (e'_3)^2), \\
    e_5 e'_5 &= q^0 ([e'_3 e_{13}] + q e'_4), \\
    e_6 e'_6 &= q^0 ([e'_4 e_{14}] + q^2 [e_{10} e_{12} e_{13}^2]).
\end{align*}
Here we used $e_1 e'_3=[e_1 e'_3]$, $e_2 e'_4 =[e_2 e'_4]$, $e_5 e'_3 =[e_5 e'_3]$, $e_6 e'_4 = [e_6 e'_4]$ to obtain $q^0$. It is not hard to verify that these relations are indeed satisfied by the elementary webs shown in \cref{fig:flip_sequence_web}. 
The check for the third step $\mu_3\mu_4$ follows from the first step by turning the topside down. 
\end{description}
Thus the toric frames $M^{(\omega)}$ are related to each other by the quantum exchange relations. Then it follows from \cref{prop:q-seed condition,lem:compatibility_check} that the pair $(B^{(\omega)},\Pi^{(\omega)})$ satisfies the compatibility relation for all $\omega \in \Exch'_{\mathfrak{sp}_4,\Sigma}$. The assertion is proved. 
\end{proof}

It follows from the above theorem that there exists a canonical mutation class $\sfs_q(\mathfrak{sp}_4,\Sigma)$ containing the quantum seeds $(B^{(\omega)},\Pi^{(\omega)},M^{(\omega)})$ associated with the vertices $\omega \in \Exch'_{\mathfrak{sp}_4,\Sigma}$, and we get the following algebras over $\bZ_q$:
\begin{align*}
    \CA^q_{\mathfrak{sp}_4,\Sigma} \subset \UCA^q_{\mathfrak{sp}_4,\Sigma} \subset \mathrm{Frac}\Skein{\Sigma}^q.
\end{align*}
Here we write $\CA^q_{\mathfrak{sp}_4,\Sigma}:=\CA_{\sfs_q(\mathfrak{sp}_4,\Sigma)}$ and $\UCA^q_{\mathfrak{sp}_4,\Sigma}:=\UCA_{\sfs_q(\mathfrak{sp}_4,\Sigma)}$ by simplifying the notation. 
By construction, we already know that the cluster variables associated with the vertices $\omega \in \Exch'_{\mathfrak{sp}_4,\Sigma}$ in the surface subgraph are realized in $\Skein{\Sigma}^{\bZ_q}[\partial^{-1}]$. 


\subsection{Inclusion \texorpdfstring{$\Skein{\Sigma}^{\bZ_q}[\partial^{-1}] \subseteq \CA^q_{\mathfrak{sp}_4,\Sigma}$}{S(sp4,S) in CA(sp4,S)} via the sticking trick}\label{subsec:S in A}

\begin{prop}\label{S=A_triangle}
When $\Sigma=T$ is a triangle, we have 
\begin{align*}
    \Skein{T}^{\bZ_q}[\partial^{-1}]=\CA^q_{\mathfrak{sp}_4,T}=\UCA^q_{\mathfrak{sp}_4,T}.
\end{align*}
\end{prop}

\begin{proof}
In this case, the six elementary webs in $\Skein{T}^q$ listed in \cref{thm:Eweb-T} exactly correspond to the six quantum cluster variables in $\CA^q_{\mathfrak{sp}_4,T}$. Moreover, observe that all the six elementary webs are simple Wilson lines. 
Hence 
\begin{align*}
    \CA^q_{\mathfrak{sp}_4,T} = \langle \Eweb{T} \rangle_{\bZ_q}[\partial^{-1}] = \langle \Desc{T} \rangle_{\bZ_q}[\partial^{-1}] =\Skein{T}^{\bZ_q}[\partial^{-1}].
\end{align*}
Since the mutation class $\sfs_q(\mathfrak{sp}_4,T)$ is of acyclic exchange type, we have $\CA^q_{\mathfrak{sp}_4,T} = \UCA^q_{\mathfrak{sp}_4,T}$ (\cite[Proposition 8.17]{Muller16}). Thus the assertion is proved. 
\end{proof}

\begin{thm}\label{thm:S in A}
For a connected unpunctured marked surface $\Sigma$ with at least two special points, we have an inclusion
\begin{align*}
    \Skein{\Sigma}^{\bZ_q}[\partial^{-1}]\subset \CA^q_{\mathfrak{sp}_4,\Sigma}.
\end{align*}
\end{thm}

\begin{proof}
By \cref{thm:generator-Zv-form}, it suffices to prove that each element in $\SimpWil{\Sigma}$, namely each simple Wilson line $\alpha$ of type $1$, is a cluster variable. From the definition, there exists a quadrilateral $Q$ that contains $\alpha$ such that its interior is embedded, and the stated ends of $\alpha$ lie on opposite sides of $Q$. Here we allow one of the other side is collapsed so that $Q$ is in fact a triangle in $\Sigma$. 
Then we find that $\alpha$ appears as a cluster variable in $\CA^q_{\mathfrak{sp}_4,Q}$ from \cref{fig:flip_sequence_web} (the variables $e_1$, $e_3$, $e'_1$ and $e'_3$), except for the pattern $\tikz[baseline=0.5cm,scale=0.85]{\draw[wline] (0,1) -- (0.5,0.75);
\draw[wline] (1,0) -- (0.5,0.25);
\draw[webline] (0,0) -- (0.5,0.25);
\draw[webline] (1,1) -- (0.5,0.75);
\draw[webline] (0.5,0.25) -- (0.5,0.75);
\draw[blue,rounded corners=0] (0,0) -- (1,0) -- (1,1) -- (0,1)--cycle;
\foreach \i in {0,1} \foreach \j in {0,1} \markedpt{\i,\j};}$. This one is also found in a cluster adjacent to the decorated triangulation (III) in the proof of \cref{prop:compatibility}. See \cref{sect:quad}. 
The assertion is proved.
\end{proof}

\paragraph{\textbf{Classical limit.}}
In the classical limit $q=1$, by \cref{rem:GS geometry} (3), we have $\CA_{\mathfrak{sp}_4,\Sigma}=\UCA_{\mathfrak{sp}_4,\Sigma}=\cO(\A_{Sp_4,\Sigma}^\times)$ if $\Sigma$ has at least two marked points,
and they are generated by the matrix entries of simple Wilson lines in the vector representation $V(\varpi_1)$ of $Sp_4$ \cite[Theorem 1, Corollary 3.20 and Proposition 2.1]{IOS}. Moreover, they correspond to elements in $\SimpWil{\Sigma}$ (recall \cref{rem:Wilson_line}). Combining with \cref{thm:generator-Zv-form}, we get the following:

\begin{thm}\label{prop:Tree=A=U}
In the classical limit $q=1$, we have
\begin{align*}
    \mathscr{S}_{\mathfrak{sp}_4,\Sigma}^1[\partial^{-1}]=\CA_{\mathfrak{sp}_4,\Sigma}=\UCA_{\mathfrak{sp}_4,\Sigma}.
\end{align*}
Here $\mathscr{S}_{\mathfrak{sp}_4,\Sigma}^1 :=\mathscr{S}_{\mathfrak{sp}_4,\Sigma}\otimes_{\cR} \bC$ is the specialization of the $\mathfrak{sp}_4$-skein algebra, where the homomorphism $\cR \to \bC$ is given by $q \mapsto 1$.
\end{thm}

\subsection{Inclusion \texorpdfstring{$\Skein{\Sigma}^{\bZ_q}[\partial^{-1}] \subseteq \UCA^q_{\mathfrak{sp}_4,\Sigma}$}{S(sp4,S) in U(sp4,S)} via the cutting trick, quantum Laurent positivity}\label{subsec:S in U}

Here we give a direct way to compute the inclusion $\Skein{\Sigma}^{\bZ_q}[\partial^{-1}]\subset \UCA^q_{\mathfrak{sp}_4,\Sigma}$ based on the cluster expansion result (\cref{thm:Cweb-exp}). 

\begin{thm}\label{thm:comparison_skein_cluster}
For any unpunctured marked surface $\Sigma$, we have an
inclusion
\begin{align*}
    \Skein{\Sigma}^{\bZ_q}[\partial^{-1}] \subseteq \UCA^q_{\mathfrak{sp}_4,\Sigma}.
\end{align*}
\end{thm}
We remark that this theorem also follows from the quantum Laurent phenomenon  (\cref{thm:q-Laurent}) and \cref{thm:S in A} if we assume that $\Sigma$ has at least two special points.

\begin{proof}
Let $\bD \in \Exch_{\mathfrak{sp}_4,\Sigma}$ be a decorated triangulation. 
Thanks to \cref{thm:q-Laurent}, it suffices to check the inclusion $\Skein{\Sigma}^{\bZ_q}[\partial^{-1}] \subset \UCA^q_{\mathfrak{sp}_4,\Sigma}(\bD)$ into the upper bound associated with $\bD$. It is a direct consequence of \cref{thm:Cweb-exp} that any element in $\Skein{\Sigma}^{\bZ_q}[\partial^{-1}]$ has a Laurent expression with coefficients in $\bZ_q$ in the web cluster $\mathcal{C}_{\bD}$. Thus $\Skein{\Sigma}^{\bZ_q}[\partial^{-1}] \subset T_{\bD}$. 

Observe that all the vertices of $\Exch_{\mathfrak{sp}_4,\Sigma}$ adjacent to $\bD$ are decorated cell decompositions. Indeed, the mutation at a face vertex preserves the underlying triangulation, and the mutation at an edge vertex lies in a flip sequence. 
Therefore it suffices to consider an ideal cell decomposition $(\Delta;E)$ and verify that
any element $x \in \Skein{\Sigma}^{\bZ_q}$ has a Laurent expression in each web cluster of the form 
\begin{align*}
    \mathcal{C}_{(\Delta;E)}:=\bigcup_{T \in t(\Delta),~T \nsupseteq E} \mathcal{C}_T \cup \mathcal{C}_{Q_E},
\end{align*}
where $Q_E$ denotes the unique quadrilateral containing $E$, $\mathcal{C}_T$ (resp. $\mathcal{C}_{Q_E}$) is a web cluster in $\Skein{T}^q$ (resp. $\Skein{Q_E}^q$). 

By the cutting trick (\cref{lem:cutteing-trick}), there exists a monomial $J_{(\Delta;E)}$ of arcs along $e_{\interior}(\Delta)\setminus \{E\}$ of either types such that the product $x\cdot J_{(\Delta;E)}$ is expanded into a linear combination
of products of webs contained in $\Skein{T}^{\bZ_q}[\partial^{-1}]$ or $\Skein{Q_E}^{\bZ_q}[\partial^{-1}]$ (see \cref{ex:cutting} for an example). Here each term in the expansion belongs to the $\bZ_q$-form by \cref{lem:crossroads-in-Zv}. 
With a notice that each triangle $T$ or the quadrilateral $Q_E$ have at least two special points (since $\Sigma$ has no punctures), we have inclusions 
\begin{align*}
    &\Skein{T}^{\bZ_q}[\partial^{-1}] \subset \CA^q_{\mathfrak{sp}_4,T} \subset \UCA^q_{\mathfrak{sp}_4,T}, \\
    &\Skein{Q_E}^{\bZ_q}[\partial^{-1}] \subset \CA^q_{\mathfrak{sp}_4,Q_E} \subset \UCA^q_{\mathfrak{sp}_4,Q_E}
\end{align*}
by \cref{thm:S in A}. Then it follows that each web appearing in the expression of $x\cdot J_{(\Delta;E)}$ has a Laurent expansion with coefficients in $\bZ_q$ in the web cluster $\mathcal{C}_{(\Delta;E)}$. Therefore
so does $x$. 
Thus we get the inclusion $\Skein{\Sigma}^{\bZ_q}[\partial^{-1}] \subset \UCA^q_{\mathfrak{sp}_4,\Sigma}(\bD) = \UCA^q_{\mathfrak{sp}_4,\Sigma}$, as desired.
\end{proof}

\begin{ex}\label{ex:cutting}
Let us consider the following web $W$ and its Laurent expression in the cluster $\cC_{(\Delta;E)}$:
\begin{align*}
\begin{tikzpicture}[scale=0.1]
    \coordinate (P0) at (90:20);
        \coordinate (P1) at (162:20);
        \coordinate (P2) at (234:20);
        \coordinate (P3) at (306:20);
        \coordinate (P4) at (378:20);
        \coordinate (Q0) at (90:14);
        \coordinate (Q1) at (162:14);
        \coordinate (Q2) at (234:14);
        \coordinate (Q3) at (306:14);
        \coordinate (Q4) at (378:14);
        \draw[blue,dashed, thick] (P0) -- (P2);
        \draw[blue] (P0) -- (P3);
        \draw[overarc] (P0) -- (Q1) -- (Q2) -- (Q4) -- (P0);
        \draw[wline] (P1) -- (Q1);
        \draw[wline] (P2) -- (Q2);
        \draw[wline] (P4) -- (Q4);
        \draw[fill] (P0) circle (20pt);
        \draw[fill] (P1) circle (20pt);
        \draw[fill] (P2) circle (20pt);
        \draw[fill] (P3) circle (20pt);
        \draw[fill] (P4) circle (20pt);
        \draw[blue] (P0) -- (P1) -- (P2) -- (P3) -- (P4) -- cycle;
        \draw[red] (P1) node[below left]{$W$};
        \draw[blue] (P3) node[right=2em]{$(\Delta;E)$};
        \draw[blue] (P3) node[above=3em]{$T$};
        \draw[blue] ($(P0)!0.5!(P2)$) node[right]{$E$};
        \draw[blue] ($(P3)!0.5!(P2)$) node[above]{$Q_E$};
\end{tikzpicture}
\end{align*}
By using the cutting trick (\cref{lem:cutteing-trick}), we get
\begin{align*}
    \tikz[baseline=-.6ex, scale=.1]
    {
        \coordinate (P0) at (90:10);
        \coordinate (P1) at (162:10);
        \coordinate (P2) at (234:10);
        \coordinate (P3) at (306:10);
        \coordinate (P4) at (378:10);
        \coordinate (Q0) at (90:7);
        \coordinate (Q1) at (162:7);
        \coordinate (Q2) at (234:7);
        \coordinate (Q3) at (306:7);
        \coordinate (Q4) at (378:7);
        \draw[wline] (P0) -- (P3);
        \draw[webline] (P0) to[bend right] (P3);
        \draw[overarc] (P0) -- (Q1) -- (Q2) -- (Q4) -- (P0);
        \draw[wline] (P1) -- (Q1);
        \draw[wline] (P2) -- (Q2);
        \draw[wline] (P4) -- (Q4);
        \draw[fill] (P0) circle (20pt);
        \draw[fill] (P1) circle (20pt);
        \draw[fill] (P2) circle (20pt);
        \draw[fill] (P3) circle (20pt);
        \draw[fill] (P4) circle (20pt);
        \draw[blue] (P0) -- (P1) -- (P2) -- (P3) -- (P4) -- cycle;
    }
    &
    =q^2
    \ \tikz[baseline=-.6ex, scale=.1]
    {
        \coordinate (P0) at (90:10);
        \coordinate (P1) at (162:10);
        \coordinate (P2) at (234:10);
        \coordinate (P3) at (306:10);
        \coordinate (P4) at (378:10);
        \coordinate (Q0) at (90:7);
        \coordinate (Q1) at (162:7);
        \coordinate (Q2) at (234:7);
        \coordinate (Q3) at (306:7);
        \coordinate (Q4) at (378:7);
        \draw[wline] (P0) -- (P3);
        \draw[webline] (P0) -- (Q1) -- (Q2) -- (P0);
        \draw[webline] (P3) -- (Q4) -- (P0);
        \draw[wline] (P1) -- (Q1);
        \draw[wline] (P2) -- (Q2);
        \draw[wline] (P4) -- (Q4);
        \draw[fill] (P0) circle (20pt);
        \draw[fill] (P1) circle (20pt);
        \draw[fill] (P2) circle (20pt);
        \draw[fill] (P3) circle (20pt);
        \draw[fill] (P4) circle (20pt);
        \draw[blue] (P0) -- (P1) -- (P2) -- (P3) -- (P4) -- cycle;
    }\ 
    +q
    \ \tikz[baseline=-.6ex, scale=.1]
    {
        \coordinate (P0) at (90:10);
        \coordinate (P1) at (162:10);
        \coordinate (P2) at (234:10);
        \coordinate (P3) at (306:10);
        \coordinate (P4) at (378:10);
        \coordinate (Q0) at (90:7);
        \coordinate (Q1) at (162:7);
        \coordinate (Q2) at (234:7);
        \coordinate (Q3) at (306:7);
        \coordinate (Q4) at (378:7);
        \draw[webline] (P0) -- (Q1) -- (Q2) -- ($(Q3)+(-2,2)$) -- (P3);
        \draw[webline] (P0) -- ($(Q3)+(2,2)$) -- (Q4) -- (P0);
        \draw[wline] (P0) -- ($(Q3)+(-2,2)$);
        \draw[webline] ($(Q3)+(-2,2)$) --  (P3);
        \draw[wline] (P1) -- (Q1);
        \draw[wline] (P2) -- (Q2);
        \draw[wline] (P3) -- ($(Q3)+(2,2)$); 
        \draw[wline] (P4) -- (Q4);
        \draw[fill] (P0) circle (20pt);
        \draw[fill] (P1) circle (20pt);
        \draw[fill] (P2) circle (20pt);
        \draw[fill] (P3) circle (20pt);
        \draw[fill] (P4) circle (20pt);
        \draw[blue] (P0) -- (P1) -- (P2) -- (P3) -- (P4) -- cycle;
    }
    +q^{-1}
    \ \tikz[baseline=-.6ex, scale=.1]
    {
        \coordinate (P0) at (90:10);
        \coordinate (P1) at (162:10);
        \coordinate (P2) at (234:10);
        \coordinate (P3) at (306:10);
        \coordinate (P4) at (378:10);
        \coordinate (Q0) at (90:7);
        \coordinate (Q1) at (162:7);
        \coordinate (Q2) at (234:7);
        \coordinate (Q3) at (306:7);
        \coordinate (Q4) at (378:7);
        \draw[webline] (P0) -- (Q1) -- (Q2) -- ($(Q3)+(-2,2)$) -- (P0);
        \draw[wline] (P3) -- ($(Q3)+(-2,2)$);
        \draw[wline] (P0) -- ($(Q3)+(2,2)$); 
        \draw[webline] (P3) -- ($(Q3)+(2,2)$) -- (Q4) -- (P0);
        \draw[wline] (P1) -- (Q1);
        \draw[wline] (P2) -- (Q2);
        \draw[wline] (P4) -- (Q4);
        \draw[fill] (P0) circle (20pt);
        \draw[fill] (P1) circle (20pt);
        \draw[fill] (P2) circle (20pt);
        \draw[fill] (P3) circle (20pt);
        \draw[fill] (P4) circle (20pt);
        \draw[blue] (P0) -- (P1) -- (P2) -- (P3) -- (P4) -- cycle;
    }
    +q^{-2}
    \ \tikz[baseline=-.6ex, scale=.1]
    {
        \coordinate (P0) at (90:10);
        \coordinate (P1) at (162:10);
        \coordinate (P2) at (234:10);
        \coordinate (P3) at (306:10);
        \coordinate (P4) at (378:10);
        \coordinate (Q0) at (90:7);
        \coordinate (Q1) at (162:7);
        \coordinate (Q2) at (234:7);
        \coordinate (Q3) at (306:7);
        \coordinate (Q4) at (378:7);
        \draw[wline] (P0) to[bend right] (P3);
        \draw[webline] (P0) -- (Q1) -- (Q2) -- (P3);
        \draw[webline] (P0) to[bend right] (Q4) -- (P0);
        \draw[wline] (P1) -- (Q1);
        \draw[wline] (P2) -- (Q2);
        \draw[wline] (P4) -- (Q4);
        \draw[fill] (P0) circle (20pt);
        \draw[fill] (P1) circle (20pt);
        \draw[fill] (P2) circle (20pt);
        \draw[fill] (P3) circle (20pt);
        \draw[fill] (P4) circle (20pt);
        \draw[blue] (P0) -- (P1) -- (P2) -- (P3) -- (P4) -- cycle;
    }\ \\
    &=q^2
    \ \tikz[baseline=-.6ex, scale=.1]
    {
        \coordinate (P0) at (90:10);
        \coordinate (P1) at (162:10);
        \coordinate (P2) at (234:10);
        \coordinate (P3) at (306:10);
        \coordinate (P4) at (378:10);
        \coordinate (Q0) at (90:7);
        \coordinate (Q1) at (162:7);
        \coordinate (Q2) at (234:7);
        \coordinate (Q3) at (306:7);
        \coordinate (Q4) at (378:7);
        \draw[wline] (P0) -- (P3);
        \draw[webline] (P0) -- (Q1) -- (Q2) -- (P0);
        \draw[webline] (P3) -- (Q4) -- (P0);
        \draw[wline] (P1) -- (Q1);
        \draw[wline] (P2) -- (Q2);
        \draw[wline] (P4) -- (Q4);
        \draw[fill] (P0) circle (20pt);
        \draw[fill] (P1) circle (20pt);
        \draw[fill] (P2) circle (20pt);
        \draw[fill] (P3) circle (20pt);
        \draw[fill] (P4) circle (20pt);
        \draw[blue] (P0) -- (P1) -- (P2) -- (P3) -- (P4) -- cycle;
    }\ 
    +q
    \ \tikz[baseline=-.6ex, scale=.1]
    {
        \coordinate (P0) at (90:10);
        \coordinate (P1) at (162:10);
        \coordinate (P2) at (234:10);
        \coordinate (P3) at (306:10);
        \coordinate (P4) at (378:10);
        \coordinate (Q0) at (90:7);
        \coordinate (Q1) at (162:7);
        \coordinate (Q2) at (234:7);
        \coordinate (Q3) at (306:7);
        \coordinate (Q4) at (378:7);
        \draw[webline] (P0) -- (Q1) -- (Q2) -- ($(Q3)+(-2,2)$) -- (P3);
        \draw[webline] (P0) -- ($(Q3)+(2,2)$) -- (Q4) -- (P0);
        \draw[wline] (P0) -- ($(Q3)+(-2,2)$);
        \draw[webline] ($(Q3)+(-2,2)$) --  (P3);
        \draw[wline] (P1) -- (Q1);
        \draw[wline] (P2) -- (Q2);
        \draw[wline] (P3) -- ($(Q3)+(2,2)$); 
        \draw[wline] (P4) -- (Q4);
        \draw[fill] (P0) circle (20pt);
        \draw[fill] (P1) circle (20pt);
        \draw[fill] (P2) circle (20pt);
        \draw[fill] (P3) circle (20pt);
        \draw[fill] (P4) circle (20pt);
        \draw[blue] (P0) -- (P1) -- (P2) -- (P3) -- (P4) -- cycle;
    }
    +q^{-1}
    \ \tikz[baseline=-.6ex, scale=.1]
    {
        \coordinate (P0) at (90:10);
        \coordinate (P1) at (162:10);
        \coordinate (P2) at (234:10);
        \coordinate (P3) at (306:10);
        \coordinate (P4) at (378:10);
        \coordinate (Q0) at (90:7);
        \coordinate (Q1) at (162:7);
        \coordinate (Q2) at (234:7);
        \coordinate (Q3) at (306:7);
        \coordinate (Q4) at (378:7);
        \draw[webline] (P0) -- (Q1) -- (Q2) -- ($(Q3)+(-2,2)$) -- (P0);
        \draw[wline] (P3) -- ($(Q3)+(-2,2)$);
        \draw[wline] (P1) -- (Q1);
        \draw[wline] (P2) -- (Q2);
        \draw[wline] (P4) to[bend left] (P0);
        \draw[webline] (P0) --  (P3);
        \draw[fill] (P0) circle (20pt);
        \draw[fill] (P1) circle (20pt);
        \draw[fill] (P2) circle (20pt);
        \draw[fill] (P3) circle (20pt);
        \draw[fill] (P4) circle (20pt);
        \draw[blue] (P0) -- (P1) -- (P2) -- (P3) -- (P4) -- cycle;
    }\ .
\end{align*}
Observe that each term belongs to the $\bZ_q$-forms of the skein algebras of $T$ and $Q_E$. Hence it can be written as a Laurent polynomial in the web cluster $\cC_{(\Delta;E)}$. 
\end{ex}

Now we complete the proof of \cref{introthm:comparison}.

\begin{proof}[Proof of \cref{introthm:comparison}]
The algebra inclusions are already proved in \cref{thm:S in A,thm:comparison_skein_cluster}. The mapping class group equivariance is clear from the construction (see \cite{IYsl3} for more explanation). 
What remains are the comparison of gradings and anti-involutions. Thanks to the web cluster expansion result (\cref{thm:Cweb-exp}), it suffices to compare these structures on the quantum torus $T_\bD$ generated by the elementary webs associated with a single decorated triangulation $\bD$. Then the coincidence of the mirror-reflection and the bar-involution is clear from their definitions, and the coincidence of the endpoint grading and the ensemble grading has been shown in \cref{lem:comparison_grading}. 
\end{proof}

\begin{dfn}\label{def:GS-univ}
An element of $\UCA^q_{\mathfrak{sp}_4,\Sigma}\otimes \cR$ is called a \emph{quantum GS-universally positive Laurent polynomial over $\bZ_q$ (resp. $\cR$)} if it admits an expression as a Laurent polynomial in the cluster $\mathbf{A}^{\bD}$ associated with 
any decorated triangulation $\bD$ of $\Sigma$ with coefficients in $\bZ_+[q^{\pm 1/2}]$ (resp. $\cR_+=\bZ_+[q^{\pm 1/2},1/[2]_q]$). 
\end{dfn}
Here ``GS'' stands for Goncharov--Shen. The quantum GS-universal positivity over $\bZ_q$ is weaker than the quantum universal positivity \cite{FG09}, which asks the positivity of Laurent expression for \emph{any} cluster. 
The following is a rephrasing of \cref{thm:positivity-web}:

\begin{thm}[Quantum Laurent positivity of webs]\label{thm:positivity_cluster}
Any elevation-preserving web with respect to an ideal triangulation $\Delta$ is expressed as a Laurent polynomial with coefficients in $\cR_+$ in the quantum cluster associated with any decorated triangulation $\bD=(\Delta,m_\Delta,\bs_\Delta)$ over $\Delta$. In particular, by \cref{ex:Desc}, the following webs are quantum GS-universally positive Laurent polynomials
\begin{itemize}
    \item over $\bZ_q$: elements in $\Desc{\Sigma}$;
    \item over $\cR$: the geometric bracelets or the bangles
    (\cref{fig:bracelet}) of type $2$ along simple loops. 
\end{itemize}
\end{thm}

\section{Gallery of web clusters in a quadrilateral}\label{sect:quad}
In this section, we present examples of web clusters in a quadrilateral that support \cref{conj:tree-variable}, and exchange relations among them.
Indecomposability of elementary webs in this section follows from the injective homomorphism constructed in \cref{thm:S in A} and the indecomposability of cluster variables \cite{GLS13}.
Let $\Sigma$ be a quadrilateral, which is a disk with four special points.
We use the labeling in \cref{fig:labeling_square_case} for boundary webs.

In the following presentation of web clusters of $\Skein{\Sigma}$, we display the $6$ elementary webs other than boundary webs as follows.
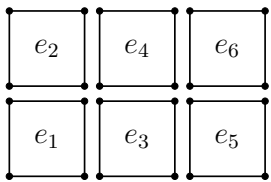
\begin{center}
    \begin{tikzpicture}[scale=.8]
        \begin{scope}
            \node[scale=.8] at (0,0)
            {
                \ \tikz[baseline=-.6ex, scale=.05]{
                    \foreach \i in {1,2,...,6}
                    \foreach \j in {1,2,...,6}
                    {
                        \coordinate (P\i\j) at (5*\j,-5*\i);
                    }
                    \coordinate (C) at ($(P11)!.5!(P66)$);
                    \node[scale=1.2] at (C) {$e_1$};
                    \draw[thick] (P11) -- (P61) -- (P66) -- (P16) -- cycle;
                    \foreach \i in {1,6}
                    \foreach \j in {1,6}
                    {
                        \draw[fill] (P\i\j) circle (30pt);
                    }
                }\ 
            };
            \node[scale=.8] at (1.5,0)
            {
                \ \tikz[baseline=-.6ex, scale=.05]{
                    \foreach \i in {1,2,...,6}
                    \foreach \j in {1,2,...,6}
                    {
                        \coordinate (P\i\j) at (5*\j,-5*\i);
                    }
                    \coordinate (C) at ($(P11)!.5!(P66)$);
                    \node[scale=1.2] at (C) {$e_3$};
                    \draw[thick] (P11) -- (P61) -- (P66) -- (P16) -- cycle;
                    \foreach \i in {1,6}
                    \foreach \j in {1,6}
                    {
                        \draw[fill] (P\i\j) circle (30pt);
                    }
                }\ 
            };
            \node[scale=.8] at (3,0)
            {
                \ \tikz[baseline=-.6ex, scale=.05]{
                    \foreach \i in {1,2,...,6}
                    \foreach \j in {1,2,...,6}
                    {
                        \coordinate (P\i\j) at (5*\j,-5*\i);
                    }
                    \coordinate (C) at ($(P11)!.5!(P66)$);
                    \node[scale=1.2] at (C) {$e_5$};
                    \draw[thick] (P11) -- (P61) -- (P66) -- (P16) -- cycle;
                    \foreach \i in {1,6}
                    \foreach \j in {1,6}
                    {
                        \draw[fill] (P\i\j) circle (30pt);
                    }
                }\ 
            };
            \node[scale=.8] at (0,1.5)
            {
                \ \tikz[baseline=-.6ex, scale=.05]{
                    \foreach \i in {1,2,...,6}
                    \foreach \j in {1,2,...,6}
                    {
                        \coordinate (P\i\j) at (5*\j,-5*\i);
                    }
                    \coordinate (C) at ($(P11)!.5!(P66)$);
                    \node[scale=1.2] at (C) {$e_2$};
                    \draw[thick] (P11) -- (P61) -- (P66) -- (P16) -- cycle;
                    \foreach \i in {1,6}
                    \foreach \j in {1,6}
                    {
                        \draw[fill] (P\i\j) circle (30pt);
                    }
                }\ 
            };
            \node[scale=.8] at (1.5,1.5)
            {
                \ \tikz[baseline=-.6ex, scale=.05]{
                    \foreach \i in {1,2,...,6}
                    \foreach \j in {1,2,...,6}
                    {
                        \coordinate (P\i\j) at (5*\j,-5*\i);
                    }
                    \coordinate (C) at ($(P11)!.5!(P66)$);
                    \node[scale=1.2] at (C) {$e_4$};
                    \draw[thick] (P11) -- (P61) -- (P66) -- (P16) -- cycle;
                    \foreach \i in {1,6}
                    \foreach \j in {1,6}
                    {
                        \draw[fill] (P\i\j) circle (30pt);
                    }
                }\ 
            };
            \node[scale=.8] at (3,1.5)
            {
                \ \tikz[baseline=-.6ex, scale=.05]{
                    \foreach \i in {1,2,...,6}
                    \foreach \j in {1,2,...,6}
                    {
                        \coordinate (P\i\j) at (5*\j,-5*\i);
                    }
                    \coordinate (C) at ($(P11)!.5!(P66)$);
                    \node[scale=1.2] at (C) {$e_6$};
                    \draw[thick] (P11) -- (P61) -- (P66) -- (P16) -- cycle;
                    \foreach \i in {1,6}
                    \foreach \j in {1,6}
                    {
                        \draw[fill] (P\i\j) circle (30pt);
                    }
                }\ 
            };
        \end{scope}
    \end{tikzpicture}    
\end{center}
Moreover, in the presentation of quantum exchange/skein relation upon each mutation $\mu_k$, non-primed variables $e_i$, $i=1,\dots,6$ denote the elementary webs in the web cluster before that mutation and the primed one $e'_k$ denotes the new elementary web appearing after the mutation. 

\subsection{Web clusters adjacent to (I)--(V)}
Here are web clusters adjacent to the decorated triangulations (I)--(V) in the proof of \cref{thm:comparison_skein_cluster}.
New elementary webs that do not come from any decorated triangulations appear in the replacement of $e_3$ or $e_4$ in each case, which are easily seen to be tree-type. Let us denote the mutation replacing $e_3$ (resp. $e_4$) starting from the web cluster $(N)$ by $\mu_3^{(N)}$ (resp. $\mu_4^{(N)}$) for $N \in \{\mathrm{I,II,III,IV,V}\}$.

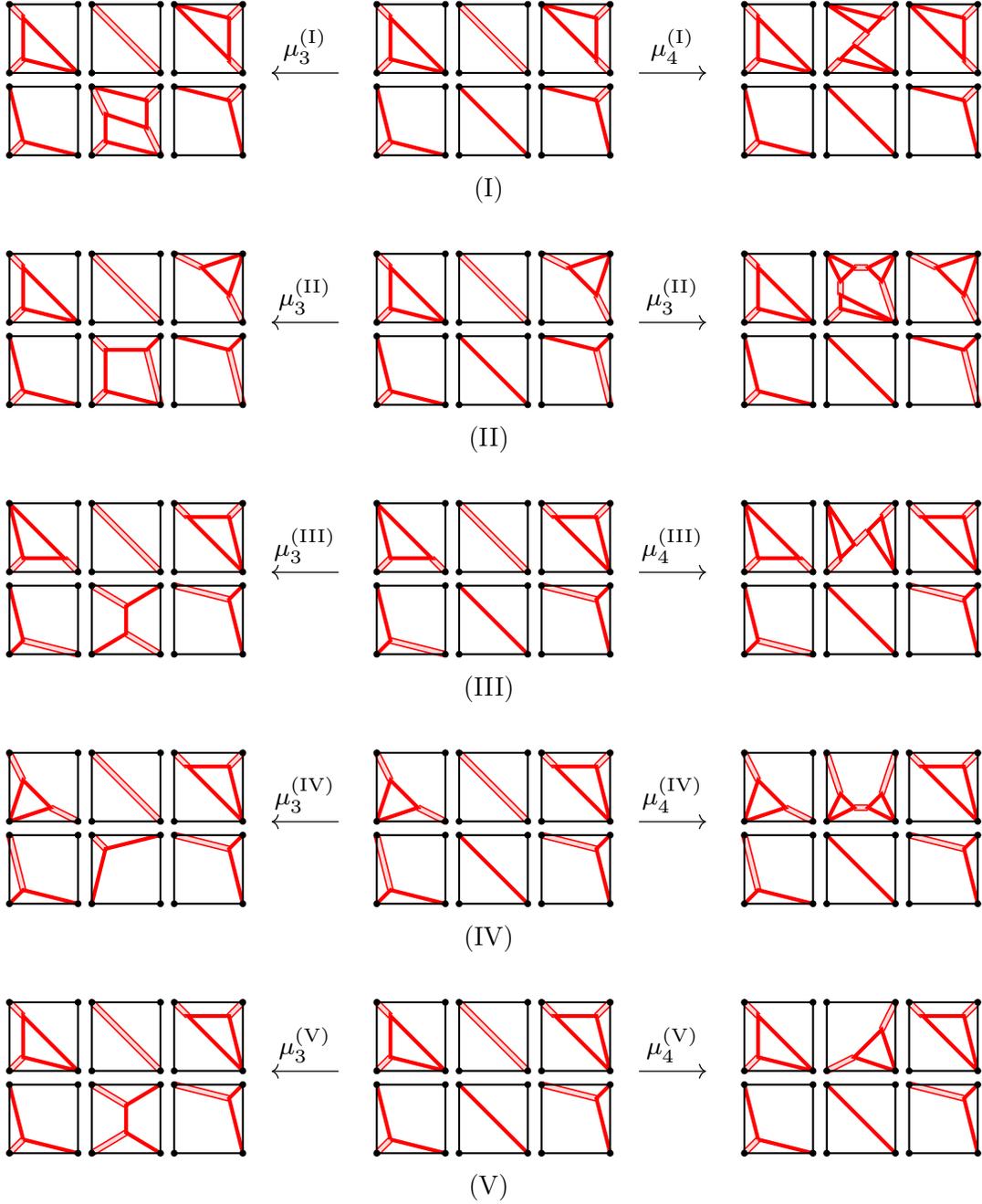
\begin{tikzcd}
    
\end{tikzcd}

We have the following skein relations among them:
\begin{align*}
    \mu_3^{\mathrm{(I)}}\colon e_3e_3'&=q[e_2e_6]+[e_1e_5e_4],&\mu_4^{\mathrm{(I)}}\colon e_4e_4'&=[e_2e_6]+q^2[e_1e_5e_4],\\
    \mu_3^{\mathrm{(II)}}\colon e_3e_3'&=q^{\frac{1}{2}}[e_2e_6]+q^{-\frac{1}{2}}[e_1e_4e_{11}],& \mu_4^{\mathrm{(II)}}\colon e_4e_4'&=q^{-1}[e_2e_5^2]+q[e_3^2e_6e_{10}],\\
    \mu_3^{\mathrm{(III)}}\colon e_3e_3'&=[e_{1}e_{5}]+q[e_{4}e_{9}e_{13}],& \mu_4^{\mathrm{(III)}}\colon e_4e_4'&=[e_2e_6]+q^{-2}[e_3^2e_{8}e_{12}],\\
    \mu_3^{\mathrm{(IV)}}\colon e_3e_3'&=q^{-\frac{1}{2}}[e_5e_7]+q^{\frac{1}{2}}[e_1e_{13}],& \mu_4^{\mathrm{(IV)}}\colon e_4e_4'&=[e_2e_6]+q^2[e_{1}^{2}e_{12}]\\
    \mu_3^{\mathrm{(V)}}\colon e_3e_3'&=[e_1e_5]+q[e_{2}e_{13}],& \mu_4^{\mathrm{(V)}}\colon e_4e_4'&=q[e_6e_{10}]+q^{-1}[e_{2}e_{12}].
\end{align*}
They are exactly the quantum exchange relations induced by $\mu_i^{(N)}$. 

\subsection{Other tree-type elementary webs}
Observe that the examples of elementary webs in the previous section are all made from the patterns

.
Let us see more examples of elementary webs made from other types of trees.
We start from the web cluster $(\mathrm{IV})$.

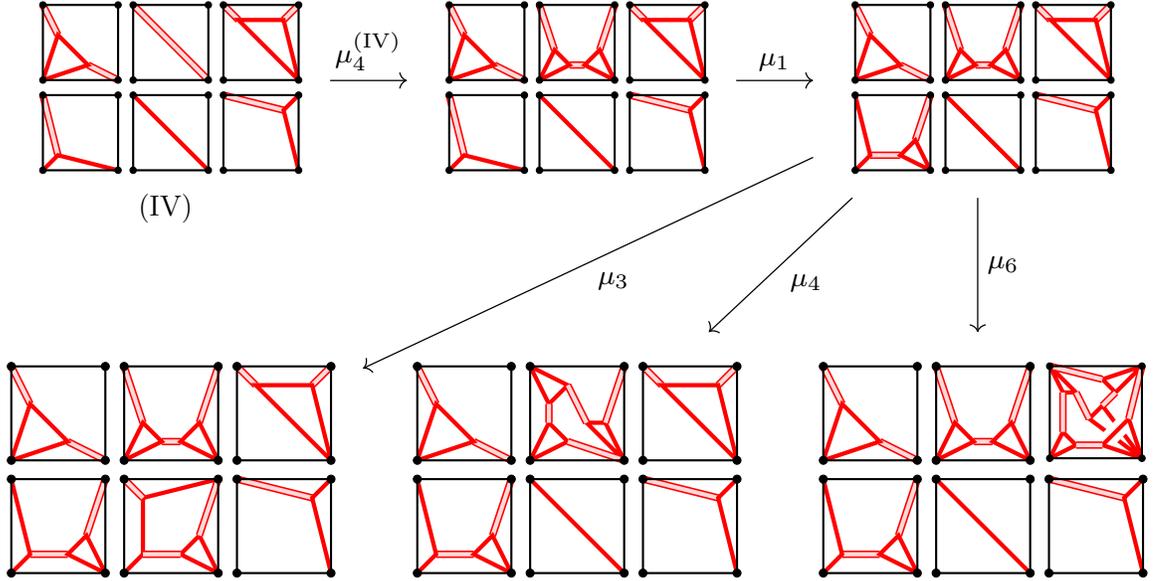
\begin{tikzcd}
%
\end{tikzcd}

The skein/quantum exchange relations among them are:
\begin{align*}
    \mu_1\colon e_1e_1'&=[e_6e_7e_9]+q^{-1}[e_3e_4],&\mu_3\colon e_3e_3'&=[e_6e_9e_{13}]+q^{-1}[e_1e_5],\\
    \mu_4\colon e_4e_4'&=q^{-1}[e_{1}^2e_{2}]+q[e_{6}e_{7}^2e_{9}^2e_{12}],&\mu_6\colon e_6e_6'&=q^{-1}[e_3e_4e_{11}e_{13}]+q[e_1^2e_{5}^2].
\end{align*}

\subsection{An infinite sequence of web clusters}
We construct an infinite sequence of web clusters starting from the web cluster associated with the decorated triangulation
$
    \mathord{
        \ \tikz[baseline=.5em, scale=.1]{
            \draw[thick] (0,0) rectangle (6,6);
            \draw[blue] (0,6) -- (6,0);
            \node at (1.5,2.5) {\scriptsize ${+}$};
            \node at (4,4) {\scriptsize ${-}$};
            \node at (3.5,1) {\scriptsize $\ast$};
            \node at (5,2) {\scriptsize $\ast$};
            \fill (0,0) circle (15pt);
            \fill (6,0) circle (15pt);
            \fill (6,6) circle (15pt);
            \fill (0,6) circle (15pt);
        }\ 
    }.
$
We denote by $p_0,p_1,p_2,p_3$ the special points in $\Sigma$ in the counterclockwise order from the upper left corner.
We represent an elementary web in $\Sigma$ by a diagram in a band with infinitely many linearly ordered special points labeled by $p_0, p_1, p_2, p_3$ on its bottom side. See \cref{fig:covering}. 
The band is viewed as an infinite cyclic covering of $\Sigma$, and the actual elementary web is obtained through the projection.
At a crossing of the projection of the diagram, we draw the $\mathfrak{sp}_4$-web in such a way that the right arc in the band becomes over-passing.
If all the crossings in the projected diagram are arborizable or at special points, then we can uniquely determine the flat $\mathfrak{sp}_4$-web representing this projected diagram.
\begin{figure}
    \begin{tikzpicture}
        \node at (0,0)
        {
            \ \tikz[baseline=-.6ex, scale=.1]{
                \foreach \i in {0,1,...,10}
                {
                    \coordinate (A\i) at (5*\i,0);
                    \coordinate (B\i) at (5*\i,15);
                }
                \draw[dashed] ($(B0)+(-5,0)$) -- ($(B10)+(5,0)$);
                \draw[dashed, thick] ($(A0)+(-5,0)$) -- ($(A10)+(5,0)$);
                \bdryline{(A0)}{(A10)}{2cm}
                \foreach \j [evaluate=\j as \k using {int(mod(\j-1,4))}] in {1,2,...,9}
                {
                    \draw[dotted] (A\j) -- (B\j);
                    \fill (A\j) circle (20pt);
                    \node at (A\j) [below]{\scriptsize $p_{\k}$};
                }
            }\ 
        };
        \node at (5,0)
        {
            \ \tikz[baseline=-.6ex, scale=.1]{
                \draw[thick, ->] (0,0) -- (10,0);
                \node at (5,0)[below]{\small projection};
            }\ 
        };
        \node at (8,0)
        {
            \ \tikz[baseline=-.6ex, scale=.1]{
                \foreach \i in {0,0.5,...,4}
                {
                    \coordinate (P\i) at (90*\i+135:10);
                }
                \foreach \j/\k in {0/1,1/2,2/3,3/4}
                {
                    \bdryline{(P\j)}{(P\k)}{2cm}
                }
                \foreach \j in {0,1,2,3}
                {
                    \draw[dotted] (P\j) -- (0,0);
                    \node at ($(0,0)!1.4!(P\j)$) {\scriptsize $p_{\j}$};
                    \fill (P\j) circle (20pt);
                }
            }\ 
        };
    \end{tikzpicture}
    \caption{The bottom side corresponds to the boundary of $\Sigma$ and the top side to the center of $\Sigma$.}
    \label{fig:covering}
\end{figure}
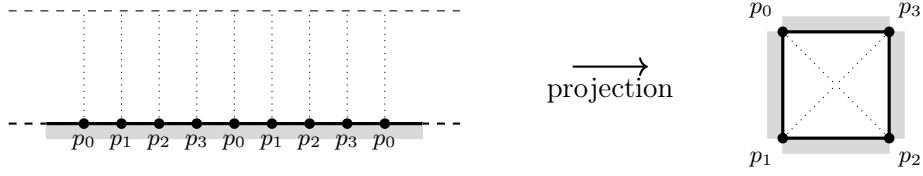

For $i \in \bZ$ and $k\in\bZ_{\geq 0}$, let us consider the elementary webs $x_{i}^{(k)}$ and $y_{i}^{(k)}$ given by
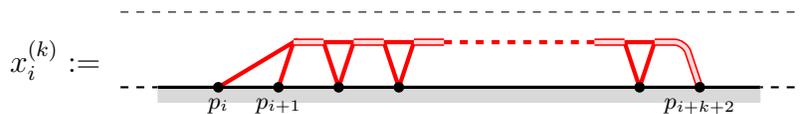
\begin{align*}
    x_{i}^{(k)}&:=
    \mathord{
        \ \tikz[baseline=.5em, scale=.1]{
            \foreach \i in {0,1,...,10}
            {
                \coordinate (A\i) at (8*\i,0);
                \coordinate (B\i) at (8*\i,10);
            }
            \draw[webline] (A3) -- ($(A3)+(-2,6)$) -- ($(A3)+(2,6)$) -- (A3);
            \draw[webline] (A4) -- ($(A4)+(-2,6)$) -- ($(A4)+(2,6)$) -- (A4);
            \draw[webline] (A8) -- ($(A8)+(-2,6)$) -- ($(A8)+(2,6)$) -- (A8);
            \draw[webline] (A1) -- ($(A2)+(2,6)$) -- (A2);
            \draw[wline] ($(A3)+(-2,6)$) -- ($(A2)+(2,6)$);
            \draw[wline] ($(A4)+(-2,6)$) -- ($(A3)+(2,6)$);
            \draw[wline] ($(A5)+(-2,6)$) -- ($(A4)+(2,6)$);
            \draw[webline, dashed] ($(A5)+(-2,6)$) -- ($(A7)+(2,6)$);
            \draw[wline] ($(A8)+(-2,6)$) -- ($(A7)+(2,6)$);
            \draw[wline, rounded corners] (A9) -- ($(A9)+(-2,6)$) -- ($(A8)+(2,6)$);
            \draw[dashed] ($(B0)+(-5,0)$) -- ($(B10)+(5,0)$);
            \draw[dashed, thick] ($(A0)+(-5,0)$) -- ($(A10)+(5,0)$);
            \bdryline{(A0)}{(A10)}{2cm}
            \foreach \j in {1,2,3,4,8,9}
            {
                \fill (A\j) circle (20pt);
            }
            \node at (A1) [below]{\scriptsize $p_i$};
            \node at (A2) [below]{\scriptsize $p_{i+1}$};
            \node at (A9) [below]{\scriptsize $p_{i+k+2}$};
        }\ 
    },\\[1em]
    y_{i}^{(k)}&:=
    \mathord{
        \ 
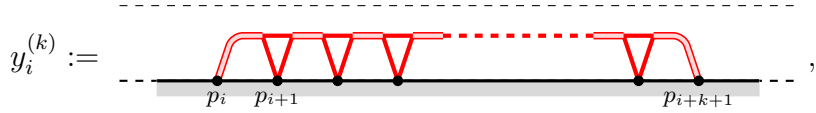
\tikz[baseline=.5em, scale=.1]{
            \foreach \i in {0,1,...,10}
            {
                \coordinate (A\i) at (8*\i,0);
                \coordinate (B\i) at (8*\i,10);
            }
            \draw[webline] (A2) -- ($(A2)+(-2,6)$) -- ($(A2)+(2,6)$) -- (A2);
            \draw[webline] (A3) -- ($(A3)+(-2,6)$) -- ($(A3)+(2,6)$) -- (A3);
            \draw[webline] (A4) -- ($(A4)+(-2,6)$) -- ($(A4)+(2,6)$) -- (A4);
            \draw[webline] (A8) -- ($(A8)+(-2,6)$) -- ($(A8)+(2,6)$) -- (A8);
            \draw[wline] ($(A3)+(-2,6)$) -- ($(A2)+(2,6)$);
            \draw[wline] ($(A4)+(-2,6)$) -- ($(A3)+(2,6)$);
            \draw[wline] ($(A5)+(-2,6)$) -- ($(A4)+(2,6)$);
            \draw[webline, dashed] ($(A5)+(-2,6)$) -- ($(A7)+(2,6)$);
            \draw[wline] ($(A8)+(-2,6)$) -- ($(A7)+(2,6)$);
            \draw[wline, rounded corners] ($(A2)+(-2,6)$) -- ($(A1)+(2,6)$) -- (A1);
            \draw[wline, rounded corners] (A9) -- ($(A9)+(-2,6)$) -- ($(A8)+(2,6)$);
            \draw[dashed] ($(B0)+(-5,0)$) -- ($(B10)+(5,0)$);
            \draw[dashed, thick] ($(A0)+(-5,0)$) -- ($(A10)+(5,0)$);
            \bdryline{(A0)}{(A10)}{2cm}
            \foreach \j in {1,2,3,4,8,9}
            {
                \fill (A\j) circle (20pt);
            }
            \node at (A1) [below]{\scriptsize $p_i$};
            \node at (A2) [below]{\scriptsize $p_{i+1}$};
            \node at (A9) [below]{\scriptsize $p_{i+k+1}$};
        }\ 
    },
\end{align*}
where subscripts $i$ for $p$, $x$, and $y$ are considered modulo $4$.
We denote by $e_{i,j}$ the simple arc of type~$1$ between the special points $p_i$ and $p_j$ for any distinct integers $i$ and $j$. 
\begin{prop}\label{prop:infty-web-clusters}
    For $i,k \in \bZ_{\geq 0}$, we have the relation
        \begin{align}
            \eta_{i}^{(k)}\colon x_{i}^{(k+4)}x_{i}^{(k)}=q^{m_k}\bigg[ e_{i,i+2}e_{i,i+1}e_{i+1,i+2}\bigg(\prod_{j=i}^{k-1}(e_{j+2,j+3})^{2}\bigg)y_{i+k+2}^{(3)}\bigg]+q^{n_k}(x_{i+2}^{(k+2)})^2, \label{eq:infty-exch}
        \end{align}
    for some $m_k,n_k \in \frac{1}{2}\bZ$.
    Moreover, the collections $\cC_{k}:=\{x_{2k+2}^{(2k)},x_{2k}^{(2k+2)},y_{0}^{(1)},y_{0}^{(3)},y_{2}^{(1)},e_{0,3}\}\cup\partial_{\Sigma}$ are web clusters for any $k \in \bZ_{\geq 0}$.
\end{prop}
\begin{proof}
    It suffices to show \eqref{eq:infty-exch} for $i=0$, since $x_{i}^{(k)}$ is obtained by shifting $x_{0}^{(k)}$ to the right by $i$ special points.
    In the calculation below, we tacitly exchange over- and under-passing arcs when a crossing arborizes at special points, discarding the explicit description of the exponents of $q$ appearing in such exchanges.
    \begin{align*}
        &x_{0}^{(k+4)}x_{0}^{(k)}
        =q^{\bullet}
        \mathord{
            \ 
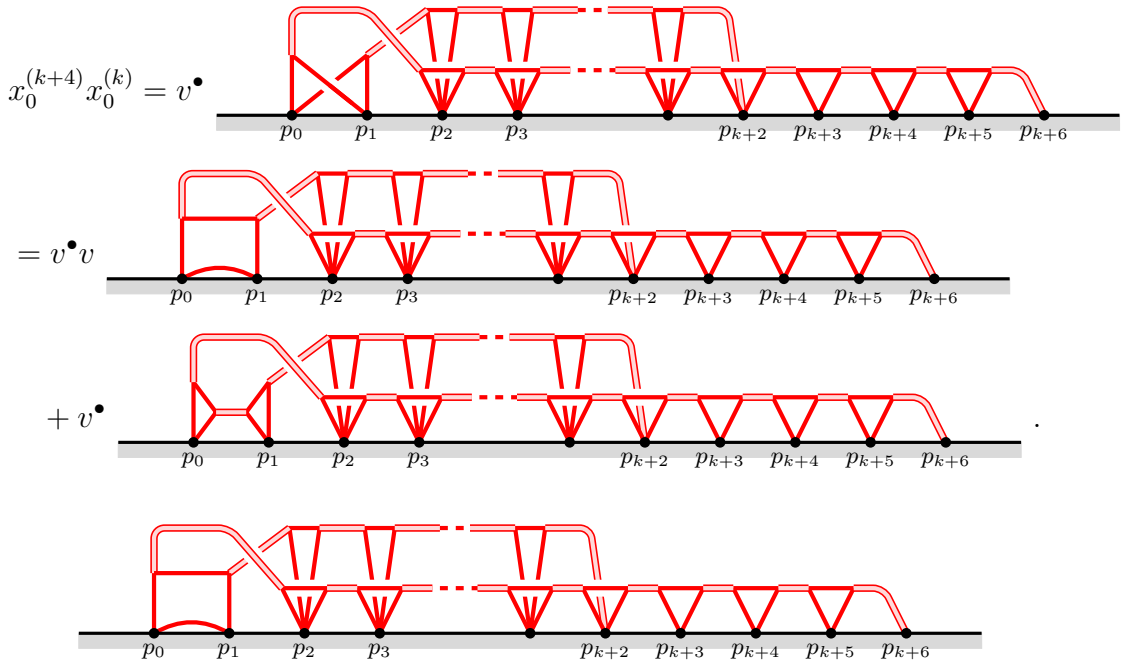
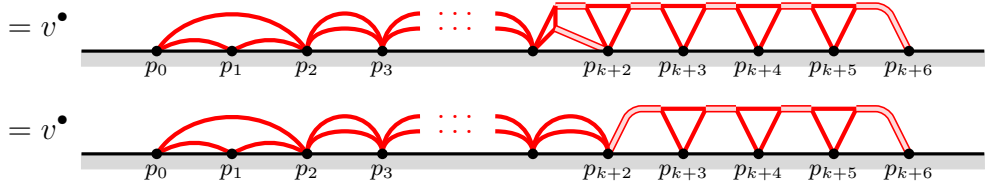
\tikz[baseline=.5em, scale=.1]{
                \foreach \i in {0,1,...,14}
                {
                    \coordinate (A\i) at (10*\i,0);
                    \coordinate (B\i) at (10*\i,8);
                }
                \begin{scope}
                    \draw[webline, dashed] ($(B5)+(-2,6)$) -- ($(B5)+(2,6)$);
                    \foreach \j in {3,4,6}
                    {
                        \draw[webline] (A\j) -- ($(B\j)+(-2,6)$) -- ($(B\j)+(2,6)$) -- (A\j);
                    }
                    \foreach \i [evaluate=\i as \j using {int(\i-1)}] in {4,5,6}
                    {
                        \draw[wline] ($(B\i)+(-2,6)$) -- ($(B\j)+(2,6)$);
                    }
                    \draw[wline] ($(B3)+(-2,6)$) -- (B2);
                    \draw[webline] (A1) -- (B2);
                    \draw[webline] (A2) -- (B2);
                    \draw[wline, rounded corners] (A7) -- ($(B7)+(-2,6)$) -- ($(B6)+(2,6)$);
                \end{scope}
                \begin{scope}
                    \draw[webline, dashed] ($(A5)+(-2,6)$) -- ($(A9)+(2,6)$);
                    \foreach \j in {3,4,6,7,8,9,10}
                    {
                        \draw[overarc] ($(A\j)+(-3,6)$) -- ($(A\j)+(3,6)$);
                        \draw[webline] (A\j) -- ($(A\j)+(-3,6)$);
                        \draw[webline] ($(A\j)+(3,6)$) -- (A\j);
                    }
                    \foreach \i [evaluate=\i as \j using {int(\i-1)}] in {4,5,6,7,8,9,10}
                    {
                        \draw[wline] ($(A\i)+(-3,6)$) -- ($(A\j)+(3,6)$);
                    }
                    \draw[overarc] (B1) -- ($(A2)!.2!(B1)$);
                    \draw[webline] (A2) -- ($(A2)!.5!(B1)$);
                    \draw[overwline, rounded corners] (B1) -- ($(B1)+(0,6)$) -- ($(B2)+(0,6)$) -- ($(A3)+(-3,6)$);
                    \draw[webline] (A1) -- (B1);
                    \draw[wline, rounded corners] (A11) -- ($(A11)+(-3,6)$) -- ($(A10)+(3,6)$);
                \end{scope}
                \bdryline{(A0)}{(A12)}{2cm}
                \foreach \j in {1,2,3,4,6,7,8,9,10,11}
                {
                    \fill (A\j) circle (20pt);
                }
                \node at (A1) [below]{\scriptsize $p_{0}$};
                \node at (A2) [below]{\scriptsize $p_{1}$};
                \node at (A3) [below]{\scriptsize $p_{2}$};
                \node at (A4) [below]{\scriptsize $p_{3}$};
                \node at (A7) [below]{\scriptsize $p_{k+2}$};
                \node at (A8) [below]{\scriptsize $p_{k+3}$};
                \node at (A9) [below]{\scriptsize $p_{k+4}$};
                \node at (A10) [below]{\scriptsize $p_{k+5}$};
                \node at (A11) [below]{\scriptsize $p_{k+6}$};
            }\ 
        }\\
        &=q^{\bullet}q
        \mathord{
            \ \tikz[baseline=.5em, scale=.1]{
                \foreach \i in {0,1,...,14}
                {
                    \coordinate (A\i) at (10*\i,0);
                    \coordinate (B\i) at (10*\i,8);
                }
                \begin{scope}
                    \draw[webline, dashed] ($(B5)+(-2,6)$) -- ($(B5)+(2,6)$);
                    \foreach \j in {3,4,6}
                    {
                        \draw[webline] (A\j) -- ($(B\j)+(-2,6)$) -- ($(B\j)+(2,6)$) -- (A\j);
                    }
                    \foreach \i [evaluate=\i as \j using {int(\i-1)}] in {4,5,6}
                    {
                        \draw[wline] ($(B\i)+(-2,6)$) -- ($(B\j)+(2,6)$);
                    }
                    \draw[wline] ($(B3)+(-2,6)$) -- (B2);
                    \draw[webline] (A2) -- (B2);
                    \draw[wline, rounded corners] (A7) -- ($(B7)+(-2,6)$) -- ($(B6)+(2,6)$);
                \end{scope}
                \begin{scope}
                    \draw[webline, dashed] ($(A5)+(-2,6)$) -- ($(A9)+(2,6)$);
                    \foreach \j in {3,4,6,7,8,9,10}
                    {
                        \draw[overarc] ($(A\j)+(-3,6)$) -- ($(A\j)+(3,6)$);
                        \draw[webline] (A\j) -- ($(A\j)+(-3,6)$);
                        \draw[webline] ($(A\j)+(3,6)$) -- (A\j);
                    }
                    \foreach \i [evaluate=\i as \j using {int(\i-1)}] in {4,5,6,7,8,9,10}
                    {
                        \draw[wline] ($(A\i)+(-3,6)$) -- ($(A\j)+(3,6)$);
                    }
                    \draw[overwline, rounded corners] (B1) -- ($(B1)+(0,6)$) -- ($(B2)+(0,6)$) -- ($(A3)+(-3,6)$);
                    \draw[webline]  (B1) -- (B2);
                    \draw[webline]  (A1) to[bend left] (A2);
                    \draw[webline] (A1) -- (B1);
                    \draw[wline, rounded corners] (A11) -- ($(A11)+(-3,6)$) -- ($(A10)+(3,6)$);
                \end{scope}
                \bdryline{(A0)}{(A12)}{2cm}
                \foreach \j in {1,2,3,4,6,7,8,9,10,11}
                {
                    \fill (A\j) circle (20pt);
                }
                \node at (A1) [below]{\scriptsize $p_{0}$};
                \node at (A2) [below]{\scriptsize $p_{1}$};
                \node at (A3) [below]{\scriptsize $p_{2}$};
                \node at (A4) [below]{\scriptsize $p_{3}$};
                \node at (A7) [below]{\scriptsize $p_{k+2}$};
                \node at (A8) [below]{\scriptsize $p_{k+3}$};
                \node at (A9) [below]{\scriptsize $p_{k+4}$};
                \node at (A10) [below]{\scriptsize $p_{k+5}$};
                \node at (A11) [below]{\scriptsize $p_{k+6}$};
            }\ 
        }\\
        &\quad+q^{\bullet}
        \mathord{
            \ \tikz[baseline=.5em, scale=.1]{
                \foreach \i in {0,1,...,14}
                {
                    \coordinate (A\i) at (10*\i,0);
                    \coordinate (B\i) at (10*\i,8);
                }
                \begin{scope}
                    \draw[webline, dashed] ($(B5)+(-2,6)$) -- ($(B5)+(2,6)$);
                    \foreach \j in {3,4,6}
                    {
                        \draw[webline] (A\j) -- ($(B\j)+(-2,6)$) -- ($(B\j)+(2,6)$) -- (A\j);
                    }
                    \foreach \i [evaluate=\i as \j using {int(\i-1)}] in {4,5,6}
                    {
                        \draw[wline] ($(B\i)+(-2,6)$) -- ($(B\j)+(2,6)$);
                    }
                    \draw[wline] ($(B3)+(-2,6)$) -- (B2);
                    \draw[webline] (A2) -- (B2);
                    \draw[wline, rounded corners] (A7) -- ($(B7)+(-2,6)$) -- ($(B6)+(2,6)$);
                \end{scope}
                \begin{scope}
                    \draw[webline, dashed] ($(A5)+(-2,6)$) -- ($(A9)+(2,6)$);
                    \foreach \j in {3,4,6,7,8,9,10}
                    {
                        \draw[overarc] ($(A\j)+(-3,6)$) -- ($(A\j)+(3,6)$);
                        \draw[webline] (A\j) -- ($(A\j)+(-3,6)$);
                        \draw[webline] ($(A\j)+(3,6)$) -- (A\j);
                    }
                    \foreach \i [evaluate=\i as \j using {int(\i-1)}] in {4,5,6,7,8,9,10}
                    {
                        \draw[wline] ($(A\i)+(-3,6)$) -- ($(A\j)+(3,6)$);
                    }
                    \draw[webline] (A1) -- ($(A1)!.5!(B2)+(-2,0)$) -- (B1);
                    \draw[webline] (A2) -- ($(A2)!.5!(B1)+(2,0)$) -- (B2);
                    \draw[wline] ($(A1)!.5!(B2)+(-2,0)$) -- ($(A2)!.5!(B1)+(2,0)$);
                    \draw[overwline, rounded corners] (B1) -- ($(B1)+(0,6)$) -- ($(B2)+(0,6)$) -- ($(A3)+(-3,6)$);
                    \draw[webline] (A1) -- (B1);
                    \draw[wline, rounded corners] (A11) -- ($(A11)+(-3,6)$) -- ($(A10)+(3,6)$);
                \end{scope}
                \bdryline{(A0)}{(A12)}{2cm}
                \foreach \j in {1,2,3,4,6,7,8,9,10,11}
                {
                    \fill (A\j) circle (20pt);
                }
                \node at (A1) [below]{\scriptsize $p_{0}$};
                \node at (A2) [below]{\scriptsize $p_{1}$};
                \node at (A3) [below]{\scriptsize $p_{2}$};
                \node at (A4) [below]{\scriptsize $p_{3}$};
                \node at (A7) [below]{\scriptsize $p_{k+2}$};
                \node at (A8) [below]{\scriptsize $p_{k+3}$};
                \node at (A9) [below]{\scriptsize $p_{k+4}$};
                \node at (A10) [below]{\scriptsize $p_{k+5}$};
                \node at (A11) [below]{\scriptsize $p_{k+6}$};
            }\ .
        }
    \end{align*}
    Applying \cref{lem:face-vanishing} to the first term, we get:
    \begin{align*}
        &\mathord{
            \ \tikz[baseline=.5em, scale=.1]{
                \foreach \i in {0,1,...,14}
                {
                    \coordinate (A\i) at (10*\i,0);
                    \coordinate (B\i) at (10*\i,8);
                }
                \begin{scope}
                    \draw[webline, dashed] ($(B5)+(-2,6)$) -- ($(B5)+(2,6)$);
                    \foreach \j in {3,4,6}
                    {
                        \draw[webline] (A\j) -- ($(B\j)+(-2,6)$) -- ($(B\j)+(2,6)$) -- (A\j);
                    }
                    \foreach \i [evaluate=\i as \j using {int(\i-1)}] in {4,5,6}
                    {
                        \draw[wline] ($(B\i)+(-2,6)$) -- ($(B\j)+(2,6)$);
                    }
                    \draw[wline] ($(B3)+(-2,6)$) -- (B2);
                    \draw[webline] (A2) -- (B2);
                    \draw[wline, rounded corners] (A7) -- ($(B7)+(-2,6)$) -- ($(B6)+(2,6)$);
                \end{scope}
                \begin{scope}
                    \draw[webline, dashed] ($(A5)+(-2,6)$) -- ($(A9)+(2,6)$);
                    \foreach \j in {3,4,6,7,8,9,10}
                    {
                        \draw[overarc] ($(A\j)+(-3,6)$) -- ($(A\j)+(3,6)$);
                        \draw[webline] (A\j) -- ($(A\j)+(-3,6)$);
                        \draw[webline] ($(A\j)+(3,6)$) -- (A\j);
                    }
                    \foreach \i [evaluate=\i as \j using {int(\i-1)}] in {4,5,6,7,8,9,10}
                    {
                        \draw[wline] ($(A\i)+(-3,6)$) -- ($(A\j)+(3,6)$);
                    }
                    \draw[overwline, rounded corners] (B1) -- ($(B1)+(0,6)$) -- ($(B2)+(0,6)$) -- ($(A3)+(-3,6)$);
                    \draw[webline]  (B1) -- (B2);
                    \draw[webline]  (A1) to[bend left] (A2);
                    \draw[webline] (A1) -- (B1);
                    \draw[wline, rounded corners] (A11) -- ($(A11)+(-3,6)$) -- ($(A10)+(3,6)$);
                \end{scope}
                \bdryline{(A0)}{(A12)}{2cm}
                \foreach \j in {1,2,3,4,6,7,8,9,10,11}
                {
                    \fill (A\j) circle (20pt);
                }
                \node at (A1) [below]{\scriptsize $p_{0}$};
                \node at (A2) [below]{\scriptsize $p_{1}$};
                \node at (A3) [below]{\scriptsize $p_{2}$};
                \node at (A4) [below]{\scriptsize $p_{3}$};
                \node at (A7) [below]{\scriptsize $p_{k+2}$};
                \node at (A8) [below]{\scriptsize $p_{k+3}$};
                \node at (A9) [below]{\scriptsize $p_{k+4}$};
                \node at (A10) [below]{\scriptsize $p_{k+5}$};
                \node at (A11) [below]{\scriptsize $p_{k+6}$};
            }\ 
        }\\
       &=q^{\bullet}
        \mathord{
            \ \tikz[baseline=.5em, scale=.1]{
                \foreach \i in {0,1,...,14}
                {
                    \coordinate (A\i) at (10*\i,0);
                    \coordinate (B\i) at (10*\i,8);
                }
                \begin{scope}
                    \foreach \j in {7,8,9,10}
                    {
                        \draw[overarc] ($(A\j)+(-3,6)$) -- ($(A\j)+(3,6)$);
                        \draw[webline] (A\j) -- ($(A\j)+(-3,6)$);
                        \draw[webline] ($(A\j)+(3,6)$) -- (A\j);
                    }
                    \foreach \i [evaluate=\i as \j using {int(\i-1)}] in {7,8,9,10}
                    {
                        \draw[wline] ($(A\i)+(-3,6)$) -- ($(A\j)+(3,6)$);
                    }
                    \draw[webline]  (A1) to[bend left] (A2);
                    \draw[webline]  (A2) to[bend left] (A3);
                    \draw[webline]  (A1) to[bend left=2cm] (A3);
                    \draw[webline]  (A3) to[out=north, in=west] ($(A3)!.5!(A4)+(0,3)$) to[out=east, in=north] (A4);
                    \draw[webline]  (A3) to[out=north, in=west] ($(A3)!.5!(A4)+(0,5)$) to[out=east, in=north] (A4);
                    \draw[webline]  (A4) to[out=north, in=west] ($(A4)!.5!(A5)+(0,3)$);
                    \draw[webline]  (A4) to[out=north, in=west] ($(A4)!.5!(A5)+(0,5)$);
                    \draw[webline]  ($(A5)!.5!(A6)+(0,3)$) to[out=east, in=north] (A6);
                    \draw[webline]  ($(A5)!.5!(A6)+(0,5)$) to[out=east, in=north] (A6);
                    \node at ($(A5)+(0,5)$) [red]{$\dots$};
                    \node at ($(A5)+(0,3)$) [red]{$\dots$};
                    \draw[wline, rounded corners] (A11) -- ($(A11)+(-3,6)$) -- ($(A10)+(3,6)$);
                    \draw[webline] (A6) -- ($(A6)+(3,6)$);
                    \draw[webline] ($(A6)+(3,6)$) -- ($(A6)+(3,3)$);
                    \draw[webline] ($(A6)+(3,3)$) -- (A6);
                    \draw[wline] ($(A6)+(3,3)$) -- (A7);
                \end{scope}
                \bdryline{(A0)}{(A12)}{2cm}
                \foreach \j in {1,2,3,4,6,7,8,9,10,11}
                {
                    \fill (A\j) circle (20pt);
                }
                \node at (A1) [below]{\scriptsize $p_{0}$};
                \node at (A2) [below]{\scriptsize $p_{1}$};
                \node at (A3) [below]{\scriptsize $p_{2}$};
                \node at (A4) [below]{\scriptsize $p_{3}$};
                \node at (A7) [below]{\scriptsize $p_{k+2}$};
                \node at (A8) [below]{\scriptsize $p_{k+3}$};
                \node at (A9) [below]{\scriptsize $p_{k+4}$};
                \node at (A10) [below]{\scriptsize $p_{k+5}$};
                \node at (A11) [below]{\scriptsize $p_{k+6}$};
            }\ 
        }\\
        &=q^{\bullet}
        \mathord{
            \ \tikz[baseline=.5em, scale=.1]{
                \foreach \i in {0,1,...,14}
                {
                    \coordinate (A\i) at (10*\i,0);
                    \coordinate (B\i) at (10*\i,8);
                }
                \begin{scope}
                    \foreach \j in {8,9,10}
                    {
                        \draw[overarc] ($(A\j)+(-3,6)$) -- ($(A\j)+(3,6)$);
                        \draw[webline] (A\j) -- ($(A\j)+(-3,6)$);
                        \draw[webline] ($(A\j)+(3,6)$) -- (A\j);
                    }
                    \foreach \i [evaluate=\i as \j using {int(\i-1)}] in {9,10}
                    {
                        \draw[wline] ($(A\i)+(-3,6)$) -- ($(A\j)+(3,6)$);
                    }
                    \draw[wline, rounded corners] ($(A8)+(-3,6)$) -- ($(A7)+(3,6)$) -- (A7);
                    \draw[webline]  (A1) to[bend left] (A2);
                    \draw[webline]  (A2) to[bend left] (A3);
                    \draw[webline]  (A1) to[bend left=2cm] (A3);
                    \draw[webline]  (A3) to[out=north, in=west] ($(A3)!.5!(A4)+(0,3)$) to[out=east, in=north] (A4);
                    \draw[webline]  (A3) to[out=north, in=west] ($(A3)!.5!(A4)+(0,5)$) to[out=east, in=north] (A4);
                    \draw[webline]  (A4) to[out=north, in=west] ($(A4)!.5!(A5)+(0,3)$);
                    \draw[webline]  (A4) to[out=north, in=west] ($(A4)!.5!(A5)+(0,5)$);
                    \draw[webline]  ($(A5)!.5!(A6)+(0,3)$) to[out=east, in=north] (A6);
                    \draw[webline]  ($(A5)!.5!(A6)+(0,5)$) to[out=east, in=north] (A6);
                    \node at ($(A5)+(0,5)$) [red]{$\dots$};
                    \node at ($(A5)+(0,3)$) [red]{$\dots$};
                    \draw[webline]  (A6) to[out=north, in=west] ($(A6)!.5!(A7)+(0,3)$) to[out=east, in=north] (A7);
                    \draw[webline]  (A6) to[out=north, in=west] ($(A6)!.5!(A7)+(0,5)$) to[out=east, in=north] (A7);
                    \draw[wline, rounded corners] (A11) -- ($(A11)+(-3,6)$) -- ($(A10)+(3,6)$);
                \end{scope}
                \bdryline{(A0)}{(A12)}{2cm}
                \foreach \j in {1,2,3,4,6,7,8,9,10,11}
                {
                    \fill (A\j) circle (20pt);
                }
                \node at (A1) [below]{\scriptsize $p_{0}$};
                \node at (A2) [below]{\scriptsize $p_{1}$};
                \node at (A3) [below]{\scriptsize $p_{2}$};
                \node at (A4) [below]{\scriptsize $p_{3}$};
                \node at (A7) [below]{\scriptsize $p_{k+2}$};
                \node at (A8) [below]{\scriptsize $p_{k+3}$};
                \node at (A9) [below]{\scriptsize $p_{k+4}$};
                \node at (A10) [below]{\scriptsize $p_{k+5}$};
                \node at (A11) [below]{\scriptsize $p_{k+6}$};
            }\ 
        }.
    \end{align*}
    We can take another lift of $\mathfrak{sp}_4$-web in the second term as
    \begin{align*}
        \mathord{
            \ 
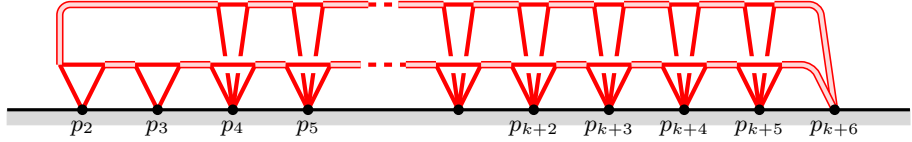
\tikz[baseline=.5em, scale=.1]{
                \foreach \i in {0,1,...,14}
                {
                    \coordinate (A\i) at (10*\i,0);
                    \coordinate (B\i) at (10*\i,8);
                }
                \begin{scope}
                    \draw[webline, dashed] ($(B5)+(-2,6)$) -- ($(B5)+(2,6)$);
                    \foreach \j in {3,4,6,7,8,9,10}
                    {
                        \draw[webline] (A\j) -- ($(B\j)+(-2,6)$) -- ($(B\j)+(2,6)$) -- (A\j);
                    }
                    \foreach \i [evaluate=\i as \j using {int(\i-1)}] in {4,5,6,7,8,9,10}
                    {
                        \draw[wline] ($(B\i)+(-2,6)$) -- ($(B\j)+(2,6)$);
                    }
                    \draw[wline, rounded corners] (A11) -- ($(B11)+(-2,6)$) -- ($(B10)+(2,6)$);
                \end{scope}
                \begin{scope}
                    \draw[webline, dashed] ($(A5)+(-2,6)$) -- ($(A9)+(2,6)$);
                    \foreach \j in {1,2,3,4,6,7,8,9,10}
                    {
                        \draw[overarc] ($(A\j)+(-3,6)$) -- ($(A\j)+(3,6)$);
                        \draw[webline] (A\j) -- ($(A\j)+(-3,6)$);
                        \draw[webline] ($(A\j)+(3,6)$) -- (A\j);
                    }
                    \foreach \i [evaluate=\i as \j using {int(\i-1)}] in {2,3,...,10}
                    {
                        \draw[wline] ($(A\i)+(-3,6)$) -- ($(A\j)+(3,6)$);
                    }
                    \draw[wline, rounded corners] ($(A1)+(-3,6)$) -- ($(B1)+(-3,6)$) -- ($(B3)+(-2,6)$);
                    \draw[wline, rounded corners] (A11) -- ($(A11)+(-3,6)$) -- ($(A10)+(3,6)$);
                \end{scope}
                \bdryline{(A0)}{(A12)}{2cm}
                \foreach \j in {1,2,3,4,6,7,8,9,10,11}
                {
                    \fill (A\j) circle (20pt);
                }
                \node at (A1) [below]{\scriptsize $p_{2}$};
                \node at (A2) [below]{\scriptsize $p_{3}$};
                \node at (A3) [below]{\scriptsize $p_{4}$};
                \node at (A4) [below]{\scriptsize $p_{5}$};
                \node at (A7) [below]{\scriptsize $p_{k+2}$};
                \node at (A8) [below]{\scriptsize $p_{k+3}$};
                \node at (A9) [below]{\scriptsize $p_{k+4}$};
                \node at (A10) [below]{\scriptsize $p_{k+5}$};
                \node at (A11) [below]{\scriptsize $p_{k+6}$};
            }\ .
        }
    \end{align*}
    Furthermore, one can confirm the above $\mathfrak{sp}_4$-web is the same as
    \begin{align*}
        \mathord{
            \ 
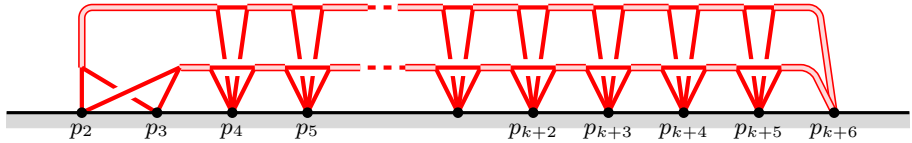
\tikz[baseline=.5em, scale=.1]{
                \foreach \i in {0,1,...,14}
                {
                    \coordinate (A\i) at (10*\i,0);
                    \coordinate (B\i) at (10*\i,8);
                }
                \begin{scope}
                    \draw[webline, dashed] ($(B5)+(-2,6)$) -- ($(B5)+(2,6)$);
                    \foreach \j in {3,4,6,7,8,9,10}
                    {
                        \draw[webline] (A\j) -- ($(B\j)+(-2,6)$) -- ($(B\j)+(2,6)$) -- (A\j);
                    }
                    \foreach \i [evaluate=\i as \j using {int(\i-1)}] in {4,5,6,7,8,9,10}
                    {
                        \draw[wline] ($(B\i)+(-2,6)$) -- ($(B\j)+(2,6)$);
                    }
                    \draw[wline, rounded corners] (A11) -- ($(B11)+(-2,6)$) -- ($(B10)+(2,6)$);
                \end{scope}
                \begin{scope}
                    \draw[webline, dashed] ($(A5)+(-2,6)$) -- ($(A9)+(2,6)$);
                    \foreach \j in {3,4,6,7,8,9,10}
                    {
                        \draw[overarc] ($(A\j)+(-3,6)$) -- ($(A\j)+(3,6)$);
                        \draw[webline] (A\j) -- ($(A\j)+(-3,6)$);
                        \draw[webline] ($(A\j)+(3,6)$) -- (A\j);
                    }
                    \foreach \i [evaluate=\i as \j using {int(\i-1)}] in {3,4,...,10}
                    {
                        \draw[wline] ($(A\i)+(-3,6)$) -- ($(A\j)+(3,6)$);
                    }
                    \draw[wline, rounded corners] ($(A1)+(0,6)$) -- ($(B1)+(0,6)$) -- ($(B3)+(-2,6)$);
                    \draw[wline, rounded corners] (A11) -- ($(A11)+(-3,6)$) -- ($(A10)+(3,6)$);
                    \draw[webline] (A2) -- ($(A1)+(0,6)$);
                    \draw[overarc] (A1) -- ($(A2)+(3,6)$);
                    \draw[webline] (A1) -- ($(A1)+(0,6)$);
                    \draw[webline] (A2) -- ($(A2)+(3,6)$);
                \end{scope}
                \bdryline{(A0)}{(A12)}{2cm}
                \foreach \j in {1,2,3,4,6,7,8,9,10,11}
                {
                    \fill (A\j) circle (20pt);
                }
                \node at (A1) [below]{\scriptsize $p_{2}$};
                \node at (A2) [below]{\scriptsize $p_{3}$};
                \node at (A3) [below]{\scriptsize $p_{4}$};
                \node at (A4) [below]{\scriptsize $p_{5}$};
                \node at (A7) [below]{\scriptsize $p_{k+2}$};
                \node at (A8) [below]{\scriptsize $p_{k+3}$};
                \node at (A9) [below]{\scriptsize $p_{k+4}$};
                \node at (A10) [below]{\scriptsize $p_{k+5}$};
                \node at (A11) [below]{\scriptsize $p_{k+6}$};
            }\ 
        }
    \end{align*}
    again by \cref*{lem:face-vanishing}.
    Hence,
    \begin{align*}
        x_{0}^{(k+4)}x_{0}^{(k)}=q^{\bullet}\bigg[ e_{0,2}e_{0,1}e_{1,2}\left(\prod_{j=2}^{k+1}(e_{j,j+1})^2\right)y_{k+2}^{(3)}\bigg]+q^{\bullet}(x_{2}^{(k+2)})^2.
    \end{align*}
    Firstly, it is easy to see that $\mathfrak{sp}_4$-webs in $\cC_k$ are $v$-commutative by the arborization relations in \cref{lem:arborization}, and that $\cC_{k}$ is related to $\cC_{k+1}$ by $\eta_{2k+2}^{(2k)}$ in \eqref{eq:infty-exch}.
\end{proof}
The mutation sequence among the web clusters $\cC_k$ are the following:
    
    
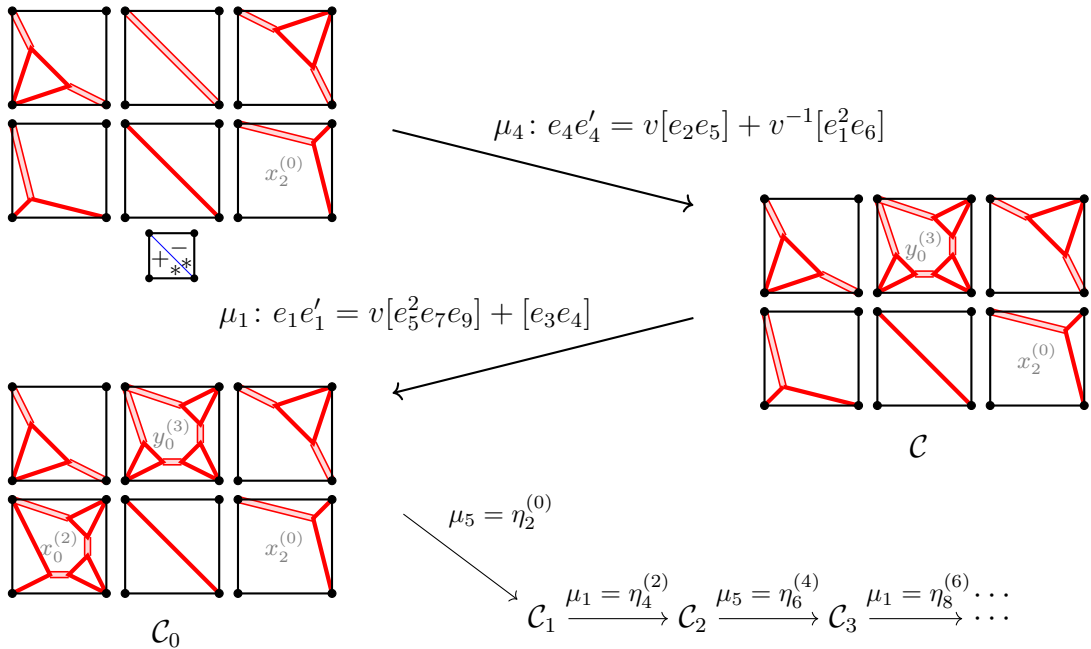
\begin{tikzpicture}
        \begin{scope}[xshift=6.5cm, yshift=0cm]
            \draw[->, thick] (-2,.5) -- (2,-.5);
            \node at ($(-2,.5)!.3!(2,-.5)$) [above right]{$\mu_{4}\colon e_4e_4'=q[e_2e_5]+q^{-1}[e_1^2e_6]$};
        \end{scope}
        \begin{scope}[xshift=6.5cm, yshift=-2.5cm]
            \draw[->, thick] (2,.5) -- (-2,-.5);
            \node at ($(2,.5)!.3!(-2,-.5)$) [above left]{$\mu_{1}\colon e_1e_1'=q[e_5^2e_7e_9]+[e_3e_4]$};
        \end{scope}
        \begin{scope}[xshift=6.5cm, yshift=-5cm]
            \node (C0) at (-2,.5) {};
            \node (C1) at (0,-1) {$\cC_1$};
            \node (C2) at (2,-1) {$\cC_2$};
            \node (C3) at (4,-1) {$\cC_3$};
            \node (C4) at (6,-1) {$\cdots$};
            \node at (6,-.7) {$\cdots$};
            \node at ($(C0)!.3!(C1)$) [above right]{\footnotesize $\mu_{5}=\eta_2^{(0)}$};
            \node at ($(C1)!.5!(C2)$) [above]{\footnotesize $\mu_{1}=\eta_4^{(2)}$};
            \node at ($(C2)!.5!(C3)$) [above]{\footnotesize $\mu_{5}=\eta_6^{(4)}$};
            \node at ($(C3)!.5!(C4)$) [above]{\footnotesize $\mu_{1}=\eta_8^{(6)}$};
            \draw[->] (C0) edge (C1) (C1) edge (C2) (C2) edge (C3) (C3) edge (C4);
        \end{scope}
        \begin{scope}
            \node at (0,0)
            {
                \ \tikz[baseline=-.6ex, scale=.05]{
                    \foreach \i in {1,2,...,6}
                    \foreach \j in {1,2,...,6}
                    {
                        \coordinate (P\i\j) at (5*\j,-5*\i);
                    }
                    \coordinate (C) at ($(P11)!.5!(P66)$);
                    \draw[wline] (P52) -- (P11);
                    \draw[webline] (P52) -- (P61);
                    \draw[webline] (P52) -- (P66);
                    \draw[thick] (P11) -- (P61) -- (P66) -- (P16) -- cycle;
                    \foreach \i in {1,6}
                    \foreach \j in {1,6}
                    {
                        \draw[fill] (P\i\j) circle (30pt);
                    }
                }\ 
            };
            \node at (0,1.5)
            {
                \ \tikz[baseline=-.6ex, scale=.05]{
                    \foreach \i in {1,2,...,6}
                    \foreach \j in {1,2,...,6}
                    {
                        \coordinate (P\i\j) at (5*\j,-5*\i);
                    }
                    \coordinate (C) at ($(P11)!.5!(P66)$);
                    \draw[wline] (P11) -- (P32);
                    \draw[wline] (P66) -- (P54);
                    \draw[webline] (P61) -- (P32) -- (P54) -- (P61);
                    \draw[thick] (P11) -- (P61) -- (P66) -- (P16) -- cycle;
                    \foreach \i in {1,6}
                    \foreach \j in {1,6}
                    {
                        \draw[fill] (P\i\j) circle (30pt);
                    }
                }\ 
            };
            \node at (1.5,0)
            {
                \ \tikz[baseline=-.6ex, scale=.05]{
                    \foreach \i in {1,2,...,6}
                    \foreach \j in {1,2,...,6}
                    {
                        \coordinate (P\i\j) at (5*\j,-5*\i);
                    }
                    \coordinate (C) at ($(P11)!.5!(P66)$);
                    \draw[webline] (P11) -- (P66);
                    \draw[thick] (P11) -- (P61) -- (P66) -- (P16) -- cycle;
                    \foreach \i in {1,6}
                    \foreach \j in {1,6}
                    {
                        \draw[fill] (P\i\j) circle (30pt);
                    }
                }\ 
            };
            \node at (1.5,-1.2)
            {
                \ \tikz[baseline=.5em, scale=.1]{
                    \draw[thick] (0,0) rectangle (6,6);
                    \draw[blue] (0,6) -- (6,0);
                    \node at (1.5,2.5) {\scriptsize ${+}$};
                    \node at (4,4) {\scriptsize ${-}$};
                    \node at (3.5,1) {\scriptsize $\ast$};
                    \node at (5,2) {\scriptsize $\ast$};
                    \fill (0,0) circle (15pt);
                    \fill (6,0) circle (15pt);
                    \fill (6,6) circle (15pt);
                    \fill (0,6) circle (15pt);
                }\ 
            };
            \node at (1.5,1.5)
            {
                \ \tikz[baseline=-.6ex, scale=.05]{
                    \foreach \i in {1,2,...,6}
                    \foreach \j in {1,2,...,6}
                    {
                        \coordinate (P\i\j) at (5*\j,-5*\i);
                    }
                    \coordinate (C) at ($(P11)!.5!(P66)$);
                    \draw[wline] (P11) -- (P66);
                    \draw[thick] (P11) -- (P61) -- (P66) -- (P16) -- cycle;
                    \foreach \i in {1,6}
                    \foreach \j in {1,6}
                    {
                        \draw[fill] (P\i\j) circle (30pt);
                    }
                }\ 
            };
            \node at (3,0)
            {
                \ \tikz[baseline=-.6ex, scale=.05]{
                    \foreach \i in {1,2,...,6}
                    \foreach \j in {1,2,...,6}
                    {
                        \coordinate (P\i\j) at (5*\j,-5*\i);
                    }
                    \coordinate (C) at ($(P11)!.5!(P66)$);
                    \node at (C) [gray]{\scriptsize $x_{2}^{(0)}$};
                    \draw[wline] (P11) -- (P25);
                    \draw[webline] (P66) -- (P25) -- (P16);
                    \draw[thick] (P11) -- (P61) -- (P66) -- (P16) -- cycle;
                    \foreach \i in {1,6}
                    \foreach \j in {1,6}
                    {
                        \draw[fill] (P\i\j) circle (30pt);
                    }
                }\ 
            };
            \node at (3,1.5)
            {
                \ \tikz[baseline=-.6ex, scale=.05]{
                    \foreach \i in {1,2,...,6}
                    \foreach \j in {1,2,...,6}
                    {
                        \coordinate (P\i\j) at (5*\j,-5*\i);
                    }
                    \coordinate (C) at ($(P11)!.5!(P66)$);
                    \draw[wline] (P11) -- (P23);
                    \draw[wline] (P66) -- (P45);
                    \draw[webline] (P16) -- (P23) -- (P45) -- (P16);
                    \draw[thick] (P11) -- (P61) -- (P66) -- (P16) -- cycle;
                    \foreach \i in {1,6}
                    \foreach \j in {1,6}
                    {
                        \draw[fill] (P\i\j) circle (30pt);
                    }
                }\ 
            };
        \end{scope}
        \begin{scope}[xshift=10cm, yshift=-2.5cm]
            \node at (0,0)
            {
                \ \tikz[baseline=-.6ex, scale=.05]{
                    \foreach \i in {1,2,...,6}
                    \foreach \j in {1,2,...,6}
                    {
                        \coordinate (P\i\j) at (5*\j,-5*\i);
                    }
                    \coordinate (C) at ($(P11)!.5!(P66)$);
                    \draw[wline] (P52) -- (P11);
                    \draw[webline] (P52) -- (P61);
                    \draw[webline] (P52) -- (P66);
                    \draw[thick] (P11) -- (P61) -- (P66) -- (P16) -- cycle;
                    \foreach \i in {1,6}
                    \foreach \j in {1,6}
                    {
                        \draw[fill] (P\i\j) circle (30pt);
                    }
                }\ 
            };
            \node at (0,1.5)
            {
                \ \tikz[baseline=-.6ex, scale=.05]{
                    \foreach \i in {1,2,...,6}
                    \foreach \j in {1,2,...,6}
                    {
                        \coordinate (P\i\j) at (5*\j,-5*\i);
                    }
                    \coordinate (C) at ($(P11)!.5!(P66)$);
                    \draw[wline] (P11) -- (P32);
                    \draw[wline] (P66) -- (P54);
                    \draw[webline] (P61) -- (P32) -- (P54) -- (P61);
                    \draw[thick] (P11) -- (P61) -- (P66) -- (P16) -- cycle;
                    \foreach \i in {1,6}
                    \foreach \j in {1,6}
                    {
                        \draw[fill] (P\i\j) circle (30pt);
                    }
                }\ 
            };
            \node at (1.5,0)
            {
                \ \tikz[baseline=-.6ex, scale=.05]{
                    \foreach \i in {1,2,...,6}
                    \foreach \j in {1,2,...,6}
                    {
                        \coordinate (P\i\j) at (5*\j,-5*\i);
                    }
                    \coordinate (C) at ($(P11)!.5!(P66)$);
                    \draw[webline] (P11) -- (P66);
                    \draw[thick] (P11) -- (P61) -- (P66) -- (P16) -- cycle;
                    \foreach \i in {1,6}
                    \foreach \j in {1,6}
                    {
                        \draw[fill] (P\i\j) circle (30pt);
                    }
                }\ 
            };
            \node at (1.5,-1.2){$\cC$};
            \node at (1.5,1.5)
            {
                \ \tikz[baseline=-.6ex, scale=.05]{
                    \foreach \i in {1,2,...,6}
                    \foreach \j in {1,2,...,6}
                    {
                        \coordinate (P\i\j) at (5*\j,-5*\i);
                    }
                    \coordinate (C) at ($(P11)!.5!(P66)$);
                    \node at (C) [gray]{\scriptsize $y_{0}^{(3)}$};
                    \draw[wline] (P11) -- (P42);
                    \draw[wline] (P53) -- (P54);
                    \draw[wline] (P45) -- (P35);
                    \draw[wline] (P24) -- (P11);
                    \draw[webline] (P61) -- (P42) -- (P53) -- (P61);
                    \draw[webline] (P66) -- (P54) -- (P45) -- (P66);
                    \draw[webline] (P16) -- (P35) -- (P24) -- (P16);
                    \draw[thick] (P11) -- (P61) -- (P66) -- (P16) -- cycle;
                    \foreach \i in {1,6}
                    \foreach \j in {1,6}
                    {
                        \draw[fill] (P\i\j) circle (30pt);
                    }
                }\ 
            };
            \node at (3,0)
            {
                \ \tikz[baseline=-.6ex, scale=.05]{
                    \foreach \i in {1,2,...,6}
                    \foreach \j in {1,2,...,6}
                    {
                        \coordinate (P\i\j) at (5*\j,-5*\i);
                    }
                    \coordinate (C) at ($(P11)!.5!(P66)$);
                    \node at (C) [gray]{\scriptsize $x_{2}^{(0)}$};
                    \draw[wline] (P11) -- (P25);
                    \draw[webline] (P66) -- (P25) -- (P16);
                    \draw[thick] (P11) -- (P61) -- (P66) -- (P16) -- cycle;
                    \foreach \i in {1,6}
                    \foreach \j in {1,6}
                    {
                        \draw[fill] (P\i\j) circle (30pt);
                    }
                }\ 
            };
            \node at (3,1.5)
            {
                \ \tikz[baseline=-.6ex, scale=.05]{
                    \foreach \i in {1,2,...,6}
                    \foreach \j in {1,2,...,6}
                    {
                        \coordinate (P\i\j) at (5*\j,-5*\i);
                    }
                    \coordinate (C) at ($(P11)!.5!(P66)$);
                    \draw[wline] (P11) -- (P23);
                    \draw[wline] (P66) -- (P45);
                    \draw[webline] (P16) -- (P23) -- (P45) -- (P16);
                    \draw[thick] (P11) -- (P61) -- (P66) -- (P16) -- cycle;
                    \foreach \i in {1,6}
                    \foreach \j in {1,6}
                    {
                        \draw[fill] (P\i\j) circle (30pt);
                    }
                }\ 
            };
        \end{scope}
        \begin{scope}[xshift=0cm, yshift=-5cm]
            \node at (0,0)
            {
                \ \tikz[baseline=-.6ex, scale=.05]{
                    \foreach \i in {1,2,...,6}
                    \foreach \j in {1,2,...,6}
                    {
                        \coordinate (P\i\j) at (5*\j,-5*\i);
                    }
                    \coordinate (C) at ($(P11)!.5!(P66)$);
                    \node at (C) [gray]{\scriptsize $x_{0}^{(2)}$};
                    \draw[wline] (P53) -- (P54);
                    \draw[wline] (P45) -- (P35);
                    \draw[wline] (P24) -- (P11);
                    \draw[webline] (P61) -- (P53) -- (P11);
                    \draw[webline] (P66) -- (P54) -- (P45) -- (P66);
                    \draw[webline] (P16) -- (P35) -- (P24) -- (P16);
                    \draw[thick] (P11) -- (P61) -- (P66) -- (P16) -- cycle;
                    \foreach \i in {1,6}
                    \foreach \j in {1,6}
                    {
                        \draw[fill] (P\i\j) circle (30pt);
                    }
                }\ 
            };
            \node at (0,1.5)
            {
                \ \tikz[baseline=-.6ex, scale=.05]{
                    \foreach \i in {1,2,...,6}
                    \foreach \j in {1,2,...,6}
                    {
                        \coordinate (P\i\j) at (5*\j,-5*\i);
                    }
                    \coordinate (C) at ($(P11)!.5!(P66)$);
                    \draw[wline] (P11) -- (P32);
                    \draw[wline] (P66) -- (P54);
                    \draw[webline] (P61) -- (P32) -- (P54) -- (P61);
                    \draw[thick] (P11) -- (P61) -- (P66) -- (P16) -- cycle;
                    \foreach \i in {1,6}
                    \foreach \j in {1,6}
                    {
                        \draw[fill] (P\i\j) circle (30pt);
                    }
                }\ 
            };
            \node at (1.5,0)
            {
                \ \tikz[baseline=-.6ex, scale=.05]{
                    \foreach \i in {1,2,...,6}
                    \foreach \j in {1,2,...,6}
                    {
                        \coordinate (P\i\j) at (5*\j,-5*\i);
                    }
                    \coordinate (C) at ($(P11)!.5!(P66)$);
                    \draw[webline] (P11) -- (P66);
                    \draw[thick] (P11) -- (P61) -- (P66) -- (P16) -- cycle;
                    \foreach \i in {1,6}
                    \foreach \j in {1,6}
                    {
                        \draw[fill] (P\i\j) circle (30pt);
                    }
                }\ 
            };
            \node at (1.5,-1.2){$\cC_{0}$};
            \node at (1.5,1.5)
            {
                \ \tikz[baseline=-.6ex, scale=.05]{
                    \foreach \i in {1,2,...,6}
                    \foreach \j in {1,2,...,6}
                    {
                        \coordinate (P\i\j) at (5*\j,-5*\i);
                    }
                    \coordinate (C) at ($(P11)!.5!(P66)$);
                    \node at (C) [gray]{\scriptsize $y_{0}^{(3)}$};
                    \draw[wline] (P11) -- (P42);
                    \draw[wline] (P53) -- (P54);
                    \draw[wline] (P45) -- (P35);
                    \draw[wline] (P24) -- (P11);
                    \draw[webline] (P61) -- (P42) -- (P53) -- (P61);
                    \draw[webline] (P66) -- (P54) -- (P45) -- (P66);
                    \draw[webline] (P16) -- (P35) -- (P24) -- (P16);
                    \draw[thick] (P11) -- (P61) -- (P66) -- (P16) -- cycle;
                    \foreach \i in {1,6}
                    \foreach \j in {1,6}
                    {
                        \draw[fill] (P\i\j) circle (30pt);
                    }
                }\ 
            };
            \node at (3,0)
            {
                \ \tikz[baseline=-.6ex, scale=.05]{
                    \foreach \i in {1,2,...,6}
                    \foreach \j in {1,2,...,6}
                    {
                        \coordinate (P\i\j) at (5*\j,-5*\i);
                    }
                    \coordinate (C) at ($(P11)!.5!(P66)$);
                    \node at (C) [gray]{\scriptsize $x_{2}^{(0)}$};
                    \draw[wline] (P11) -- (P25);
                    \draw[webline] (P66) -- (P25) -- (P16);
                    \draw[thick] (P11) -- (P61) -- (P66) -- (P16) -- cycle;
                    \foreach \i in {1,6}
                    \foreach \j in {1,6}
                    {
                        \draw[fill] (P\i\j) circle (30pt);
                    }
                }\ 
            };
            \node at (3,1.5)
            {
                \ \tikz[baseline=-.6ex, scale=.05]{
                    \foreach \i in {1,2,...,6}
                    \foreach \j in {1,2,...,6}
                    {
                        \coordinate (P\i\j) at (5*\j,-5*\i);
                    }
                    \coordinate (C) at ($(P11)!.5!(P66)$);
                    \draw[wline] (P11) -- (P23);
                    \draw[wline] (P66) -- (P45);
                    \draw[webline] (P16) -- (P23) -- (P45) -- (P16);
                    \draw[thick] (P11) -- (P61) -- (P66) -- (P16) -- cycle;
                    \foreach \i in {1,6}
                    \foreach \j in {1,6}
                    {
                        \draw[fill] (P\i\j) circle (30pt);
                    }
                }\ 
            };
        \end{scope}
    \end{tikzpicture}

One can verify that the relations $\eta_i^{(k)}$ are indeed quantum exchange relations. The quiver mutation for the first step $\mu_4$ is given by
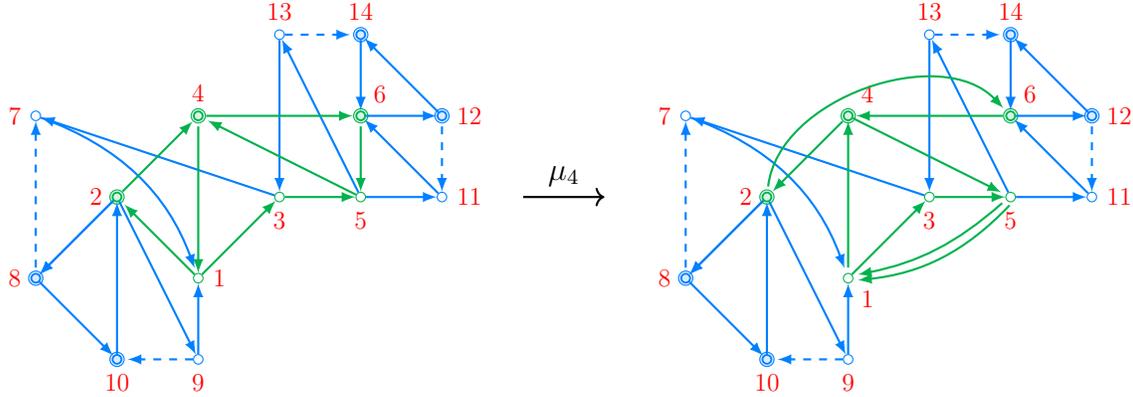
\begin{align*}
\begin{tikzpicture}[scale=0.9]
[>=latex]
\pgfmathsetmacro{\r}{1.2};
{\color{mygreen}
\draw(\r,\r) circle(2pt) coordinate(x1) node[red,scale=0.8,right=0.2em]{$1$};
\draw(2*\r,2*\r) circle(2pt) coordinate(x3) node[red,scale=0.8,below=0.2em]{$3$};
\draw(3*\r,2*\r) circle(2pt) coordinate(x5) node[red,scale=0.8,below=0.2em]{$5$};
\dnode{0,2*\r}{mygreen} circle(2pt) coordinate(x2) node[red,scale=0.8,left=0.2em]{$2$};
\dnode{\r,3*\r}{mygreen} circle(2pt) coordinate(x4) node[red,scale=0.8,above=0.2em]{$4$};
\dnode{3*\r,3*\r}{mygreen} circle(2pt) coordinate(x6) node[red,scale=0.8,above right=0.2em]{$6$};
}
{\color{myblue}
\draw(-\r,3*\r) circle(2pt) coordinate(x7) node[red,scale=0.8,left=0.2em]{$7$};
\draw(\r,0) circle(2pt) coordinate(x9) node[red,scale=0.8,below=0.2em]{$9$};
\draw(4*\r,2*\r) circle(2pt) coordinate(x11) node[red,scale=0.8,right=0.2em]{$11$};
\draw(2*\r,4*\r) circle(2pt) coordinate(x13) node[red,scale=0.8,above=0.2em]{$13$};
\dnode{-\r,\r}{myblue} circle(2pt) coordinate(x8) node[red,scale=0.8,left=0.2em]{$8$};
\dnode{0,0}{myblue} circle(2pt) coordinate(x10) node[red,scale=0.8,below=0.2em]{$10$};
\dnode{4*\r,3*\r}{myblue} circle(2pt) coordinate(x12) node[red,scale=0.8,right=0.2em]{$12$};
\dnode{3*\r,4*\r}{myblue} circle(2pt) coordinate(x14) node[red,scale=0.8,above=0.2em]{$14$};
}
{\color{myblue}
\qarrow{x9}{x1}
\qarrow{x2}{x9}
\qarrow{x10}{x2}
\qarrow{x2}{x8}
\qarrow{x8}{x10}
\uniarrow{x7}{x1}{shorten <=2pt,shorten >=4pt,bend left=25}
\qarrow{x3}{x7}
\qarrow{x2}{x8}
\qarrow{x5}{x13}
\qarrow{x13}{x3}
\qarrow{x12}{x14}
\qarrow{x14}{x6}
\qarrow{x5}{x11}
\qarrow{x11}{x6}
\qarrow{x6}{x12}
\uniarrow{x9}{x10}{dashed,shorten <=2pt,shorten >=4pt}
\uniarrow{x8}{x7}{dashed,shorten <=4pt,shorten >=2pt}
\uniarrow{x12}{x11}{dashed,shorten <=4pt,shorten >=2pt}
\uniarrow{x13}{x14}{dashed,shorten <=2pt,shorten >=4pt}
}
{\color{mygreen}
\qarrow{x1}{x3}
\qarrow{x3}{x5}
\qsarrow{x2}{x4}
\qsarrow{x4}{x6}
\qstarrow{x6}{x5}
\qsharrow{x5}{x4}
\qstarrow{x4}{x1}
\qsharrow{x1}{x2}
}
\draw[->,thick] (5*\r,2*\r) --node[midway,above]{$\mu_4$} (6*\r,2*\r);
\begin{scope}[xshift=8*\r cm]
{\color{mygreen}
\node[inner sep=1.5pt] (c1) at (\r,\r) {};
\node[inner sep=1.5pt] (c5) at (3*\r,2*\r) {};
\draw(\r,\r) circle(2pt) coordinate(x1) node[red,scale=0.8,below right=0.2em]{$1$};
\draw(2*\r,2*\r) circle(2pt) coordinate(x3) node[red,scale=0.8,below=0.2em]{$3$};
\draw(3*\r,2*\r) circle(2pt) coordinate(x5) node[red,scale=0.8,below=0.2em]{$5$};
\dnode{0,2*\r}{mygreen} circle(2pt) coordinate(x2) node[red,scale=0.8,left=0.2em]{$2$};
\dnode{\r,3*\r}{mygreen} circle(2pt) coordinate(x4) node[red,scale=0.8,above right=0.2em]{$4$};
\dnode{3*\r,3*\r}{mygreen} circle(2pt) coordinate(x6) node[red,scale=0.8,above right=0.2em]{$6$};
}
{\color{myblue}
\draw(-\r,3*\r) circle(2pt) coordinate(x7) node[red,scale=0.8,left=0.2em]{$7$};
\draw(\r,0) circle(2pt) coordinate(x9) node[red,scale=0.8,below=0.2em]{$9$};
\draw(4*\r,2*\r) circle(2pt) coordinate(x11) node[red,scale=0.8,right=0.2em]{$11$};
\draw(2*\r,4*\r) circle(2pt) coordinate(x13) node[red,scale=0.8,above=0.2em]{$13$};
\dnode{-\r,\r}{myblue} circle(2pt) coordinate(x8) node[red,scale=0.8,left=0.2em]{$8$};
\dnode{0,0}{myblue} circle(2pt) coordinate(x10) node[red,scale=0.8,below=0.2em]{$10$};
\dnode{4*\r,3*\r}{myblue} circle(2pt) coordinate(x12) node[red,scale=0.8,right=0.2em]{$12$};
\dnode{3*\r,4*\r}{myblue} circle(2pt) coordinate(x14) node[red,scale=0.8,above=0.2em]{$14$};
}
{\color{myblue}
\qarrow{x9}{x1}
\qarrow{x2}{x9}
\qarrow{x10}{x2}
\qarrow{x2}{x8}
\qarrow{x8}{x10}
\uniarrow{x7}{x1}{shorten <=2pt,shorten >=4pt,bend left=25}
\qarrow{x3}{x7}
\qarrow{x2}{x8}
\qarrow{x5}{x13}
\qarrow{x13}{x3}
\qarrow{x12}{x14}
\qarrow{x14}{x6}
\qarrow{x5}{x11}
\qarrow{x11}{x6}
\qarrow{x6}{x12}
\uniarrow{x9}{x10}{dashed,shorten <=2pt,shorten >=4pt}
\uniarrow{x8}{x7}{dashed,shorten <=4pt,shorten >=2pt}
\uniarrow{x12}{x11}{dashed,shorten <=4pt,shorten >=2pt}
\uniarrow{x13}{x14}{dashed,shorten <=2pt,shorten >=4pt}
}
{\color{mygreen}
\qarrow{x1}{x3}
\qarrow{x3}{x5}
\qsarrow{x4}{x2}
\qsarrow{x6}{x4}
\qsharrow{x4}{x5}
\qstarrow{x1}{x4}
\uniarrow{x2}{x6}{shorten <=4pt,shorten >=4pt,bend left=65}
\uniarrow{c5.210}{c1.30}{shorten <=2pt,shorten >=2pt,bend left=15}
\uniarrow{c5.300}{c1.-30}{shorten <=2pt,shorten >=2pt,bend left=20}
}
\end{scope}
\end{tikzpicture}
\end{align*}
Then we find the Kronecker sub-quiver $\tikz{
\node[inner sep=1.5pt] (c1) at (0,0) {};
\node[inner sep=1.5pt] (c2) at (1,0) {};
\draw(0,0) circle(2pt) node[red,above=0.2em,scale=0.8]{$5$};
\draw(1,0) circle(2pt) node[red,above=0.2em,scale=0.8]{$1$};
\qarrow{c1.40}{c2.140}
\qarrow{c1.-40}{c2.-140}}$. It is well-known that the mutation sequences $(\mu_5\mu_1)^n$ for $n\geq 0$ on such a Kronecker quiver are non-trivial, and produce distinct quantum clusters. 
In this way, we obtain an infinite sequence of (web) clusters. Observe that $x_i^{(k)}$'s are all tree-type elementary webs which are not invariant under $\mathrm{DT}$.

\begin{figure}[ht]
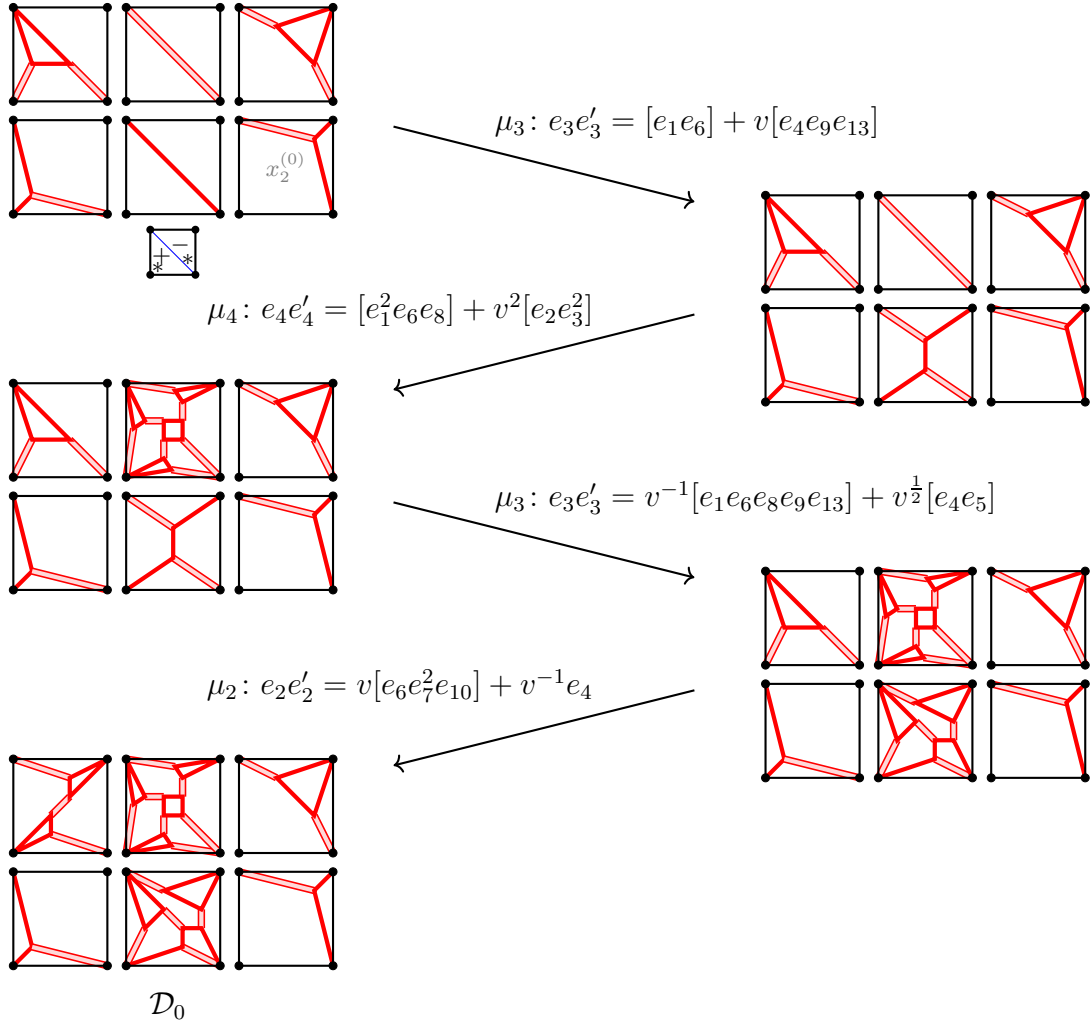


    \caption{Web clusters along a weight $2$ sequence.}
    \label{fig:infinite_weight2}
\end{figure}

\begin{rem} 
For a Kronecker quiver over vertices $\{i,j\}$ with a common weight and $2$ arrows between them, it is known that the sequence of $g$-vectors along the mutation sequence $(\mu_i\mu_j)^n$ has linear growth, and its recurrence relation stabilizes to a linear one (see, for instance, \cite[Proposition 6.6]{IK-2}). Hence the sequence $\mathrm{deg}_m \in X^*(H_\A)$ of ensemble degrees of the cluster variables appearing alternately on the two vertices along this mutation sequence, which are projections of $g$-vectors, also has linear growth. Let us define $\mathrm{deg}_\infty:=\lim_{m \to \infty}\mathrm{deg}_m/m$, which we call the \emph{asymptotic ensemble degree} of this mutation sequence. 

For the above mutation sequence $(\mu_5\mu_1)^n$ of (web) clusters, the elementary webs/cluster variables have weight $1$ and the asymptotic ensemble degree can be easily seen to be 
\begin{align}\label{eq:degree_weight1}
    \mathrm{deg}_\infty=(2\varpi_1,2\varpi_1,2\varpi_1,2\varpi_1)\in \mathsf{P}_+^{\oplus 4}.
\end{align}
In a similar analysis, we can find a Kronecker sub-quiver over the weight $2$ vertices $\{4,6\}$, and hence an infinite sequence of (web) clusters of weight $2$ starting from the cluster $
    \cD_0:=\mu_2\mu_3\mu_4\mu_3\bigg(\mathord{
        \ \tikz[baseline=.5em, scale=.1]{
            \draw[thick] (0,0) rectangle (6,6);
            \draw[blue] (0,6) -- (6,0);
            \node at (1.5,2.5) {\scriptsize ${+}$};
            \node at (4,4) {\scriptsize ${-}$};
            \node at (1,1) {\scriptsize $\ast$};
            \node at (5,2) {\scriptsize $\ast$};
            \fill (0,0) circle (15pt);
            \fill (6,0) circle (15pt);
            \fill (6,6) circle (15pt);
            \fill (0,6) circle (15pt);
        }\ 
    }\bigg).
$ Its asymptotic degree is
\begin{align}\label{eq:degree_weight2}
    \mathrm{deg}_\infty=(2\varpi_1+\varpi_2,2\varpi_1+\varpi_2,2\varpi_1+\varpi_2,2\varpi_1+\varpi_2) \in \mathsf{P}_+^{\oplus 4}.
\end{align}
The corresponding web clusters until $\cD_0$ has been confirmed, which are shown in \cref{fig:infinite_weight2}. 
The $\mathfrak{sp}_4$-webs appearing here are all tree-type elementary webs, where
$
\ 
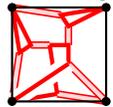
\tikz[baseline=-.6ex, scale=.05, yshift=18cm]{
    \foreach \i in {1,2,...,6}
    \foreach \j in {1,2,...,6}
    {
        \coordinate (P\i\j) at (5*\j,-5*\i);
    }
    \coordinate (C) at ($(P11)!.5!(P66)$);
    \draw[wline] (P61) -- ($(P31)+(2,-2)$);
    \draw[wline] (P66) -- ($(P63)+(2,2)$);
    \draw[wline] (P11) -- ($(P14)+(-2,-2)$);
    \draw[wline] (P44) -- (P66);
    \draw[wline] (P32) -- (P33);
    \draw[wline] (P24) -- (P34);
    \draw[webline]  (P44) -- (P34) -- (P33);
    \draw[webline] (P61) -- ($(P63)+(2,2)$);
    \draw[webline] (P33) -- ($(P63)+(2,2)$);
    \draw[overarc] (P44) -- (P61);
    \draw[webline] (P11) -- ($(P31)+(2,-2)$) -- (P32) -- (P11);
    \draw[webline] (P16) -- ($(P14)+(-2,-2)$) -- (P24) -- (P16);
    \draw[thick] (P11) -- (P61) -- (P66) -- (P16) -- cycle;
    \foreach \i in {1,6}
    \foreach \j in {1,6}
    {
        \draw[fill] (P\i\j) circle (30pt);
    }
}\
$
gives a lift of $e_4$ in the web cluster $\cD_0$.
We believe that the $\mathfrak{sp}_4$-webs $\{z_2^{(3)},z_2^{(7)},z_2^{(11)},\ldots\}$ give the  corresponding sequence of web clusters of weight $2$ from $\cD_0$, 
where
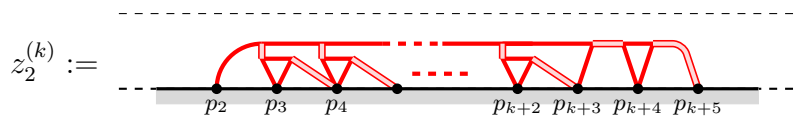
\begin{align*}
    z_{2}^{(k)}&:=
    \mathord{
        \ \tikz[baseline=.5em, scale=.1]{
            \foreach \i in {0,1,...,10}
            {
                \coordinate (A\i) at (8*\i,0);
                \coordinate (B\i) at (8*\i,10);
            }
            \draw[webline] (A2) -- ($(A2)+(-2,4)$) -- ($(A2)+(2,4)$) -- (A2);
            \draw[webline] (A3) -- ($(A3)+(-2,4)$) -- ($(A3)+(2,4)$) -- (A3);
            \draw[webline] (A6) -- ($(A6)+(-2,4)$) -- ($(A6)+(2,4)$) -- (A6);
            \draw[webline] (A8) -- ($(A8)+(-2,6)$) -- ($(A8)+(2,6)$) -- (A8);
            \draw[webline] (A1) to[out=north, in=west] ($(A2)+(-2,6)$);
            \draw[webline] ($(A2)+(-2,6)$) -- ($(A3)+(-2,6)$);
            \draw[webline] ($(A3)+(-2,6)$) -- ($(A4)+(-2,6)$);
            \draw[webline] ($(A5)+(-2,6)$) -- ($(A6)+(2,6)$);
            \draw[webline] ($(A6)+(-2,6)$) -- ($(A7)+(2,6)$);
            \draw[wline] (A3) -- ($(A2)+(2,4)$);
            \draw[wline] (A4) -- ($(A3)+(2,4)$);
            \draw[wline] (A7) -- ($(A6)+(2,4)$);
            \draw[wline] ($(A2)+(-2,4)$) -- ($(A2)+(-2,6)$);
            \draw[wline] ($(A3)+(-2,4)$) -- ($(A3)+(-2,6)$);
            \draw[wline] ($(A6)+(-2,4)$) -- ($(A6)+(-2,6)$);
            \draw[webline, dashed] ($(A4)+(-2,6)$) -- ($(A5)+(-2,6)$);
            \draw[webline, dashed] ($(A4)+(2,2)$) -- ($(A5)+(2,2)$);
            \draw[webline] (A7) -- ($(A7)+(2,6)$);
            \draw[wline]  ($(A7)+(2,6)$) -- ($(A8)+(-2,6)$);
            \draw[wline, rounded corners] (A9) -- ($(A9)+(-2,6)$) -- ($(A8)+(2,6)$);
            \draw[dashed] ($(B0)+(-5,0)$) -- ($(B10)+(5,0)$);
            \draw[dashed, thick] ($(A0)+(-5,0)$) -- ($(A10)+(5,0)$);
            \bdryline{(A0)}{(A10)}{2cm}
            \foreach \j in {1,2,3,4,6,7,8,9}
            {
                \fill (A\j) circle (20pt);
            }
            \node at (A1) [below]{\scriptsize $p_2$};
            \node at (A2) [below]{\scriptsize $p_{3}$};
            \node at (A3) [below]{\scriptsize $p_{4}$};
            \node at (A6) [below]{\scriptsize $p_{k+2}$};
            \node at (A7) [below]{\scriptsize $p_{k+3}$};
            \node at (A8) [below]{\scriptsize $p_{k+4}$};
            \node at (A9) [below]{\scriptsize $p_{k+5}$};
        }\ 
    }.
\end{align*}
We have verified the first two quantum exchange relations $e_4z_2^{(3)}=q^{\bullet}[e_{6}^2e_{7}e_{9}^2e_{11}^2e_{12}e_{13}e_{14}]+q^{\bullet}[e_2e_3^{2}]$ and $e_6z_2^{(7)}=q^{\bullet}[e_2e_3^2e_{11}^2e_{14}]+q^{\bullet}(z_2^{(3)})^2$, where the non-primed variables are elementary webs in $\cD_0$.

It is an interesting observation that the asymptotic ensemble degrees \eqref{eq:degree_weight1}, \eqref{eq:degree_weight2} are invariant under rotations (\emph{i.e.}, the cluster Donaldson--Thomas transformation). The authors do not know if a general theory exists behind this phenomenon.
\end{rem}


\end{document}